\documentclass[12pt]{article}
\usepackage{amsmath,amssymb,amsfonts,amsthm}
\usepackage[all]{xy}

\newcounter{item}[section]
\newcounter{kirshr}
\newcounter{kirsha}
\newcounter{kirshb}
\newenvironment{enumroman}{\setcounter{kirshr}{1}
\begin{list}{(\roman{kirshr})}{\usecounter{kirshr}} }{\end{list}}
\newenvironment{enumarab}{\setcounter{kirshb}{1}
\begin{list}{(\arabic{kirshb})}{\usecounter{kirshb}} }{\end{list}}
\newenvironment{athm}[1]{\vskip3mm\par\noindent
{\bf #1 }. \slshape }
{\upshape\par\vskip10pt minus3pt}
\newtheorem{theorem}{Theorem}[section]

\newtheorem{lemma}[theorem]{Lemma}
\newtheorem{corollary}[theorem]{Corollary}
\newenvironment{demo}[1]{\noindent{\bf #1.}\upshape\mdseries}
{\nopagebreak{\hfill\rule{2mm}{2mm}\nopagebreak}\par\normalfont}
\theoremstyle{definition}

\newtheorem{example}[theorem]{Example}
\newtheorem{definition}[theorem]{Definition}

\def\R{\mathbb{R}}

\def\C{{\mathfrak{C}}}
\def\Fm{{\mathfrak{Fm}}}

\def\At{{\bf At}}
\def\Nr{{\mathfrak{Nr}}}
\def\Fr{{\mathfrak{Fr}}}
\def\Sg{{\mathfrak{Sg}}}
\def\Fm{{\mathfrak{Fm}}}
\def\A{{\mathfrak{A}}}
\def\B{{\mathfrak{B}}}
\def\C{{\mathfrak{C}}}
\def\D{{\mathfrak{D}}}
\def\M{{\mathfrak{M}}}
\def\N{{\mathfrak{N}}}

\def\CA{{\bf CA}}
\def\QA{{\bf QA}}

\def\QEA{{\bf QEA}}

\def\Df{{\bf Df}}

\def\Lf{{\bf Lf}}

\def\PA{{\bf PA}}
\def\PEA{{\bf PEA}}

\def\K{{\bf K}}
\def\K{{\bf K}}

\def\RCA{{\bf RCA}}

\def\Rd{{\ Rd}}
\def\(R)RA{{\bf (R)RA}}
\def\RA{{\bf RA}}

\def\R{\mathbb{R}}

\def\Sc{{\bf Sc}}

\def\tr{{\sf tr}}
\def\c #1{{\cal #1}}
 \def\CA{{\sf CA}}
\def\B{{\sf B}}
\def\G{{\sf G}}
\def\w{{\sf w}}
\def\y{{\sf y}}
\def\g{{\sf g}}

\def\r{{\sf r}}
\def\K{{\sf K}}

 \def\Cm{{\mathfrak{Cm}}}
\def\Nr{{\mathfrak{Nr}}}

\def\restr #1{{\restriction_{#1}}}
\def\cyl#1{{\sf c}_{#1}}
\def\diag#1#2{{\sf d}_{#1#2}}
\def\sub#1#2{{\sf s}^{#1}_{#2}}

\def\Ra{{\mathfrak{Ra}}}
\def\Ca{{\mathfrak{Ca}}}
\def\set#1{\{#1\} }
\def\Ra{{\mathfrak{Ra}}}
\def\Nr{{\mathfrak{Nr}}}
\def\Tm{{\mathfrak{Tm}}}
\def\A{{\mathfrak{A}}}
\def\B{{\mathfrak{B}}}
\def\C{{\mathfrak{C}}}
\def\D{{\mathfrak{D}}}

\def\A{{\mathfrak{A}}}
\def\B{{\mathfrak{B}}}
\def\C{{\mathfrak{C}}}
\def\D{{\mathfrak{D}}}

\def\L{{\mathfrak{L}}}
\def\Rd{{\mathfrak{Rd}}}

\def\At{{\mathfrak{At}}}
\def\L{{\mathfrak{L}}}
\def\Bl{{\mathfrak{Bl}}}
\def\CA{{\bf CA}}
\def\RA{{\bf RA}}

\def\RCA{{\bf RCA}}

\def\G{{\bf G}}

\def\F{{\mathfrak{F}}}
\def\At{{\sf{At}}}
\def\N{\mathbb{N}}
\def\R{\mathfrak{R}}

\def\L{{\mathfrak L}}

\def\sub#1#2{{\sf s}^{#1}_{#2}}
\def\cyl#1{{\sf c}_{#1}}
\def\diag#1#2{{\sf d}_{#1#2}}

\def\c #1{{\cal #1}}

\def\pa{$\forall$}
\def\pe{$\exists$}

\def\ef{Ehren\-feucht--Fra\"\i ss\'e}

\def\nodes{{\sf nodes}}

\def\restr #1{{\restriction_{#1}}}

\def\Ra{{\mathfrak{Ra}}}
\def\Nr{{\mathfrak{Nr}}}

\def\CA{{\bf CA}}
\def\RCA{{\bf RCA}}
\def\c#1{{\mathcal #1}}

\def\set#1{ \{#1\}}

\def\Ca{{\mathfrak Ca}}

\def\pe{$\exists$}
\def\pa{$\forall$}
\def\Cm{{\mathfrak Cm}}
\def\Sg{{\mathfrak Sg}}
\def\P{{\mathfrak P}}
\def\Rl{{\mathfrak Rl}}
\def\N{{\cal N}}
\def\d{Dedekind-MacNeille}

\def\At{{\sf At}}
\def\Ig{{\sf Ig}}

\def\rng{{\sf rng}}
\def\dom{{\sf dom}}
\def\QRA{{\sf QRA}}

\def\w{{\sf w}}
\def\g{{\sf g}}
\def\y{{\sf y}}
\def\r{{\sf r}}
\def\RCA{\sf RCA}

\def\cyl#1{{\sf c}_{#1}}
\def\sub#1#2{{\sf s}^{#1}_{#2}}
\def\diag#1#2{{\sf d}_{#1#2}}

\def\ws{winning strategy}
\def\ef{Ehren\-feucht--Fra\"\i ss\'e}

\def\y{{\sf y}}
\def\g{{\sf g}}
\def\r{{\sf r}}
\def\w{{\sf w}}

\def\QPEA{{\sf QPEA}}
 
\def\RQEA{{\sf RQEA}}

\def\y{{\sf y}}
\def\g{{\sf g}}
\def\r{{\sf r}}
\def\w{{\sf w}}
\def\RA{{\sf RA}}
\def\CA{{\sf CA}}
\def\QRA{\sf QRA}
\def\CA{{\sf CA}}
\def\Ra{\mathfrak{Ra}}
\def\Lf{{\sf Lf}}
\def\RCA{\sf RCA}

\def\Kn{{\sf Kn}}

\def\G{\sf G}
\def\EF{\sf EF}
\def\PEA{\sf PEA}

\def\N{\mathbb N}

\def\K{{\bf K}}
\def\QPA{{\bf QPA}}
\def \QPEA{{\bf QPEA}}

\def\RQEA{{\bf RQEA}}

\def\cyl#1{{\sf c}_{#1}}
\def\diag#1#2{{\sf d}_{#1#2}}

\def\c #1{{\cal #1}}

\def\pa{$\forall$}
\def\pe{$\exists$}

\def\ef{Ehren\-feucht--Fra\"\i ss\'e}

\def\Ca{{\mathfrak{Ca}}}
\def\Rl{{\mathfrak{Rl}}}

\def\nodes{{\sf nodes}}

\def\restr #1{{\restriction_{#1}}}

\def\A{{\cal{A}}}
\def\B{{\mathfrak{B}}}
\def\C{{\mathfrak{C}}}
\def\D{{\mathfrak{D}}}
\def\P{{\mathfrak{P}}}

\def\Ra{{\mathfrak{Ra}}}
\def\Nr{{\mathfrak{Nr}}}
\def\F{{\mathfrak{F}}}
\def\CA{{\bf CA}}
\def\RCA{{\bf RCA}}
\def\c#1{{\mathcal #1}}

\def\Ca{{\mathfrak Ca}}

\def\pe{$\exists$}
\def\pa{$\forall$}
\def\Cm{{\mathfrak Cm}}
\def\Sg{{\mathfrak Sg}}
\def\CPEA{{\sf CPEA}}

\def\At{{\sf At}}

\def\rng{{\sf rng}}
\def\dom{{\sf dom}}

\def\dim{{\sf dim}}

\def\w{{\sf w}}
\def\g{{\sf g}}
\def\y{{\sf y}}
\def\r{{\sf r}}

\def\cyl#1{{\sf c}_{#1}}
\def\sub#1#2{{\sf s}_{[{#1}/ {#2}]}}
\def\diag#1#2{{\sf d}_{#1#2}}

\def\swap#1#2{{\sf s}_{[#1, #2]}}
\def\ef{Ehren\-feucht--Fra\"\i ss\'e}

\def\y{{\sf y}}
\def\g{{\sf g}}
\def\r{{\sf r}}

\def\w{{\sf w}}

\def\G{{\bold G}}
\def\Sc{{\sf Sc}}
\def\Df{{\sf Df}}
\def\PA{{\sf PA}}
\def\Id{{\sf Id}}
\def\QEA{{\sf QEA}}
\def\s{{\sf s}}
\def\QPA{{\sf QPA}}
\def\QPEA{{\sf QPEA}}
\def\CA{{\sf CA}}
\def\K{{\sf K}}
\def\RCA{{\sf RCA}}
\def\RQEA{{\sf RQEA}}
\def\A{{\mathfrak{A}}}

\def\QA{{\sf QA}}

\def\Co{{\sf Co}}
\title{Cylindric and polyadic algebras, new perspectives}
\author{Tarek Sayed Ahmed}

\begin{document}
\maketitle

\begin{abstract} We generalize the notions of Monk's system of varieties definable by a schema
to integrate finite dimensions. We give a general method
in this new framework
that allows one to lift a plethora of results on neat embeddings to the infinite dimensional case. Several examples are given.
We prove that that given $\alpha\geq \omega,$ $r\in \omega$ and $k\geq 1$, there is an algebra
$\B^r\in \Nr_{\alpha}\K_{\alpha+k}\sim S\Nr_{\alpha}\K_{\alpha+k}$ such that $\Pi_r\A_r/F\in \K_{\alpha}$ for any $r\in \omega$,
where $\K$ is either the class of cylindric algebras or
quasi-polyadic equality algebras $(\QEA)$.
We discuss its analogue for diagonal free algebras like Pinter's algebras and quasi-polyadic algebras.

In another direction, we use Andr\'eka's methods of splitting to
show that $\RQEA_{\alpha}$ (representable $\QEA_{\alpha}$s)
cannot be axiomatized by a finite schema over its substitution free reduct discarding substitution operations
indexed by transpositions, and is not finitely axiomatizable
over its diagonal free reduct.

We also give a general definition encompassing the two notions of Monk's and Halmos'
schema, in infinite dimensions,  which we call a generalized system of varieties definable by schema,  and give many concrete examples.
In this context we show that the free  $MV$ polyadic algebras, and free algebras
in various countable
reducts of Heyting polyadic algebras, have an interpolation property, deducing that the former class have the
super-amalgamation property.
We also show that Ferenczi's cylindric-polyadic algebras have
the super-amalgamation property.

Having at hand the notion of general schema,  our results on neat embeddings
are presented from the view point of category theory,
encompassing both the cylindric and polyadic
paradigm.

Finally, a class of cylindric-like algebras, namely, N\'emeti's directed $\CA$s,
that manifest polyadic behavior according
to our categorial classification is presented. Deep metalogical manifestations of such a class, which is the cylindric analogue
of Tarski's quasi-projective relation algebras,
like G\"odel's second incompleteness theorems and forcing in set theory are elaborated upon.
\end{abstract}

\section{Introduction}

Algebraic logic starts from certain special logical considerations, abstracts from them, places them in a general algebraic context
and via this generalization makes contact with other branches of mathematics (like set theory and topology).
It cannot be overemphasized that algebraic logic is more algebra than logic,
nor more logic than algebra; in this paper we argue that algebraic logic, particularly the theory of cylindric algebras,
has become sufficiently interesting and deep
to acquire a distinguished status among other subdisciplines of mathematical logic.

This paper focuses on the notion of neat reducts, a notion central in the theory of Tarski's cylindric algebras and Halmos'
polyadic algebras, the two pillars of Tarskian algebraic logic. The first part of the
paper has a survey character, in so far as:
\begin{enumarab}
\item It contains known results and proofs.
\item Many technical details
are occasionally skipped giving only the idea of the proof, but in all cases it will give more than a glimpse of
the underlying idea.
\item Inter-connections between different theorems and concepts are elaborated upon
on a meta-level.
\end{enumarab}

However, we emphasize that the article is not at all only purely expository. First, the second part of the paper 
contains only new results.
Even the first part contains new ideas, novel approaches to old ones, and it offers
new perspectives on previously known results by comparing them to new ones, see example definition \ref{2.12} and its consequences.
Indeed deep theorems in algebraic logic, and more generally in different branches mathematics
acquire new dimensions, as long as new related results are obtained.

On the whole, we prove a diversity of new results (in part 2) collect and possibly abstract
older ones that are apparently unrelated (in part 1)
presenting them all in a unified framework. The highly abstract
general framework we present embodies the notion of systems of varieties
definable by a schema, here generalized to allow finite dimensions and Halmos' schema,
and category theory. 

Throughout the paper certain concepts, that we find worthwhile
highlighting, and deep enough to deserve special attention, may be discussed in some detail
and lengthly elaborated upon, in the process stressing on similarities and differences with other related concepts, both old and new.

The late \cite{1}, the notation of which we follow, gives a panoramic picture of the research in the area in the last
thirty years till the present day, emphasizing the bridges it has built
with other areas, like combinatorics, graph theory, data base, theory stochastics and other fields.
This paper also  surveys, refines and adds to the latest developments
of the subject, but on a smaller scale, of course. In particular, we focus more on `pure' Tarskian
algebraic logic.  Departing from there, we build different  bridges with category theory
non-classical logic, G\"odel's incompleteness theorems and forcing in set theory.

We try, as much as possible, to give a wide panorama
of results on a central key notion in the representation theory of algebraic logic, namely, that of neat reducts
and the related one of neat embedding which casts its shadow over the entire field.
Our results will be later on
wrapped up in the language of arrows, namely, category theory.

We hope that this work also provides rapid dissemination of the latest
research in the field of algebraic logic, particularly in the representation theory of both cylindric and polyadic
algebras, intimately related to the notion of neat embeddings.

Due to its length the paper is divided into two parts. Each part can be read seperately.

Let $\K$ be a cylindric - like class (like for instance quasi-polyadic algebras), so that for all $\alpha$, $\K_{\alpha}$ is a
variety consisting of $\alpha$ dimensional algebras, and for
$\alpha<\beta$, $\Nr_{\alpha}\B$ and $\Rd_{\alpha}\B$ are defined, so that the former, the $\alpha$ neat reduct of $\B$,
is a subalgebra of the latter, the $\alpha$ reduct of $\B$ and the latter  is in $\K_{\beta}$.

We start by the following two questions on subreducts and subneat reducts.
\begin{enumarab}

\item Given ordinals $\alpha<\beta$ and $\A\in \K_{\alpha},$ is there a $\B\in \K_{\beta}$ such that

(i) $\A$  embeds into the $\alpha$ reduct of $\B$?

(ii) $\A$ embeds into the $\alpha$ {\it neat} reduct of $\B$?

\item Assume that it does neatly embed into $\B$, and assume further that $\A$ (as a set) generates $\B$, is then $\B$ unique up to isomorphisms
that fix $\A$ pointwise? does $\A$ exhaust the set of $\alpha$ dimensional elements of $\B,$ namely, the neat $\alpha$
reduct of $\B$?

\end{enumarab}

If $\B$ and $\A$ are like in the second item, then $\B$ is called a {\it minimal dilation} of $\A$.
The first question is related to {\it completeness} theorems, while the second
has to do with {\it definability} properties.

The second question is completely settled in \cite{neat} and \cite{conference}, where it is shown that
minimal dilations may not be unique up to isomorphisms that fix
$\A$ pointwise, witness theorem \ref{Sc} below.
For an intense discussion of this phenomena in connection to
various amalgamation properties for classes of infinite-dimensional representable
cylindric-like algebras, and in the context of confirming several conjectures
of Tarski's on neat embeddings,  the reader is referred to \cite{Sayed}.
Theorem \ref{Sc} formulated below (in the second part of the paper) gives the gist of such connections.

In  the first part of the paper we will, therefore,  focus on the first question.

When reflecting about neat reducts several possible questions cannot help but spring to mind, each with its own range of intuition.
Let us start from the very beginning. The first natural question, as indicated above, that can cross one's mind is:

Is it true that every algebra neatly embeds into another algebra having only one extra dimension?
having $k$ extra dimension $k>1$ ($k$ possibly infinite) ?
And could it at all be possible that an $\alpha$ dimensional algebra
neatly embeds into another one in $\alpha+k$ dimensions but does not neatly
embed into any $\alpha+k+1$ algebra? Now that we have the class of neat reducts
$\Nr_n\K_m$ in front of us, can we classify
it for each $n<m$? Here by classifying we mean defining; for example is it a variety, if not, is it
elementary, if not, is it pseudo-elementary, that is a reduct of an elementary class? and if so is its elementary
theory recursively enumerable?

These are all fair questions, some are somewhat deep, and indeed difficult to answer. Such questions have provoked
extensive research that
have engaged algebraic logicians for years.

Once vexing outstanding problems in logic,
some of such questions for finite dimensions
were settled by the turn of the millennium, and others a few years later,
after thorough  dedicated trials, and dozens of publications providing
partial answers.

It is now known that for finite $n\geq 3$, in symbols for any such $n>2$, and $k\geq 1$, the
variety $S\Nr_n\CA_{n+ k+1}$ is not even finitely axiomatizable over $S\Nr_n\CA_{n+k}$;
this is a known result for cylindric algebras due to
Hirsch, Hodkinson and Maddux, answering problem 2.12 in \cite{HMT1}, witness theorems \ref{thm:cmnr1}, \ref{thm:cmnr}.

It is now also known that for $1<n<m$, $n$ finite,  the class $\Nr_n\CA_m$ is not first order definable, least a variety; however,
it is pseudo- elementary, and its elementary theory is recursively enumerable.
These last results, with a precursor by Andr\'eka and N\'emeti,
addressing the case $n=2$,  is due to the present author answering problem 2.11 in \cite{HMT1} and 4.4 in \cite{HMT2}, 
witness theorem
\ref{neatreduct}.

Both results lifts to  infinite dimensions; the first by replacing finite axiomatizability by finite schema axiomatizability
using an ingenious lifting argument of Monk's, a task implemented
by Robin Hirsch and Sayed-Ahmed \cite{t}
for other concrete cylindric-like algebras (like Pinter's substitution algebras) for all dimensions,
using a finer version of Monk's lifting argument, refined, and abstracted here,
to pass from the finite
to the transfinite, witness theorem \ref{2.12}. But below we shall have occasion to prove a stronger result, than the one proved in \cite{t},
when we address
only cylindric algebras and quasi-polyadic algebras {\it with} equality.

In the second case, the result  lifts, too,
using essentially the same argument as the finite dimensional case, by replacing certain
set algebras used in the proof for the finite dimensional case, see theorem \ref{neatreduct},
by {\it weak} set algebras, retaining as much as possible the proof for the finite dimensional case, for weak set algebras have a
{\it finitary} character \cite{IGPL}. A weak set algebra has top element consisting
of  the  set of sequences having the same infinite length agreeing {\it co-finitely} with a fixed in advance sequence having
the same length.

In the latter case, the lifting does not involve any conceptual addition, in the former case it does.
We therefore
find that formulating the lifting argument
in a very general setting is a task worthwhile implementing.

Indeed, such a lifting argument has the potential to lift many deep results on neat embeddings to the infinite dimensional case,
when the hard work is already accomplished in the finite
dimensional case. 

Another technical lifting argument will be given below, showing that when  finite reducts
of an algebra $\A$ satisfy certain properties then this algebra has the neat embedding property, and so
in the cases we have a neat embedding theorem, like in most cylindric-like algebras dealt with here,
$\A$ turns out  representable.  

A result of the author's, is that in the cylindric case, there are some representable algebras that do not satisfy such properties
formulated in the hypothesis of theorem \ref{firstlifting} below, 
and this answers an open question in \cite{HMT1}, namely, problem $2.13$, \cite{2.13}.
The fact that problems $2.11, 2.12$ and $2.13$ in \cite{HMT1}
address the notion of neat reducts clearly indicates that such a notion is central in cylindric algebra representation theory,
and that it involved, and still involves  quite a few  subtleties and intricacies.

For example the solution to problem 2.11 and several variations on it,
re-appeared in the introduction of \cite{HMT2} as `proofs of Tarski's whose co-authors Henkin and Monk
could not reconstruct'.  All these problems are now solved by the author, with a precursor
\cite{SL}, which is a joint publication with Istvan N\'emeti.
The reader may consult \cite{Sayedneat} in this connection for an overview, where it is shown how the answer
of such questions follow from the apparently remote
fact that the variety of representable cylindric algebras of infinite dimensions
does not have the amalgamation property, a classical  result of Pigozzi's.
This tie between amalgamation properties and neat embedding properties is best 
visualised in a general framework using the language of category theory by viewing
the neat reduct operator as a functor;  this task will be  implemented in theorem \ref{cat} below.

Very roughly the last theorem says that the amalgamation property holds in the target 
category if and only if the neat reduct functor has a right adjoint. 
In polyadic algebra terminogy this right adjoint is called a dilation, in a dilation dimensions are stretched,
when forming neat reducts, they are, dually,  compressed.
If the neat reduct functor turns out to be an equivalence, with inverse the `dilation' functor, 
then the stronger super amalgamation property holds in the target 
category, and this has a converse, theorem \ref{cat}. 
This is the case indeed  with polyadic algebras of infinite dimension, 
but is not the case with cylindric algebras of any finite dimension $>2$,
in this last case the neat reduct functor does not have a right adjoint \cite{conference}.

The idea of lifting here, we further emphasize, can be traced back to Monk 
lifting his classical non-finite axiomatizability result for representable $\CA_n$,
$n$ finite $n>3$ to the transfinite.

To make our argument as general as much as possible, we introduce the new notion of a system of varieties
of {\it Boolean algebras} with operators definable by a schema.
It is like the definition of a Monk's schema, except that we integrate finite dimensions, in such a way that
the $\omega$ dimensional case, uniquely determining higher dimensions, is a natural limit of all $n$ dimensional varieties for finite $n$.

This is crucial for our later investigations.
The definition is general enough to handle our algebras, and narrow enough to prove what we
need.

Splitting techniques of Andr\'eka's will be used to prove, conceptually different new
{\it relative non-finite axiomatizability results} for infinite dimensions,
though one of them has been already done for countable dimensions by the author
\cite{splitting}, theorems \ref{splitting1}, \ref{splitting2}. 
Here we extend it to the uncountable case, a task implemented in \cite{recent}.

Examples of the so-called blow up and blur constructions will be outlined, and a
model-theoretic proof that the class of neat reducts is not first order axiomatizable will be given, both viewed as instances of splitting.

Splitting techniques basically involve splitting an atom in an algebra
into a number of smaller atoms, possibly infinite, getting another super-algebra.
These atoms are big, in the sense that they are cylindrically
equivalent to their copy. This dual nature of such atoms, being small {\it precisely because they are atoms},
and on the other hand, being big
in at least the above sense, helps among many other things,
to obtain non-representable algebras by splitting  an atom or more 
in representable algebras.

Such algebras are {\it barely non-representable}
in the sense that their `small' subalgebras are representable for example all subalgebras generated by $k$ elements,
where $k$ is a fixed in advance finite number;  and various reducts thereof are
also representable.  This technique typically proves impossibility of universal axiomatizations of a class $\K$
of representable algebras
using only $k$  many variables over a strict reduct of $\K$.
This can be done, when the number of  atoms obtained after splitting
are finite and are larger than $k$, modulo some non-trivial
combinatorics, witness theorem \ref{splitting1}.  Another instance of splitting will be used below to render
another relative non-finite  axiomatizability result, theorem \ref{splitting2}. Several instances of splitting are sprinkled
throughout the whole of the first part  of the paper.

Such delicate constructions  turn out  apt
for obtaining relative  non finite-axiomatizability results,
a task impressively initiated by
Andr\'eka \cite{Andreka}, using an ingenious combination of Monk-like constructions
and {\it dilations} in the sense of \cite{HMT2}. This term does not have anything to do
with the term dilation encountered a few paragraphs before
in the context of stretching dimensions.

Here  `dilation' means that  in an atomic algebra, or rather its atom structure, 
if its not rightdown impossible 
to insert an atom, then insert one.  Monk like algebras, on the other hand 
are finite, hence atomic. The atoms are given colours,
and the operations are defined so as to avoid monochromatic traingle. If the atoms are more than colours (which is the case with splitting methods),
then any representation will produce an inconsistency namely a monochromatic traingle.
Such algebras wil be witnessed below in theorems \ref{thm:cmnr1}, \ref {thm:cmnr}.

Another delicate construction that we will have 
occasion to use (twice) is the rainbow construction to prove several algebraic results that have subtle 
impact on finite variable 
fragments of first order logic, showing that 
non-principal types in countable first order $L_n$ theories, may not be omitted even if we substantially broaden 
the class of permissable models - relativizing semantic - in the process 
allowing models which are severely relativized.

We show that certain classes consisting 
of algebras that have a neat embedding property are not elementary, theorem \ref{rainbow},  and that 
other such classes, that are varieties 
are not atom- canonical, theorem \ref{blowupandblur}. Such classes allow relativized representations, 
and so the two negative algebraic 
results exclude relativized (complete) 
representations (representations respecting infinitary meets) 
for abstract algebras, thereby implying failure of omitting types even when we consider 
models in clique guarded semantics. 

In all cases, algebras adressed are countable  and atomic, 
and the type that cannot be omitted  in relativized models will be the non-principal type of co-atoms; it is non-principal because its meet
is the zero element. In `omiting types' terminology non-principal types are sometimes 
referred to as non-isolated types, which means that they cannot
be isolated by a formula, often called a witness \cite{ANT}. 
The terminology non-principal is more suitable here when we work in an algebraic setting, 
because in this context `non-principal' 
actually refers to a non-principal Boolean filter.

In the second part of the paper providing only new results, 
we will define a generalized system definable by a schema that encompasses also polyadic algebras in the infinite dimensional case;
encompassing both Monk's schema and Halmos' schema which exist at least implicitly in a scattered and fragmented
form in the  literature, definition \ref{Halmos}.

Then both systems are approached, in a yet higher level of abstraction, using category theory,
and at this level amalgamation properties (presented in the language of arrows)
are investigated in connection to properties of the neat reduct operator viewed as a functor,
as indicated above, a recent theme initiated by the author in \cite{conference}, formulated at the boundaries of variants, modifications and expansions
of quantifier logics - possibly non claasical like intuitionistic logic and many valued logic - algebraic logic and
category theory. Several non-classical examples of such generalized systems are given, and 
the metalogical property of interpolation is proved for $MV$ polyadic algebras and various reducts of Heyting polyadic algebras. 

We show that the neat reduct operator can also be defined on N\'emeti's directed $\CA$s,
establishing, when treated as a functor in a natural way, an equivalence
between classes of algebras viewed as concrete algebraic categories in different finite dimensions $>2$.
Here the concrete category of Tarski's quasi-relation 
algebras ($\sf QRA$) is used as a transient category, for each finite dimension 
the category of directed cylindric algebras of this dimension is 
equivalent to $\sf QRA$, hence they are all equivalent.

Metalogical repercussions,  in the context of finite variable fragments of untyped higher order
logic, are discussed in connection to G\"odel's second incompleteness theorem, and forcing in set theory.
This is not surprising for this class is the cylindric counterpart of $\sf QRA$s
which were originally designed to formulate set
theory without variables.

At some places some historical notes are added, and accordingly the reader will find
that many of the results are presented in a temporal (historic) way, moving
from old results to newer refinements or/and generalizations.
The first part of the paper has a survey character. The second part  contains only new results.

In their places, the main results,  theorems \ref{2.12}, \ref{new}, \ref{splitting1},
\ref{splitting2}, \ref{neatreduct},  \ref{Fer}, \ref{Sc}, \ref{cat},  and defininitions \ref{Monk}, \ref{Halmos}
are summarized in the outline of the paper to be given after we fix our notation. The new results, mostly in the second part of the paper,
on their part, will be
highlighted.

\subsection*{On the notation}

We follow, as stated above, the notation in \cite{1}. But, for the reader's convenience,
we include the following list of notation that will be used throughout the  paper.
Other than that our notation is fairly standard or self explanatory.
Unusual notation will be explained at  its first occurrence in the text.

For cardinals $m, n$ we write $^mn$ for the set of maps from $m$ to $n$.
If $U$ is an ultrafilter over $\wp(I)$ and if $\A_i$ is some structure (for $i\in I$)
we write either  $\Pi_{i\in I}\A_i/U$ or $\Pi_{i/U}\A_i$ for the ultraproduct of the $\A_i$ over $U$.
Fix some ordinal $n\geq 2$.
For $i, j<n$ the replacement $[i/j]$ is the map that is like the identity on $n$, except that $i$ is mapped to $j$ and the transposition
$[i, j]$ is the like the identity on $n$, except that $i$ is swapped with $j$.     A map $\tau:n\rightarrow n$ is  finitary if
the set $\set{i<n:\tau(i)\neq i}$ is finite, so  if $n$ is finite then all maps $n\rightarrow n$ are finitary.
It is known, and indeed not hard to show, that any finitary permutation is a product of transpositions
and any finitary non-injective map is a product of replacements.

The standard reference for all the classes of algebras to be dealt with is  \cite{HMT2}.
Each class in  $\set{\Df_n, \Sc_n, \CA_n, \PA_n, \PEA_n, \sf QEA_n, \sf QEA_n}$ consists of Boolean algebras with extra operators,
as shown in figure~\ref{fig:classes}, where $\diag i j$ is a nullary operator (constant), $\cyl i, \s_\tau,  \sub i j$ and $\swap i j$
are unary operators, for $i, j<n,\; \tau:n\rightarrow n$.

For finite $n$, polyadic algebras are the same as quasi-polyadic algebra and for the infinite
dimensional case we restrict our attention to quasi-polyadic algebras in $\sf QA_n, \QEA_n$.
Each class is defined by a finite set of equation schema. Existing in a somewhat scattered form in the literature, equations defining
$\sf Sc_n, \sf QA_n$ and $\QEA_n$ are given in the appendix of \cite{t} and also in \cite{AGMNS}.
For $\CA_n$ we follow the standard axiomatization given in definition 1.1.1 in \cite{HMT1}.
For any operator $o$ of any of these signatures, we write $\dim(o)\; (\subseteq n)$
for the set of  dimension ordinals used by $o$, e.g. $\dim(\cyl i)=\set i, \; \dim (\sub i j)=\set{i, j}$.  An algebra $\A$ in $\QEA_n$
has operators that can define any operator of $\sf QA_n, \CA_n,\;\Sc_n$ and $\sf Df_n$. Thus we may obtain the
reducts $\Rd_{\sf K}(\A)$ for $\K\in\set{\QEA_n, \sf QA_n, \CA_n, \Sc_n, \sf Df_n}$ and it turns out that the reduct always
satisfies the equations defining the relevant class so $\Rd_K(\A)\in \K$.
Similarly from any algebra $\A$ in any of the classes $\sf QEA_n, \sf QA_n, \sf CA_n, \sf Sc_n$
we may obtain the reduct $\Rd_\Sc(\c A)\in\Sc_n$ \cite{AGMNS}.

\begin{figure}
\[\begin{array}{l|l}
\mbox{class}&\mbox{extra operators}\\
\hline
\Df_n&\cyl i:i<n\\
\Sc_n& \cyl i, \s_{[i/j]} :i, j<n\\
\CA_n&\cyl i, \diag i j: i, j<n\\
\PA_n&\cyl i, \s_\tau: i<n,\; \tau\in\;^nn\\
\PEA_n&\cyl i, \diag i j,  \s_\tau: i, j<n,\;  \tau\in\;^nn\\
\QA_n&  \cyl i, \s_{[i/j]}, \s_{[i, j]} :i, j<n  \\
\QEA_n&\cyl i, \diag i j, \s_{[i/j]}, \s_{[i, j]}: i, j<n
\end{array}\]
\caption{Non-Boolean operators for the classes\label{fig:classes}}
\end{figure}
Let $\K\in\set{\QEA, \QA, \CA, \Sc, \Df}$, let $\A\in \K_n$ and let $2\leq m\leq n$ (possibly infinite ordinals).
The \emph{reduct to $m$ dimensions} $\Rd_m(\K_n)\in\K_m$ is obtained from $\A$ by discarding all operators with indices $m\leq i<n$.
The \emph{neat reduct to $m$ dimensions}  is the algebra  $\Nr_m(\A)\in \K_m$ with universe $\set{a\in\A: m\leq i<n\rightarrow \cyl i a = a}$
where  all the operators are induced from $\A$ (see \cite[definition~2.6.28]{HMT1} for the $\CA$ case).
More generally, for $\Gamma\subseteq n$ we write $\Nr_\Gamma\A$ for the algebra whose universe
is $\set{a\in\A: i \in n\setminus\Gamma\rightarrow \cyl i a=a}$ with all the operators $o$ of $\A$ where $\dim(o)\subseteq\Gamma$.
Let $\A\in\K_m, \; \B\in \K_n$.  An injective homomorphism $f:\A\rightarrow \B$ is a \emph{neat embedding} if the range of $f$
is a subalgebra of $\Nr_m\B$.
The notions of neat reducts and  neat embeddings have proved useful in
analyzing the number of variables needed in proofs,
as well as for proving representability results,
via the so-called neat embedding theorems.

Let $m\leq n$ be ordinals and let $\rho:m\rightarrow n$ be an injection.
For any $n$-dimensional algebra $\B$ (substitution, cylindric or quasi-polyadic algebra with or without equality)
we define an $m$-dimensional algebra $\Rd^\rho\B$, with the same universe and Boolean structure as
$\B$, where the $(ij)$th diagonal of $\Rd^\rho\B$ is $\diag {\rho(i)}{\rho(j)}\in\B$
(if diagonals are included in the signature of the algebra), the $i$th cylindrifier is $\cyl{\rho(i)}$, the $i$ for $j$
replacement operator is the operator $\s^{\rho(i)}_{\rho(j)}$ of $\A$, the $ij$ transposition operator is $\s_{\rho(i)\rho(j)}$
(if included in the signature), for $i, j<m$.  It is easy to check, for $\K\in\set{\Df,\Sc, \CA, \QPA, \QPEA}$,
that if $\B\in\K_n$ then $\Rd^\rho\B\in\K_m$.    Also, for $\B\in\K_n$ and $x\in \B$,
we define $\Rl_x(\B)$ by `restriction to $x$', so the universe is the set of elements of $\B$ below $x$, where the Boolean unit is $x$,
Boolean zero and sum are not changed, Boolean complementation is relative to $x$,
and the result of applying any non-Boolean operator is obtained by using the operator for $\B$
and intersecting with $x$. It is not always the case
that $\Rl_x\B$ is a $\K_{n}$ (we can lose commutativity of cylindrifiers).

For given algebras $\A$ and $\B$ having the same signature $Hom(\A, \B)$ denotes the set of al homomorphisms from $\A$ to $\B$.

\section*{Layout}

Due to its length, the paper is divided to two parts.

Follows is a rundown of the main results.

\section*{Part 1}

\begin{enumarab}

\item We abstract a lifting argument essentially due to Monk, implemented using ultraproducts
and stretching dimensions in a step-by-step way, in the general context of systems of Boolean algebras with operators, witness
theorem \ref{2.12}.
The main novelty here is integrating finite dimensions uniformly, witness definition \ref{Monk}.

\item We give many applications to existing finite dimensional Monk-like algebras
lifting non-finite axiomatizability results for classes of algebras (all varieties) that
have a neat embedding property to the transfinite, and we prove the new result (on neat embeddings) as
in the abstract, witness theorem \ref{new}. This is proved in \cite{recent2}; here we compare such a result
to its analogue of diagonal free reducts of cylindric algebras.

\item We  present a different method, namely, Andr\'eka's splitting methods.
Such methods are more basic, though they prove stronger non-finite axiomatizability results
for the infinite dimensional case
(involving excluding universal axiomatizations containing finitely
many variables), witness theorems \ref{splitting1}, \ref{splitting2}.
These are proved in \cite{recent}

\item We give one model-theoretic proof to show that for any class $\K$ between $\Sc$ and $\PEA$,
and for any pair of finite ordinals $1<n<m$ when $n$ is finite, that
the class $\Nr_n\K_m$ is not first order definable, unifying proofs in \cite{Fm, note, MLQ, IGPL}.
We give instances of the so-called blow up and blur construction, a very indicative
term introduced in \cite{ANT},
in the framework of splitting.

\item We use rainbow constructions (rainbows, for short) twice.
Once to show that several classes of algebras 
having a {\it complete} neat embedding property are not elementary, theorem \ref{rainbow}.
Our second rainbow,  blows up and blur a finite rainbow polyadic rainbow algebra,
showing that certain varieties larger than that of the representable algebras are not
atom-canonical, witness theorem \ref{blowupandblur}. 
Applications to omitting types theorems for finite variable fragments,
when we consider clique guarded sematics are given, theorem \ref{OTT}.

\end{enumarab}

\section*{Part 2}
Part 2 contains only new results.
\begin{enumarab}

\item We give a new unifying definition for Monk's schema and Halmos' schema, witness theorem \ref{Halmos}.
Several new results on the super-amalgamation property for instances of such schema
addressing $MV$ polyadic algebras,
various proper reducts of Heyting and Ferenczi's
cylindric,  polyadic algebras
of infinite dimension are given,  theorems, \ref{Fer}, \ref{mv}, \ref{heyting}.

\item Pursueing the line of research in \cite{conference},
we give a categorial   formulation of some of our results on neat embeddings, wrapping up the abstract context of
generalized systems of varieties, hitherto introduced, in the language of arrows or category theory,
witness theorem \ref{cat}.

\item Studying cylindric-like algebras, namely, N\'emeti's directed cylindric algebras,
that manifest polyadic behaviour according to our categorial classification.
We use a transient category, namely, $\sf QRA$, the concrete category of relation
algebras with quasi-projections,
to show that directed cylindric algebras in any finite dimension $n, m\geq 3$
are equivalent,  as concrete categories. A new result on the amalgamation property for
$\QRA$ is obtained, witness theorem \ref{SUPAP}, and using the hitherto established equivalence of categories the same result is obtained
for directed cylindric algebras of any dimension $>2$.

\item Connections of such directed cylindric algebras with G\"odels second incompleteness theorem
and forcing in set theory are established, witness
theorems \ref{second}, \ref{forcing}.
\end{enumarab}

Throughout this article, the high level of abstraction embodied in category theory and  in dealing with highly abstract notions, like
{\it systems of varieties definable by a schema}, is motivated and exemplified by well known concrete examples,
so that this level of abstraction can be kept from becoming a high
level of obfuscation.

\section*{Part 1}

\section{Generalizing Monk's Lifting to systems of varieties}

\subsection{The definition of complete systems definable by a schema}

We will generalize Monk's notion of systems of algebras definable by
a schema in two different directions.
In the first case we count in finite dimensions (in a uniform) manner, so that the process of lifting is at least meaningful
in this context.

In the second case we generalize Monk's schema but only for infinite
dimension to cover the polyadic paradigm.
In particular, in this (second) case which we call a {\it generalized system of varieties definable by a schema},
we can handle infinitary substitutions, that occur
in the signature of infinite dimensional polyadic algebras both with
and without equality. However, this will be done later, witness theorem \ref{Halmos}.

Let us now  concentrate on generalizing the notion of systems of varieties definable by a Monk schema to
a system of varieties where we can include
finite dimensions.

To make our argument of the lifting process, as general as much as possible, we introduce the new notion of a system of varieties
of {\it Boolean algebras} with operators definable by a schema.
It is similar to the definition of a Monk's schema \cite{HMT2}, except that we integrate finite dimensions, in such a way that
the $\omega$ dimensional case, uniquely determining higher dimensions, is a natural limit of all the $n$ dimensional varieties for finite $n$.
This is crucial for our later investigations.
The definition is general enough to handle our algebras, and narrow enough to prove what we
need.

We generalize the notion of Monk's schema allowing finite dimensions, but not necessarily all,
so that our systems are indexed by all ordinals $\geq m$ and $m$ could be finite.
(We can allow also proper infinite subsets of $\omega$, but we do not need that much.)

The main advantage in this approach is that it shows that a plethora of results proved for infinite
dimensions (like non-finite schema axiomatizability of the representable algebras)
really depend on the analogous result proved for every finite dimension starting at a certain finite $n$ which is usually $3$.
The hard work is done for the finite dimensional case.
The rest is a purely syntactical ingenious
lifting process invented by Monk.
First the general definition, then we will provide  two technical examples.

\begin{definition}\label{Monk}
\begin{enumroman}
\item Let $2\leq m\in \omega.$ A finite $m$ type schema is a quadruple $t=(T, \delta, \rho,c)$ such that
$T$ is a set, $\delta$ and $\rho$ maps $T$ into $\omega$, $c\in T$, and $\delta c=\rho c=1$ and $\delta f\leq m$ for all $f\in T$.
\item A type schema as in (i) defines a similarity type $t_{n}$ for each $n\geq m$ as follows. The domain $T_{n}$ of $t_{n}$ is
$$T_{n}=\{(f, k_0,\ldots, k_{\delta f-1}): f\in T, k\in {}^{\delta f}n\}.$$
For each $(f, k_0,\ldots, k_{\delta f-1})\in T_{n}$ we set $t_{n}(f, k_0,\ldots, k_{\delta f-1})=\rho f$.
\item A system $(\K_{n}: n\geq m)$ of classes
of algebras is of type schema $t$ if for each $n\geq m,$ $\K_{n}$ is a class of algebras of type
$t_{n}$.
\end{enumroman}
\end{definition}

\begin{definition}\label{Monk} Let $t$ be  a finite $m$ type schema.
\begin{enumroman}
\item With each $m\leq n\leq \beta$ we associate a language $L_{n}^t$ of type $t_{n}$: for each $f\in T$ and
$k\in {}^{\delta f}n,$ we have a function symbol $f_{k0,\ldots, k(\delta f-1)}$ of rank $\rho f$
\item Let $m\leq \beta\leq n$, and let $\eta\in {}^{\beta}n$ be an injection.
We associate with each term $\tau$ of $L_{\beta}^t$ a term $\eta^+\tau$ of $L_{n}^t$.
For each $\kappa\in \omega, \eta^+ v_{\kappa}=v_{\kappa}$. If $f\in T, k\in {}^{\delta f}\alpha$, and $\sigma_1,\ldots, \sigma_{\rho f-1}$
are terms of $L_{\beta}^t$, then
$$\eta^+f_{k(0),\ldots, k(\delta f-1)}\sigma_0,\ldots, \sigma_{\rho f-1}=f_{\eta(k(0)),\ldots, \eta(k(\delta f-1))}\eta^+\sigma_0\ldots \eta^+\sigma_{\rho f-1}.$$
Then we associate with each equation $e=\sigma=\tau$ of $L_{\beta}^t$ the equation
$\eta^+\sigma=\eta^+\tau$ of $L_{\alpha}^t$, which we denote by
$\eta^+(e)$.

\item A system $\K=(\K_n: n\geq m)$ of finite $m$ type schema $t$
is a {\it complete system of varieties definable by a schema},
if there is a system $(\Sigma_n: n\geq m)$ of equations such that ${\sf Mod}(\Sigma_n)=\K_n$, and
for $n\leq m<\omega$ if $e\in \Sigma_n$ and
$\rho: n\to m$ is an injection, then $\rho^+e\in \Sigma_m$; $(\K_{\alpha}: \alpha\geq \omega)$
is a system of varieties definable by a schema
and $\Sigma_{\omega}=\bigcup_{n\geq m}\Sigma_n$.
\end{enumroman}
\end{definition}

\begin{definition}
\begin{enumarab}
\item Let $\alpha, \beta$ be ordinals, $\A\in \K_{\beta}$ and $\rho:\alpha\to \beta$ be an injection.
We assume for simplicity of notation, and without any loss,  that in addition to cylindrifiers,
we have only one unary function symbol $f$ such that $\rho(f)=\delta(f)=1.$
(The arity is one, and $f$ has only one index.)
Then $\Rd_{\alpha}^{\rho}\A$ is the $\alpha$ dimensional
algebra obtained for $\A$ by defining for $i\in \alpha$ $\sf c_i$ by $\sf c_{\rho(i)}$, and $f_i$ by $f_{\rho(i)}$.
$\Rd_{\alpha}\A$ is $\Rd_{\alpha}^{\rho}\A$ when $\rho$ is the inclusion map.

\item As in the first part we assume only the existence of one unary operator with one index.
Let $\A\in \K_{\beta}$, and $x\in A$. The dimension set of $x$, denoted by $\Delta x$, is the set $\Delta x=\{i\in \alpha: {\sf c}_ix\neq x\}$.
We assume that if $\Delta x\subseteq \alpha$, then $\Delta f(x)\leq \alpha$.
Then $\Nr_{\alpha}\B$ is the subuniverse of $\Rd_{\alpha}\B$ consisting
only of $\alpha$ dimensional elements.

\item For $K\subseteq \K_{\beta}$ and an injection $\rho:\alpha\to \beta$,
then $\Rd_{\alpha}^{\rho}K=\{\Rd^{\rho}_{\alpha}\A: \A\in K\}$ and $\Nr_{\alpha}K=\{\Nr_{\alpha}\A: \A\in K\}.$
\end{enumarab}
\end{definition}
The class $S\Nr_{\alpha}\K_{\alpha+\omega}$ has special significance since it coincides in the most
known cases to
the class of representable algebras.

\subsection{Monk's argument}

In the following theorem, we use a very similar argument of lifting implemented by Monk to show that the class
of infinite dimensional cylindric algebras cannot be axiomatizable by a finite Monk's schema, using and llifting
the analogous result
for infinite dimension $>2$ \cite[theorem 3.2.87]{HMT2}.
So let us warm up by the first lifting argument;
the second argument, though in essence very similar, will be more involved technically.

We show, in our next theorem,  how properties that hold for all finite reducts of an infinite dimensional algebra forces it
to have the neat embedding property.
The proof does not use any properties not formalizable in complete systems of varieties definable by a schema;
it consists of non-trivial manipulation of reducts and neat reducts
via an ultraproduct construction, used to `stretch' dimensions.

The theorem is an abstraction of \cite[theorem 2.6.50]{HMT1}.

\begin{theorem}\label{firstlifting} Let $\alpha$ be an infinite ordinal and $\A\in \K_{\alpha}$ such that for every finite injective map $\rho$ into $\alpha$,
and for every  $x,y\in \A$, $x\neq y$, there is a function $h$ and $k<\alpha$ such
that $h$ is an endomorphism of $\Rd^{\rho}\A$, $k\in \alpha\sim \rng(\rho)$, ${\sf c}_k\circ h=h$ and $h(x)\neq h(y)$.
Then $\A\in {\sf UpS}\Nr_{\alpha}\K_{\alpha+\omega}$.
\end{theorem}
\begin{demo}{Proof}  We first prove that the following holds for any $l<\omega$.
For every $k<\omega$, for every injection $\rho:k\to \alpha$, and every $x, y\in A$, $x\neq y$,
there exists $\sigma,h$ such that $\sigma:k+l\to \alpha$ is an injection, $\rho\subseteq \sigma$,
$h$ is an endomorphism of $\Rd_k^{\rho}\A$, $c_{\sigma_u}\circ h=h$, whenever $k\leq u\leq k+l$, and $h(x)\neq h(y)$.

We proceed by induction on $l$. This holds trivially for $l=0$, and it is easy to see that it is true for $l=1$
Suppose now that it holds for given $l\geq 1$.
Consider $k, \rho$, and $x,y$ satisfying the premise. By the induction hypothesis there are
$\sigma, h$, $\sigma:\alpha+l\to \alpha$ an injection, $\rho\subseteq \sigma$,
$h$ is an endomorphism of $\Rd_k^{\rho}\A$, ${\sf c}_{\sigma_u}\circ h=h$
whenever  $k\leq u< k+l$, and $h(x)\neq h(y).$ But then there exist $k,v$ such that $k$ is an endomorphism of
$\Rd_{k+l}^{\sigma}\A$, $v\in \alpha\sim \rng\sigma$, ${\sf c}_v\circ k=k$ and $k\circ h(x)\neq k\circ h(y)$.
Let $\sigma'$ be defined by $\sigma'\upharpoonright \alpha+(l+k)=\sigma$ and $\sigma'(k+l)=v$, and let
$h'={\sf c}_{\sigma'_u}\circ h$. It is easy to check that
$\sigma'$ and $h'$ complete the induction step.

We have $h\in Hom(\Rd_k^{\rho}\A, \Nr_{k}\B)$ where $\B=\Rd_{k+l}^{\sigma}\A$. Then $\Rd_k^{\rho}\A\in S\Nr_k\K_{\alpha+\omega}.$
For brevity let $\D=\Rd_k^{\rho}\A$. For each $l<\omega,$ let $\B_l\in \K_{k+l}$ such that $\D\subseteq \Nr_k\B_l$.
For all such $l$, let $\C_l$ be an algebra have the same similarity type as of $\K_{\omega}$ be such $\B_l=\Rd_{k+l}\C_l$.
Let $F$ be a non-principal ultrafilter on $\omega$, and let $\G=\Pi_{\l<\omega}\C_l/F$. Let
$$G_n=\{\Gamma\cap (\omega\sim n):\Gamma\in F\}.$$
Then for all $\mu<\omega$, we have
\begin{align*}
\Rd_{k+\mu}\G
&=\Pi_{\eta<\omega} \Rd_{k+\mu}\C_{\eta}/F\\
&\cong \Pi_{\mu\leq \eta<\omega}\Rd_{k+\mu}\C_{\eta}/G_{\mu}\\
&=\Pi_{\mu\leq \eta<\omega}\Rd_{\beta+\mu}\B_{\eta}/G_{\mu}.
\end{align*}
We have shown that $\G\in \K_{\omega}$.
Define $h$ from $\D$ to $\G$, via
$$x\to (x: \eta<\omega)/F.$$
Then $h$ is an injective homomorphism from $\D$ into $\Nr_k\G$.
We have $\G\in \K_{\omega}$ We now show that there exists $\B$ in $\K_{\alpha+\omega}$ such that $\D\subseteq \Nr_k\B$.
(This is a typical instance where reducts are used to 'stretch dimensions', not to compress them).

One proceeds inductively, at successor ordinals (like $\omega+1$) as follows.
Let $\rho:\omega+1\to \omega$ be an injection such that  $\rho(i)=i$, for each $i\in \alpha$. Then
$\Rd^{\rho}\G\in \K_{\omega+1}$ and $\G=\Nr_{\omega}{\Rd^{\rho}}\G$. At limits one uses ultraproducts like above.

Thus $\Rd_k^{\rho}\A\subseteq \Nr_k\B$ for some $\B\in \K_{\alpha+\omega}$. Let $\sigma$ be a permutation of
$\alpha+\omega$ such that $\sigma\upharpoonright k=\rho$ and $\sigma(j)=j$ for all $j\geq \omega$.
Then
$$\Rd_k^{\rho}\A\subseteq \Nr_k\B=\Nr_k\Rd_{\alpha+\omega}^{\sigma}\Rd_{\alpha+\omega}^{\sigma^{-1}}\B.$$
Then for any $u$ such that $\sigma[k]\subseteq u\subseteq \alpha+\omega$, we have

$$\Nr_k\Rd_{\alpha+\omega}^{\sigma}\Rd_{\alpha+\omega}^{\sigma^{-1}}\B\subseteq \Rd_k^{\sigma|k}\Nr_{uu}\Rd_{\alpha+\omega}^{\sigma^{-1}}\B.$$
Thus $\Rd_k^{\rho}\A\in S\Rd^{\rho}\Nr_{\alpha}\K_{\alpha+\omega},$ and this holds for any injective finite sequence $\rho$.

Let $I$ be the set of all finite one to one sequences with range in $\alpha$.
For $\rho\in I$, let $M_{\rho}=\{\sigma\in I:\rng\rho\subseteq \rng\sigma\}$.

Let $U$ be an ultrafilter of $I$ such that $M_{\rho}\in U$ for every $\rho\in I$. Exists, since $M_{\rho}\cap M_{\sigma}=M_{\rho\cup \sigma}.$
Then for $\rho\in I$, there is $\B_{\rho}\in \K_{\alpha+\omega}$ such that
$\Rd^{\rho}\A\subseteq \Rd^{\rho}\B_{\rho}$. Let $\C=\Pi\B_{\rho}/U$; it is in ${Up}\K_{\alpha+\omega}$.
Define $f:\A\to \Pi\B_{\rho}$ by $f(a)_{\rho}=a$, and finally define $g:\A\to \C$ by $g(a)=f(a)/U$.
Then $g$ is an embedding, and we are done.
\end{demo}

We note that handling reducts are easier than handling neat ones.
Problems involving
neat reducts tend to be messy, in the positive sense.

So lets get over with the easy part of reducts.
The infinite dimensional case follows from the definition of a system of varieties definable by schema, namely, for any such system
we have $\K_{\alpha}= {\sf HSP}\Rd^{\rho}\K_{\beta}$
for any pair of infinite ordinals $\alpha<\beta$ and any injection $\rho:\alpha\to \beta$.

For finite dimensions this is not true \cite[theorem 2.6.14]{HMT1}.
But for all algebras considered ${\sf S}\Rd_{\alpha}\K_{\beta}$ is a variety, hence the
desired conclusion; which is that every algebra is a subreduct of
an algebra in any pre-assigned higher dimension.
We can strengthen this, quite easily, to:

\begin{theorem} Let $t$ be a type system and $\K=(\K_{\alpha}:\alpha\geq \omega)$ be a system of type schema $t$.
Then the following conditions are equivalent
\begin{enumroman}
\item $\K$ is definable by a schema
\item If $\omega\leq \alpha\leq \beta$, and injective $\rho\in {}^{\alpha}\beta$, $\K_{\alpha}=HSP\Rd^{\rho}\K_{\beta}$
 \end{enumroman}
\end{theorem}
\begin{demo}{Proof} \cite[theorem 5.6.13]{HMT2}
\end{demo}
\begin{theorem} For any pair of infinite ordinals $\alpha<\beta,$ we have
$\K_{\alpha}={\sf UpUr}\Rd_{\alpha}\K_{\beta}$
\end{theorem}
\begin{proof} For simplicity we assume that we have one unary operation $f$, with $\rho(f)=1$.
The general case is the same.
Let $\A\in \K_{\alpha}$. Let $I=\{\Gamma: \Gamma\subseteq \beta, |\Gamma|<\omega\}$.
Let $I_{\Gamma}=\{\Delta \subseteq I, \Gamma\subseteq \Delta\}$, and let $F$ be an ultrafilter such that
$I_{\Gamma}\in F$ for all $\Gamma\in I$. Notice that $I_{\Gamma_1}\cap I_{\Gamma_2}=I_{\Gamma_1\cup \Gamma_2}$
so this ultrafilter exists.
For each $\Gamma\in I$, let $\rho(\Gamma)$ be an injection from $\Gamma$ into $\alpha$
such that $Id\upharpoonright \Gamma\cap \alpha\subseteq \rho(\Gamma)$, and let $\B_{\Gamma}$ be an algebra having same similarity type
as $\K_{\beta}$such that for $k\in \Gamma$
$f_k^{\B_{\Gamma}}= f_{\rho(\Gamma)[k]}^{\A}$. Then
$\D=\Pi \B_{\Gamma}/F\in \K_{\beta}$ and
$f:\A\to \Rd_{\alpha}\D$ defined via $a\mapsto (a: \Gamma\in I)/F$
is an elementary embedding.
\end{proof}

As mentioned in the introduction, it is not the case that, for $2<m<n$ every algebra in $\CA_m$ is the neat reduct of an algebra in $\CA_n$,
nor need it even be a subalgebra of a neat reduct of an algebra in $\CA_n$.  Furthermore, $S\Nr_m\CA_{m+k+1}\neq S\Nr_m\CA_m$,
whenever $3\leq m<\omega$ and $k<\omega$.

The hypothesis in the next theorem presupposes the existence of certain finite
dimensional algebras, that might look at first sight complicated, but their choice is not haphazard at all;
they are  rather an abstraction of cylindric algebras existing in the literature witnessing the last
proper inclusions.

The main idea, that leads to the conclusion of the theorem,
is to use such finite dimensional algebras to obtain an an analogous result for the infinite dimensional case.
Accordingly, we found it convenient to streamline Monk's argument who did exactly that for cylindric algebras \cite[theorem 3.2.87]{HMT2},
but we do it in
the wider context of systems of varieties of Boolean algebras with operators definable by a schema as defined above in \ref{Monk}.

Strictly speaking Monk's lifting argument is weaker,
the infinite dimensional constructed algebras are merely non-representable, in our case they are not only non-representable,
but they are {\it not subneat reducts} of  algebras in a given pre
assigned dimension; this is a technical significant difference, that needs some
non-trivial fine tuning in the proof. Another significant difference from Monk's lifting argument is that here
we have to use commutativity of ultraproducts.

The inclusion of finite dimensions in our formulation, was therefore not a luxury, nor was it motivated by
aesthetic reasons, and nor was it merely an artefact of Monk's definition.
It is motivated by the academic worthiness of the result
(for infinite dimensions).

\begin{theorem}\label{2.12} Let $(\K_{\alpha}: \alpha\geq 2)$ be a complete system of varieties definable by a schema.
Assume that for $3\leq m<n<\omega$,
there is $m$ dimensional  algebra $\C(m,n,r)$ such that
\begin{enumarab}
\item $\C(m,n,r)\in S\Nr_m\K_n$
\item $\C(m,n,r)\notin S\Nr_m\K_{n+1}$
\item $\Pi_{r\in \omega} \C(m, n,r)\in S\Nr_m\K_{n+1}$
\item For $m<n$ and $k\geq 1$, there exists $x_n\in \C(n,n+k,r)$ such that $\C(m,m+k,r)\cong \Rl_{x}\C(n, n+k, r).$
\end{enumarab}
Then for any ordinal $\alpha\geq \omega$ and $k\geq 1$,  $S\Nr_{\alpha}\K_{\alpha+k+1}$
is not axiomatizable by a finite schema over $S\Nr_{\alpha}\K_{\alpha+k}$
\end{theorem}

\begin{proof} The proof is a lifting argument essentially due to Monk, by 'stretching' dimensions  using only properties of
reducts  and ultraproducts, formalizable in the context of a system of varieties definable by a schema.
The proof is divided into 3 parts,
and it is a generalization of a  lifting argument
in \cite{t}, addressing certain cylindric-like algebras. Both are generalizations of Monk's argument in
\cite[theorem 3.2.87]{HMT2}; but our formulation is far more general.
The first first is preparing for the next two, where the theorem is proved.

\begin{enumarab}

\item  Let $\alpha$ be an infinite ordinal. Let $X$ be any finite subset of $\alpha$ and
let $$I=\set{\Gamma:X\subseteq\Gamma\subseteq\alpha,\; |\Gamma|<\omega}.$$
For each $\Gamma\in I$ let $M_\Gamma=\set{\Delta\in I:\Delta\supseteq\Gamma}$
and let $F$ be any ultrafilter over $I$
such that for all $\Gamma\in I$ we have $M_\Gamma\in F$
(such an ultrafilter exists because $M_{\Gamma_1}\cap M_{\Gamma_2} = M_{\Gamma_1\cup\Gamma_2}$).

For each $\Gamma\in I$ let $\rho_\Gamma$ be a bijection from $|\Gamma|$ onto $\Gamma$.
For each $\Gamma\in I$ let $\A_\Gamma, \B_\Gamma$ be $\K_\alpha$-type algebras.

We claim that

(*) If  for each $\Gamma\in I$ we have
$\Rd^{\rho_\Gamma}\A_\Gamma=\Rd^{\rho_\Gamma}\B_\Gamma,$ then
we have
$$\Pi_{\Gamma/F}\A_\Gamma=\Pi_{\Gamma/F}\B_\Gamma.$$
The proof is standard using  Los' theorem.

Indeed,
$\Pi_{\Gamma/F}\A_\Gamma$, $\Pi_{\Gamma/F}\Rd^{\rho_\Gamma}\A_\rho$
and  $\Pi_{\Gamma/F}\B_\Gamma$ all have the same universe,  by  assumption.
Also each operator $o$ of $\K_\alpha$ is also the same for
both ultraproducts, because $\set{\Gamma\in I:\dim(o)\subseteq\rng(\rho_\Gamma)} \in F$.

Now we claim that

(**) if $\Rd^{\rho_\Gamma}\A_\Gamma \in \K_{|\Gamma|}$, for each $\Gamma\in I,$
then $\Pi_{\Gamma/F}\A_\Gamma\in \K_\alpha$.
For this, it suffices to prove that each of the defining axioms for $\K_\alpha$ holds for $\Pi_{\Gamma/F}\A_\Gamma$.

Let $\sigma=\tau$ be one of the defining equations for $\K_{\alpha}$,
the number of dimension variables is finite, say $n$.

Take any $i_0, i_1,\ldots,  i_{n-1}\in\alpha$. We have to prove that
$$\Pi_{\Gamma/F}\A_\Gamma\models \sigma(i_0,\ldots, i_{n-1})=\tau(i_0\ldots,  i_{n-1}).$$
Suppose that they are all in $\rng(\rho_\Gamma)$, say $i_0=\rho_\Gamma(j_0), \; i_1=\rho_\Gamma(j_1), \;\ldots, i_{n-1}=\rho_\Gamma(j_{n-1})$,
then $\Rd^{\rho_\Gamma}\A_\Gamma\models \sigma(j_0, \ldots ,j_{n-1})=\tau(j_0, \ldots j_{n-1})$,
since $\Rd^{\rho_\Gamma}\A_\Gamma\in\K_{|\Gamma|}$, so $\A_\Gamma\models\sigma(i_0,\ldots, i_{n-1})=\tau(i_0\ldots, i_{n-1})$.

Hence $\set{\Gamma\in I:\A_\Gamma\models\sigma(i_0, \ldots, i_{n-1})=\tau(i_0, \ldots,  i_{n-1})}\supseteq\set{\Gamma\in I:i_0,\ldots,  i_{n-1}
\in\rng(\rho_\Gamma}\in F.$
It now easily follows that
$$\Pi_{\Gamma/F}\A_\Gamma\models\sigma(i_0,\ldots,  i_{n-1})=\tau(i_0, \ldots,  i_{n-1}).$$
Thus $\Pi_{\Gamma/F}\A_\Gamma\in\K_\alpha$, and we are done.

\item Let $k\geq  1 $ and $r\in \omega$. Let $\alpha$, $I$, $F$
and $\rho_{\Gamma}$ be as above and assume the hypothesis of the theorem.

Let ${\C}_{\Gamma}^r$ be an algebra similar to $\K_{\alpha}$ such that
\[\Rd^{\rho_\Gamma}{\C}_{\Gamma}^r={\c C}(|\Gamma|, |\Gamma|+k,r).\]
Let
\[\B^r=\Pi_{\Gamma/F\in I}\C_{\Gamma}^r.\]
We will prove that
\begin{enumerate}
\item\label{en:1} $\B^r\in S\Nr_\alpha\K_{\alpha+k}$ and
\item\label{en:2} $\B^r\not\in S\Nr_\alpha\K_{\alpha+k+1}$.  \end{enumerate}

For the first part, for each $\Gamma\in I$ we know that $\C(|\Gamma|+k, |\Gamma|+k, r) \in\K_{|\Gamma|+k}$ and
$\Nr_{|\Gamma|}\C(|\Gamma|+k, |\Gamma|+k, r)\cong\C(|\Gamma|, |\Gamma|+k, r)$.
Let $\sigma_{\Gamma}$ be a one to one function
 $(|\Gamma|+k)\rightarrow(\alpha+k)$ such that $\rho_{\Gamma}\subseteq \sigma_{\Gamma}$
and $\sigma_{\Gamma}(|\Gamma|+i)=\alpha+i$ for every $i<k$.

Let $\A_{\Gamma}$ be an algebra similar to a
$\K_{\alpha+k}$ such that
$\Rd^{\sigma_\Gamma}\A_{\Gamma}=\C(|\Gamma|+k, |\Gamma|+k, r)$.  By (**)
with  $\alpha+k$ in place of $\alpha$,\/ $m\cup \set{\alpha+i:i<k}$
in place of $X$,\/ $\set{\Gamma\subseteq \alpha+k: |\Gamma|<\omega,\;  X\subseteq\Gamma}$
in place of $I$, and with $\sigma_\Gamma$ in place of $\rho_\Gamma$,
we know that  $\Pi_{\Gamma/F}\A_{\Gamma}\in \K_{\alpha+k}$.

Now we prove that $\B^r\subseteq \Nr_\alpha\Pi_{\Gamma/F}\A_\Gamma$.  Recall that $\B^r=\Pi_{\Gamma/F}\C^r_\Gamma$ and note
that $C^r_{\Gamma}\subseteq A_{\Gamma}$
(the universe of $C^r_\Gamma$ is $\C(|\Gamma|, |\Gamma|+k, r)$; the universe of $A_\Gamma$ is $\C(|\Gamma|+k, |\Gamma|+k, r)$).
 So, for each $\Gamma\in I$,
\begin{align*}
\Rd^{\rho_{\Gamma}}\C_{\Gamma}^r&=\C((|\Gamma|, |\Gamma|+k, r)\\
&\cong\Nr_{|\Gamma|}\C(|\Gamma|+k, |\Gamma|+k, r)\\
&=\Nr_{|\Gamma|}\Rd^{\sigma_{\Gamma}}\A_{\Gamma}\\
&=\Rd^{\sigma_\Gamma}\Nr_\Gamma\A_\Gamma\\
&=\Rd^{\rho_\Gamma}\Nr_\Gamma\A_\Gamma
\end{align*}
By (*) we deduce that
$$\Pi_{\Gamma/F}\C^r_\Gamma\cong\Pi_{\Gamma/F}\Nr_\Gamma\A_\Gamma\subseteq\Nr_\alpha\Pi_{\Gamma/F}\A_\Gamma,$$
proving \eqref{en:1}.

Now we prove \eqref{en:2}.
For this assume, seeking a contradiction, that $\B^r\in S\Nr_{\alpha}\K_{\alpha+k+1}$,
$\B^r\subseteq \Nr_{\alpha}\C$, where  $\C\in \K_{\alpha+k+1}$.
Let $3\leq m<\omega$ and  $\lambda:m+k+1\rightarrow \alpha +k+1$ be the function defined by $\lambda(i)=i$ for $i<m$
and $\lambda(m+i)=\alpha+i$ for $i<k+1$.
Then $\Rd^\lambda(\C)\in \K_{m+k+1}$ and $\Rd_m\B^r\subseteq \Nr_m\Rd^\lambda(\C)$.
For each $\Gamma\in I$,\/  let $I_{|\Gamma|}$ be an isomorphism
\[{\C}(m,m+k,r)\cong \Rl_{x_{|\Gamma|}}\Rd_m {\C}(|\Gamma|, |\Gamma+k|,r).\]
Exists by assumption. Let $x=(x_{|\Gamma|}:\Gamma)/F$ and let $\iota( b)=(I_{|\Gamma|}b: \Gamma)/F$ for  $b\in \C(m,m+k,r)$.
Then $\iota$ is an isomorphism from $\C(m, m+k,r)$ into $\Rl_x\Rd_m\B^r$.
Then $\Rl_x\Rd_{m}\B^r\in S\Nr_m\K_{m+k+1}$.
It follows that  $\C (m,m+k,r)\in S\Nr_{m}\K_{m+k+1}$ which is a contradiction and we are done.

Now we prove the third part of the theorem, putting the superscript $r$ to use.
Recall that $\B^r=\Pi_{\Gamma/F}\C^r_\Gamma$, where $\C^r_\Gamma$ has the type of $\K$
and $\Rd^{\rho_\Gamma}\C^r_\Gamma=\C(|\Gamma|, |\Gamma|+k, r)$.
We know that
$$\Pi_{r/U}\Rd^{\rho_\Gamma}\C^r_\Gamma=\Pi_{r/U}\C(|\Gamma|, |\Gamma|+k, r) \subseteq \Nr_{|\Gamma|}\A_\Gamma,$$
for some $\A_\Gamma\in\K_{|\Gamma|+k+1}$.

Let $\lambda_\Gamma:|\Gamma|+k+1\rightarrow\alpha+k+1$ extend $\rho_\Gamma:|\Gamma|\rightarrow \Gamma \; (\subseteq\alpha)$ and satisfy
\[\lambda_\Gamma(|\Gamma|+i)=\alpha+i\]
for $i<k+1$.  Let $\F_\Gamma$ be a $\K_{\alpha+k+1}$ type algebra such that $\Rd^{\lambda_\Gamma}\F_\Gamma=\A_\Gamma$.
As before, $\Pi_{\Gamma/F}\F_\Gamma\in\K_{\alpha+k+1}$.  And
\begin{align*}
\Pi_{r/U}\B^r&=\Pi_{r/U}\Pi_{\Gamma/F}\C^r_\Gamma\\
&\cong \Pi_{\Gamma/F}\Pi_{r/U}\C^r_\Gamma\\
&\subseteq \Pi_{\Gamma/F}\Nr_{|\Gamma|}\c A_\Gamma\\
&=\Pi_{\Gamma/F}\Nr_{|\Gamma|}\Rd^{\lambda_\Gamma}\F_\Gamma\\
&\subseteq\Nr_\alpha\Pi_{\Gamma/F}\F_\Gamma,
\end{align*}
proving the required.

\end{enumarab}
\end{proof}
An obvious modification of the proof shows that:

\begin{theorem}\label{Monk} Assume that for $2<m<n<\omega$, there is
$\C(m, n)\in \K_m$ such that for all $k\in \omega$,
$\C(m,m+k)\subseteq \Nr_m\C(n, n+k)$, $\C(m, m+k)\notin S\Nr_m\K_{m+\omega}$  and the ultraproduct (on $k$)
is in $S\Nr_n\K_{\omega}$.
Furthermore, for $m<n$ and $k\geq 1$, there exists $x_n\in \C(n,n+k)$ such that $\C(m,m+k)\cong \Rl_{x}\C(n, n+k, r).$
Then $S\Nr_{\alpha}\K_{\alpha+\omega}$ cannot be axiomatizable by a finite schema.
\end{theorem}

In what follows we give concrete instances of our lifting argument, then use it to solve the infinite analogue of the famous problem
numbered 2.12 in \cite{HMT2} (this is already done in \cite{t} but in the cylindric case our result is stronger),
that has provoked extensive research for decades,
including, to name a few algebraic logicians who worked on it, we mention Andr\'eka, N\'emeti, Monk, Maddux, Hirsch, Hodkinson
Sayed Ahmed and Venema. The problem was first approached by Monk; more accurately, it was raised by Monk.

Our following exposition also has a historic (temporal) perspective, starting from Monk's original algebras, all the way to refined
Monk-like algebras constructed by Hirsch and Hodkinson solving problem 2.12 in \cite{HMT2},
culminating in the result on neat embeddings formulated
in theorem \ref{new} and highlighted in the abstract.

\subsection{Monk's algebras}

Monk defined the required algebras, witnessing the non finite axiomatizability of $\RCA_n$, for finite $n\geq 3$,
via their atom structure. We write following Monk's notation $n\not{R}m$ if $(n,m)\notin R$ and  for a cylindric atom structure $\G$,
we write, following the notation of \cite{HMT1}.
$\Ca\G$ for its complex (cylindric) algebra.

An $n$ dimensional atom structure is a triple
$\G=(G, T_i, E_{ij})_{i,j\in n}$
such that $T_i\subseteq G\times G$ and $E_{ij}\subseteq G$, for all $i, j\in n$. An atom structure so defined, is a cylindric atom structure if
its complex algebra $\Ca\G\in \CA_n$. $\Ca\C$ is the algebra
$$(\wp(G), \cap, \sim T_i^*, E_{ij}^*)_{i,j\in n},$$ where
$$T_i^*(X)=\{a\in G: \exists b\in X: (a,b)\in T_i\}$$
and
$$E_{i,j}^*=E_{i,j}.$$
Cylindric algebras are axiomatized by so-called Sahlqvist equations, and therefore it is easy to spell out first order correspondents
to such equations characterizing
atom structures of cylindric algebras.

\begin{definition}
For $3 \leq m\leq n < \omega$, ${\G}_{m, n}$ denotes
the cylindric atom structure such that ${\G}_{m, n} = (G_{m, n},
T_i, E_{i,j})_{i, j < m} $ of dimension
$m$ which is defined as follows:
$G_{m, n}$ consists of all pairs $(R, f)$ satisfying
the following conditions:
\begin{enumarab}
\item $R$ is equivalence relation on $m$,
\item $f$ maps $\{ (\kappa, \lambda) : \kappa, \lambda < n,
\kappa  \not{R} \lambda\}$ into $n$,
\item for all $\kappa, \lambda < m$, if $\kappa \not{R} \lambda$
then $f_{\kappa \lambda } = f_{\lambda \kappa},$
\item for all $\kappa, \lambda, \mu < m$, if $\kappa \not{R}
\lambda R \mu$ then $f_{\kappa \lambda } = f_{\kappa \mu}$,
\item for all $\kappa, \lambda, \mu < n$, if $\kappa \not{R}
\lambda \not{R} \mu \not{R} \kappa$ then $|f_{\kappa \lambda },
f_{\kappa \mu}, f_{\lambda \mu}| \neq 1.$
\end{enumarab}
For $\kappa < m$ and $(R, f), (S, g) \in G(m,n)$ we define
\begin{eqnarray*}
&(R, f) T_\kappa (S, g) ~~ \textrm{iff} ~~ R \cap {}^2(n
\smallsetminus
\{\kappa\}) = S \cap {}^2(m \smallsetminus \{\kappa\}) \\
& \textrm{and for all} ~~ \lambda, \mu \in m \smallsetminus
\{\kappa\}, ~~ \textrm{if} ~~ \lambda \not{R} \mu~~ \textrm{then} ~~
f_{\lambda \mu } = g_{\lambda \mu}.
\end{eqnarray*}
For any $ \kappa, \lambda <m$, set
$$ E_{\kappa \lambda} = \{ ( R, f) \in G(m,n) : \kappa R \lambda \}.$$
\end{definition}

Monk proves that this indeed defines a cylindric atom structure, he defines
the $m$ dimensional cylindric algebra $\C(m,n)=\Ca(\G(m,n)),$ then he proves:

\begin{theorem}
\begin{enumarab}
\item For $3\leq m\leq n<\omega$ and $n-1\leq \mu< \omega,$ $\Nr_m\C(n,\mu)\cong \C(m,\mu)$.
In particular, $\C(m, m+k)\cong \Nr_m(\C(n, n+k))$ for every finite $k$.
\item Let $x_n=\{(R,f)\in G(n, n+k): R=(R\cap {}^2n) \cup (Id\upharpoonright {}^2(n\sim m))\\
\text { for all $u, v$,} uRv, f(u,v)\in n+k,
\text { and
for all } \mu\in n\sim m, v<\mu,\\ f(\mu, v)=\mu+k\}.$

Then
$\C(n, n+k)\cong \Rl_x\Rd_n\C(m, m+k).$
\end{enumarab}
\end{theorem}
\begin{demo}{Proof} \cite[theorems 3.2.77-3.2.86]{HMT2}.
\end{demo}
\begin{theorem} The class $\RCA_{\alpha}$ is not axiomatized by a finite schema.
\end{theorem}
\begin{demo}{Proof} By theorem \ref{Monk} and $\RCA_{\alpha}=S\Nr_{\alpha}\CA_{\alpha+\omega}$.
\end{demo}

Johnsson defined a polyadic atom structure based on the structure $\G(m,n)$. First a helpful piece of notation:
For relations $R$ and $G$, $R\circ G$ is the relation
$$\{(a,b): \exists c (a,c)\in R, (c, b)\in S\}.$$
Now Johnson extended the atom structure $\G(m,n)$ by

$(R,f)\equiv_{ij}(S,g)$ iff $f(i,j)=g(j,i)$ and if $(i,j)\in R$, then $R=S$, if not, then $R=S\circ [i,j]$, as composition of relations.

Strictly speaking, Johnsson did not define substitutions quite in this way; because he had all finite transformations, not only transpositions.
Then, quasi-polyadic algebras was not formulated in schematizable form, a task accomplished by Sain and Thompson \cite{ST} much later; and
indeed they were able to prove the following:

\begin{theorem}(Sain-Thompson)\cite{ST} For any infinite ordinal $\alpha$,
the classes ${\sf RQPA}_{\alpha}$ and ${\sf RQPEA}_{\alpha}$ are not finite schema axiomatizable.
\end{theorem}
\begin{demo}{Proof} One proof uses the fact that ${\sf RQPA}_{\alpha}=S\Nr_{\alpha}{\sf QPA}_{\alpha+\omega}$, and that the diagonal free reduct
of Monk's algebras (hence their infinite dilations) are not representable. Another can be distilled from \cite{t}; this will be elaborated on below.
A completely analogous result holds for Pinter's algebras ${\sf Sc}$s,
using also finite dimensional Pinter's algebras, which are the reducts of Monk's algebras described
above, satisfying the hypothesis of theorem \ref{2.12}; a finer result is also
found in \cite{t}.
\end{demo}
Sain Thompson's paper also contains the result that the class of representable quasi polyadic equality algebras {\it with equality}
of infinite dimension cannot be axiomatized by any set of universal formulas containing only finitely
many variables. However, the proof has a serious gap, that was partially corrected by the present author in \cite{splitting}
addressing only countable (infinite) dimensions; here we extend the correction to  the uncountable case, too.

\subsection{Hirsch and Hodkinson's Monk-like  algebras}

Now we apply the lifting argument to existing finite dimensional algebras.
We prove the conclusion of theorem \ref{2.12}, for cylindric algebras
solving the infinite dimensional version of the famous 2.12 problem in algebraic logic.
We use {\it existing} Monk-like algebras, that
renders a stronger result than the one proved in \cite{t};
though the latter has the bonus of proving an analogous weaker non-finite axiomatizability result for infinite dimensions
for {\it various diagonal free reducts} of cylindric algebras and quasi-polaydic equality ones.
In \cite{t} Monk's lifting-argument, or rather its refined version formulated in theorem \ref{2.12}, was also used.

The finite dimensional algebras we use are constructed by Hirsch and Hodkinson; and
they are based on a relation algebra construction, having $n$ dimensional hyperbasis.
Such combinatorial algebras have affinity with what has become to be know in the literature
as Monk's algebras and/or Maddux's algebras.

Related algebras were constructed in \cite{t}.
We recall the construction of Hirsch and Hodkinson in \cite[section 15.2, starting p.466]{HHbook}.
They prove their result for cylindric algebras. Here, by noting that
their atom structures are also symmetric; it permits expansion by substitutions,
so we can slightly extend the result to polyadic equality algebras, too.
Then using theorem \ref{2.12} we prove a strong non-finite axiomatizability result for quasi-polyadic equality
algebras, witness theorem \ref{new}.

\begin{definition}\label{relation}\cite[section 15.2, starting p. 466]{HHbook}.
Define  relation algebras $\A(n,r)$ having two parameters $n$ and $r$ with $3\leq n<\omega$ and $r<\omega$.
Let $\Psi$ satisfy $n,r\leq \Psi<\omega$. We specify the atom structure of $\A(n,r)$.
\begin{itemize}
\item The atoms of $\A(n,r)$ are $id$ and $a^k(i,j)$ for each $i<n-1$, $j<r$ and $k<\psi$.
\item All atoms are self converse.
\item We can list the forbidden triples $(a,b,c)$ of atoms of $\A(n,r)$- those such that
$a.(b;c)=0$. Those triples that are not forbidden are the consistent ones. This defines composition: for $x,y\in A(n,r)$ we have
$$x;y=\{a\in \At(\A(n,r)); \exists b,c\in \At\A(n,r): b\leq x, c\leq y, (a,b,c) \text { is consistent }\}$$
Now all permutations of the triple $(Id, s,t)$ will be inconsistent unless $t=s$.
Also, all permutations of the following triples are inconsistent:
$$(a^k(i,j), a^{k'}(i,j), a^{k''}(i,j')),$$
if $j\leq j'<r$ and $i<n-1$ and $k,k', k''<\Psi$.
All other triples are consistent.
\end{itemize}
\end{definition}
Hirsch and Hodkinson invented means to pass from
relation algebras to $n$ dimensional cylindric algebras,
when the relation algebras in question have what they call a hyperbasis, cf. \cite[definition 12.11]{HHbook}.
$n$ dimensional hyperbasis generalize the notion of Madux's cylindric basis, by allowing labelled hyperedges of arbitrary finite
lengths.

Unless otherwise specified,
$\A=(A,+,\cdot, -,  0,1,\breve{} , ;, \Id)$
will denote an arbitrary relation algebra with $\breve{}$ standing for converse, and $;$ standing for composition, and
$\Id$ standing for the identity relation.

\begin{definition}\cite[definition12.21]{HHbook}. Let $3\leq m\leq n\leq k<\omega$, and let $\Lambda$ be a non-empty set (of labels).
An $n$ wide $m$ dimensional $\Lambda$ hypernetwork over $\A$ is a map
$N:{}^{\leq n}m\to \Lambda\cup \At\A$ such that $N(\bar{x})\in \At\A$ if $|\bar{x}|=2$ and $N(\bar{x})\in \Lambda$ if $|\bar{x}|\neq 2$,
with the following properties:
\begin{itemize}
\item $N(x,x)\leq \Id$ ( that is $N(\bar{x})\leq \Id$ where $\bar{x}=(x,x)\in {}^2n.)$

\item $N(x,y)\leq N(x,z);N(z,y)$
for all $x,y,z<m$
\item If $\bar{x}, \bar{y}\in {}^{\leq n}m$, $|\bar{x}|=|\bar{y}|$
and $N(x_i,y_i)\leq \Id$ for all $i<|\bar{x}|$, then $N(\bar{x})=N(\bar{y})$

\item when $n=m$, then $N$ is called an $n$ dimensional $\Lambda$ hypernetwork.

\end{itemize}
\end{definition}

\begin{definition}\cite[definitions 12.11, 12.21]{HHbook}.  Let $M,N$ be $n$ wide $m$ dimensional $\Lambda$ hypernetworks.
\begin{enumarab}
\item For $x<m$ we write $M\equiv_xN$ if $M(\bar{y})=N(\bar{y})$ for all $\bar{y}\in {}^{\leq n}(m\sim \{x\})$
\item More generally, if $x_0,\ldots, x_{k-1}<m$ we write $M\equiv_{x_0,\ldots,x_{k-1}}N$
if $M(\bar{y})=N(\bar{y})$ for all $\bar{y}\in {}^{\leq n}(m\sim \{x_0,\ldots, x_{k-1}\}).$
(Using an expression of Maddux, here $M$ and $N$ agree off of
$\{x_0,\ldots x_{n-1}\}$.)
\item If $N$ is an $n$ wide $m$ dimensional $\Lambda$ -hypernetwork over $\A$, and $\tau:m\to m$ is any map, then
$N\circ \tau$ denotes the $n$ wide $m$ dimensional $\Lambda$ hypernetwork over $\A$ with labellings defined by
$$N\circ \tau(\bar{x})=N(\tau(\bar{x})) \text { for all }\bar{x}\in {}^{\leq n}m.$$
That is
$$N\circ \tau(\bar{x})=N(\tau(x_0),\ldots,\tau(x_{l-1}))$$
\end{enumarab}
\end{definition}

\begin{lemma} Let $N$ be an $n$ dimensional $\Lambda$ hypernetwork over $\A$ and $\tau:n\to n$ be a map.
Then $N\circ \tau$, which we often denote simply by $N\tau$, is also a network.
\end{lemma}
\begin{demo}{Proof} \cite[lemma 12.7]{HHbook}.
\end{demo}

\begin{definition}\cite[definition 12.21]{HHbook2}.
The set of all $n$ wise $m$ dimensional hypernetworks will be denoted by $H_m^n(\A,\Lambda)$.
An $n$ wide $m$ dimensional $\Lambda$
hyperbasis for $\A$ is a set $H\subseteq H_m^n(\A,\lambda)$ with the following properties:
\begin{itemize}
\item For all $a\in \At\A$, there is an $N\in R$ such that $N(0,1)=a$
\item For all $N\in R$ all $x,y,z<n$ with $z\neq x,y$ and for all $a,b\in \At\A$ such that
$N(x,y)\leq a;b$ there is $M\in R$ with $M\equiv_zN, M(x,z)=a$ and $M(z,y)=b$
\item For all $M,N\in H$ and $x,y<n$, with $M\equiv_{xy}N$, there is $L\in H$ such that
$M\equiv_xL\equiv_yN$. Here the notation $M\equiv_S N$ means that, using Maddux's terminology that $M$ and $N$ agree off of $S$.

\item For a $k$ wide $n$ dimensional hypernetwork $N$, we let $N|_m^k$ the restriction of the map $N$ to $^{\leq k}m$.
For $H\subseteq H_n^k(\A,\lambda)$ we let $H|_k^m=\{N|_m^k: N\in H\}$.
\item When $n=m$, $H_n(\A,\Lambda)$ is called an $n$ dimensional hyperbases.
\end{itemize}
We say that $H$ is symmetric, if whenever $N\in H$ and $\sigma:m\to m$, then $N\circ\sigma\in H$.
\end{definition}
We note that $n$ dimensional hyperbasis are extensions of Maddux's notion of cylindric basis.

\begin{theorem} If $H$ is a $m$ wide $n$ dimensional $\Lambda$ symmetric
hyperbases for $\A$, then $\Ca H\in {\sf PEA}_n$.
\end{theorem}
\begin{demo}{Proof}
Let $H$ be the set of $m$ wide $n$ dimensional $\Lambda$ symmetric hypernetworks for $\A$.
The domain of $\Ca(H)$ is $\wp(H)$.
The Boolean operations are defined as expected (as complement and union of sets). For $i,j<n$ the diagonal is defined by
$${\sf d}_{ij}=\{N\in H: N(i,j)\leq \Id\}$$
and for $i<n$ we define the cylindrifier ${\sf c}_i$ by
$${\sf c}_iS=\{N\in H: \exists M\in S(N\equiv_i M\}.$$
Now the polyadic operations are defined by
$${\sf p}_{ij}X=\{N\in H: \exists M\in S(N=M\circ [i,j])\}$$

Then $\Ca(H)\in {\sf PEA}_n$. Furthermore, $\A$ embeds into
$\Ra(\Ca(H))$ via
$$a\mapsto \{N\in H: N(0,1)\leq a\}.$$
\end{demo}

\begin{theorem} Let $3\leq m\leq n\leq k<\omega$ be given.
Then $\Ca(H|^k_m)\cong \Nr_m(\Ca(H))$
\end{theorem}
\begin{demo}{Proof}\cite[theorem 12.22]{HHbook}
\end{demo}

The set $H=H_n^{n+1}(\A(n,r), \Lambda)$ of all $(n+1)$ wide $n$
dimensional $\Lambda$ hypernetworks over $\A(n,r)$ is an $n+1$ wide $n$
dimensional {\it symmetric} $\Lambda$ hyperbasis.
$H$ is symmetric, if whenever $N\in H$ and $\sigma:m\to m$, then $N\circ\sigma\in H$.
Hence $\A(n,r)$ embeds into the $\sf Ra$ reduct of $\Ca(H).$
\begin{theorem}\label{thm:cmnr1}

Let $\A(n,r)$ be as in definition \ref{relation}. Assume that $3\leq m\leq n$, and let
$$\C(m,n,r)=\Ca(H_m^{n+1}(\A(n,r),  \omega)).$$ Then the following hold assumptions hold (the first two items are the assumptions
in theorem \ref{2.12})

\begin{enumarab}

\item For any $r$ and $3\leq m\leq n<\omega$, we
have $\C(m,n,r)\in \Nr_m{\sf PEA}_n$ and $\Rd_{ca}\C(m,n,r)\notin S\Nr_m{\sf CA_{m+1}}$
and $\Pi_r\C(m,n,r)\in \Nr_n\PEA_{m+1}$.

\item  For $m<n$ and $k\geq 1$, there exists $x_n\in \C(n,n+k,r)$ such that $\C(m,m+k,r)\cong \Rl_{x}\C(n, n+k, r).$

\item For any ordinal $\alpha>2$, possibly infinite, $S\Nr_{\alpha}\CA_{\alpha+k+1}$ is not axiomatizable by a finite schema
over $S\Nr_{\alpha}\CA_{\alpha+k}.$
An analogous result holds for quasi-polyadic equality algebras. In fact, there exists
$\B^r\in \Nr_{\alpha}\sf QEA_{\alpha+k}$
with $\Rd_{ca}\B^r\notin S\Nr_{\alpha}\CA_{\alpha+k}$,
such that $\Pi_r \B/F\in S\Nr_{\alpha}\sf QEA_{\alpha +k +1}$
for any non principal ultrafilter $F$ on $\omega$.
\end{enumarab}
\end{theorem}
\begin{proof}
\begin{enumarab}
\item $H=H_n^{n+1}(\A(n,r), \omega)$ is a wide $n$ dimensional $\omega$ symmetric hyperbases, so $\Ca H\in {\sf PEA}_n.$
But $H_m^{n+1}(\A(n,r),\omega)=H|_m^{n+1}$.
Thus
$$\C_r=\Ca(H_m^{n+1}(\A(n,r), \omega))=\Ca(H|_m^{n+1})\cong \Nr_m\Ca H$$
\item From \cite[theorem15.1 (4)]{HHbook}
\item For $m<n$, let $$x_n=\{f\in F(n,n+k,r): m\leq j<n\to \exists i<m f(i,j)=Id\}.$$
Then $x_n\in \c C(n,n+k,r)$ and ${\sf c}_ix_n\cdot {\sf c}_jx_n=x_n$ for distinct $i, j<m$.
Furthermore
\[{I_n:\c C}(m,m+k,r)\cong \Rl_{x_n}\Rd_m {\c C}(n,n+k, r).\]
via
$$I_n(S)=\{f\in F(n, n+k, r): f\upharpoonright m\times m\in S,$$
$$\forall j(m\leq j<n\to  \exists i<m\; f(i,j)=Id)\}.$$

\item Follows from theorem \ref{2.12}.
\end{enumarab}
\end{proof}

Now we can recover Monk's and Sain and Thompson's results, from the following reasoning
which also involves a lifting argument.
We know that $\Rd_{ca}\C^r$ is not in $S\Nr_n\CA_{n+1},$ hence not
representable, but we actually have $\Pi_{r\in \omega}\C^r\in \sf RQEA_{n}$.
This can be proved by showing that \pe\ has a \ws\  in the $\omega$ rounded game of an elementary
substructure of $\Pi \C_r/F$, as we proceed to show.

However, in item (1) above, we will need and use that the
ultraproduct is actually {\it  representable}, so that it neatly embeds into $\omega$ extar dimensions, not just $k+1$.
This is stated in \cite{HHbook} with several scattered hints formulated in the text or in the exercises.
We fill in some (but not all) of the details.

Let $G_{\omega}^k$ be the usual atomic game defined on atomic networks with $k$ pebbles and $\omega$
rounds \cite[definition 11.1, lemma 11. 2, definition 11.3]{HHbook}.

We first address the relation algebras on which the $\C_r$s are based.
Given $k$, then we show that for any $r\geq  k^2$, we have \pe\ has a \ws\ in $G_{\omega}^k$
in $\A(n,r)$. This implies using ultraproducts and an elementary chain argument that \pe\
has a \ws\ in the $\omega$ rounded game in an elementary substructure of $\Pi\A(n,r)/F$,
hence the former is representable and then so is the latter because
${\sf RRA}$ is a variety.

Consider the only move by \pa\, namely, a triangle move $N(x,y,z,a,b).$
To label the edges $(w,z)$ where $w\neq k$, $w\neq x,y,z$, \pe\ uses $a^0(i, j_{w})$ where $i<n-1$
and $a,b$ not less that $a(i)$, where
the numbers $j_w(w\in k\in \{x,y,z\}$ are distinct elements
of $\{j<r: \neg \exists u,v\in k\sim {z}(N(u,v)\leq a(i,j))\}$. Because
$r$ {\it is} large enough,  this set has size at least $|k\sim \{x,y,z\}|$ and so
it  is impossible to find $j_w$. Then
any triangle labelled by $a^k(i,j)$ $a^{k'}(i,j')$ and
$a^{k''}(i, j'')$, the indices $j', j, j''$ are distinct and so
the triangle is consistent.

To show that $\Pi \C^r/F$ is also representable, it suffices to find a representation of
$\A\prec \Pi \A(n,r)/F$  that embeds {\it all}
$m$ dimensional hypernetworks, respecting $\equiv_i$ for all $i<m$.

We construct such a representation in a step-by- step fairly standard manner.
Our base representation $M_0$ may not embed all $m$ dimensional hypernetworks; we extend to one that
does.
Assume that $M$ is a relation algebra representation of $\A$.

Exists by the above argument.
We build a chain of hypergraphs $M_t:t<\omega$ their limit (defined in a precise
sense) will as required.  Each $M_t$ will have hyperedges of length $m$ labelled by
atoms of $\A$, the labelling of other hyperedges will be done later.

We proceed by a step by step manner, where the defects are treated one by one, and they are diffused at the limit obtaining the required
hypergraph, which also consists of
two parts, namely, edges and hyperedges \cite[proposition 13.37]{HHbook2}.

This limiting hypergraph will be compete (all  edges and hypergraphs will be labelled;
which might not be the case with $M_t$ for $t<\omega$.)
Every atom-hyperedge will indeed be labelled by an atom of $\A$ and hyperedges with length
$\neq  m$ will also be labelled. Let $M_0=M$.

We require inductively that $M_t$  satisfies:

Any $m$ tuple of  $M_t$ is contained in $\rng(v)$ for some $N\in H$ and some embedding $v:N\to  M_t$.
such that this embedding satisfies the following two conditions:

(a) if $i,j<m$  an  edge of $M_t$, then
$M_t(v(i) ,v(j))=N(i,j),$

(b) Whenever $a\in {}^{\leq ^m} m$ with $|a|\neq 2$, then $v(a)$ is a
hyperedge of $M_t$ and is labelled by $N(\bar{a})$.

(Note that an embedding might not be injective).

We assume that during the game all  hypernetworks played are strict, that is if
$N(x,y)\leq {\sf Id}$, then $x=y$.
For the base of the induction we take   $M_0=M$. Assume that we have built $M_t$ $(t<\omega)$.
We proceed as follows.

We define $M_{t+1}$ such that for every quadruple
$(N, v, k, N')$ where $N, N'\in H$, $k<m$ and $M\equiv_k N'$ and $v:N\to M_t$
is an embedding, then  the restriction $v\upharpoonright n\sim \{k\}$ extends to an embedding
$v'; N'\to M_{t+1}$

\begin{itemize}

\item For each such $(N, v,k, N')$ we add just one new node $\pi$,
and we add the edges $(\pi, v(i))$ $(v(i), \pi)$ for each $i\in m\sim\{k\}$ and $(\pi, \pi)$
that are pairwise distinct. Such new edges are labeled by $N'(k,k)$, $(\pi, v(i))$ and $(v(i), \pi)$. This is well defined.
We extend $ v$ by defining $v'(k)=\pi$.

\item We add a new hyperedge $v'(\bar{a})$ for every $\bar{a}$  of length $\neq m$,
with $k\in \rng (\bar{a})$ and give it the label $N'(\bar{a})$.

\end{itemize}
Then $M_{t+1}$ will be $M_t$ with its old labels, atom hyperedges,  labelled hyperedges
together with the new ones define as above.

It is straightforward to check that the inductive hypothesis  are preserved.

Now define a labelled hypergraph as follows:
$\nodes(M)=\bigcup \nodes(M_t)$,
for any $\bar{x}$ that is an atom-hyperedge; then it is a one in some $M_t$
and its label is defined in $M$ by $M_t(\bar{x})$

The hyperedges are $n$ tuples $(x_0,\ldots, x_{m-1})$.
For each such tuple, we let $t<\omega$,
such that $\{x_0\ldots, x_{m-1}\} \subseteq  M_t$,
and we set $M(x_0,\ldots, x_{m-1})$   to be the unique
$N\in H$ such that there is an embedding
$v:N\to M$ with  $\bar{x}\subseteq \rng(v).$
This can be easily checked to be well defined. Indeed, uniqueness
is clear from the definition of embedding. Now we show existence. Note that there is an $N\in H$
and an embedding $v:N\to M_t$ with $(x_0, \ldots, x_{m-1})\subseteq \rng(v)$.
So take $\tau:m\to m$, such that $x_i=v(\tau(i))$ for each $i<m$.
As $H$ is symmetric, $N\tau\in H$, and clearly $v\circ \tau: N\tau\to M_t$
is also an embedding. But $x_i=v\circ \tau(u)$ for each $i<m$,
then let $M(x_0,\ldots, x_{m-1})=N\tau.$

Now we we define the representation $M$.
Let $L(A)$ be the signature obtained by adding an a binary relation symbol for each element
of $\A$.
Define $M\models r(\bar{x})$ if $\bar{x}$ is an edge and $M(\bar{x})\leq r.$
This representation embeds every $m$ hypernetwork.

The hyperedges are $m$ tuples $(x_0,\ldots, x_{m-1})$.
For each such tuple, we let $t<\omega$,
such that $\{x_0\ldots, x_{m-1}\} \subseteq  M_t$,
and we set $M(x_0,\ldots, x_{m-1})$   to be the unique
$N\in H$ such that there is an embedding
$v:N\to M$ with  $\bar{x}\subseteq \rng(v).$
This can be easily checked to be well defined.

Now we we define the representation $M$ of $\C_r$ that embeds all $m$ hypernetworks.
We drop the subscript $r$ to simplify the notation.
Let $L(C)$ be the signature obtained by adding an $n$ ary relation symbol for each element
of $\C$.
Now define for $r\in L(C)$
$$M\models r(\bar{x})\text { iff  } M(\bar{x})\in  r.$$
(Note that $M(\bar{x})$
is a hypernetwork, while $r$ is as set of hypernetworks).

Now we review the result in \cite{t}. The third item in our coming theorem \ref{thm:cmnr},
which is the main theorem in \cite{t}, is {\it strictly weaker} than the second item in theorem \ref{thm:cmnr1}. 
In the latter the ultraproduct was proved representable; in theorem \ref{thm:cmnr}, it is proved that the ultraproduct 
can make it only to one extra dimension, not infinitely many. 

In fact, the algebras used are Monk-like 
{\it finite} polyadic equality algebras, 
while the cylindric and polyadic equality algebras constructed before were {\it infinite} since they were constructed 
from hyperbasis having infinitely
many hyperlabels. 

Nevertheless,  in all cases the 
construction can be based on {\it the same relation algebra}, namely a variant of the relation algebra $\A(n,r)$
as specified above,  
except that now the finite atom structure $m$ dimensional 
polyadic equality algebra formed consists of all 
basic matrices of dimension $m$. No labels are involved so
the algebras are finite. We know from the above that $\A(n,r)$ has an $n+1$ wide 
$m$  
dimensional hyperbasis; in particular it has 
an $m$ dimensional hyperbasis,  
and so it has  an $m$ dimensional 
symmetric cylindric basis, that is, a polyadic equality basis $(m<n).$ 

Such basis are the atom structures of the new finite 
$m$ dimensional polyadic equality algebras. 
The construction of such algebras is due to Robin Hirsch.

\begin{theorem}\label{thm:cmnr} Let $3\leq m\leq n$ and $r<\omega$.
\begin{enumerate} 
\renewcommand{\theenumi}{\Roman{enumi}}
\item $\C(m, n, r)\in \Nr_m\QPEA_n$,\label{en:one}
\item $\Rd_{\Sc}\C(m, n, r)\not\in S\Nr_m\Sc_{n+1}$, \label{en:two}
\item $\Pi_{r/U} \C(m, n, r)$ is elementarily equivalent to a countable polyadic equality algebra $\C\in\Nr_m\QPEA_{n+1}$.  \label{en:four}
\end{enumerate} 
\end{theorem}
We define the algebras $\C(m,n,r)$ for $3\leq m\leq n<\omega$ and $r$ 
and then give a sketch of \eqref{en:two}. To prove \eqref{en:four}, $r$ has to be a linear order.
We start with:
\begin{definition}\label{def:cmnr}
Define a function $\kappa:\omega\times\omega\rightarrow\omega$ by $\kappa(x, 0)=0$ 
(all $x<\omega$) and $\kappa(x, y+1)=1+x\times\kappa(x, y))$ (all $x, y<\omega$).
For $n, r<\omega$ let 
\[\psi(n, r)=
\kappa((n-1)r, (n-1)r)+1.\]
All of this is simply to ensure that $\psi(n, r)$ is sufficiently big compared to $n, r$ for the proof of non-embeddability to work.  
The second parameter $r<\omega$ may be considered as a finite linear order of length $r$.    
We may extend the definition of $\psi$ to the case where its second parameter is an 
arbitrary linear order by letting $\psi(n, r)=\omega$ for any infinite linear order $r$. 

For any  $n<\omega$ and any linear order $r$, let 
\[Bin(n, r)=\set{Id}\cup\set{a^k(i, j):i< n-1,\;j\in r,\;k<\psi(n, r)}\] 
where $Id, a^k(i, j)$ are distinct objects indexed by $k, i, j$.
Let $3\leq m\leq n<\omega$ and let $r$ be any linear order.
Let $F(m, n, r)$ be the set of all  functions $f:m\times m\to Bin(n, r)$ 
such that $f$ is symmetric ($f(x, y)=f(y, x)$ for all $x, y<m$) 
and for all $x, y, z<m$ we have $f(x, x)=Id,\;f(x, y)=f(y, x)$, and $(f(x, y), f(y, z), f(x, z))\not\in Forb$, 
where $Forb$ (the \emph{forbidden} triples) is the following set of triples
 \[ \begin{array}{c}
 \set{(Id, b, c):b\neq c\in Bin(n, r)}\\
 \cup \\
 \set{(a^k(i, j), a^{k'}(i,j), a^{k^*}(i, j')): k, k', k^*< \psi(n, r), \;i<n-1, \; j'\leq j\in r}.
 \end{array}\]
Here $Bin(n,r)$ is an atom structure of a finite relation relation  
and $Forb$ specifies its operations by specifying forbidden triples. 

This atom structure defines a relation algebra; 
that is very similar but not identical to $\A(n,r)$. There the upperbound was $nr^{nr}$ here it is specified by the Ramsey function 
$\psi$.  

Now any such $f\in F(m,n,r)$ is a basic matrix on this atom structure 
in the sense of Maddux, and the whole lot of them will be a 
basis, constituting the atom structure of algebras we want.
So here instead of an infinite set of hypernetworks with hyperlabels from $\omega$ our cylindric-like algebras are finite.

Now accessibility relations coresponding to substitutions, cylindrifiers are defined, as expected on matrices,  
as follows.
For any $f, g\in F(m, n,  r)$ and $x, y<m$ we write $f\equiv_{xy}g$ if for all $w, z\in m\setminus\set {x, y}$ we have $f(w, z)=g(w, z)$.  
We may write $f\equiv_x g$ instead of $f\equiv_{xx}g$.  For $\tau:m\to m$ we write $(f\tau)$ for the function defined by 
\begin{equation}\label{eq:ftau}(f\tau)(x, y)=f(\tau(x), \tau(y)).
\end{equation}
Clearly $(f\tau)\in F(m, n, r)$. 
Accordingly, the universe of $\C(m, n, r)$ is the power set of $F(m, n, r)$ and the operators (lifting from the atom structure)
are
\begin{itemize}
\item  the Boolean operators $+, -$ are union and set complement, 
\item  the diagonal $\diag xy=\set{f\in F(m, n, r):f(x, y)=Id}$,
\item  the cylindrifier $\cyl x(X)=\set{f\in F(m, n, r): \exists g\in X\; f\equiv_xg }$ and
\item the polyadic $\s_\tau(X)=\set{f\in F(m, n, r): f\tau \in X}$,
\end{itemize}
for $x, y<m,\;  X\subseteq F(m, n, r)$ and  $\tau:m\to m$.
 \end{definition}
\medskip

\begin{enumarab}

\item It is straightforward to see 
that  $3\leq m,\; 2\leq n$ and $r<\omega$ 
the algebra $\C(m, n, r)$ satisfies all of the axioms defining $\PEA_m$
except, perhaps, the commutativity of cylindrifiers $\cyl x\cyl y(X)=\cyl y\cyl x(X)$, which it satisfies because
$F(m,n,r)$ is a symmetric cylindric basis, so that overlapping 
matrices amalgamate.
Furthermore, if  $3\leq m\leq m'$ then $\C(m, n, r)\cong\Nr_m\C(m', n, r)$
via $$X\mapsto \set{f\in F(m', n, r): f\restr{m\times m}\in X}.$$

\item If we replace $nr^{nr}$ in the previous case  by the new Ramsey function 
$\psi(n,r)$  
then the relation algebra used before will be identical to one introduced here,
which we denote, by a slight abuse of notation, by $\A(n,r)$.

Furthermore, this change will not alter our previous result. 
The new ultraproduct 
$\Pi_{r/U}\A(n,r)$ is still representable; one proves that \pe\ can win the game $G^k_{\omega}$ when $r\geq k^2$
as above,  and so is the infinite 
polyadic equality algebra $\C_r=\Ca(H_m^{n+1}(\A(n,r), \omega)$ 
is also representable. Its representability, can be proved, 
by finding a representation of an elementary countable subalgebra of $\Pi_{r/U} \A(n,r)$
that embeds all $m$ dimensional hypernetworks, and then as above, using this representation one defines a representation 
of a countable elementary subalgebra of $\C_r$, which immediately implies that $\C_r$ 
is representable (as a polyadic equality algebra of dimension $m$).

We give a sketch of proof of \ref{thm:cmnr}(\ref{en:two}), which is the heart and soul of the proof, and it is quite similar 
to its $\CA$ analogue 4.69-475 in \cite{HHbook2}. We will also refer to the latter when the proofs overlap, or are very similar.

Assume for contradiction  that 
$\Rd_{\Sc}\C(m, n, r)\subseteq\Nr_m\C$ 
for some $\C\in \Sc_{n+1}$, some finite $m, n, r$. 
Then it can be shown inductively 
that there must be a large  set $S$ of distinct elements of $\C$, 
satisfying certain inductive assumptions, which we outline next.  
For each $s\in S$ and $i, j<n+2$ there is an element $\alpha(s, i, j)\in Bin(n, r)$ obtained from $s$ 
by cylindrifying all dimensions in $(n+1)\setminus\set{i, j}$, then using substitutions to replace $i, j$ by $0, 1$.  
Then one shows that $(\alpha(s, i, j), \alpha(s, j, k), \alpha(s, i, k))\not\in Forb$.

The induction hypothesis say, most importantly, that $\cyl n(s)$ is constant, for $s\in S$, 
and for $l<n$  there are fixed $i<n-1,\; j<r$ such that for all $s\in S$ we have $\alpha(s, l, n)\leq a(i, j)$.  
This defines, like in the proof of theorem 15.8 in \cite{HHbook2} p.471, two functions $I:n\rightarrow (n-1),\; J:n\rightarrow r$ 
such that $\alpha(s, l, n)\leq a(I(l), J(l))$ for all $s\in S$.  The \emph{rank} ${\sf rk}(I, J)$ of $(I, J)$ (as defined in definition 15.9 in \cite{HHbook2}) is   
the sum (over $i<n-1$) of the maximum $j$ with $I(l)=i,\; J(l)=j$ (some $l<n$) or $-1$ if there is no such $j$.  

Next it is proved that there is a set $S'$ with index functions $(I', J')$, still relatively large 
(large in terms of the number of times we need to repeat the induction step) 
where the same induction hypotheses hold but where ${\sf rk}(I', J')>{\sf rk}(I, J)$.  (See \cite{HHbook2}, where for $t<nr$, $S'$ was denoted by $S_t$
and proof of property (6) in the induction hypothesis  on p.474 of \cite{HHbook2}.)

By repeating this enough times (more than $nr$ times) we obtain a non-empty set $T$ 
with index functions of rank strictly greater than $(n-1)\times(r-1)$, an impossibility. (See \cite{HHbook2}, where for $t<nr$, $S'$ was denoted by $S_t$.)

We sketch the induction step.  Since $I$ cannot be injective there must be distinct $l_1, l_2<n$ 
such that $I(l_1)=I(l_2)$ and $J(l_1)\leq J(l_2)$.  We may use $l_1$ as a "spare dimension" 
(changing the index functions on $l$ will not reduce the rank).  
 Since $\cyl n(s)$ is constant, we may fix $s_0\in S$ 
and choose a new element $s'$ below $\cyl l s_0\cdot \sub n l\cyl  l s$, 
with certain properties.  Let $S^*=\set{s': s\in S\setminus\set{s_0}}$.
We wish to re-establish the induction hypotheses for $S^*$, and many of these are simple to check.  
Although suitable functions $I', J'$ may not exist on the whole of $S$, but $S$ remains
large enough to enable selecting a 
subset $S'$ of $S^*$, still large in terms of the number of remaining times the induction step must be applied.
The required functions $I', J'$ now exist (for all but one value of $l<n$ the values $I'(l), J'(l)$ are determined by $I, J$, 
for  one value of $l$ there are at most $(n-1)r$ possible values, hence on a large subset the choices agree).  

Next it can be shown that $J'(l)\geq J(l)$ for all $l<n$.   Since 
$$(\alpha(s, i, j), \alpha(s, j, k), \alpha(s, i, k))\not\in Forb$$ 
and by the definition of $Forb$  
either $\rng(I')$ properly extends $\rng(I)$ or there is $l<n$ such that $J'(l)>J(l)$, hence  ${\sf rk}(I', J')>{\sf rk} (I, J)$.

Although in the previous sketch the construction is based on the same relation algebra $\A(n,r)$ as before,
we do not guarantee that the ultraproduct on $r$
of $\C(m,n,r)$ ($2<m<n<\omega)$ based on $\A(n,r)$, like in the cylindric case, is representable.
It is not all clear that we can lift the 
representability $\Pi_r\A(n,r)$ to $\Pi_r\C(m,n,r)$ for any $m\leq n<\omega$ 
which was the case when we had diagonal elements.
In fact, what was proved in \cite{t} is the weaker item (III) of the above theorem.

\item 
To prove the first two parts, the algebras we considered had 
an atom structure consisting of basic $m$ dimensional 
matrices, hence all the 
elements are generated by two dimensional elements.

In the third part 
elements are {\it essentially} three dimensional, 
so here we deal with three dimensional matrices, or tensors, if you like.

The parameter defining the algebras between two succesive dimensions, namely $r$ (which varies over $\omega$) is no longer
only a number; it has a 
linear order.  This order induces a third dimension. 

In this case the ultraproduct of the constructed 
$m$ dimensional algebras, where $m<n$, over $r$,
is only in $\Nr_m\K_{n+1}$ $(\K$ between $\Sc$ and $\PEA$),  and it is not clear that it is representable like in the $\CA$ case, 
when we have diagonal elements.

In other words, the ultraproduct could make it only to one extra dimension and {\it not}  to $\omega$ many, 
witness theorem \ref{thm:cmnr1}.

A standard L\"os argument shows that 
$\Pi_{r/U}\C(m, n, r) \cong\C(m, n, \Pi_{r/U} r)$ and $\Pi_{r/U}r$ 
contains an infinite ascending sequence. 
Here we will have to extend the definition of $\psi$ 
by letting $\psi(n, r)=\omega,$ for any infinite linear order $r$.

It  is proved in \cite{t} is the weaker item (III) of the above theorem, where it is proved that the infinite algebra 
$\C(m,n, J)\in \Nr_n\PEA_{m+1}$
when $J$ is an infinite linear order as above, and clearly $\Pi_{r/U} r$ is such. 

\end{enumarab}

Our next theorem is stronger than the result established for infinite dimension in \cite{t}
when restricted to cylindric and quasi-polyadic equality algebras.
However,  the latter has the bonus that it proves a weaker non-finite axiomatizability 
result for diagonal free algebras like Pinter's and quasi-polyadic algebras.

Here  the ultraproduct for the infinite dimensional case, like the finite dimensional case,
will be {\it representable}, not just neatly embedding in algebras have 
$\alpha+k+1$ spare dimensions. 

This will follow from the fact that the lifted algebras 
of finite dimension $>2$ have representable ultraproducts.
This is the statement, to be proved, in our next theorem restricted to finite dimensions $>2$.

\begin{theorem}\label{new} Let $\alpha >2$.  Let $\K$ be any of $\CA$ or $\QEA$, then for any $r\in \omega$, for any
$k\geq 1$, there exists $\B^{r}\in S\Nr_{\alpha}\K_{\alpha+k}\sim S\Nr_{\alpha}\K_{\alpha+k+1}$ such
$\Pi_{r\in \omega}\B^r\in {\sf RK}_{\alpha}$.
In particular, ${\sf RK}_{\alpha}$ is not axiomatizable
by a finite schema over $S\Nr_{\alpha}\K_{\alpha+k}$.
\end{theorem}

\begin{proof}

Let $k\in \omega$.
We have already dealt with the finite dimensional case.
Assume that $\alpha$ is infinite. We deal with $\CA$ for simplicity of notation. The more general case
is the same.
We know that
$S\Nr_{\alpha}\CA_{\alpha+k+1}\subset S\Nr_{\alpha}\CA_{\alpha+k};$
Fix $r\in \omega$. As before,
let $I=\{\Gamma: \Gamma\subseteq \alpha,  |\Gamma|<\omega\}$.
For each $\Gamma\in I$, let $M_{\Gamma}=\{\Delta\in I: \Gamma\subseteq \Delta\}$,
and let $F$ be an ultrafilter on $I$ such that $\forall\Gamma\in I,\; M_{\Gamma}\in F$.
For each $\Gamma\in I$, let $\rho_{\Gamma}$
be a one to one function from $|\Gamma|$ onto $\Gamma.$

Let ${\C}_{\Gamma}^r$ be an algebra similar to $\CA_{\alpha}$ such that
\[\Rd^{\rho_\Gamma}{\C}_{\Gamma}^r={\C}(|\Gamma|, |\Gamma|+k,r).\]
Let
\[\B^r=\prod_{\Gamma/F\in I}\C_{\Gamma}^r.\]
From theorem \ref{2.12}, we have:
\begin{enumerate}
\item\label{en:1} $\B^r\in S\Nr_\alpha\CA_{\alpha+k}$ and
\item\label{en:2} $\B^r\not\in S\Nr_\alpha\CA_{\alpha+k+1}$.
\end{enumerate}
But since the finite dimensional ultraproducts are now representable,
we also have
$$\Pi_{r/U}\Rd^{\rho_\Gamma}\C^r_\Gamma=\Pi_{r/U}\C(|\Gamma|, |\Gamma|+k, r) \subseteq \Nr_{|\Gamma|}\A_\Gamma, $$
for some $\A_\Gamma\in\CA_{|\Gamma|+\omega}$.

Following exactly the argument is in theorem \ref{2.12}, using also the same notation,
let $\lambda_\Gamma:|\Gamma|+k+1\rightarrow\alpha+k+1$
extend $\rho_\Gamma:|\Gamma|\rightarrow \Gamma \; (\subseteq\alpha)$ and satisfy
\[\lambda_\Gamma(|\Gamma|+i)=\alpha+i\]
for $i<k+1$.  Let $\F_\Gamma$ be a $\CA_{\alpha+\omega}$ type algebra such that $\Rd^{\lambda_\Gamma}\F_\Gamma=\A_\Gamma$.
As before, $\Pi_{\Gamma/F}\F_\Gamma\in\CA_{\alpha+\omega}$.  And
\begin{align*}
\Pi_{r/U}\B^r&=\Pi_{r/U}\Pi_{\Gamma/F}\C^r_\Gamma\\
&\cong \Pi_{\Gamma/F}\Pi_{r/U}\C^r_\Gamma\\
&\subseteq \Pi_{\Gamma/F}\Nr_{|\Gamma|}\A_\Gamma\\
&=\Pi_{\Gamma/F}\Nr_{|\Gamma|}\Rd^{\lambda_\Gamma}\F_\Gamma\\
&\subseteq\Nr_\alpha\Pi_{\Gamma/F}\F_\Gamma,
\end{align*}
By the neat embedding theorem, we have $\Pi_{\Gamma/F}\F_{\Gamma}\in \RCA_{\alpha}$
and we are done.
\end{proof}

\section{Splitting in infinite dimensions}

In a different direction, Andr\'eka \cite{Andreka}  generalized some of  the results excluding axiomatizations using finitely many variables
a task done by Jonsson for relation algebras. Andr\'eka using more basic methods,
avoiding totally the notion of schema. In this connection Andr\'eka invented the method of {\it splitting} \cite{Andreka}
to show that there is no universal axiomatization of
the class of representable cylindric algebras of any infinite dimension containing only finitely many variables.

This is stronger than Monk's result; furthermore, the method proved powerful enough to prove much more refined non-finite
axiomatizability results. Andr\'eka's result was apparently generalized to
quasi-polyadic algebras with equality by Sain and Thompson in their seminal
paper \cite{SL}, but there is a
serious error in the proof,
that was corrected by the present author, but only for countable dimensions.

The idea, traced back to Jonsson for relation algebras, consists of constructing for every finite $k\in \omega$
a non-representable algebra, all of whose $k$ -generated subalgebras are representable.

Andr\'eka ingeniously transferred such an idea to cylindric algebras, and to fully implement it, she invented
the {\it nut cracker} method of splitting.
The subtle splitting technique invented by Andr\'eka can be summarized as follows.

Take a fairly simple representable algebra generated by an atom, and  break
up or {\it split} the atom into enough (finitely many)  $k$ atoms,
forming a larger algebra, that is in fact non-representable; in fact, its cylindric reduct will not be representable, due to the  incompatibility
between the number of atoms, and the number of elements in the domain of a representation. However, the ''small" subalgebras
namely, those generated by $k$ elements of such an algebra will be
representable.

This incompatibility is usually witnessed by an application of Ramsey's theorems, like in Monk's first Monk-like algebra,
but this is not always the case (as we will show below)
like  Andr\'eka' splitting. This is definitely an asset, because the proofs in this  case are much  more basic,
and often even proving stronger results.
While Monk and rainbow algebras can prove subtle results concerning non finite axiomatizability for finite
dimensions, splitting works best in the infinite dimensional case.

We give two instances of the splitting technique.
The first answers a question of Andr\'eka's
formulated on p. 193 of \cite{Andreka} for any infinite ordinal
$\alpha$, and corrects a serious error in the seminal paper \cite{ST}. The result is proved for countable dimension
in \cite{splitting}. The extension to uncountable ordinals involves transfinite
induction.

\begin{theorem}\label{splitting1} Let $\alpha\geq \omega$. Then the variety $\RQEA_{\alpha}$
cannot be axiomatized with with a set $\Sigma$ of quantifier free
formulas containing finitely many variables. In fact, for any $k<\omega$, and any set of quantifier free formulas
$\Sigma$ axiomatizing $\RQEA_{\alpha}$,
$\Sigma$ contains a formula with more than $k$ variables in which some diagonal
element occurs. In particular, the variety $\RQEA_{\alpha}$ is not axiomatizable over its diagonal free reduct
with a set of universal formulas containing infinitely many variables.
\end{theorem}
\begin{proof}  The theorem is proved only for $\omega$ in \cite{splitting} (the proof generalizes to any countable ordinal without much ado).
This is an outline.
Fix $k\geq 1$. Algebras $\A_{k,n}$  for $n\in \omega\sim \{0\}$, with the following properties, are constructed:
$\A_{k,n}$ is of the form $(A_{k,n}, +, \cdot ,-, {\sf c}_i, {\sf s}_{\tau}, {\sf d}_{ij})_{i,j\in \omega, \tau\in G_n}$,
where $G_n$ is the symmetric group on $n$
satisfying the following:
\begin{enumarab}
\item $\Rd_{ca}\A_{k,n}\notin \RCA_{\omega}$.
\item Every $k$-generated subalgebra of $\A_{k,n}$ is representable.
\item  There is a one to one mapping $h:\A_{k,n}\to (\B(^{\omega}W), {\sf c}_i, {\sf s}_{\tau}, {\sf d}_{ij})_{i,j<\omega, \tau\in G_n}$
such that $h$ is a homomorphism with respect to all operations of $\A_{k,n}$ except for the
diagonal elements.
\end{enumarab}
Here $k$-generated means generated by $k$ elements.
Then it can be shown hat  for $n<m$, $\A_{k,n}$ is a subreduct (subalgebra of a reduct)
of $\A_{k,m}$.

Then the directed union $\A_k=\bigcup_{n\in \omega}\A_{k,n}$, was proved to be as required when the signature is countable.
Now to lift this result to possibly uncountable ordinals, we proceed inductively. Fix finite $k\geq 1$.
For any ordinal $\alpha\geq \omega$, we assume that for all $\kappa\in \alpha$, there is an algebra $\A_{k,\kappa}$
such that $\Rd_{ca}\A_{k,\kappa}\in \RCA_{\alpha}$, and it has all substitution operations corresponding to
transposition from $\kappa$,  namely,
${\sf s}_{[i,j]}$ $i,j\in \kappa$, satisfying all three conditions.

We further assume that  $\beta<\mu<\alpha$, then $\A_{k, \beta}$ is a subreduct of $\A_{k, \mu}$.
The base of the induction is valid.

Now, as in the countable case, take $\A_k=\bigcup_{\mu\in \alpha}\A_{k,\mu}$, and define the operations the natural way; this limit is of the
same similarity type as ${\sf QPEA}_{\alpha}$, and is a well defined algebra, since $\alpha$ is well ordered.
We claim that it is as desired. Clearly it is not representable (for its cylindric reduct is not representable)

Let $|G|\leq k$. Then $G\subseteq \A_{k,\mu}$ for some $\mu\in \alpha$.
We show that $\Sg^{\A_k}G$ has to be representable.
Let $\tau=\sigma$  be valid in the variety
$\RQEA_{\alpha}$. We show that it  is valid in $\Sg^{\A_k}G$.
Let $v_1,\ldots v_k$ be the variables occurring in this equation, and let $b_1,\ldots b_k$ be arbitrary elements of $\Sg^{\A_k}G$.
We show that $\tau(b_1,\ldots, b_k)=\sigma(b_1\ldots b_k)$. Now there are terms
$\eta_1\ldots, \eta_k$ written up from elements of $G$ such that $b_1=\eta_1\ldots, b_k=\eta_k$, then we need to show that
$\tau(\eta_1,\ldots, \eta_k)=\sigma(\eta_1, \ldots, \eta_k).$
This is another equation written up from elements of $G$, which is also valid in $\RQEA_{\alpha}$.
Let $n$ be an upper bound for the indices occurring in this equation and let $\mu>n$ be such that $G\subseteq \A_{k,\mu}$.
Then the above equation is valid in $\Sg^{\Rd_{\mu}\A}G$ since the latter is representable.
Hence the equation $\tau=\sigma$ holds in $\Sg^{\A_k}G$ at the evaluation $b_1,\ldots b_k$ of variables.

For $\mu\in \alpha$, let  $\Sigma_{\mu}^v$ be the set of universal formulas using only $\mu$
substitutions and $k$ variables valid in $\RQEA_{\omega}$,
and let $\Sigma_{\mu}^d$ be the set of universal formulas
using only $\mu$ substitutions and no diagonal elements valid in $\RQEA_{\omega}$.
By $\mu$ substitutions we understand the set
$\{{\sf s}_{[i,j]}: i,j\in \mu\}.$
Then $\A_{k,\mu}\models \Sigma_{\mu}^v\cup \Sigma_{\mu}^d$.
$\A_{k,\mu}\models \Sigma_{\mu}^v$
because the $k$ generated subalgebras of $\A_{k,\mu}$ are representable,
while $\A_{k,\mu}\models \Sigma_{\mu}^d$ because $\A_{k,{\mu}}$
has a representation that preserves all operations except
for diagonal elements.

Indeed, let $\phi\in \Sigma_{\mu}^d$, then there is a representation of $\A_{k,{\mu}}$ in which all operations
are the natural ones except for the diagonal elements.
This means that (after discarding the diagonal elements) there is a one to one homomorphism
$h:\A^d\to \P^d$ where $\A^d=(A_{k,n}, +, \cdot , {\sf c}_k, {\sf s}_{[i,j]}, {\sf s} _i^j)_{k\in \alpha, i,j\in \mu}\text { and }
\P^d=(\B(^{\alpha}W), {\sf c}_k^W, {\sf s}_{[i,j]}^W, {\sf s}_{[i|j]}^W)_{k\in \alpha, i,j\in \mu},$
for some infinite set $W$.

Now let $\P=(\B(^{\alpha}W), {\sf c}_k^W, {\sf s}_{[i,j]}^W, {\sf s}_{[i|j]}^W, {\sf d}_{kl}^W)_{k,l\in \alpha, i,j\in \mu}.$
Then we have that $\P\models \phi$ because $\phi$ is valid
and so  $\P^d\models \phi$ due to the fact that  no diagonal elements  occur in $\phi$.
Then $\A^d\models \phi$ because $\A^d$ is isomorphic to a subalgebra of $\P^d$ and $\phi$ is quantifier free. Therefore
$\A_{k,\mu}\models \phi$.
Let $$\Sigma^v=\bigcup_{\mu\in \alpha}\Sigma_{\mu}^v
\text { and }\Sigma^d=\bigcup_{\mu\in \alpha}\Sigma_{\mu}^d$$
Hence $\A_k\models \Sigma^v\cup \Sigma^d.$

For if not then there exists a quantifier free  formula $\phi(x_1,\ldots x_m)\in \Sigma^v\cup \Sigma^d$,
and $b_1,\ldots, b_m$ such that $\phi[b_1,\ldots, b_n]$ does not hold in $\A_k$. We have $b_1\ldots, b_m\in \A_{k,\mu}$
for some $\mu\in \alpha$.

Take $\mu$ large enough $\geq i$ so that
$\phi\in \Sigma_{\mu}^v\cup \Sigma_{\mu}^d$.   Then $\A_{k,\mu}$ does not model $\phi$, a contradiction.
Now let $\Sigma$ be  a set of quantifier free formulas axiomatizing  $\RQEA_{\alpha}$, then $\A_k$ does not model $\Sigma$ since $\A_k$ is not
representable, so there exists a formula $\phi\in \Sigma$ such that
$\phi\notin \Sigma^v\cup \Sigma^d.$
Then $\phi$ contains more than $k$ variables and a diagonal constant occurs in $\phi$.

\end{proof}

In the above construction, infinitely many finite splittings
(not just one which is done in \cite{ST}), increasing in number but always finite, are implemented
constructing infinitely many algebras, whose similarity types contain only finitely many
substitutions. This is the main novelty occurring in \cite{t} {\it  a modification of Andr\'eka's method of splitting to adapt to
the quasi-polyadic equality case.} Such constructed non-representable algebras,
form a chain, and our desired algebra is their
directed union, so that it is basically an $\omega$ step by step construction not just one step construction.

{\it The conceptual error in Sain's Thompson paper is claiming that the small subalgebras
of the non-representable algebra, obtained by performing {\it only one splitting into infinitely many atoms}, are representable; this is not necessarily
true at all.}

For finite dimensions, as it happens,  this {\it is true} because the splitting is implemented relative to a {\it finite} set of substitutions
and representability  involves a combinatorial trick depending on {\it counting the number} of substitutions.
This technique no longer holds in the presence of infinitely many substitutions because  we simply cannot count them.
So the splitting have to be done relative to reducts containing only finitely many substitutions and
their desired algebra, witnessing the complexity of
axiomatizations, is the limit
of such reducts, synchronized by the reduct operator.

\begin{corollary}\label{splitting2} The variety $\RQEA_{\alpha}$ is not axiomatizable over its diagonal free reduct
with a set of universal formulas containing infinitely many variables.
The variety $\RQEA_{\alpha}$ is not axiomatizable over its diagonal free reduct
with a set of universal formulas containing infinitely many variables.
\end{corollary}

In the next theorem we will also be sketchy.
\begin{theorem}\label{splitting2} For any $\alpha\geq \omega$, the class $\sf RQEA_{\alpha}$
is not finite axiomatizable over $\RCA_{\alpha}$ by a finite schema
\end{theorem}
\begin{demo}{Sketch of proof}
This follows by a refinement of the result in \cite{ANS}.
Let $n\geq 3$.
Let $Z_0=Z_1=n=\{0,1,2, \ldots n-1\}$ and $Z_i=\{(n-1)i-1, (n-1)i\}$ for $i>1$.
Let $p:\omega\to \omega$ be defined by $p(i)=(n-1)i$. Let
$$V={}^{\omega}\N^{(p)}=\{s\in {}^{\omega}\N: |\{i\in \omega: s_i\neq (n-1)i\}|<\omega\}.$$
$V$ will be the unit of our algebra.
Let $${\sf P}Z=Z_0\times Z_1\times Z_2\ldots \cap V.$$
Let $$t=Z_2\times Z_3\times\ldots \cap V.$$
We split $t$ into two parts, measured by the deviation from $p$:
$$X=\{s\in t: |\{i\in \omega\sim 2: s(i)\neq (n-1)i\}| \text { is even }\},$$
$$Y=\{s\in t: |\{i\in \omega\sim 2: s(i)\neq (n-1)i\}| \text { is odd }\}.$$
Then we define two relations $R, B$ on $n$, such that $\dom R=\dom B=n$ and $R\cap B=\emptyset$, for example, we set:
$$R=\{(u,v): u\in n, v=u+1(mod n)\},$$
$$B=\{(u,v): u\in n, v=u+2(mod n)\}.$$

$FT_{\omega}$ denotes the set of all finite transformations on $\omega$.
Let $\pi(\omega)=\{\tau\in FT_{\omega}: \tau \text { is a bijection }\}.$
Next we define {\it the  crucial} relation on $V,$
$$a=\{s\in {\sf P}Z: (s\upharpoonright 2\in R \text { and } s\upharpoonright \omega\sim 2\in X)\text { or }
(s\upharpoonright 2\in B\text { and }s\upharpoonright \omega\sim 2\in Y\}.$$
We also set
$$d=\{s\in {\sf P}Z: s_0=s_1\},$$
and we let $P$ be the permutated versions of $a$ and $d$, that is,
$$P=\{{\sf S}_{\tau}x: \tau \in \pi(\omega),\ \  x\in \{a,d\}\}.$$
Here, and elsewhere, ${\sf S}_{\tau}$ is the usual set-theoretic substitution operation corresponding to $\tau$.

For $W\in {}^{\omega}\rng Z^{(Z)}$, let
$${\sf P}W=\{s\in V: (\forall i\in \omega) s_i\in W_i\}.$$
Let $Eq(\omega)$ be the set of all equivalence relations on $\omega$.
For $E\in Eq(\omega)$, let $e(E)=\{s\in V: ker s=E\}$.
Let
$$T=\{{\sf P}W\cdot e(E): W\in {}^{\omega}\rng Z^{(Z)},  (\forall \delta\in \pi(\omega))W\neq Z\circ \delta, E\in Eq(\omega)\}.$$
Finally, let $$\At =P\cup T,$$
and
$$A_n=\{\bigcup X: X\subseteq \At\}.$$

We have
\begin{enumarab}
\item  $A_n$ is a subuniverse of the full cylindric weak set algebra
$$\langle \wp(V), + ,\cdot,  -, {\sf c}_i, {\sf d}_{ij}\rangle_{i,j\in \omega}.$$
Furthermore $\A_n$ is atomic and $\At \A_n=\At\sim \{0\}$.
\item  $\A_n$ can be expanded to a quasi-polyadic equality algebra $\B_n$ such that $\B_n$ is not representable.
\end{enumarab}
Like  \cite{ANS} Claim 1 undergoing the obvious changes
Like \cite{ANS} Claim 2, also undergoing the obvious changes, in particular the polyadic operations
are defined by:

Let $ \tau, \delta \in FT_{\omega}$. We say that $``\tau,
\delta$ transpose" iff $( \delta0-\delta1).(\tau \delta 0 - \tau
\delta 1)$ is negative.
Let $P'\subseteq P$ be the set of permutated versions of $a$.
Now we first define $ {\sf p}_\sigma : \At \rightarrow A$ for every $ \sigma
\in FT_{\omega}$.
\[ {\sf p}_\sigma({\sf S}_\delta a )  =
\begin{cases}
{\sf S}_{\sigma \circ \delta \circ [0,1] } a& \textrm{if}~~~``
\sigma, \delta
~~~\textrm{transpose}" \\
{\sf S}_{\sigma \circ \delta} a & \textrm{otherwise}
\end{cases}\]

$$ {\sf p}_\sigma (x) = {\sf S}_\sigma x \quad \textrm{if} \quad x \in At
\sim P'.$$
Then we set:
$${\sf p}_\sigma (\sum X) = \sum \{ {\sf p}_\sigma (x) : x \in X \} \quad \textrm{for}
\quad X \subseteq \At.$$
The defined operations satisfy the polyadic axioms, so that the expanded algebra $\B_n$ with the polyadic operations
is a quasi-polyadic equality algebra that is not representable.

Let $F$ be the cofinite ultrafilter over $\omega,$
then $\Pi \B_n/F$ is representable.
Of course the cylindric reduct of the ultraproduct is representable.
The point is to represent the substitutions; particularly those corresponding to transpositions.
The non-representability follows from an incompatibility condition between
the cardinality of $Z_0$ which is $n$ and how the abstract substitutions are defined.
This is expressed by the fact that this definition forces $|a|$ to be strictly greater than $n.$
When one forms the ultraproduct, this 'cardinality incompatibility condition' disappears, $Z_0$ is infinitely
stretched to have cardinality $\omega$. The abstract polyadic operations  coincide with usual
concrete ones.
\end{demo}

\section{Blow up and Blur constructions}

The splitting method due to Andr\'eka  has other several subtle sophisticated re-incarnations in the literature, not only
in the infinite dimensional case, but also in the finite dimensional case.
In fact it does has affinity to Monk's construction.

To mention only a few instances, witness  the so-called blow up and blur constructions in \cite{ANT}, \cite{can},
and the constructions in \cite{IGPL} and  \cite{Sayedneat}.

In the last
two references a finite atom structure is split twice (meaning that in each case every atom is split into infinitely many)
to give non-elementary equivalent algebras (based on the resulting two splitted atom structures) one a neat reduct
the other is not, this will be approached in detail in the next section.

In the former two references the atoms in finite Monk-like algebras both relation and cylindric, respectively
are each split to
infinitely many to prove deep results on \d\ completions, atom canonicity
and omitting types.

The idea of a blow up and blur construction, a little bit more, than in a nut shell:

Let $n$ be finite $n>2$.
Assume that $\RCA_n\subseteq \K$, and $S\K=\K$.
Start with a finite algebra $\C$ outside $\K$. Blow up and blur $\C$, by splitting
each atom to infinitely many, getting a new atom structure $\At$. In this process a (finite) set of blurs are used.

They do not blur the complex algebra, in the sense that $\C$ is there on this global level.
The algebra $\Cm\At$ will not be in $\K$
because $\C\notin \K$ and $\C$ embeds into $\Cm\At$.
Here the completeness of the complex algebra will play a major role,
because every element of $\C$,  is mapped, roughly, to {\it the join} of
its splitted copies which exist in $\Cm\At$ because it is complete.

These precarious joins prohibiting membership in $\K$ {\it do not }exist in the term algebra, only finite-cofinite joins do,
so that the blurs blur $\C$ on the
this level; $\C$ does not embed in $\Tm\At.$

In fact, the the term algebra will  not only be in $\K$, but actually it will be in the possibly larger $\RCA_n$.
This is where the blurs play their other role. Basically non-principal ultrafilters, the blurs
are used as colours to represent  $\Tm\At\A$.
In the process of representation we cannot use {\it only} principal ultrafilters,
because $\Tm\At$ cannot be completely representable for this would give that $\Cm\At\A$
is representable. But the blurs will actually provide a {\it complete representation} of the {\it canonical extension}
of $\Tm\At$, in symbols $\Tm\At^+$; the algebra whose underlying set consists of all ultrafilters of $\Tm\At\A$. The atoms of $\Tm\At$
are coded in the principal ones,  and the remaining non- principal ultrafilters, ot the blurs,
will be finite, used as colours to completely represent $\Tm\At^+$, in the process representing $\Tm\At$.

Such subtle constructions cast their shadow over the entire field of algebraic logic.
Lately, it has become fashionable in algebraic logic to
study representations of abstract algebras that has a complete representation \cite{Sayed} for an extensive overview.
A representation of $\A$ is roughly an injective homomorphism from $f:\A\to \wp(V)$
where $V$ is a set of $n$-ary sequences; $n$ is the dimension of $\A$, and the operations on $\wp(V)$
are concrete  and set theoretically
defined, like the Boolean intersection and cylindrifiers or projections.
A complete representation is one that preserves  arbitrary disjuncts carrying
them to set theoretic unions.
If $f:\A\to \wp(V)$ is such a representation, then $\A$ is necessarily
atomic and $\bigcup_{x\in \At\A}f(x)=V$.

Let us focus on cylindric algebras.
It is known that there are countable atomic $\RCA_n$s when $n>2$,
that have no complete representations;
in fact, the class of completely representable $\CA_n$s when $n>2$, is not even elementary \cite[corollary 3.7.1]{HHbook2}.

Such a phenomena is also closely
related to the algebraic notion of {\it atom-canonicity} which is an important persistence property in modal logic
and to the metalogical property of  omitting types in finite variable fragments of first order logic
\cite[theorems 3.1.1-2, p.211, theorems 3.2.8, 9, 10]{Sayed}.
A variety $V$ of Boolean algebras with operators is atom-canonical,
if whenever $\A\in V$, and $\A$ is atomic, then the complex algebra of its atom structure, $\Cm\At\A$ for short, is also
in $V$.

If $\At$ is a weakly representable but
not strongly  representable, then
$\Cm\At$ is not representable; this gives that $\RCA_n$ for $n>2$ $n$ finite, is {\it not} atom-canonical.
Also $\Cm\At\A$  is the \d\ completion of $\A$, and so obviously $\RCA_n$ is not closed under \d\ completions.

On the other hand, $\A$ cannot be completely  representable for, it can be shown without much ado,  that
a complete representation of $\A$ induces a representation  of $\Cm\At\A$ \cite[definition 3.5.1, and p.74]{HHbook2}.

Finally, if $\A$ is countable, atomic and has no complete representation
then the set of co-atoms (a co-atom is the complement of an atom), viewed in the corresponding Tarski-Lindenbaum algebra,
$\Fm_T$, as a set of formulas, is a non principal-type that cannot be omitted in any model of $T$; here $T$ is consistent
if $|A|>1$. This last connection  was first established by the  author leading up to \cite{ANT} and more, see e.g \cite{HHbook2}.

An extensive overview of such connections
can be found in \cite{Sayed}.

The key idea, as we saw,
of the construction of a Monk-like algebra is not so hard.
Such algebras are finite.
The atoms are given colours, and  cylindrifications and diagonal elements are defined by stating that monochromatic triangles
are inconsistent. If a Monk's algebra has many more atoms than colours,
it follows from Ramsey's Theorem that any representation
of the algebra must contain an inconsistent monochromatic triangle, so the algebra
is not representable. The close affinity with the splitting technique
is that {\it here the colours play the role of the base $|U_0|=m$ of the alleged
representation of $\A_{k,n}$ as above.}
Witness too \cite{HHbook}
(section on completions)  in the  context of splitting atoms, also each to infinitely many, in finite  rainbow
relation algebras \cite[lemmas, 17, 32, 34, 35, 36]{HHbook}.

Let us get more familiar with the blow up and blur construction. 
We use the notation of the above cited lemmas in \cite{HHbook2} without warning, and our proof will be very brief just stressing the main ideas.
Let $\R=\A_{K_m, K_n}$, $m>n>2$.
Let $T$ be the term algebra obtained by splitting the reds. Then $T$ has exactly two blurs
$\delta$ and $\rho$. $\rho$ is a flexible non principal ultrafilter consisting of reds with distinct indices and $\delta$ is the reds
with common indices.
Furthermore, $T$ is representable, but $\Cm\At T\notin  S\Ra\CA_{m+2}$, in particular, it is not representable. \cite [lemma 17.32]{HHbook2}

First it is obvious that  \pe\ has a \ws\ over $\At\R$ using in $m+2$ rounds,
hence $\R\notin \RA_{m+2}$, hence is not
in $S\Ra\CA_{m+2}$.  $\Cm\At T$ is also not in the latter class
$\R$ embeds into it, by mapping ever red
to the join of its copies. Let $D=\{r_{ll}^n: n<\omega, l\in n\}$, and $R=\{r_{lm}^n, l,m\in n, l\neq m\}$.
If $X\subseteq R$, then $X\in T$ if and only if $X$ is finite or cofinite in $R$ and same for subsets of $D$ \cite[lemma 17.35]{HHbook}.
Let $\delta=\{X\in T: X\cap D \text { is cofinite in $D$}\}$,
and $\rho=\{X\in T: X\cap R\text { is cofinite in $R$}\}$.
Then these are {\it the} non principal ultrafilters, they are the blurs and they are enough
to (to be used as colours), together with the principal ones, to represent $T$ as follows \cite[bottom of p. 533]{HHbook2}.
Let $\Delta$ be the graph $n\times \omega\cup m\times \{{\omega}\}$.
Let $\B$ be the full rainbow algebras over $\At\A_{K_m, \Delta}$ by
deleting all red atoms $r_{ij}$ where $i,j$ are
in different connected components of $\Delta$.

Obviously \pe\ has a \ws\ in $\EF_{\omega}^{\omega}(K_m, K_m)$, and so it has a \ws\ in
$G_{\omega}^{\omega}(\A_{K_m, K_m})$.
But $\At\A_{K_m, K_m}\subseteq \At\B\subseteq \At_{K_m, \Delta}$, and so $\B$ is representable.

One  then defines a bounded morphism from $\At\B$ to the the canonical extension
of $T$, which we denote by $T^+$, consisting of all ultrafilters of $T$. The blurs are images of elements from
$K_m\times \{\omega\}$, by mapping the red with equal double index,
to $\delta$, for distinct indices to $\rho$.
The first copy is reserved to define the rest of the red atoms the obvious way.
(The underlying idea is that this graph codes the principal ultrafilters in the first component, and the non principal ones in the second.)
The other atoms are the same in both structures. Let $S=\Cm\At T$, then $\Cm S\notin S\Ra\CA_{m+2}$ \cite[lemma 17.36]{HHbook}.

Note here that the \d\ completion of $T$ is not representable while its canonical extension is
{\it completely representable}, via the representation defined above.
However, $T$ itself is {\it not} completely representable, for a complete representation of $T$ induces a representation of its \d\ completion,
namely, $\Cm\At\A$.

In the next secion we use two rainbow constructions to prove two diffrent algebraic results that render almost the same result
for omitting types theorems in finite variable fragments of first order logic.

\subsection{Rainbows and Omitting types in Clique guarded semantics}

The rainbow construction in algebraic logic is invented by Hirsch and Hodkinson \cite{HHbook}.
This ingenious construction reduces finding subtle differences between seemingly
related notions or concepts using a very simple \ef\ forth pebble game between two players \pa\ and \pe\ on two very simple structures.
From those structures a relation or cylindric algebra can be constructed and a \ws\ for either player lifts to a \ws\ on the atom structure
of the algebra, though the number of pebbles and rounds increases in the algebra.

In the case of relation algebras, the atoms are coloured, so that the games are played on colours. For cylindric algebras,
matters are a little bit more complicated
because games are played on so-called coloured graphs, which are models of the rainbow signature coded in an
$L_{\omega_1,\omega}$ theory. The atom structure consists of finite
coloured graphs rather than colours.

Nevertheless, the essence of the two construction is very similar, because in the cylindric
algebra constructed from $A$ and $B$, the relation algebra atom structure based also on $A$ and $B$
is coded in the cylindric atom structure, but the latter has additional shades of yellow
that are used to label $n-1$ hyperedges coding the cylindric information.

The strategy for \pe\ in a rainbow game is try white, if it doesn't work try black and finally if it doesn't work
try red. In the latter case she is kind of cornered, so it is the hardest part in the game.
She never uses green. 

In the cylindric algebra case, the most difficult part
for \pe\ is to label the edge between apexes of two cones (a cone is a special coloured graph) having
a common base, when she is also forced a red.

The colours used in both constructions maty also be slightly different. For example in the proof 
of \cite{HH} \pe\ had the option to play a black colour, but in the cylindric algebra this colour was not avaiable to label edges in the 
coloured graphs. In this case the responses by \pe\ in the relation algebra rainbow game 
were handled by the shades of yellow in the corresponding cylindric rainbow game.

But in  both cases  the choice of a red, when she is forced one,
 is the most difficult part for \pa\ and if she succeeds in every round then she wins.

Indeed,  it is always the case that \pa\ wins on a red clique,
using his greens to force \pa\  play an inconsistent triple of reds.
In case of cylindric algebra \pa\ bombards \pe\ with cones having green tints, based on the same base.

We start by doing the algebra. We prove two results. One is that certain classes of algebras that have a strong neat embedding property
in higher dimensions, possible only finite, are not elementary (here by strong we mean that the algebra in 
question {\it completely} embds into the neat reduct
of the algebra in higher dimensions).
The second result 
is that several varieties also consisting of algebras having a usual neat embedding property, in higher dimensions,
are not atom canonical. In both cases we use rainbows.
Next we do the logic.

We need some preparing to do.

\begin{definition}\label{subs}
Let $n$ be an ordinal. An $s$ word is a finite string of substitutions $({\sf s}_i^j)$,
a $c$ word is a finite string of cylindrifications $({\sf c}_k)$.
An $sc$ word is a finite string of substitutions and cylindrifications
Any $sc$ word $w$ induces a partial map $\hat{w}:n\to n$
by
\begin{itemize}

\item $\hat{\epsilon}=Id$

\item $\widehat{w_j^i}=\hat{w}\circ [i|j]$

\item $\widehat{w{\sf c}_i}= \hat{w}\upharpoonright(n\smallsetminus \{i\}).$

\end{itemize}
\end{definition}

If $\bar a\in {}^{<n-1}n$, we write ${\sf s}_{\bar a}$, or more frequently
${\sf s}_{a_0\ldots a_{k-1}}$, where $k=|\bar a|$,
for an an arbitrary chosen $sc$ word $w$
such that $\hat{w}=\bar a.$
$w$  exists and does not
depend on $w$ by \cite[definition~5.23 ~lemma 13.29]{HHbook}.
We can, and will assume \cite[Lemma 13.29]{HHbook}
that $w=s{\sf c}_{n-1}{\sf c}_n.$
[In the notation of \cite[definition~5.23,~lemma~13.29]{HHbook},
$\widehat{s_{ijk}}$ for example is the function $n\to n$ taking $0$ to $i,$
$1$ to $j$ and $2$ to $k$, and fixing all $l\in n\setminus\set{i, j,k}$.]
The following is the $\CA$ analogue of \cite[lemma~19]{r}.
In the next definition we extend the definition of atomic cylindric networks to polyadic ones.
For diagonal free reducts, the definition is modified the obvious way. For example for $\Df$s, we only have the
first condition, for $\sf Sc$s we do not have the second,
and the fourth ${\sf s}_{[i,j]}$ and $[i,j]$ are  replaced, respectively, by
${\sf s}_i^j$ and $[i|j].$

Let $\delta$ be a map. Then $\delta[i\to d]$ is defined as follows. $\delta[i\to d](x)=\delta(x)$
if $x\neq i$ and $\delta[i\to d](i)=d$. We write $\delta_i^j$ for $\delta[i\to \delta_j]$.

\begin{definition}
From now on let $2< n<\omega.$ Let $\C$ be an atomic ${\sf PEA}_{n}$.
An \emph{atomic  network} over $\C$ is a map
$$N: {}^{n}\Delta\to \At\C$$
such that the following hold for each $i,j<n$, $\delta\in {}^{n}\Delta$
and $d\in \Delta$:
\begin{itemize}
\item $N(\delta^i_j)\leq {\sf d}_{ij}$
\item $N(\delta[i\to d])\leq {\sf c}_iN(\delta)$
\item $N(\bar{x}\circ [i,j])= {\sf s}_{[i,j]}N(\bar{x})$ for all $i,j<n$.

\end{itemize}
\end{definition}

Note than $N$ can be viewed as a hypergraph with set of nodes $\Delta$ and
each hyperedge in ${}^{\mu}\Delta$ is labelled with an atom from $\C$.
We call such hyperedges atomic hyperedges.
We write $\nodes(N)$ for $\Delta.$ We let $N$ stand for the set of nodes
as well as for the function and the network itself. Context will help.
We assume that $\nodes(N)\subseteq \N$.
Formulated for $\sf Df$s only, the next definition applies to all algebras considered.

\begin{definition}\label{def:games}
Let $2\leq n<\omega$. For any ${\sf PEA_n}$
atom structure $\alpha$ and $n\leq m\leq
\omega$, we define two-player games $G(\alpha),$ \; and
$F^m(\alpha)$,
each with $\omega$ rounds.
\begin{enumarab}
\item   $G(\alpha)$, or simply $G$ is the usual $\omega$ rounded atomic game on networks \cite{HHbook2}.
$G_k$ is $G$ truncated to $k$ rounds.
\item $F^m(\alpha)$, or $F^m$  is very similar, except that \pa\ can choose from only $m>n$ pebbles,
but he has the option to  re-use them.

Let $m\leq \omega$.
In a play of $F^m(\alpha)$ the two players construct a sequence of
networks $N_0, N_1,\ldots$ where $\nodes(N_i)$ is a finite subset of
$m=\set{j:j<m}$, for each $i$.

In the initial round of this game \pa\
picks any atom $a\in\alpha$ and \pe\ must play a finite network $N_0$ with
$\nodes(N_0)\subseteq  m$,
such that $N_0(\bar{d}) = a$
for some $\bar{d}\in{}^{n}\nodes(N_0)$.

In a subsequent round of a play of $F^m(\alpha)$ \pa\ can pick a
previously played network $N$ an index $l<n$, a {\it face}
$F=\langle f_0,\ldots, f_{n-2} \rangle \in{}^{n-2}\nodes(N),\; k\in
m\setminus\set{f_0,\ldots, f_{n-2}}$, and an atom $b\in\alpha$ such that
$$b\leq {\sf c}_lN(f_0,\ldots, f_i, x,\ldots f_{n-2}).$$
The choice of $x$ here is arbitrary,
as the second part of the definition of an atomic network together with the fact
that $\cyl i(\cyl i x)=\cyl ix$ ensures that the right hand side does not depend on $x$.

This move is called a \emph{cylindrifier move} and is denoted
$$(N, \langle f_0, \ldots, f_{n-2}\rangle, k, b, l)$$
or simply by $(N, F,k, b, l)$.
In order to make a legal response, \pe\ must play a
network $M\supseteq N$ such that
$M(f_0,\ldots f_{i-1}, k, f_{i+1},\ldots f_{n-2}))=b$
and $\nodes(M)=\nodes(N)\cup\set k$.

\pe\ wins $F^m(\alpha)$ if she responds with a legal move in each of the
$\omega$ rounds.  If she fails to make a legal response in any
round then \pa\ wins.
\end{enumarab}

We need some more technical lemmas which are  generalizations of lemmas formulated for relation algebras
in \cite{r}.

The next definition is also formulated ${\sf Sc}$s, and of course to applies to all its expansion studied here.

\begin{definition}\label{def:hat}
For $m\geq 5$ and $\C\in\Sc_m$, if $\A\subseteq\Nr_n(\C)$ is an
atomic cylindric algebra and $N$ is an $\A$-network with $\nodes(N)\subseteq m$, then we define
$\widehat N\in\C$ by
\[\widehat N =
 \prod_{i_0,\ldots i_{n-1}\in\nodes(N)}{\sf s}_{i_0, \ldots i_{n-1}}N(i_0\ldots i_{n-1})\]
$\widehat N\in\C$ depends
implicitly on $\C$.
\end{definition}

In what follows we write $\A\subseteq_c \B$ if $\A$ is a complete subalgebra of $\B$.

\begin{lemma}\label{lem:atoms2}
Let $n<m$ and let $\A$ be an atomic $\sf Sc_n$,
$\A\subseteq_c\Nr_n\C$
for some $\C\in\CA_m$.  For all $x\in\C\setminus\set0$ and all $i_0, \ldots i_{n-1} < m$ there is $a\in\At(\A)$ such that
${\sf s}_{i_0\ldots i_{n-1}}a\;.\; x\neq 0$.
\end{lemma}
\begin{proof}
We can assume, see definition  \ref{subs},
that ${\sf s}_{i_0,\ldots i_{n-1}}$ consists only of substitutions, since ${\sf c}_{m}\ldots {\sf c}_{m-1}\ldots
{\sf c}_nx=x$
for every $x\in \A$. We have ${\sf s}^i_j$ is a
completely additive operator (any $i, j$), hence ${\sf s}_{i_0,\ldots i_{\mu-1}}$
is too  (see definition~\ref{subs}).
So $\sum\set{{\sf s}_{i_0\ldots i_{n-1}}a:a\in\At(\A)}={\sf s}_{i_0\ldots i_{n-1}}
\sum\At(\A)={\sf s}_{i_0\ldots i_{n-1}}1=1$,
for any $i_0,\ldots i_{n-1}<n$.  Let $x\in\C\setminus\set0$.  It is impossible
that ${\sf s}_{i_0\ldots i_{n-1}}\;.\;x=0$ for all $a\in\At(\A)$ because this would
imply that $1-x$ was an upper bound for $\set{{\sf s}_{i_0\ldots i_{n-1}}a:
a\in\At(\A)}$, contradicting $\sum\set{{\sf s}_{i_0\ldots i_{n-1}}a :a\in\At(\A)}=1$.
\end{proof}

For networks $M, N$ and any set $S$, we write $M\equiv^SN$
if $N\restr S=M\restr S$, and we write $M\equiv_SN$
if the symmetric difference $\Delta(\nodes(M), \nodes(N))\subseteq S$ and
$M\equiv^{(\nodes(M)\cup\nodes(N))\setminus S}N$. We write $M\equiv_kN$ for
$M\equiv_{\set k}N$.

We write $Id_{-i}$ for the function $\{(k,k): k\in n\smallsetminus\{i\}\}.$
For a network $N$ and a partial map $\theta$ from $n$ to $n$, that is $\dom\theta\subseteq n$, then
$N\theta$ is the network whose labelling is defined by $N\theta(\bar{x})=N(\tau(\bar{x}))$
where for $i\in n$, $\tau(i)=\theta(i)$ for $i\in \dom\theta$
and $\tau(i)=i$ otherwise.
Recall that $F^m$ is the usual atomic game
on networks, except that the nodes are $m$ and \pa\ can re use nodes. Then:

Let $\K$ be any class between $\Sc$ and $\PEA$.
\begin{theorem}\label{thm:n}
Let $n<m$, and let $\A$ be an atomic $\K_m.$
If $\A\in S_c\Nr_{n}\K_m, $
then \pe\ has a \ws\ in $F^m(\At\A)$ (the latter involves only cylindrifier moves so it applies to algebras in $\K$).
In particular, if $\A$ is countable and completely representable, then \pe\ has a \ws\ in $F^{\omega}(\At\A).$
In the latter case, since $F^{\omega}(\At\A)$ is equivalent to the usual atomic rounded game on networks,
the converse is also true.
\end{theorem}
\begin{proof}
The proof of the first part is based on repeated use of
lemma ~\ref{lem:atoms2}. We first show:
\begin{enumerate}
\item For any $x\in\C\setminus\set0$ and any
finite set $I\subseteq m$ there is a network $N$ such that
$\nodes(N)=I$ and $x\;.\;\widehat N\neq 0$.
\item
For any networks $M, N$ if
$\widehat M\;.\;\widehat N\neq 0$ then $M\equiv^{\nodes(M)\cap\nodes(N)}N$.
\end{enumerate}
We define the edge labelling of $N$ one edge
at a time. Initially no hyperedges are labelled.  Suppose
$E\subseteq\nodes(N)\times\nodes(N)\ldots  \times\nodes(N)$ is the set of labelled hyper
edges of $N$ (initially $E=\emptyset$) and
$x\;.\;\prod_{\bar c \in E}{\sf s}_{\bar c}N(\bar c)\neq 0$.  Pick $\bar d$ such that $\bar d\not\in E$.
Then there is $a\in\At(\c A)$ such that
$x\;.\;\prod_{\bar c\in E}{\sf s}_{\bar c}N(\bar c)\;.\;{\sf s}_{\bar d}a\neq 0$.

Include the edge $\bar d$ in $E$.  Eventually, all edges will be
labelled, so we obtain a completely labelled graph $N$ with $\widehat
N\neq 0$.
it is easily checked that $N$ is a network.

For the second part, if it is not true that
$M\equiv^{\nodes(M)\cap\nodes(N)}N$ then there are is
$\bar c \in{}^{1}\nodes(M)\cap\nodes(N)$ such that $M(\bar c )\neq N(\bar c)$.
Since edges are labelled by atoms we have $M(\bar c)\cdot N(\bar c)=0,$
so $0={\sf s}_{\bar c}0={\sf s}_{\bar c}M(\bar c)\;.\; {\sf s}_{\bar c}N(\bar c)\geq \widehat M\;.\;\widehat N$.

Next, we show that:

\begin{enumerate}
\item\label{it:-i}
If $i\not\in\nodes(N)$ then ${\sf c}_i\widehat N=\widehat N$.

\item \label{it:-j} $\widehat{N Id_{-j}}\geq \widehat N$.

\item\label{it:ij} If $i\not\in\nodes(N)$ and $j\in\nodes(N)$ then
$\widehat N\neq 0 \rightarrow \widehat{N[i/j]}\neq 0$,
where $N[i/j]=N\circ [i|j]$

\item\label{it:theta} If $\theta$ is any partial, finite map $n\to n$
and if $\nodes(N)$ is a proper subset of $n$,
then $\widehat N\neq 0\rightarrow \widehat{N\theta}\neq 0$.
\end{enumerate}

The first part is easy.
The second
part is by definition of $\;\widehat{\;}$. For the third part, suppose
$\widehat N\neq 0$.  Since $i\not\in\nodes(N)$, by part~\ref{it:-i},
we have ${\sf c}_i\widehat N=\widehat N$.  By cylindric algebra axioms it
follows that $\widehat N\cdot {\sf d}_{ij}\neq 0$.  From the above
there is a network $M$ where $\nodes(M)=\nodes(N)\cup\set i$ such that
$\widehat M\cdot\widehat N\cdot {\sf d}_{ij}\neq 0$.  From the first part, we
have $M\supseteq N$ and $M=N[i/j]$.
Hence $\widehat{N[i/j]}\neq 0$.

For the final part
(cf. \cite[lemma~13.29]{HHbook}), since there is
$k\in n\setminus\nodes(N)$, \/ $\theta$ can be
expressed as a product $\sigma_0\sigma_1\ldots\sigma_t$ of maps such
that, for $s\leq t$, we have either $\sigma_s=Id_{-i}$ for some $i<n$
or $\sigma_s=[i/j]$ for some $i, j<n$ and where
$i\not\in\nodes(N\sigma_0\ldots\sigma_{s-1})$.
Now apply parts~\ref{it:-j} and \ref{it:ij}.

Now we prove the required. Suppose that 
$\A\subseteq_c\Nr_n\C$ for some $\C\in\K_m$ then \pe\ always
plays networks $N$ with $\nodes(N)\subseteq n$ such that
$\widehat N\neq 0$. In more detail, in the initial round, let \pa\ play $a\in \At\A$.
\pe\ plays a network $N$ with $N(0, \ldots n-1)=a$. Then $\widehat N=a\neq 0$.
At a later stage suppose \pa\ plays the cylindrifier move
$(N, \langle f_0, \ldots, f_{n-2}\rangle, k, b, l)$
by picking a
previously played network $N$ and $f_i\in \nodes(N), \;l<n,  k\notin \{f_i: i<n-2\}$,
and $b\leq {\sf c}_k N(f_0,\ldots,  f_{i-1}, x, f_{i+1}, \ldots, f_{n-2})$.

Let $\bar a=\langle f_0\ldots f_{i-1}, k, f_{i+1}, \ldots f_{n-2}\rangle.$
Then we have ${\sf c}_k\widehat N\cdot {\sf s}_{\bar a}b\neq 0$.
and so by the above there is a network  $M$ such that
$\widehat{M}\cdot\widehat{{\sf c}_kN}\cdot {\sf s}_{\bar a}b\neq 0$.

Hence
$M(f_0,\dots, f_{i-1}, k, f_{i-2}, \ldots, f_{n-2})=b$, and $M$ is the required response.
\end{proof}

We define a new class $\K$ consisting of coloured graphs
(which are the models of the rainbow signature satisfying the $L_{\omega_1, \omega}$ theory as formulated
in \cite{HHbook2}). Let $\N^{-{1}}$ be $\N$ with reverse order, and let $g:\N\to \N^{-1}$ be the identity map, we denote $g(n)$ by $-n$.
We assume that $0\in \N$.
We take ${\sf G}=\N^{-1}$ and our reds is $\N$.

Translating the rainbow signature to coloured graphs,
we get that our class of models (coloured graphs) $\K$ consists of:

\begin{enumarab}

\item $M$ is a complete graph.

\item $M$ contains no triangles (called forbidden triples)
of the following types:

\vspace{-.2in}
\begin{eqnarray}
&&\nonumber\\
(\g, \g^{'}, \g^{*}), (\g_i, \g_{i}, \w),
&&\mbox{any }i\in n-1\;  \\
(\g^j_0, \g^k_0, \w_0)&&\mbox{ any } j, k\in \N\\
\label{forb:pim}(\g^i_0, \g^j_0, \r_{kl})&&\mbox{unless } \set{(i, k), (j, l)}\mbox{ is an order-}\\
&&\mbox{ preserving partial function }\N^{-1}\to\N\nonumber\\
\label{forb:match}(\r_{ij}, \r_{j'k'}, \r_{i^*k^*})&&\mbox{unless }i=i^*,\; j=j'\mbox{ and }k'=k^*
\end{eqnarray}
and no other triple of atoms is forbidden.

\item If $a_0,\ldots   a_{n-2}\in M$ are distinct, and no edge $(a_i, a_j)$ $i<j<n$
is coloured green, then the sequence $(a_0, \ldots a_{n-2})$
is coloured a unique shade of yellow.
No other $(n-1)$ tuples are coloured shades of yellow.

\item If $D=\set{d_0,\ldots  d_{n-2}, \delta}\subseteq M$ and
$M\upharpoonright D$ is an $i$ cone with apex $\delta$, inducing the order
$d_0,\ldots  d_{n-2}$ on its base, and the tuple
$(d_0,\ldots d_{n-2})$ is coloured by a unique shade
$\y_S$ then $i\in S.$

\end{enumarab}

\end{definition}
\begin{definition}
Let $i\in \N^{-1}$, and let $M$ be a coloured graph  consisting of $n$ nodes
$x_0,\ldots  x_{n-2}, z$. We call $M$ an $i$ - cone if $M(x_0, z)=\g^0_i$
and for every $1\leq j\leq n-2$ $M(x_j, z)=\g_j$,
and no other edge of $M$
is coloured green.
$(x_0,\ldots x_{n-2})$
is called the center of the cone, $z$ the apex of the cone
and $i$ the tint of the cone.
\end{definition}

\begin{theorem}\label{rainbow}
Let $3\leq n<\omega$. Then there exists an atomic $\PEA_n$ with countably many atoms
such that $\Rd_{Sc}\A\notin S_c\Nr_n\Sc_{n+3}$, and there exists
a countable $\B\in S_c\Nr_n\QEA_{\omega}$ such that $\A\equiv \B$ (hence $\B$ is also atomic). In particular, for
for any class $\K$, such that $S_c\Nr_n\CA_{\omega}\subseteq \K\subseteq S_c\Nr_n\CA_{n+3}$, $\K$ is not elementary, and
for any ${\sf L}\in \{\CA, \QA, \QEA, \Sc\}$, the class of completely representable algebras of dimension $n$
is not elementary.
\end{theorem}
\begin{demo} {Proof}

\begin{enumarab}

\item We define a polyadic equality algebra of
dimension $n$. It is a rainbow polyadic equality algebra.

We first specify its atom structure.
Let $$\At=\{a:n \to M, M\in \K, \text { $a$ is surjective  and }
\nodes(M)\subseteq \N\}.$$
We write $M_a$ for the element of $\At$ for which
$a:\alpha\to M$ is a surjection.
Let $a, b\in \At$ define the
following equivalence relation: $a \sim b$ if and only if
\begin{itemize}
\item $a(i)=a(j)\text { and } b(i)=b(j)$

\item $M_a(a(i), a(j))=M_b(b(i), b(j))$ whenever defined

\item $M_a(a(k_0)\dots a(k_{n-2}))=M_b(b(k_0)\ldots b(k_{n-1}))$ whenever
defined
\end{itemize}
Let $\At$ be the set of equivalences classes. Then define
$$[a]\in E_{ij} \text { iff } a(i)=a(j)$$
$$[a]T_i[b] \text { iff }a\upharpoonright n\smallsetminus \{i\}=b\upharpoonright n\smallsetminus \{i\}.$$
Define accessibility relations corresponding to the polyadic (transpositions) operations as follows:
$$[a]S_{ij}[b] \text { iff } a\circ [i,j]=b$$
This, as easily checked, defines a $\PEA_n$
atom structure.
Let $3\leq n<\omega$.
Let $\C$ be the complex polyadic equality algebra over $\At$. We will show that $\Rd_{Sc}\C$ is
not in $S_c\Nr_n\Sc_{n+3}$
but an elementary extension of $\C$ belongs to
$S_c\Nr_{n}\QEA_{\omega}.$

\item But first we translate the atomic games $G_k$ above \cite[definition 27(3)]{HH}, \cite{HHbook2}
to games on coloured graphs \cite[lemma 30]{HH}.

Let $N$ be an atomic $\C$ network.
Let $x,y$ be two distinct nodes occurring in the
$n$ tuple $\bar z$. $N(\bar z)$ is an atom of $\C$
which defines an edge colour of
$x,y$. Using the fact that the dimension is at least $3$,
the edge colour depends only on $x$ and $y$
not on the other elements of
$\bar z$ or the positions of $x$ and $y$ in $\bar z$.
Similarly $N$ defines shades of white for certain $(n-1)$ tuples.  In this way $N$
translates
into a coloured graph.

This translation has an inverse. More precisely, we have:
Let $M\in \K$ be arbitrary. Define $N_{M}$
whose nodes are those of $M$
as follows. For each $a_0,\ldots a_{n-1}\in M$, define
$N_{M}(a_0,\ldots,  a_{n-1})=[\alpha]$
where $\alpha: n\to M\upharpoonright \set{a_0,\ldots, a_{n-1}}$ is given by
$\alpha(i)=a_i$ for all $i<n$.
Then, as easily checked,  $N_{M}$ is an atomic $\C$ network.

Conversely, let $N$ be any non empty atomic $\C$ network.
Define a complete coloured graph $M_N$
whose nodes are the nodes of $N$ as follows:
\begin{itemize}
\item For all distinct $x,y\in M_N$ and edge colours $\eta$, $M_N(x,y)=\eta$
if and only if  for some $\bar z\in ^nN$, $i,j<n$, and atom $[\alpha]$, we have
$N(\bar z)=[\alpha]$, $z_i=x$, $z_j=y$ and the edge $(\alpha(i), \alpha(j))$
is coloured $\eta$ in the graph $\alpha$.

\item For all $x_0,\ldots, x_{n-2}\in {}^{n-1}M_N$ and all yellows $\y_S$,
$M_N(x_0,\ldots,  x_{n-2})= \y_S$ if and only if
for some $\bar z$ in $^nN$, $i_0,\ldots,  i_{n-2}<n$
and some atom $[\alpha]$, we have
$N(\bar z)=[\alpha]$, $z_{i_j}=x_j$ for each $j<n-1$ and the $n-1$ tuple
$\langle \alpha(i_0),\ldots \alpha(i_{n-2})\rangle$ is coloured
$\y_S.$ Then $M_N$ is well defined and is in $\K$.
\end{itemize}
The following is then, though tedious and long,  easy to  check:
For any $M\in \K$, we have  $M_{N_{M}}=M$,
and for any $\C$ network
$N$, $N_{{M}_N}=N.$
This translation makes the following equivalent formulation of the
games $F^m(\At\C),$ originally defined on networks.

\item The new  graph version of
the game \cite[p.27-29]{HH} builds a nested sequence $M_0\subseteq M_1\subseteq \ldots $.
of coloured graphs.

Let us start with the game $F^m$.
\pa\ picks a graph $M_0\in \K$ with $M_0\subseteq m$ and
$\exists$ makes no response
to this move. In a subsequent round, let the last graph built be $M_i$.
$\forall$ picks
\begin{itemize}
\item a graph $\Phi\in \K$ with $|\Phi|=n$
\item a single node $k\in \Phi$
\item a coloured graph embedding $\theta:\Phi\smallsetminus \{k\}\to M_i$
Let $F=\phi\smallsetminus \{k\}$. Then $F$ is called a face.
\pe\ must respond by amalgamating
$M_i$ and $\Phi$ with the embedding $\theta$. In other words she has to define a
graph $M_{i+1}\in C$ and embeddings $\lambda:M_i\to M_{i+1}$
$\mu:\phi \to M_{i+1}$, such that $\lambda\circ \theta=\mu\upharpoonright F.$
\end{itemize}
In the above game the nodes of the graphs played are bounded by $m$. For $G$ and its truncated versions $G_k$
there is no restriction on the number of nodes used during the game.

\item Now \pe\ has a \ws\ in $G_k(\At\C)$ for all finite $k\geq n$, his strategy is exactly like the \ws\ of \pe\ in
\cite[lemma 31]{HH}.

Strictly speaking  in \cite{HH} the game was played on $\CA_{\N, \N}$
but the reversing of the order in $\N$, does not affect \pe\ s \ws\ .
This reversing of the order was introduced to make \pa\
the game  $F^{n+3}$. As we shall see, this also gives the stronger result than that in \cite{HH}, stated in our theorem.

And indeed we claim that  \pa\ has a \ws\ in $F^{n+3}$ \cite[theorem 33, lemma 41]{r}.
He uses the  usual strategy in rainbow algebras,
by bombarding \pe\ with cones on the same base and different green tints.
In this way  \pe\ is forced decreasing sequence in $\N$.
To implement his strategy we will see that he needs $n+3$ pebbles that he can use and re-use.

In the initial round \pa\ plays a graph $M$ with nodes $0,1,\ldots n-1$ such that $M(i,j)=\w$
for $i<j<n-1$
and $M(i, n-1)=\g_i$
$(i=1, \ldots, n-2)$, $M(0,n-1)=\g_0^0$ and $M(0,1,\ldots, n-2)=\y_{\N}$.

In the following move \pa\ chooses the face $(0,\ldots, n-2)$ and demands a node $n$
with $M_2(i,n)=\g_i$ $(i=1,\ldots, n-2)$, and $M_2(0,n)=\g_0^{-1}.$

\pe\ must choose a label for the edge $(n+1,n)$ of $M_2$. It must be a red atom $r_{mn}$. Since $-1<0$ we have $m<n$.
In the next move \pa\ plays the face $(0, \ldots, n-2)$ and demands a node $n+1$, with $M_3(i,n)=\g_i$ $(i=1,\ldots n-2)$,
such that  $M_3(0,n+2)=\g_0^{-2}$.
Then $M_3(n+1,n)$ and $M_3(n+1,n-1)$ both being red, the indices must match.
$M_3(n+1,n)=r_{ln}$ and $M_3(n+1, n-1)=r_{lm}$ with $l<m$.

In the next round \pa\ plays $(0,1,\ldots n-2)$ and reuses the node $2$ such that $M_4(0,2)=\g_0^{-3}$.
This time we have $M_4(n,n-1)=\r_{jl}$ for some $j<l\in \N$.
Continuing in this manner leads to a decreasing sequence in $\N$.

By theorem \ref{thm:n}, it follows that $\Rd_{Sc}\C\notin S_c\Nr_n\Sc_{n+3}$,
but it is elementary equivalent
to a countable completely representable algebra. Indeed, using ultrapowers and an elementary chain argument,
we obtain $\B$ such $\C\equiv \B$ \cite[lemma 44]{r},

In more detail, we have seen that for $k<\omega$, \pe\ has a \ws\ $\sigma_n$ in $G_k(\At\C)$.
We can assume that $\sigma_k$ is deterministic.
Let $\B$ be a non-principal ultrapower of $\C$.  We can show that
\pe\ has a \ws\ $\sigma$ in $G(\At\B)$ --- essentially she uses
$\sigma_k$ in the $k$'th component of the ultraproduct so that at each
round of $G(\At\B)$,  \pe\ is still winning in co-finitely many
components, this suffices to show she has still not lost.

We can also assume that $\C$ is countable. If not then replace it by its subalgebra generated by the countably many atoms.
(called the term algebra); \ws\ s that depended only on the atom structure persist
for both players.

Now one can use an
elementary chain argument to construct countable elementary
subalgebras $\C=\A_0\preceq\A_1\preceq\ldots\preceq\B$.  For this,
let $\A_{i+1}$ be a countable elementary subalgebra of $\B$
containing $\A_i$ and all elements of $\B$ that $\sigma$ selects
in a play of $G(\At\B)$ in which \pa\ only chooses elements from
$\A_i$. Now let $\A'=\bigcup_{i<\omega}\A_i$.  This is a
countable elementary subalgebra of $\B$, hence necessarily atomic,  and \pe\ has a \ws\ in
$G(\At\A')$, so by \cite[theorem 3.3.3]{HHbook2}, $\A'$ is completely representable.

\item To see the last statement, namely that $\A'\in S_c\Nr_n\QEA_{\omega}$,  we prove something slightly more general.
Assume that $M$ is the base of
a complete representation of $\A'$, whose
unit is a generalized space,
that is, $1^M=\bigcup {}^nU_i$, where $U_i\cap U_j=\emptyset$ for distance $i$ and $j$, in some
index set $I$. For $i\in I$, let $E_i={}^nU_i$, pick $f_i\in {}^{\omega}U_i$, let $W_i=\{f\in  {}^{\omega}U_i: |\{k\in \omega: f(k)\neq f_i(k)\}|<\omega\}$,
and let ${\sf c}_i=\wp(W_i)$. Then $\C_i$ is atomic; indeed the atoms are the singletons.

Let $x\in \Nr_n\C_i$, that is ${\sf c}_ix=x$ for all $n\leq i<\omega$.
Now if  $f\in x$ and $g\in W_i$ satisfy $g(k)=f(k) $ for all $k<n$, then $g\in x$.
Hence $\Nr_n \C_i$
is atomic;  its atoms are $\{g\in W_i:  \{g(0),\ldots g(n-1)\}\subseteq U_i\}.$
Define $h_i: \A\to \Nr_n\C_i$ by
$$h_i(a)=\{f\in W_i: \exists a\in \At\A: M\models a(f(0)\ldots f(n-1))\}.$$
Let $\C=\prod _i \C_i$. Let $\pi_i:\C\to \C_i$ be the $i$th projection map.
Now clearly  $\C$ is atomic, because it is a product of atomic algebras,
and its atoms are $(\pi_i(\beta): \beta\in \At(\C_i)\}$.
Now  $\A$ embeds into $\Nr_n\C$ via $I:a\mapsto (\pi_i(a) :i\in I)$. If $a\in \Nr_n\C$,
then for each $i$, we have $\pi_i(x)\in \Nr_n\C_i$, and if $x$ is non
zero, then $\pi_i(x)\neq 0$. By atomicity of $\C_i$, there are $\bar{m}\in D$ such that
$\{g\in W_i: g(k)=[\bar{m}]_k\}\subseteq \pi_i(x)$. Hence there is an atom
of $\A$, such that $M\models ({m})$,  so $x. I(a)\neq 0.$
and we are done. Note that in this argument {\it no cardinality condition} is required.
(The reverse inclusion does not hold in general for uncountable algebras,
though it holds for atomic algebras with countably many atoms).
Hence if $\A\ $ is completely representable, then it is in $S_c\Nr_n\QEA_{\omega}$.

\end{enumarab}
\end{demo}

\begin{corollary} For any finite $m>2$ , and $n\geq m+3$, the class $S_c\Nr_m\CA_n$ is not elementary.
\end{corollary}
\begin{proof} Let $\K=S_c\Nr_m\CA_n$, then $S_c\Nr_m\CA_{\omega}\subseteq \K\subseteq S_c\Nr_m\CA_{m+3}$,
hence  it is not elementary.
\end{proof}
We now formulate, and prove,  two negative
new omitting types theorems, that are
consequences of the algebraic results proved above.
Now we blow up and blur a finite polyadic equality rainbow algebra.

\begin{theorem}\label{blowupandblur}
For any finite $n>2$, for any $\K$ between $\Sc$ and $\PEA,$
and for any $k\geq 4$,
the variety $S\Nr_n\K_{n+4}$ is not atom-canonical.
\end{theorem}
\begin{proof}

This is proved only for $\CA_n$ in \cite{can}. Here we generalize the result to any $\K$ between $\Sc$ and $\PEA$.
The idea of the proof is essentially the same. Let $\At$ be the rainbow atom structure in \cite{Hodkinson} except that we have $n+2$ greens and
$n+1$ reds.

The rainbow signature now consists of $\g_i: i<n-1$, $\g_0^i: i\in n+2$, $\r_{kl}^t: k,l\in n+1$, $t\in \omega$,
binary relations and $\y_S$ $S\subseteq n+2$,
$S$ finite and a shade of red $\rho$; the latter is outside the rainbow signature,
but it labels coloured graphs during the game, and in fact \pe\ can win the $\omega$ rounded game
and build the $n$ homogeneous model $M$ by using $\rho$ when
she is forced a red.

Then $\Tm\At$ is representable; in fact it is representable as a polyadic equality algebra;
this can be proved exactly as in \cite{Hodkinson}; by defining the polyadic operations of the set algebra $\A$
completely analogous to the one constructed in \cite{Hodkinson}
by swapping variables. This algebra is representable, $\At\A=\At$ and
$\Tm\At\A\subseteq \A$.

The atoms of $\Tm\At\A$ are coloured graphs whose edges are not labelled by
the one shade of red  $\rho$; it can also be viewed as a set algebra based on $M$
by relativizing semantics discarding assignments whose edges are labelled
by $\rho$. A coloured graph (an atom) in $\PEA_{n+2, n+1}$
is one such that at least one of its edges is labelled red.

Now $\PEA_{n+2, n+1}$ embeds into $\Cm\At\A$,
by taking every red graph to the join of its copies, which exists because $\Cm\At\A$ is complete
(these joins do not exist in the (not complete) term algebra; only joins of finite
or cofinitely many reds do, hence it serves non representability.)
A copy of a red graph is one that is isomorphic to this graph, modulo removing superscripts of reds.

Another way to express this is to take every coloured graph to the interpretation of an infinite disjunct of the ${\sf MCA}$ formulas
(as defined in \cite{Hodkinson}), and to be dealt with below; such formulas define coloured graphs whose edges are not labelled by the shade of red,
hence the atoms, corresponding
to its copies, in the relativized semantics; this defines an embedding,  because $\Cm\At\A$ is isomorphic to
the set algebra based on the same relativized semantics
using $L_{\infty,\omega}^n$ formulas in the rainbow signature.
(Notice that that the rainbow theory itself \cite[definition 3.6.9]{HHbook2} is only first order because we have only finitely
many greens).

Here again $M$ is the new  homogeneous model constructed
in the new rainbow signature, though the construction is the
same \cite{Hodkinson}.
But \pa\ can win a certain fairly simple finite rounded atomic
game \cite{can} on $\Rd_{\Sc}\PEA_{n+1, n+2},$ hence it is
outside $S\Nr_n\Sc_{n+4}$ and so is $\Cm\At\A$,  because the former is embeddable in the latter
and $S\Nr_n\Sc_{n+4}$ is a variety; in particular, it is closed
under forming subalgebras.
\end{proof}

Algebraic results are most attractive when they have non-trivial impact on first order logic.
We apply our algebraic results to show that omitting types theorem fails in finite variable fragments
even if we consider only clique  guarded semantics.
To make this last statement more precise,
we start by lifting certain notions defined for relativized representations of relation algebras in \cite[definition 13.2-4]{HHbook2} to 
cylindric-like  algebras.

For an algebra $\A\in \PEA_n$ we let let $L(A)$ be the signuate consisting 
of an $n$ ary relation symbol for each element of $\A$. For $m>n$, 
$L(A)_{\omega,\omega}^m$  is the set of  first order formulas 
taken in this signature but using only $m$ variables
while $L(A)_{\infty, \omega}^m$ is the set of $L_{\infty, \omega}$ formulas taken in 
the signature $L(A)$ using also only $m>n$ variables.

Let $n>2$ be finite. Let 
$\A\in \PEA_n$. A set $M$ is a relativized representation of $M$ if 
there exists an injective homomorphism 
$f:\A\to \wp(V)$  where $V\subseteq {}^nM$  
and $\bigcup\rng(s)=V$. For $s\in {}^nM$, we write $1(s)$ if $s\in V$.

An $n$ clique in $M$, or simply a clique in $M$, is a subset 
$C$ of $M$ such that $M\models 1(\bar{s})$ for all $\bar{s}\in {}^nC$, such a clique
can be viewed as a hypergraph such that evey $n$ 
tuple is labelled by the top element.

For $m>n$, let $C^{m}(M)=\{\bar{a}\in {}^mM: \text { $\rng(\bar{a})$ is a clique in $M$}\}.$
Let $\A\in \PEA_n$, and let $M$ be a relativized representation of $\A$.

\begin{itemize}
\item 
$M$ is said to be $m$ square ($m>n$) if witnesses for cylindrfiers can be found only localy on $m$ cliques,
that is, if whenever $\bar{s}\in C^m(M)$, $a\in A$, $i<n$ and $M\models {\sf c}_ia(\bar{s})$,
then there is a $t\in C^m(M)$ with $\bar{t}\equiv _i \bar{s}$,
and $M\models a(\bar{t})$.

\item Closer to a genuine representation, since it reflects locally commutativity of cylindrfiers, we have 
$M$ is $m$ flat if  for all $\phi\in L(A)_{\omega, \omega}^m$, for all $\bar{a}\in C^m(M)$, for all $i,j<m$, we have
$$M\models \exists x_i\exists x_j\phi\longleftarrow \exists x_j\exists x_i\phi(\bar{a}).$$

\item Let $T$ be an $L_n$ theory. 
Then $M$ is an $m$ square (flat) model of $T$ if there exists $f:\Fm_T\to \wp(V)$, where $V\subseteq {}^nM$, 
$\bigcup_{s\in V}\rng(s)=M$,  and $M$ is $m$ square (flat).
\end{itemize}

For a (possibly relativized) representation $M$, a formula $\phi$,  and an assignment $s\in {}^{\alpha}M$, we write
$M\models \phi[s]$ if $s$ satisfies $\phi$ in $M$.

Let $T$ be a countable consistent first order theory and let $\Gamma$ be a finitary  type that is realized in every model of $\Gamma$.
Then the usual Orey Henkin theorem
tells us that this type is necessarily principal, that is, it is isolated by a formula $\phi$.
We call such a $\phi$ an $m$ witness, if $\phi$ is built up of $m$ variables.
We say that $\phi$ is only a witness if is an $m$ witness for a theory containing
$m$ variables.

We denote the set of formulas in a given language by $\Fm$ and for a set of formula $\Sigma$ we write $\Fm_{\Sigma}$ for the Tarski-
Lindenbaum quotient (polyadic)
algebra.

Let $T$ be an $L_n$ theory. Then $T$ is atomic if the algebra 
$\Fm_T$ is atomic. This coincides with the usual definition of atomic theories 
(without resorting to the Tarski-Lindenbaum 
quotient algebra). Indeed if $\Fm_T$ is atomic, 
then for any formula $\alpha$ consistent with $T$, there is a formua $\beta$ consistent
with $T$ such that $T\models \beta\to \alpha$, and for any formula $\theta$  consitent 
with $T$, we have $T\models \beta\to \theta$ or $T\models \beta\to \neg \theta$.

A theorem of Vaught, using a direct application of the opmitting types theorem says, 
that in first order logic, atomic theories (here one requires that for every $n<\omega$, $\Nr_n\Fm_T$ is atomic), 
have atomic models.

Now it turns out that when we seek atomic models for atomic theories
in $L_n$, the situation is drastically different than usual first order logic.
We may not find one, even among the $n+3$
flat models. In other words, atomic models
 may not be found even if we consider clique guarded semantics.
We shall use that $\A\in S_c\Nr_n\CA_{m}$ if it has an $m$ complete 
flat representation which is an $m$ flat representation preserving infinitary meets carrying 
them to set theoretic intersection. This can be proved like its analogue for relation algebras \cite{HHbook2}.

The idea is that one neatly embeds $\A$ into the $m$ dimensional algebra $\D$, 
with universe $\{\phi^{M}: \phi\in L(A)_{\infty, \omega}^m\}$. Here 
$M$ is the complete $M$  flat representation and $\phi^M=\{s\in C^m(M): M\models \phi[s]\}$, 
where the satisfiability relation $\models$  is defined recursively the usual way. 
The operations of $\D$, in turn, are  defined as expected. 

More precisely, we have:

\begin{theorem}\label{OTT}
Assume that $2<n<\omega$.
\begin{enumarab}
\item There is a countable consistent atomic $L_n$ theory $T$
that has no $n+4$ flat atomic model

\item There is a countable consistent $L_n$ theory $T$ and a type $\Gamma$ such that
$\Gamma$ is realized in every square $n+3$ model, but does not have a witness.

\end{enumarab}
\end{theorem}
\begin{proof}
\begin{enumarab}
\item
Let $\A$ be as above; then 
$\A$ is countable atomic and representable, but $\Cm\At\A\notin S\Nr_n\PEA_{n+4}$  
Let $\Gamma'$ be the set of  atoms.

We can and will assume that $\A$ is simple. Recall that it has countably many atoms, so we can assume that it is countable
by replacing it by its term algebra if necessary.

Then $\A=\Fm_T$ for some countable consistent $L_n$ theory $T$ and because $\A$ is atomic (as an expansion of a Boolean algebra),
and so  $T$ is atomic according to the above definition.

Now let $\Gamma=\{\phi: \phi_T\in \Gamma'\}$. Then we claim
that $\Gamma$ is realized in all $n+4$ flat models
models.  Consider such  a model $\M$ of $T$. If  $\Gamma$ is not realized in $\M$,
then this gives an $n+4$ complete flat  representation of $\A=\Fm_T$,
which is impossible,  because $\Cm\At\A$ is not in $S\Nr_n\PEA_{n+4}$.

Now  assume that  $M$ is an $n+4$ square atomic model of $T$.
Then $\Gamma$ is realized in $M$, and $M$ is atomic, so there must be a witness to $\Gamma$.

Suppose that $\phi$ is such a witness, so that $T\models \phi\to \Gamma$.
Then $\A$ is simple, and so we can assume
without loss of generality, that it is set algebra with a countable
base.

Let $\M=(M,R)$  be the corresponding model to this set algebra in the sense of \cite[section 4.3]{HMT2}.
Then $\M\models T$ and $\phi^{\M}\in \A$.
But $T\models \exists x\phi$, hence $\phi^{\M}\neq 0,$
from which it follows that  $\phi^{\M}$ must intersect an atom $\alpha\in \A$ (recall that the latter is atomic).
Let $\psi$ be the formula, such that $\psi^{\M}=\alpha$. Then it cannot
be the case that
that $T\models \phi\to \neg \psi$,
hence $\phi$ is not a  witness,
and we are done.

\item No we use the rainbow construction. Let $\A$ be the term algebra of $\PEA_{\N^{-1},\N}$.
Then $\A$ is an  atomic countable representable algebra, such that its $\Sc$ reduct, is not
in $S_c\Nr_n\Sc_{n+3}$ (for the same reasons as in the above rainbow proof,  namely,  \pe\ still can win all finite rounded games,
while \pa\ can win the game $F^{n+3}$, because $\A$ and the term algebra have the same
atom structure).

Assume that $\A=\Fm_T$, and let $\Gamma$ be the set $\{\phi: \neg \phi_T \text { is an atom }\}$.
Then $\Gamma$ is a non-principal type, because $\A$ is atomic, but it has no $n+3$ flat representation
omitting $\Gamma$,  for such a representation would necessarily yield a complete $n+3$ relativized representation of $\A,$ which in turn
implies that it would be in
$S_c\Nr_n\PEA_{n+3},$ and we know that this is not the case.
\end{enumarab}
\end{proof}

\section{First order definability of the class of neat reducts and $\Ra$ reducts,
another splitting technique}

Let $\K$ be any of cylindric algebra, polyadic algebra, with and without equality, or Pinter's substitution algebra.
We give a unified model theoretic construction, to show the following:
For $n\geq 3$ and $m\geq 3$, $\Nr_n\K_m$ is not elementary, and $S_c\Nr_n\K_{\omega}\nsubseteq \Nr_n\K_m.$

But first we give the general idea.
The idea at heart is Andr\'eka's splitting. A finite  atom structure,
in our case it will be a an $n$ dimensional cartesian square,   with accessibility relations corresponding to the concrete interpretations of
cylindrifiers and diagonal elements, is fixed in advance.
Then its atoms are
split twice.  Once, each atom is split into uncountably many, and once each into uncountably many except for one atom which
is only split  into {\it countably} many atoms. These atoms are called big atoms, which mean that they are cylindrically equivalent to their
original. This is a general theme in splitting arguments, as shown above.

The first splitting gives an algebra $\A$ that is a full neat reduct of an algebra in arbitrary extra dimensions;
the second gives an algebra $\B$ that is not a full neat reduct
of an algebra in just one extra dimensions, hence in any higher
dimensions.

Both algebras are representable, and elementary equivalent
 -in fact, it can be proved that their atom structures are $L_{\infty, \omega}$ equivalent -
because first order logic does not witness this infinite cardinality twist.

The Boolean reduct of $\A$ can be viewed as a countable direct product of disjoint Boolean relativizations of $\A$,
which are also atomic (denoted
below by $\A_u$).
Each component will be still uncountable; the product will be indexed by the elements of the atom structure.

The language of Boolean algebras can now be expanded
so that $\A$ is interpretable in an expanded structure $\P$,
based on the  same atomic Boolean product.

Now $\B$ can be viewed as obtained from $\P$, by replacing one of the components of the product with an elementary
{\it countable} Boolean subalgebra, and then giving it the same interpretation.
By the Feferman Vaught theorem (which says that replacing in a product one of its components by an elementary
equivalent one, the resulting product remains elementary equivalent to the original product) we have $\B\equiv \A$.

First order logic will not see this cardinality twist, but a suitably chosen term
not term definable in the language of
$\CA_{\alpha}$  will,
witnessing that the twisted algebra $\B$ is not a neat reduct.
The proof works if the atom that was split to countably many and uncountably
many atoms, is split into $\kappa$ and $\lambda$ many atoms,
where $\kappa>\lambda$ are any infinite cardinals.

A variant of the following lemma, is available in \cite{Sayed} with a sketch of proof; it is fully
proved in \cite{MLQ}. If we require that a (representable) algebra be a neat reduct,
then quantifier elimination of the base model guarantees this, as indeed illustrated below.

However, in \cite{Sayed} different relations symbols only had distinct interpretations, meaning that they could have non-empty intersections;
here we strengthen
this condition to that they have {\it disjoint} interpretations. We need this stronger
condition to show that our constructed algebras
are atomic.

\begin{lemma} Let $V=(\At, \equiv_i, {\sf d}_{ij})_{i,j<3}$ be a finite cylindric atom structure,
such that $|\At|\geq |{}^33.|$
Let $L$ be a signature consisting of the unary relation
symbols $P_0,P_1,P_2$ and
uncountably many ternary predicate symbols.
For $u\in V$, let $\chi_u$
be the formula $\bigwedge_{u\in V}  P_{u_i}(x_i)$.
Then there exists an $L$-structure $\M$ with the following properties:
\begin{enumarab}

\item $\M$ has quantifier elimination, i.e. every $L$-formula is equivalent
in $\M$ to a Boolean combination of atomic formulas.

\item The sets $P_i^{\M}$ for $i<n$ partition $M$, for any permutation $\tau$ on $3,$
$\forall x_0x_1x_2[R(x_0,x_1,x_2)\longleftrightarrow R(x_{\tau(0)},x_{\tau(1)}, x_{\tau(2)}],$

\item $\M \models \forall x_0x_1(R(x_0, x_1, x_2)\longrightarrow
\bigvee_{u\in V}\chi_u)$,
for all $R\in L$,

\item $\M\models  \forall x_0x_1x_2 (\chi_u\land R(x_0,x_1,x_2)\to \neg S(x_0,x_1,x_2))$
for all distinct ternary $R,S\in L$,
and $u\in V.$

\item For $u\in V$, $i<3,$
$\M\models \forall x_0x_1x_2
(\exists x_i\chi_u\longleftrightarrow \bigvee_{v\in V, v\equiv_iu}\chi_v),$

\item For $u\in V$ and any $L$-formula $\phi(x_0,x_1,x_2)$, if
$\M\models \exists x_0x_1x_2(\chi_u\land \phi)$ then
$\M\models
\forall x_0x_1x_2(\exists x_i\chi_u\longleftrightarrow
\exists x_i(\chi_u\land \phi))$ for all $i<3$
\end{enumarab}
\end{lemma}
\begin{proof}\cite{MLQ}
\end{proof}
\begin{lemma}\label{term}
\begin{enumarab}

\item For $\A\in \CA_3$ or $\A\in \Sc_3$, there exist
a unary term $\tau_4(x)$ in the language of $\Sc_4$ and a unary term $\tau(x)$ in the language of $\CA_3$
such that $\CA_4\models \tau_4(x)\leq \tau(x),$
and for $\A$ as above, and $u\in \At={}^33$,
$\tau^{\A}(\chi_{u})=\chi_{\tau^{\wp(^33)}(u).}$

\item For $\A\in \PEA_3$ or $\A\in \PA_3$, there exist a binary
term $\tau_4(x,y)$ in the language of $\Sc_4$ and another  binary term $\tau(x,y)$ in the language of $\Sc_3$
such that $PEA_4\models \tau_4(x,y)\leq \tau(x,y),$
and for $\A$ as above, and $u,v\in \At={}^33$,
$\tau^{\A}(\chi_{u}, \chi_{v})=\chi_{\tau^{\wp(^33)}(u,v)}.$


\end{enumarab}
\end{lemma}

\begin{proof}

\begin{enumarab}
\item For cylindric algebras $\tau_4(x)={}_3 {\sf s}(0,1)x$ and $\tau(x)={\sf s}_1^0c_1x\cdot {\sf s}_0^1{\sf c}_0x$.

\item For polyadic algebras, it is a little bit more complicated because the former term above is definable.
In this case we have $\tau(x,y)={\sf c}_1({\sf c}_0x\cdot {\sf s}_1^0{\sf c}_1y)\cdot {\sf c}_1x\cdot {\sf c}_0y$,
and $\tau_4(x,y)={\sf c}_3({\sf s}_3^1{\sf c}_3x\cdot {\sf s}_3^0{\sf c}_3y)$.

\end{enumarab}
\end{proof}

The terms $\tau(x)$ and $\tau(x,y)$ do not use any spare dimensions, and they approximate the terms that
do. The two former terms may be called {\it approximate witnesses}, while the last two may be called {\it $4$ witnesses}.

On the `global' level, namely, in the complex algebra of the finite
splitted atom structure which is a $3$ dimensional set algebra with base $3$, they are equal, the approximate witnesses
are $4$ witness. The small algebra does not see the $4$th dimension.

In the splitted algebras they are not, they only {\it strictly dominate} the terms that resort to one extra
dimension.  The latter terms will detect the cardinality twist that first order logic misses out on.
If term definable (in the case we have a full neat reduct of an algebra in only one extra dimension)
they carry one or more  uncountable component to an uncountable one,
and this is not possible if the smaller algebra were a neat
reduct, because in such algebras, the latter component is forced to be countable.

\begin{theorem}\label{neatreduct}
There exists an atomic $\A\in \Nr_3\QEA_{\omega}$
with an elementary equivalent polyadic equality  uncountable algebra $\B$
which is strongly representable, and its $\Sc$ reduct is not in $\Nr_3\Sc_4$.
Furthermore, the latter is a complete subalgebra of the former. In particular, for any class $\K$ between $\Sc$ and $\PEA$,
and for any $m>3$, the class $\Nr_3\K_m$
is not elementary.
\end{theorem}

\begin{proof} Let $\L$ and $\M$ as above. Let
$\A_{\omega}=\{\phi^M: \phi\in \L\}.$
Clearly $\A_{\omega}$ is a locally finite $\omega$-dimensional cylindric set algebra.

We can further assume that the relation symbols are indexed by an uncountable set $I$.
and that  there is a group structure on $I$, such that for distinct $i\neq j\in I$,
we have $R_i\circ R_j=R_{i+j}$.

Take $\At=({}^33, \equiv_i, \equiv_{ij}, {\sf d}_{ij})_{i,j\in 3}$, where
for $u,v\in \At$ and $i,j\in 3$,  $u\equiv_i v$ iff $u$ and $v$ agree off $i$ and $v\equiv_{ij}u$ iff $u\circ [i,j]=v$.
We denote $^33$ by $V$.

By the symmetry condition we have $\A$ is a $\PEA_3$, and
$\A\cong \Nr_3\A_{\omega}$, the isomorphism is given by
$\phi^{\M}\mapsto \phi^{\M}.$

In fact, $\A$ is not just a polyadic equality algebras, it is also closed under all first order definable
operations using extra dimensions for quantifier elimination in $\M$ guarantees that this map is onto, so that $\A$ is the full  neat reduct.

For $u\in {}V$, let $\A_u$ denote the relativization of $\A$ to $\chi_u^{\M}$
i.e $\A_u=\{x\in A: x\leq \chi_u^{\M}\}.$ Then $\A_u$ is a Boolean algebra.
Furthermore, $\A_u$ is uncountable and atomic for every $u\in V$
because by property (iv) of the above lemma,
the sets $(\chi_u\land R(x_0,x_1,x_2)^{\M})$, for $R\in L$
are disjoint of $\A_u$. It is easy to see that $\A_u$ is actually isomorphic to the finite co-finite Boolean algebra on  a set of cardinality $I$.

Define a map $f: \Bl\A\to \prod_{u\in {}V}\A_u$, by
$f(a)=\langle a\cdot \chi_u\rangle_{u\in{}V}.$
We expand the language of the Boolean algebra $\prod_{u\in V}\A_u$ by constants in
and unary operations, in such a way that
$\A$ becomes interpretable in the expanded structure.

Let $\P$ denote the
following structure for the signature of Boolean algebras expanded
by constant symbols $1_u$, $u\in {}V$ and ${\sf d}_{ij}$, and unary relation symbols
${\sf s}_{[i,j]}$ for $i,j\in 3$:

\begin{enumarab}
\item The Boolean part of $\P$ is the Boolean algebra $\prod_{u\in {}V}\A_u$,

\item $1_u^{\P}=f(\chi_u^{\M})=\langle 0,\cdots0,1,0,\cdots\rangle$
(with the $1$ in the $u^{th}$ place)
for each $u\in {}V$,

\item ${\sf d}_{ij}^{\P}=f({\sf d}_{ij}^{\A})$ for $i,j<\alpha$.

\item ${\sf s}_{[i,j]}^{\P}(x)= {\sf s}_{[i,j]}^{\P}\langle x_u: u\in V\rangle= \langle x_{u\circ [i,j]} : u\in V\rangle.$

\end{enumarab}

Define a map $f: \Bl\A\to \prod_{u\in {}V}\A_u$, by
$$f(a)=\langle a\cdot \chi_u\rangle_{u\in{}V}.$$

Then there are quantifier free formulas
$\eta_i(x,y)$ and $\eta_{ij}(x,y)$ such that
$$\P\models \eta_i(f(a), b)\text {  iff } b=f({\sf c}_i^{\A}a),$$
and
$$\P\models \eta_{ij}(f(a), b)\text { iff } b=f({\sf s}_{[i,j]}a).$$
The one corresponding to cylindrifiers is exactly like the $\CA$ case, the one corresponding to substitutions in $y={\sf s}_{[i,j]}x.$
Now, like the $\CA$ case, $\A$ is interpretable in $\P$, and indeed the interpretation is one dimensional and quantifier free.

For $v\in V$, let $\B_v$ be a complete countable elementary subalgebra of $\A_v$.
Then proceed like the $\CA$ case, except that we take a different product, since we have a different atom structure, with unary relations
for substitutions:
Let $u_1, u_2\in V$ and let $v=\tau(u_1,u_2)$, as given in the above lemma.
Let $J=\{u_1,u_2, {\sf s}_{[i,j]}v: i\neq  j<3\}$.
Let
$$\B=((\A_{u_1}\times \A_{u_2}\times \B_{v}\times \prod_{i,j<3, i\neq j} \B_{{\sf s}_{[i,j]}v}\times\prod_{u\in V\sim J} \A_u), 1_u, {\sf d}_{ij}, {\sf s}_{[i,j]}x)$$
inheriting the same interpretation. Then by the Feferman Vaught theorem,
which says that replacing a component in a possibly infinite product by  elementary equivalent
algebra, then the resulting new product is elementary equivalent to the original one, so that $\B\equiv \P$,
hence $\B\equiv \A$. (If a structure is interpretable in another structure then any structure
elementary equivalent to the former structure is elementary equivalent to the last).

In our new product we made all the permuted versions of $\B_v$ countable, so that $B_v$ {\it remains} countable,
because substitutions corresponding to transpositions
are present in our signature, so if one of the permuted components is uncountable, then $\B_{v}$ would be uncountable, and we do not want that.

The contradiction follows from the fact that had  $\B$ been a neat reduct, say $\B=\Nr_3\D$
then the term $\tau_3$ as in the above lemma, using $4$ variables, evaluated in $\D$ will force the component $\B_v$ to be uncountable,
which is not the case by construction,
indeed $\tau_3^{\D}(f(R_i), f(R_j))=f(R_{i+j})$.

\end{proof}

We now show that $\At\A\equiv_{\infty}\At\B$, we devise a pebble game similar to the rainbow pebble game
but each player has the option to choose an element from {\it both} structures,
and not just stick to one so that it is {\it a back and forth game} not just a forth game.

Pairs of pebbles are outside the board.
\pa\ as usual starts the game by placing a pebble on an element of one of the structures. \pe\
responds by placing the other pebble on the an element on the other structure.
Between them they choose an atom $a_i$ of $\At\A$
and an atom  $b_i$ of $\At\B$, under the restriction that player \pe\
must choose from the other structure from player \pa\ at each step.
A win for \pe\ if the binary relation resulting from the choices of the two players $R=\{(a,b): a\in \At(\A), b\in \At(\B)\}$ is a partial isomorphism.

At each step, if the play so far $(\bar{a}, \bar{b})$ and \pa\ chooses an atom $a$
in one of the structures, we have one of two case.
Either $a.1_u=a$ for some $u\neq Id$
in which case
\pe\ chooses the same atom in the other structure.
Else $a\leq 1_{Id}$
Then \pe\ chooses a new atom below $1_{Id}$
(distinct from $a$ and all atoms played so far.)
This is possible since there finitely many atoms in
play and there are infinitely many atoms below
$1_{u}$.
This strategy makes \pe\ win, since atoms below $1_u$ are cylindrically equivalent to $1_u$.
Let $J$ be a back and forth system which exists.
Order $J$ by reverse inclusion, that is $f\leq g$
if $f$ extends $g$. $\leq$ is a partial order on $J$.
For $g\in J$, let $[g]=\{f\in J: f\leq g\}$. Then $\{[g]: g\in J\}$ is the base of a
topology on
$J.$

Let $\C$ be the complete
Boolean algebra of regular open subsets of $J$ with respect to the topology
defined on $J.$
Form the Boolean extension $\M^{\C}.$
We want to define an isomorphism in $\M^{\C}$ of $\breve{\A}$ to
$\breve{\B}.$
Define $G$ by
$||G(\breve{a},\breve{b})||=\{f\in {J}: f(a)=b\}$
for $c\in \A$ and $d\in \B$.
If the right-hand side,  is not empty, that is it contains a function $f$, then let
$f_0$ be the restriction of $f$ to the substructure of $\A$ generated by $\{a\}$.
Then $f_0\in J.$ Also $\{f\in J:  f(c)=d\}=[f_0]\in \C.$
$G$ is therefore a $\C$-valued relation. Now let $u,v\in \M$.
Then
$||\breve{u}=\breve{v}||=1\text { iff }u=v,$
and
$||\breve{u}=\breve{v}||=0\text { iff } u\neq v$
Therefore
$||G(\breve{a},\breve{b})\land G(\breve{a},\breve{c})||\subseteq ||\breve{b}=\breve{c}||.$
for $a\in \A$ and $b,c\in \B.$
So ``$G$ is a function." is valid.
It is one to one because its converse is also a function.
(This can be proved the same way).
Also $G$ is surjective.

One can alternatively show that $\A\equiv_{\infty\omega}\B$ using "soft model theory" as follows:
Form a Boolean extension $\M^*$ of the universe $\M$
in which the cardinalities of $\A$ and $\B$ collapse to
$\omega$.  Then $\A$ and $\B$ are still back and forth equivalent in $\M^*.$
Then $\A\equiv_{\infty\omega}\B$ in $\M^*$, and hence also in $\M$
by absoluteness of $\models$.

Consider the two relations $\cong$ (isomorphism) and $\equiv$ (elementary equivalence)
between
structures. In one sense isomorphism is a more intrinsic property of structures,
because it is defined directly in terms of
structural properties. $\equiv$, on the other hand, involves a (first order) language.
But in another sense elementary equivalence is more
intrinsic because the existence of an isomorphism can depend on some
subtle questions about the surrounding universe of sets.

For example if $\M$ is a transitive model  of set theory containing vector
spaces $V$ and $W$ of dimensions $\omega$ and $\omega_1$ over
the same countable field, then $V$ and $W$ are not isomorphic in $\M$,
but they are isomorphic in an extension of $\M$
obtained by collapsing the cardinal $\omega_1$ to $\omega.$
By contrast, the question whether structures $\A$ and $\B$ are
elementary equivalent depends only on
$\A$ and $\B$, and not on sets around them.
Then $\equiv_{\infty}$ relation which coincides with so called Boolean-isomorphisms
as illustrated above.  This notion hovers between these two notions, and is
purely structural. It can be characterized by back and forth systems.

In \cite{r} Robin Hirsch
claims to prove that $\Ra\CA_k$ is not elementary for $k\geq 5$. The proof therein has a serious mistake, and the result it provides
is weaker, strictly so \footnote{This loophole in the proof was pointed out to Robin Hirsch by the author,
and an errata already accepted, will  appear soon in the Journal of Symbolic Logic recounting the whole story}.
We conjecture that our ideas above can also work for relation algebras,
but it seems to us that in this context delivering concrete witnesses
and a finite atom structure is more complicated.
So in our next theorem we stipulate the existence of both an approximate witness and witness that coincide
on an atom structure $\At$ that remains unknown to us so far.

However,
we hasten to add that we strongly conjecture that the $k$ witness can be found among Jonsson's $Q's$.
Such operations are generalizations of the operation of composition
and are {\it not term definable} in the language of $\RA$s, which is a necessary condition for the argument to
work.

What we can also say, is that such an atom structure $\At$, still finite and (strongly) representable,
will be much more complicated  than the simple $^33$,  but the idea remains the same.

{\bf Splitting the atoms
of $\At$ twice as we did before, obtaining two elementary equivalent relation algebras
one in $\Ra\CA_{\omega}$ and the other not in $\Ra\CA_k$, for some finite
$k\geq 5$, pending on the $k$ witness.}

\begin{theorem} Assume that there exists  a finite cylindric algebra atom structure $\At$
of dimension $3$, $\tau_k$ a $\CA_k$
term that is not term definable in $\RA$ so that $k\geq 4$, and
a $\CA_3$ term $\tau$ term (hence definable in $\RA$) such that
$\CA_k\models \tau_k(u_1,\ldots u_m)\leq \tau(u_1,\ldots u_m)$.

Assume further that there exist $u_1, \ldots, u_m\in \At$ such that
$$\tau^{\A}(\chi_{u_1}, \chi_{u_2},\ldots,  \chi_{u_m})=\chi_{\tau^{\At}(u_1, \ldots, u_m)}.$$
Then $\Ra\CA_k$ is not elementary.
\end{theorem}

\begin{proof} Let $\M$ and $\A$ be as above.
The $\Ra$ reduct of $\A$ is a generalized reduct of $\A$,
hence $\P$, as defined above, is first order interpretable in $\Ra\A$, too.
Then  there are closed terms,  a unary relation symbol, and formulas $\eta$ and $\mu$
built out of these closed terms and the unary
relation symbol such that
$\P\models \eta(f(a), b, c)\text { iff }b= f(a\circ^{\Ra\A} c),$ and $\P\models \mu(f(a),b)\text { iff }b=\breve{a}$
where the composition is taken in $\Ra\A$.
We can assume that such formulas are defined using the closed terms
$1_u$, a unary relation symbol $R$ corresponding to the operation of converse, and a nullary operation
corresponding to the constant
$Id$ (the identity element).

As before, for each $u\in \At$, choose any countable Boolean elementary
subalgebra of $\A_{u}$, $\B_{u}$ say.
Let $u_i: 1\leq i\leq m$ be elements in $\At$ as in the hypothesis
and let $v=\tau^{\At}(u_1,\ldots, u_m)$. Let
$$\B=((\prod_{u_i: i<m}\A_{u_i}\times \B_{v}\times \times \B_{\breve{v}}\times \prod_{u\in {}V\smallsetminus \{u_1,\ldots, u_m, v, \breve{v}\}}
\A_u), 1_u, R, Id) \equiv$$
$$(\prod_{u\in V} \A_u, 1_u, R, Id)=\P.$$

Let $\B$ be the result of applying the interpretation given above to $Q$.
Then $\B\equiv \Ra\A$ as relation  algebras.

Again $\B\in {\sf RRA}$, but it is not a full $\Ra$ reduct.
For if it were, then we use the above argument, forcing
the $\tau(u_1,\ldots, u_m)$
component together with its permuted versions (because we have converse) countable;
the resulting algebra will be a an elementary subalgebra of the original one, but $\tau_k$
will make twisted countable
component uncountable,
arriving at a contradiction.

\end{proof}

\newpage
\section*{Part 2}

\section{General system of varieties}

When generalizing Monk's  schema to include finite dimensions,
we had, in the infinite dimensional case, a two sorted situation one for the indices and one
for the first order situation.

The indices
are subscripts of the operations but they  are independent, in the sense that,
they do not interact on the level of the equational
axiomatization.

In polyadic algebras of infinite dimension the substitution operations are also indexed by transformations,
but the difference is, that these transformations
have an inner structure, they are not independent;
they interact via the operation of composition. Moreover, this interaction
is coded in the standard axiomatization due to Halmos.
It is a situation similar to axiomatizing modules over rings.

So in this case we need {\it more sorts}
to handle infinitary operations like cylindrifiers when allowed on infinitely many indices and
substitutions indexed by transformations that move infinitely many points.
For polyadic algebras of finite dimension one can give an equivalent formalism getting
rid of finite transformations as indices interacting, by replacing them by operations indexed
only by double indices, reflecting transpositions and replacements. Indeed, this was accomplished by Sain and Thompson.

The idea is that these are the generators of the semigroup of finite
transformations. (This idea is common in algebraic logic, frequently occurring under the name of the {\it semigroup approach},
witness \cite{Sain}, \cite{Sain2}.)
Indeed this task is implemented elegantly in \cite{ST}.)

But first, we fix some notation.
$\sf Ord$ is the class of all ordinals. For ordinals $\alpha<\beta$, $[\alpha,\beta]$ denotes the set of ordinals $\mu $ such that
$\alpha\leq \mu \leq \beta$. $I_{\alpha}$ is the class $\{\beta\in {\sf Ord}:\beta\geq \alpha\}$
By an interval of ordinals, or simply an interval, we either mean $[\alpha, \beta]$ or $I_{\alpha}$.

\begin{definition}\label{Halmos}
\begin{enumroman}
\item A type schema is a quintuple $t=(T, \delta, \rho,c, s)$ such that
$T$ is a set, $\delta$ maps $T$ into $\omega$, $c,s\in T$, and $\delta c=\rho c=\delta s=\rho s=1$.
\item A type schema as in (i) defines a similarity type $t_{\alpha}$ for each $\alpha$ as follows. Sets
$C_{\alpha}\subseteq \wp(\alpha)$, $G_{\alpha}\subseteq {} ^{\alpha}{\alpha}$ are fixed,
and the domain $T_{\alpha}$ of $t_{\alpha}$ is
$$T_{\alpha}=\{(f, k_0,\ldots, k_{\delta f-1}): f\in T\sim\{c,s\}, k\in {}^{\delta f}\alpha\}$$
$$\cup \{(c, r): r\in C_{\alpha}\}\cup \{(q,r): r\in C_{\alpha}\}\cup \{(s,\tau): \tau\in G_{\alpha}\}.$$
For each $(f, k_0,\ldots, k_{\delta f-1})\in T_{\alpha},$ we set $t_{\alpha}(f, k_0\ldots k_{\delta f-1})=\rho f$
and we set $\rho(c,r)=\rho(q,r)=\rho (s,\tau)=1.$
\item Let $\mu$ be an interval of ordinals. A system $(\K_{\alpha}: \alpha\in \mu)$ of classes
of algebras is of type schema $t$ if for each $\alpha\in \mu$, the class $\K_{\alpha}$ is a class of algebras of type
$t_{\alpha}$.
\end{enumroman}
\end{definition}
\begin{definition} Let $L_T$ be the first order language that consists of countably many unary relational symbols $(Rel)$,
countably many function symbols $(Func)$ and countably many
constants $(Cons)$, which we denote by  $r_1, r_2\ldots$ and $f_1, f_2\ldots$ and $n_1, n_2\ldots$,
respectively. We let $L_T=Rel\cup Cons\cup Func$.
\end{definition}
\begin{definition}
\begin{enumarab}
\item A schema is a pair $(s, e)$ where $s$ is a first order formula of $L_T$ and $e$ is an equation in the language $L_{\omega}$
of $\K_{\omega}$. We denote a schema $(s, e)$ by $s\to e$. We define $Ind(L_{\omega})=\omega\cup C_{\omega}\cup G_{\omega}$.
A function $h: L_T\to L_{\omega}$ is admissible if $h$ is an injection and
$h\upharpoonright Const\subseteq \omega, h\upharpoonright Rel\subseteq C_{\omega}$ and
$h\upharpoonright Func\subseteq G_{\omega}.$
\item Let $g$ be an equation in the language of $\K_{\alpha}$. Then $g$ is an $\alpha$ instance of a schema
$s\to e$ if there exist an admissible function $h$, sets, functions and constants
$$r_1^{M},r_2^{M}\ldots f_1^{M}, f_2^{M}\ldots \in {}G_{\alpha}, n_1^{M},n_2^{M}\ldots \in \alpha$$
such
$$M=(\alpha, r_1^M, r_2^M,\ldots f_1^M, f_2^M, n_1^{M}, n_2^{M},\ldots )\models s$$
and $g$ is obtained from $e$ by replacing $h(r_i), h(f_i)$ and $h(n_i)$ by $r_i^M$, $f_i^M$ and $n_i^M$, respectively.
\end{enumarab}
\end{definition}

\begin{definition}
A system of varieties is a {\it generalized system of varieties definable by a schema}, if there exists a strictly finite set of schemes,
such that for every $\alpha$,
$\K_{\alpha}$ is axiomatized  by  the $\alpha$ dimensional instances of such schemes.
\end{definition}

Given such a system of varieties, we denote algebras in $\K_{\alpha}$ by
$$\A=(\B, {\sf c}_{(r)}, {\sf q}_{(r)}, s_{\tau})_{r\in C_r, \tau\in G_{\alpha}},$$ that is, we highlight the operations of cylindrifiers and substitutions,
and the operations in $T\sim \{c,s\}$ (of the Monk's schema part, so to speak),
with indices from $\alpha$, are encoded in $\B$.

Indeed, Monk's definition is the special case, when we forget the sort of substitutions.
That is a system of varieties is definable by Monk's schemes if $G_{\alpha}=\emptyset$ for all $\alpha$ and each schema the form
by ${\sf True}\to e$.
\begin{example}
\begin{enumarab}

\item Tarski's cylindric algebras, Pinter's substitution algebras, Halmos' quasi-polyadic algebras and Halmos' quasi-polyadic algebras with equality,
all of infinite dimension; these are indeed instances of Monk's schema.
Here $G_{\alpha}=\emptyset$ for all $\alpha\geq \omega$, and $C_{\alpha}=\{\{r\}: r\in \alpha\}$.

\item Less obvious are Halmos' {\it polyadic algebras}, of infinite dimension, as defined in \cite{HMT2}.
Such algebras are  axiomatized by Halmos schemes; hence the form a generalized system of varieties
definable by a schema of equations. For example, see \cite{Sagi}, the $\omega$ instance of $(P_{11})$ is:

$[(\forall y) (r_2(y)\longleftrightarrow \exists z(r_1(z)\land y=f_1(z)\land (\forall y,z)(r_2(y)\land r_2(z)\land y\neq z\implies\\
f_1(y)\neq f_1(z), {\sf c}_{r_1}{\sf s}_{f_1}(x)={\sf s}_{f_1}{\sf c}_{r_2}(x)].$
\end{enumarab}

\end{example}

\begin{example} Cylindric-polyadic algebras \cite{Fer}.
These are reducts of polyadic algebras of infinite dimension, where we have all substitutions, but cylindrifiers
are  allowed only on finitely many
indices. Such algebras have become fashionable lately, with the important recent the work of Ferenczi. However,
Ferenczi deals with (non-classical) versions of such algebras, where commutativity of cylindrifiers is weakened,
substantially, and he proves strong representation theorems on generalized set algebras.
Here $C_{\alpha}$ is again the set of singletons, manifesting the cylindric spirit of the algebras, while
$G_{\alpha}= {}^{\alpha}\alpha$, manifesting, in turn, its polyadic reduct.
We shall show below that all such varieties have the super-amalgamation property.
\end{example}

\begin{example}

(a) Sain's algebras \cite{Sain}: Such algebras provide a solution to one of the most central problems in algebraic logic,
namely, the so referred to in the literature as the fintizability problem.
Those  are  countable reducts of polyadic algebras, and indeed of cylindric-polyadic algebras
(as far as their similarity type is concerned; for polyadic algebras their axiomatization is identical to polyadic algebras
restricted to their similarity types).
Cylindrifiers are finite, that is, they are defined only on finitely many indices, but at least  two infinitary substitutions are there.

Like polyadic algebras, and for that matter cylindric-polyadic algebras, such classes algebras, which happen to be varieties, can
be easily formulated as a generalized system of varieties definable by a schema on the
interval $[\alpha, \alpha+\omega]$, $\alpha$ a countable ordinal.

Here we only have substitutions coming from a {\it countable} semigroup $G_{\alpha},$ and $G_{\alpha+n}$, $n\leq \omega$, is the
sub-semigroup of $^{\alpha+n}\alpha+n$ generated by $\bar{\tau}=\tau\cup Id_{(\alpha+n)\sim \alpha}$, $\tau \in G_{\alpha}$.
Such algebras, were introduced by Sain, can be modified, in case the semigroups determining their similarity types
are finitely presented, providing first order logic without equality a {\it strictly finitely based} algebraization,
witness \cite{Bulletin} for an overview of such results and more.

(b) Sain's algebras with diagonal elements \cite{Sain2}; the results therein are summarized in \cite{Sain}.
These are investigated by Sain and Gyuris, in the context of
finitizing first order logic {\it with equality}. This problem turns out to be harder, and so
the results obtained are weaker, because the class of representable algebras $V$ is not elementary for it is not closed under ultraproducts.
The authors manage to provide, however, in this case a generalized finite schema, for the class of  ${\bf H}V$; this only implies weak completeness
for the corresponding infinitary
logics; that is $\models \phi$ implies $\vdash \phi$, relative to a finitary Hilbert style axiomatization, involving only type free valid schemes.
However, there are non-empty sets of formulas $\Gamma$, such that $\Gamma\models \phi$, but there is no proof of $\phi$ from
$\Gamma$. We will approach an intuitionistic version of such algebras below, with and without diagonals, proving a representation and an
amalgamation theorem
\end{example}

It is timely to also highlight the novelties in the above definition when compared to Monk's definition of a system definable by schemes.
\begin{enumarab}

\item First, the most striking addition, is that it allows dealing with infinitary substitutions coming from a set
$G_{\alpha}$, which is usually a semigroup. Also infinitary cylindrifiers are permitted.
This, as indicated in the above examples, covers polyadic algebras,
Heyting polyadic algebras, $MV$ polyadic algebras and Ferenczi's
cylindric-polyadic algebras, together with their important reducts studied by Sain.

\item Second thing, cylindrifiers are not mandatory;
this covers many algebraizations of multi dimensional modal logics, like for example modal logics of substitutions \cite{Sagi}.

\item We also have another (universal) quantifier $q$, intended to be the dual of cylindrifiers in the case of presence of negation;
representing  universal quantification. This is appropriate for logics where we do not have negation in the classical sense,
like intuitionistic logic,  expressed algebraically by
Heyting polyadic algebras \cite{s}.

\item Finally, the system could be definable only on an interval of ordinals of the form $[\alpha,\beta]$, while the usual definition of Monk's schema
defines systems of varieties on $I_{\omega}$; without this more general condition,
we would have not been able to approach Sain's algebras \cite{Sain}, \cite{Sain2}.

\end{enumarab}

\subsection{Interpolation for cylindric -polyadic algebras}

The notion of relativized representations constitute a huge topic in both algebraic and modal logic, see the introduction of 
\cite{1} for an overview.
Historically,  in \cite{HMT1} square units got all the attention and relativization  was treated as a side issue.
However, extending original classes of models for logics to manipulate their properties is common. 
This is no mere tactical opportunism, general models just do the right thing.

The famous move from standard models to generalized models is 
Henkin's turning round  second  order logic into an axiomatizable two sorted first
order logic. Such moves are most attractive 
when they get an independent motivation. 

The idea is that we want to find a semantics that gives just the right action 
while additional effects of square set theoretic representations are separated out as negotiable decisions of formulation 
that can threaten completeness, decidability, and interpolation. 
(This comes across very much in cylindric algebras, especially in finite variable fragments of first order logic,
and classical polyadic equality algebras, in the context of Keisler's logic with equality.)

Using relativized representations Ferenczi \cite{Fer}, proved that if we weaken commutativity of cylindrifiers
and allow  relativized representations, then we get a finitely axiomatizable variety of representable 
quasi-polyadic equality algebras (analogous to the Andr\'eka-Resek Thompson ${\sf CA}$ version, cf. \cite{Sayedneat} and \cite{Fer}, 
for a discussion of the Andr\'eka-Resek Thompson breakthrough for cylindric-like algebras); 
even more this can be done without the merry go round identities.

This is in sharp contrast with the long list of  complexity results proved for the commutative case.
Ferenczi's results can be seen as establishing a hitherto fruitful contact between neat embedding 
theorems 
and relativized representations, with enriching 
repercussions and insights for both notions.

This task was also implemented by Ferenczi for infinite dimensions sidestepping negtive results for polyadic equality algebras.
Here relativization, and weakening the axioms in a smaller signature 
are done simultaneously. 
We prove a strong representability
result for several  such classes of non- commutative cylindric-polyadic algebras.
The choice of `Henkin ultrafilters' that define the required 
representation here is more delicate because we do not have full fledged commutativity 
of cylindrifiers like in the classixcal case.

We present our proof for $\sf CPEA$;
as defined in \cite[definition 6.3.7]{Fer}. This is just, we think, a representative sample.
The proof applies to many (if not all) non-commutative cylindric-polyadic algebras introduced in \cite{Fer}.

Such algebras termed cylindric-polyadic algebras of infinite dimension, will be denoted by $\alpha$, $\alpha$
an infinite ordinal, has a strong Stone-like representability result. Abstract algebras defined by a strictly finite schema of equations are
representable.
This schema is similar to that of polyadic algebras {\it with equality} with the very
important difference that commutativity of cylindrifiers is weakened. Here all substitutions are available in the signature,
so that such algebras are, like polyadic algebras, are
also examples of {\it transformation systems} having a polyadic facet.
But, on the other hand, only ordinary cylindrifiers are allowed in their signature, so in this respect, they have a cylindric facet.

Here the representable algebras have units that are unions of weak spaces, that
unlike, analogous representable cylindric algebras may not be disjoint, and thats why cylindrifiers may not commute.

We denote the class of representable algebras
by ${\sf Gp}_{\alpha}$.

It is proved in \cite{conference} that atomic such algebras are completely representable.
Complete representability is stronger, in this context, than the usual representability result proved by
Ferenczi because, such varieties are canonical
and each algebra embeds into its canonical extension which is
complete and atomic.
The complete representabiliy of the latter implies the
representability of the former.

We show that
such varieties have the super-amalgamation property.

We should mention that ${\sf CPEA_{\alpha}}$ has the super-amalgamation property, can be also proved
using the duality theory in modal logic between Kripke frames and complex algebras.

This follows from the fact that the Kripke frames (atom structures) of the algebras considered
are closed under zig zag products,  because they can be axiomatized by Horn
sentences, hence are clausifiable, and that the variety is
canonical, because it is can be easily axiomatized
by  positive equations. The rest follows from a result of Marx.

However, we chose to give a proof of
the interpolation property similar to the Henkin constructions used above.
We infer immediately using standard `bridge theorems' in algebraic
logic that ${\sf CPEA}_{\alpha}$ has the superamalgamation
property. Ferenczi's representability result can also be discerned from
our proof. The proof has affinity with the proof in \cite{super}, and for that matter, to the proof
of the same result for many valued predicate logic given in theorem \ref{mv}
below. The proof of the next theorem involves a significant contribution from Mohammad Assem, namely, checking the induction step.

\begin{theorem}\label{Fer}
Let $\beta$ be a cardinal, and $\A=\Fr_{\beta}{}\CPEA_{\alpha}$ be the free algebra on $\beta$ generators.
Let $X_1, X_2\subseteq \beta$, $a\in \Sg^{\A}X_1$ and $c\in \Sg^{\A}X_2$ be such that $a\leq c$.
Then there exists $b\in \Sg^{\A}(X_1\cap X_2)$ such that $a\leq b\leq c.$ This is the algebraic version of the Craig
interpolation property.
\end{theorem}

\begin{proof}

Assume that $\mu$ is a regular cardinal $>max(|\alpha|,|A|)$.
Let $\B\in \CPEA_{\mu}$, such that $\A=\Nr_{\alpha}\B$,
and $A$ generates $\B$. Such dilations exist.
Ultrafilters in dilations used to represent algebras in $\CPEA$ are  as before
defined via  the {\it admitted substitutions}, which
we denote by $adm.$ Recall that very admitted substitution has a domain $\dom\tau$ which is  subsets of $\alpha$ and a range,
$\rng\tau$ such that $\rng\tau\cap \alpha=\emptyset$.
One defines special  filters in the dilations $\Sg^{\B}X_1$ and in $\Sg^{\B}X_2$
like before; in particular, they have to be compatible on the common subalgebra. This needs some work.
Assume that no interpolant exists in $\Sg^{\A}(X_1\cap X_2)$.
Then, as above,
no interpolant exists in $\Sg^{\B}(X_1\cap X_2)$.
We eventually arrive at a contradiction.

Now we deal with triples since $adm$ has to come to the picture. Arrange $adm\times\mu \times \Sg^{\B}(X_1)$
and $adm\times\mu\times \Sg^{\B}(X_2)$
into $\mu$-termed sequences:
$$\langle (\tau_i,k_i,x_i): i\in \mu\rangle\text {  and  }\langle (\sigma_i,l_i,y_i):i\in \mu\rangle
\text {  respectively.}$$
is as desired.
Thus we can define by recursion (or step-by-step)
$\mu$-termed sequences of witnesses:
$$\langle u_i:i\in \mu\rangle \text { and }\langle v_i:i\in \mu\rangle$$
such that for all $i\in \mu$ we have:
$$u_i\in \mu\smallsetminus
(\Delta a\cup \Delta c)\cup \cup_{j\leq i}(\Delta x_j\cup \Delta y_j\cup \dom\tau_j\cup \rng\tau_j\cup \dom\sigma_j\cup \rng\sigma_j)$$
$$\cup \{u_j:j<i\}\cup \{v_j:j<i\}$$
and
$$v_i\in \mu\smallsetminus(\Delta a\cup \Delta c)\cup
\cup_{j\leq i}(\Delta x_j\cup \Delta y_j\cup \dom\tau_j\cup \rng\tau_j\cup \dom\sigma_j\cup \rng\sigma_j)$$
$$\cup \{u_j:j\leq i\}\cup \{v_j:j<i\}.$$
For an  algebra $\D$ we write $Bl\D$ to denote its Boolean reduct.
For a Boolean algebra $\C$  and $Y\subseteq \C$, we write
$fl^{\C}Y$ to denote the Boolean filter generated by $Y$ in $\C.$  Now let
$$Y_1= \{a\}\cup \{-{\sf s}_{\tau_i}{\sf  c}_{k_i}x_i+{\sf s}_{\tau_i}{\sf s}_{u_i}^{k_i}x_i: i\in \mu\},$$
$$Y_2=\{-c\}\cup \{-{\sf s}_{\sigma_i}{\sf  c}_{l_i}y_i+{\sf s}_{\sigma_i}{\sf s}_{v_i}^{l_i}y_i:i\in \mu\},$$
$$H_1= fl^{Bl\Sg^{\B}(X_1)}Y_1,\  H_2=fl^{Bl\Sg^{\B}(X_2)}Y_2,$$ and
$$H=fl^{Bl\Sg^{\B}(X_1\cap X_2)}[(H_1\cap \Sg^{\B}(X_1\cap X_2)
\cup (H_2\cap \Sg^{\B}(X_1\cap X_2)].$$
We claim that $H$ is a proper filter of $\Sg^{B}(X_1\cap X_2).$
To prove this it is sufficient to consider any pair of finite, strictly
increasing sequences of natural numbers
$$\eta(0)<\eta(1)\cdots <\eta(n-1)<\mu\text { and } \xi(0)<\xi(1)<\cdots
<\xi(m-1)<\mu,$$
and to prove that the following condition holds:

(1) For any $b_0$, $b_1\in \Sg^{B}(X_1\cap X_2)$ such that
$$a\cdot \prod_{i<n}(-{\sf s}_{\tau_{\eta(i)}}{\sf  c}_{k_{\eta(i)}}x_{\eta(i)}+{\sf s}_{\tau_{\eta(i)}}{\sf s}_{u_{\eta(i)}}^{k_{\eta(i)}}x_{\eta(i)})\leq b_0$$
and
$$(-c\cdot .\prod_{i<m}
(-{\sf s}_{\sigma_{\xi(i)}}{\sf  c}_{l_{\xi(i)}}y_{\xi(i)}+{\sf s}_{\sigma_{\xi(i)}}{\sf s}_{v_{\xi(i)}}^{l_{\xi(i)}}y_{\xi(i)})\leq b_1$$
we have
$$b_\cdot b_1\neq 0.$$
We prove this by a tedious induction on $n+m$.
If $n+m=0$, then (1) simply
expresses the fact that no interpolant of $a$ and $c$ exists in
$\Sg^{\B}(X_1\cap X_2).$
In more detail: if $n+m=0$, then $a_0\leq b_0$
and $-c\leq b_1$. So if $b_0\cdot b_1=0$, we get $a\leq b_0\leq -b_1\leq c.$
Now assume that $n+m>0$ and for the time being suppose that $\eta(n-1)>\xi(m-1)$.
Apply ${\sf  c}_{u_{\eta(n-1)}}$ to both sides of the first inclusion of (1).
By $u_{\eta(n-1)}\notin \Delta a$, i.e. ${\sf  c}_{u_{\eta(n-1)}}a=a$,
and by recalling that ${\sf  c}_i({\sf  c}_ix.y)={\sf  c}_ix.{\sf  c}_iy$, we get (2)
$$a\cdot {\sf  c}_{u_{\eta(n-1)}}\prod_{i<n}(-{\sf s}_{\tau_{\eta(i)}}{\sf  c}_{k_{\eta(i)}}x_{\eta(i)}+{\sf s}_{\tau_{\eta(i)}}{\sf s}_{u_{\eta(i)}}^{k_{\eta(i)}}x_{\eta(i)})\leq {\sf  c}_{u_{\eta(n-1)}}b_0.$$
Let ${\sf  c}_i^{\partial}(x)=-{\sf  c}_i(-x)$.  ${\sf  c}_i^{\partial}$
is the algebraic counterpart of the universal quantifier $\forall x_i$.
Now apply ${\sf  c}_{u_{\eta(n-1)}}^{\partial}$ to the second inclusion of (1).
By noting that ${\sf  c}_i^{\partial}$, the dual of ${\sf  c}_i$,
distributes over the Boolean meet and by
$u_{\eta(n-1)}\notin \Delta c=\Delta (-c)$ we get (3)
$$(-c)\cdot \prod_{j<m}{\sf  c}_{u_{\eta(n-1)}}^{\partial}(-{\sf s}_{\sigma_{\xi(i)}}{\sf  c}_{l_{\xi(i)}}y_{\xi(i)}+{\sf s}_{\sigma_{\xi(i)}}{\sf s}_{v_{\xi(i)}}^{l_{\xi(i)}}y_{\xi(i)})\leq {\sf  c}_{u_{\eta(n-1)}}^{\partial}b_1.$$
Before going on, we formulate (and prove) a claim that will enable us to eliminate
the quantifier ${\sf  c}_{u_{\eta(n-1)}}$ (and its dual) from (2) (and (3)) above.

For the sake of brevity set for each $i<n$ and each $j<m:$
$$z_i=-{\sf  c}_{k_{\eta(i)}}x_{\eta(i)}+{\sf s}_{u_{\eta(i)}}^{k_{\eta(i)}}x_{\eta(i)}$$  and
$$t_i=-{\sf  c}_{l_{\xi(i)}}y_{\xi(i)}+{\sf s}_{v_{\xi(i)}}^{l_{\xi(i)}}y_{\xi(i)}.$$
Then $(i)$ and $(ii)$ below hold:
$$(i)\  {\sf  c}_{u_{\eta(n-1)}}z_i=z_i\text { for }i<n-1 \text { and }{\sf  c}_{u_{\eta(n-1)}}z_{n-1}=1.$$
$$(ii)\ {\sf  c}_{u_{\eta(n-1)}}^{\partial}t_j=t_j\text { for all }j<m.$$

 {\it Proof of ${\sf  c}_{u_{\eta_{n-1}}}z_i=z_i$ for  $i<n-1$.}

Let $i<n-1$. Then by the choice of witnesses we have
$$u_{\eta(n-1)}\neq u_{\eta(i)}.$$
Also it is easy to see that for all $i,j\in \alpha$ we have
$$\Delta {\sf  c}_jx\subseteq \Delta x\text {  and that }
\Delta {\sf s}_j^ix\subseteq \Delta x\smallsetminus\{i\}\cup \{j\},$$
In particular,
$$u_{\eta(n-1)}\notin \Delta {\sf  c}_{k_{\eta(i)}}x_{\eta(i)}\text { and }
u_{\eta(n-1)}\notin \Delta ({\sf s}_{u_{\eta(i)}}^{k_{\eta(i)}}x_{\eta(i)}).$$
It thus follows that
$${\sf  c}_{u_{\eta(n-1)}}(-{\sf  c}_{k_{\eta(i)}}x_{\eta(i)})=-{\sf  c}_{k_{\eta(i)}}x_{\eta(i)}\text { and }
{\sf  c}_{u_{\eta(n-1)}} ({\sf s}_{u_{\eta(i)}}^{k_{\eta(i)}}x_{\eta(i)})={\sf s}_{u_{\eta(i)}}^{k_{\eta(i)}}
x_{\eta(i)}.$$
Finally, by ${\sf  c}_{u_{\eta(n-1)}}$ distributing over the Boolean join, we get
$${\sf  c}_{u_{\eta(n-1)}} z_i=z_i \text { for }  i<n-1.$$

{\it Proof of ${\sf  c}_{u_{\eta(n-1)}}z_{n-1}=1.$}
Computing we get, by $u_{\eta(n-1)}\notin \Delta x_{\xi(n-1)}$
and by \cite[1.5.8(i), 1.5.8(ii)]{HMT1}
the following:
\begin{align*}
&{\sf  c}_{u_{\eta(n-1)}}(-{\sf  c}_{k_{\eta(n-1)}}x_{\eta(n-1)}
+ {\sf s}_{u_{\eta(n-1)}}^{k_{\eta(n-1)}}x_{\eta(n-1)})\\
&={\sf  c}_{u_{\eta(n-1)}}-{\sf  c}_{k_{\eta(n-1)}}x_{\eta(n-1)}+ {\sf  c}_{u_{\eta(n-1)}}
{\sf s}_{u_{\eta(n-1)}}^{k_{\eta(n-1)}}x_{\eta(n-1)}\\
&=-{\sf  c}_{k_{\eta(n-1)}}x_{\eta(n-1)}+ {\sf  c}_{u_{\eta(n-1)}}{\sf s}_{u_{\eta(n-1)}}^{k_{\eta(n-1)}}
x_{\eta(n-1)}\\
&=-{\sf  c}_{k_{\eta(n-1)}}x_{\eta(n-1)}+ {\sf  c}_{u_{\eta(n-1)}}{\sf s}_{u_{\eta(n-1)}}^{k_{\eta(n-1)}}
{\sf  c}_{u_{\eta(n-1)}}x_{\eta(n-1)}\\
&=-{\sf  c}_{k_{\eta(n-1)}}x_{\eta(n-1)}+ {\sf  c}_{k_{\eta(n-1)}}{\sf s}_{k_{\eta(n-1)}}^{u_{\eta(n-1)}}
{\sf  c}_{u_{\eta(n-1)}}x_{\eta(n-1)}\\
&=-{\sf  c}_{k_{\eta(n-1)}}x_{\eta(n-1)}+ {\sf  c}_{k_{\eta(n-1)}}{\sf  c}_{u_{\eta(n-1)}}x_{\eta(n-1)}\\
&= -{\sf  c}_{k_{\eta(n-1)}}x_{\eta(n-1)}+ {\sf  c}_{k_{\eta(n-1)}}x_{\eta(n-1)}=1.
\end{align*}
With this the proof of (i) in our claim is complete.
Now we prove (ii).
Let $j<m$ . Then we have
$${\sf  c}_{u_{\eta(n-1)}}^{\partial}(-{\sf  c}_{l_{\xi(j)}}y_{\xi(j)})
=-{\sf  c}_{l_{\xi(j)}}y_{\xi(j)}$$ and
$${\sf  c}_{u_{\eta(n-1)}}^{\partial}
({\sf s}_{v_{\xi(j)}}^{l_{\xi(j)}}y_{\xi(j)})={\sf s}_{v_{\xi(j)}}^{l_{\xi(j)}}y_{\xi(j)}.$$
Indeed,  computing we get
$${\sf  c}_{u_{\eta(n-1)}}^{\partial}(-{\sf  c}_{l_{\xi(j)}}y_{\xi(j)})
=-{\sf  c}_{u_{\eta_{n-1}}}-(-{\sf  c}_{l_{\xi(j)}}y_{\xi(j)})
= -{\sf  c}_{u_{\eta(n-1)}}{\sf  c}_{l_{\xi(j)}}y_{\xi(j)}
=-{\sf  c}_{l_{\xi(j)}}y_{\xi(j)}.$$
Similarly,  we have
$${\sf  c}_{u_{\eta(n-1)}}^{\partial} ({\sf s}_{v_{\xi(j)}}^{l_{\xi(j)}}y_{\xi(j)})
=-{\sf  c}_{u_{\eta(n-1)}}- ({\sf s}_{v_{\xi(j)}}^{l_{\xi(j)}}y_{\xi(j)})$$
$$=-{\sf  c}_{u_{\eta(n-1)}} ({\sf s}_{v_{\xi(j)}}^{l_{\xi(j)}}-y_{\xi(j)})
=- {\sf s}_{v_{\xi(j)}}^{l_{\xi(j)}}-y_{\xi(j)}
= {\sf s}_{v_{\xi(j)}}^{l_{\xi(i)}}y_{\xi(j)}.$$
By ${\sf  c}_i^{\partial}({\sf  c}_i^{\partial}x+y)= {\sf  c}_i^{\partial}x+{\sf  c}_i^{\partial}y$
we get from the above that
$${\sf  c}_{u_{\eta(n-1)}}^{\partial}(t_j)={\sf  c}_{u_{\eta(n-1)}}^{\partial}({\sf  c}_{l_{\xi(j)}}y_{\xi(j)}+{\sf s}_{v_{\xi(j)}}^{l_{\xi(j)}}y_{\xi(j)})$$
$$={\sf  c}_{u_{\eta(n-1)}}^{\partial}{\sf  c}_{l_{\xi(j)}}y_{\xi(j)}+{\sf  c}_{u_{\eta(n-1)}}^{\partial}
{\sf s}_{v_{\xi(j)}}^{l_{\xi(j)}}y_{\xi(j)}$$
$$={\sf  c}_{l_{\xi(j)}}y_{\xi(j)}+ {\sf s}_{v_{\xi(j)}}^{l_{\xi(j)}}y_{\xi(j)}=t_j.$$

By the above proven claim we have
$${\sf  c}_{u_{\eta(n-1)}}\prod_{i<n}z_i={\sf  c}_{u_{\eta(n-1)}}[\prod_{i<n-1}z_i.z_n]$$
$$={\sf  c}_{u_{\eta(n-1)}}\prod_{i<n-1}z_i.
{\sf  c}_{u_{\eta(n-1)}}z_{n-1}=\prod_{i<n-1}z_i.$$
Here we are using that ${\sf  c}_i({\sf  c}_ix.y)={\sf  c}_ix.{\sf  c}_iy$.
Combined with (2) we obtain
$$a\cdot \prod_{i<n-1}(-{\sf  c}_{k_{\eta(i)}}x_{\eta(i)}+{\sf s}_{u_{\eta(i)}}^{k_{\eta(i)}}x_{\eta(i)})
 \leq {\sf  c}_{u_{\eta(n-1)}}b_0.$$
On the other hand, from our claim and (3),
 it follows that
$$(-c)\cdot \prod_{j<m}
(-{\sf  c}_{l_{\xi(j)}}y_{\xi(j)}+{\sf s}_{v_{\xi(j)}}^{l_{\xi(j)}}y_{\xi(j)})\leq {\sf  c}_{u_{\eta(n-1)}}^{\partial}b_1.$$
Now making use of the induction hypothesis we get
$${\sf  c}_{u_{\eta(n-1)}}b_0.{\sf  c}_{u_{\eta(n-1)}}^{\partial}b_1\neq 0;$$
and hence that
$$b_0.{\sf  c}_{u_{\eta(n-1)}}^{\partial}b_1\neq 0.$$
From
$$b_0.{\sf  c}_{u_{\eta(n-1)}}^{\partial}b_1\leq b_0.b_1$$
we reach the desired conclusion, i.e. that
$$b_0.b_1\neq 0.$$
The other case, when $\eta(n-1)\leq \xi(m-1)$ can be treated analogously
and is therefore left to the reader.
We have proved that $H$ is a proper filter.

Proving that $H$ is a proper filter of $\Sg^{\B}(X_1\cap X_2)$,
let $H^*$ be a (proper Boolean) ultrafilter of $\Sg^{\B}(X_1\cap X_2)$
containing $H.$
We obtain  ultrafilters $F_1$ and $F_2$ of $\Sg^{\B}(X_1)$ and $\Sg^{\B}(X_2)$,
respectively, such that
$$H^*\subseteq F_1,\ \  H^*\subseteq F_2$$
and (**)
$$F_1\cap \Sg^{\B}(X_1\cap X_2)= H^*= F_2\cap \Sg^{\B}(X_1\cap X_2).$$
Now for all $x\in \Sg^{\B}(X_1\cap X_2)$ we have
$$x\in F_1\text { if and only if } x\in F_2.$$
Also from how we defined our ultrafilters, $F_i$ for $i\in \{1,2\}$ satisfy the following
condition:

(*) For all $k<\mu$, for all $x\in \Sg^{\B}X_i$
if ${\sf  c}_kx\in F_i$ then ${\sf s}_l^kx$ is in $F_i$ for some $l\notin \Delta x.$
We obtain  ultrafilters $F_1$ and $F_2$ of $\Sg^{\B}X_1$ and $\Sg^{\B}X_2$,
respectively, such that
$$H^*\subseteq F_1,\ \  H^*\subseteq F_2$$
and (**)
$$F_1\cap \Sg^{\B}(X_1\cap X_2)= H^*= F_2\cap \Sg^{\B}(X_1\cap X_2).$$
Now for all $x\in \Sg^{\B}(X_1\cap X_2)$ we have
$$x\in F_1\text { if and only if } x\in F_2.$$
Also from how we defined our ultrafilters, $F_i$ for $i\in \{1,2\}$ are perfect.

Then define the homomorphisms, one on each subalgebra, like in \cite{Sayedneat} p. 128-129, using the perfect ultrafilters,
then freeness will enable pase these homomorphisms, to a single one defined to the set of free generators,
which we can assume to be, without any loss, to
be $X_1\cap X_2$ and it will satisfy  $h(a.-c)\neq 0$ which is a contradiction.

Then $H$ is a proper filter of $\Sg^{\B}(X_1\cap X_2)$. This can be proved by a tedious induction, with the base provided
by the fact that no interpolant exists in the dilation.
Proving that $H$ is a proper filter of $\Sg^{\B}(X_1\cap X_2)$,
let $H^*$ be a (proper Boolean) ultrafilter of $\Sg^{\B}(X_1\cap X_2)$
containing $H.$
We obtain  ultrafilters $F_1$ and $F_2$ of $\Sg^{\B}(X_1)$ and $\Sg^{\B}(X_2)$,
respectively, such that
$$H^*\subseteq F_1,\ \  H^*\subseteq F_2$$
and (**)
$$F_1\cap \Sg^{\B}(X_1\cap X_2)= H^*= F_2\cap \Sg^{\B}(X_1\cap X_2).$$
Now for all $x\in \Sg^{\B}(X_1\cap X_2)$ we have
$$x\in F_1\text { if and only if } x\in F_2.$$
Also from how we defined our ultrafilters, $F_i$ for $i\in \{1,2\}$ are {\it perfect}, a term introduced by Ferenczi.
Then one defines  homomorphisms, one on each subalgebra, like in \cite{Sayedneat} p. 128-129, using the
perfect ultrafilters to define a congruence relation on $\beta$ so that the defined homomorphisms respect diagonal elements.
Then freeness will enable paste these homomorphisms, to a single one defined to the set of free generators,
which we can assume to be, without any loss, to
be $X_1\cap X_2$ and it will satisfy  $h(a.-c)\neq 0$ which is a contradiction.
\end{proof}

\subsection{ A remark on interpolation}

Suppose we are in the classical case. For the interpolation property to hold for infinitary logics {\it with} equality
a necessary (and in some cases also sufficient)
condition is that algebraic terms
definable using the added spare dimensions are already term definable.

The typical situation (even for cylindric algebras with no restriction whatsoever, like local finiteness)
an interpolant  can {\it always} be found, but the problem is that  it might, and indeed there are situations where it must,
resort to extra dimensions (variables).

This happens in the case for instance of the so called finitary logics of infinitary relations \cite{typeless}.
This interpolant then becomes  term definable in higher dimensions, but when these terms
are actually coded by ones using only the original amount of dimensions,
we get the desired interpolant.

One way of getting around the unwarranted spare dimensions in the interpolant,
is to introduce new connectives that code extra dimensions, in which case in the original language any implication can be interpolated
by a formula using the same number of variables (and common symbols) but  possibly uses the new connectives.

In \cite{typeless} this is done to prove an interpolation
theorem for the severely incomplete typless logics studied in \cite{HMT2}.
This will also be implemented in a while in the intuitionistic context.

And here is where category theory to be addressed in a wider scope in a while,
offers a concise and for that matter an intuitive formulation.

{\it The dilation functor
(which takes an algebra to an algebra in $\omega$ extra dimensions), the interpolant is found in a dilation,
but its inverse, the neat reduct functor gets us back to our base, to our original algebra.}
This is an adjoint situation.

This statement will be proved rigorously below, witness theorem \ref{cat}.
We first formulate the statement  {\bf an interpolant can always be found} in the context of systems of varieties
definable by a schema, under certain mild conditions that cover many
cylindric-like algebras, and also non classical cases that we just encountered, namely in the context of Heyting polyadic algebras
without and indeed with diagonal elements.

In the next section, we approach in much more detail
the statement:

{\bf Neat reduct functor and its inverse the dilation functor}.

For the first bolded statement, we shall need the notion of
dimension-restricted free algebras, which is a form of relativisation
of the notion of freeness.

\begin{definition}\cite[definition 2.5.31]{HMT1}
Let $\delta$ be a cardinal.
Let $\alpha$ be an ordinal.
Let $_{\alpha} \Fr_{\delta}$ be the absolutely free algebra on $\delta$
generators and of type $t_{\alpha}.$
For an algebra
$\A,$ we write $R\in \Co\A$ if $R$ is a congruence relation on $\A.$
Let $\rho\in {}^{\delta}\wp(\alpha)$.
Let $L$ be a class having the same similarity type as
$t_{\alpha}.$
Let
$$Cr_{\delta}^{(\rho)}L=\bigcap\{R: R\in \Co_{\alpha}\Fr_{\delta},
{}_{\alpha}\Fr_{\delta}/R\in SL,
{\sf c}_k^{_{\alpha}\Fr_{\delta}}{\eta}/R=\eta/R \text { for each }$$
$$\eta<\delta \text
{ and each }k\in \alpha\smallsetminus
\rho{\eta}\}$$
and
$$\Fr_{\delta}^{\rho}L={}_{\alpha}\Fr_{\beta}/Cr_{\delta}^{(\rho)}L.$$
\end{definition}
The ordinal $\alpha$ does not figure out in $Cr_{\delta}^{(\rho)}L$ and $\Fr_{\delta}^{(\rho)}L$
though it is involved in their definition.
However, the dimension $\alpha$ will be clear from context so that no confusion is likely to ensue.
For  algebras $\A$, $\B$  and $X\subseteq \A$, $Hom(\A,\B)$ stands for the class of all
homomorphisms from $\A$ to $\B$
and $\Sg^{\A}X$, or simply $\Sg X$ when $\A$ is clear from context,
denotes the subalgebra of $\A$ generated by $X$. $\Ig^{\A}X$, or simply $\Ig X$ denotes the ideal generated by $X$.
Now free algebras, as defined, have the following universal property:

\begin{theorem} If $\A=\Fr_{\beta}^{\rho}\K_{\alpha}$ and $x=(\eta/Cr^{\rho}\K_{\alpha}: \eta<\beta)$. Then for any $\B\in \K_{\alpha}$
$y=(y_{\eta}:\eta<\beta)\in {}^{\beta}\B$, there is a unique $h\in Hom(\A,\B)$ such that $h\circ x=y$
\end{theorem}
\begin{demo}{Proof} Let $\B\in \K_{\alpha}$ and $y\in {}^{\beta}\B$ where $\Delta y_{\eta}\subseteq \Delta\eta$ for every $\eta<\beta$.
Then there is an $f\in Hom(\Fr_{\beta}, \B)$ such that $f\eta=y_{\eta}$ for each $\eta<\beta$. Then $\Fr_{\beta}/ker f\cong \Sg^{\B}Rgy\in \K_{\alpha}$
and $\delta(\eta/ker f)=\Delta y_{\eta}\subseteq \Delta \eta$; this shows that $Cr_{\beta}^{\rho)}\K_{\alpha}\subseteq ker f$. Then
$h(\eta/Cr_{\beta}^{\rho}\K_{\alpha})=f(\eta)$
is well defined and as required.
\end{demo}
Recall that:
\begin{definition} An algebra $\A$ has the {\it strong interpolation theorem}, $SIP$ for short, if for all $X_1, X_2\subseteq A$, $a\in \Sg^{\A}X_1$,
$c\in \Sg^{\A}X_2$ with $a\leq c$, there exist $b\in \Sg^{\A}(X_1\cap X_2)$ such that $a\leq b\leq c$.
\end{definition}

Apiece of notation:
\begin{athm}{Notation} For a term ${\tau}$, let  ${\sf var}({\tau})$ denote the set of variables in $\tau$ and $({\sf ind})\tau$ denote
its set of
indices.
\end{athm}

\begin{theorem} Let $\alpha$ be an ordinal and $(\K_{\alpha}: \alpha\geq \omega)$ be a system of varieties
definable by a schema. Let ${\sf c}_i^{\partial}x=-{\sf c}_i-x.$ Assume that the signature of $\K_{\omega}$ contains a
unary operation ${\sf s}_i^j$, $i,j\in \omega$ satisfying the following equations
\begin{enumarab}

\item ${\sf c}_i^{\partial}x\leq {\sf s}_i^jx\leq {\sf c}_i x$
\item ${\sf c}_i{\sf s}_i^j={\sf s}_i^jx$ if $i\neq j$
\item ${\sf s}_i^j{\sf c}_ix={\sf c}_ix$.

\item ${\sf s}_i^j{\sf c}_kx={\sf c}_k{\sf s}_i^jx$ whenever $k\notin \{i,j\}$.

\item ${\sf c}_i{\sf s}_j^ix={\sf c}_j{\sf s}_i^jx$.

\end{enumarab}

Let $K$ be a class of algebras such that $S\Nr_{\alpha}\K_{\alpha+\omega}\subseteq K \subseteq \K_{\alpha}$.
Assume  that $\Fr_{\beta}^{\rho}\K_{\alpha+\omega}$ has the interpolation property.
Then for any terms of the language of $\K_{\alpha}$, $\sigma, \tau$ say,  if
$K\models \sigma\leq \tau$, then there exists a term $\pi$ with ${\sf var}(\pi)\subseteq {\sf var}(\sigma)\cap {\sf var}(\tau)$ and
$${\sf c}_{(\Delta)}^{(\partial)}\sigma\leq \pi\leq {\sf c}_{(\Delta)}\tau$$
\end{theorem}

\begin{demo}{Proof}
Assume that
such that $K\models \sigma \leq \tau$.
We want to find an interpolant, i.e a $\pi$ as in the conclusion. Let
${\L}_{\alpha}$ be the language of $\K_{\alpha}$ and for
$n\leq \omega$, let ${\L}_{\alpha}^{n}$ or simply ${\L}^{n}$ be the language of
$\K_{\alpha+n}$. We write ${\L}$ for  ${\L}^{(0)}$.For an assignment $s:\omega\to \B$, $\B\in \K_{\alpha}$, we write $\bar{s}$ for its extension to
all terms of $\L_{\alpha}$.
Now since $\Nr_\alpha \K_{\alpha + \omega} \subseteq K$, then for every $\B\in \K_{\alpha+\omega}$,
for every $s:\omega\to \B$ such that $\rng(s)\subseteq \Nr_{\alpha}{\B}$
we have $\B\models (\sigma \leq \tau)[\bar{s}]$.

Hence, by hypothesis, there is a term $\pi$ of ${\L}^{(\omega)} $
which contains only occurrences of variables which occur in both $\sigma$ and $\tau$
and which satisfies that for all $\B\in \K_{\alpha+\omega}$, for every
$s:\omega\to \B$, such that $\rng(s)\subseteq \Nr_{\alpha}\B$,
$$\B\models (\sigma \leq \pi)[\bar{s}]\text { and } (\pi \leq\tau)[\bar{s}].$$

Extend each language ${\L}^{n)} $, $n <\omega$, by adjoining
to its signature a fixed
$\omega$-terms sequence
$ a = \langle a_0, a_1, a_2,\ldots \rangle  $ of distinct individual constants symbols.
Let $\sigma',\tau'$ and $\pi'$ be the terms of the language extending ${\L}^{(\omega)} $
that are obtained respectively from $\sigma,\tau$ and $\pi$ by replacing each variable $v_k$
in all of its occurrences by the constant symbol $a_k$.
For each $ k< \omega$ let $\Pi^{(k)}$ be the set of all identities of the form
${\sf c}_\mu a_v = a_v$.
where $ \alpha \leq \mu < \alpha + k $ and $ v < \omega.$
Now we have
$$\Sigma^{(\omega)} \cup \bigcup_{k < \omega} \Pi^{(k)} \models (\sigma' \leq  \pi') \wedge (\pi' \leq \tau')$$
where $\Sigma^{(\omega)}={\sf Eq}(\K_{\alpha+\omega})$.
Therefore, by the compactness theorem  is a finite subset $\theta$ of $\Sigma^{(\omega)} \cup \bigcup_{k <\omega} \Pi^{(k)} $
such that
$$\theta \models ( \sigma' \leq  \pi') \wedge (\pi' \leq\tau').$$
Then there there is a finite ordinal, $\delta$, say
such that $ \theta \subseteq \Sigma^{(\delta)} \bigcup\Pi^{(\delta)}$.
Now ${\sf ind}(\pi)\subseteq \alpha +\delta$.
Choose two distinct sets $ \Gamma, \triangle \subseteq \alpha$ such that $|\Gamma| = |\triangle | = \delta$
and such that neither $\Gamma$ nor $ \triangle$
contains any index occurring in
$\sigma, \tau$, or $\pi$.
Let $\mu, v$ be two sequences of length $\delta$
which enumerate the elements of $\triangle$, respectively ( and hence are necessarily one-one).
Let $\bar{\sigma}$ and $\bar{\tau}$ be the terms of ${\bf L}$ that are obtained from $\sigma$ and $\tau$, respectively,
by replacing each variable $v_k$ in all its occurrences by the
term$$ {\sf s}^{\mu_0}_{v_0} {\sf s}^{\mu_1}_{v_1}\ldots{\sf s}^{\mu_{\delta - 1}}_{v_{\delta -1}} v_k. $$
Let $\bar{\pi}$
be obtained from $\pi$ by making these same replacements
and also by replacing
every index $\lambda$ in $\pi$ such that $\lambda \in (\alpha + \delta) \sim \alpha$  to
an  ordinal in $\Gamma$; so that different ordinals are substituted for different ordinals.
Notice that $\bar{\pi}$, as well as $\bar{\sigma}$ and $\bar{\tau}$, is a term of ${\bf L}$.
Finally, let $\rho$ be any one-one function from $ \alpha + \delta$ onto $\alpha$ such that $\delta\xi = \xi$
for every $ \xi < \alpha$, $\xi\in {\sf ind} \sigma\cup {\sf ind} {\tau} \cup {\sf ind {\pi}}$
such that
$$\rho(k + k) = \mu_k ~~\textrm{for every} ~~k < \delta.$$
Then for every $ {\C} \in \K_\alpha$ we have,
$\Rd^{(\rho)} {\C} \in  \K_{\alpha+ \delta}$
and by hypothesis, we get

\begin{equation*}{\sf s}^{\mu_0}_{v_0} {\sf s}^{\mu_1}_{v_1}\ldots {\sf s}^{\mu_{\delta - 1}}_{v_{\delta- 1}}x \in \Nr_\alpha \Rd^{(\rho)} {\C}
~~\textrm{for every}~~ x\in C\end{equation*}
We  can now readily conclude that
$$\K_{\alpha}\models \bar{\sigma} \leq\bar{\pi}\text { and }\K_{\alpha}\models \bar{\pi} \leq \bar{\tau}.$$
Now neither the $\Gamma$ nor $\triangle$ contains an index  occurring in
$ \sigma, \tau$, or $\pi$
then
$$\K_{\alpha}\models \bar{\sigma} = {\sf s}^{\mu_0}_{v_0} {\sf s}^{\mu_1}_{v_1}\ldots{\sf s}^{\mu_{\delta - 1}}_{v_{\delta- 1}} \sigma
\text { and }
\K_{\alpha}\models  \bar{\tau} = {\sf s}^{\mu_0}_{v_0} {\sf s}^{\mu_1}_{v_1}\ldots {\sf s}^{\mu_{\delta - 1}}_{v_{\delta- 1}} \tau$$
Combining these results we get that
$$\K_{\alpha}\models
{\sf s}^{\mu_0}_{v_0} {\sf s}^{\mu_1}_{v_1}\ldots {\sf s}^{\mu_{\delta - 1}}_{v_{\delta- 1}} \sigma \leq \bar{\pi} \leq  {\sf s}^{\mu_0}_{v_0} {\sf s}^{\mu_1}_{v_1}....
{\sf s}^{\mu_{\delta - 1}}_{v_{\delta - 1}} \tau$$
in particular, in every member of $K$;
and the  same is true of
$$ {\sf c}^\partial_{\mu_0} {\sf c}^\partial_{\mu_1}\dots {\sf c}^\partial_{\mu_{\delta-1}} \sigma \leq \bar{\pi}
\leq {\sf c}_{\mu_0}{\sf c}_{\mu_1} \dots {\sf c}_{\mu_{\delta-1}} \tau.$$
Therefore, since $ \bar{\pi}$ is a term of ${\L}$ and it contains like $\pi$ only occurrences of variables
which occur at the same time in both $\sigma$ and $ \tau$
we have shown that the inclusion $\sigma \leq \tau$
can indeed be interpolated relative to $\sf K$.
\end{demo}

In \cite{typeless} it is shown that if we have finitely many infinitary substitutions and the language is countable,
then we can get rid of the quantifies variables, so that a genuine interpolant can be found.
However, the formulas to be interpolated are taken in the original language, while the interpolant can (and in some cases must) contain the new
unary connectives corresponding to the added substitutions; the latter operations {\it code, hence eliminate,
the finitely many quantifies variables that appear
in the double inequality of the previous theorem}.

\section{Interpolation and amalgamation in non-classical versions}


\subsection{$MV$ polyadic  algebras}

An instance of our notion of generalized systems of varieties,
here we show that $MV$ polyadic algebras (corresponding to many valued quantifier logic)
have the superamalgamation property.

Our result addresses infinitary logics.

Here the relation symbols in the signature are allowed to be infinite and quantification
can be taken on infinitely many variables. The usual case when formulas contain only finitely many variables is a special case
of our formalism. Our next theorem shows that Henkin construction, reflected algebraically by
neat embedding theorems, apply to many valued logics. Here we {\it  do not} have equality.
In our next two cases we shall deal with equality (allowing diagonal elements in our signature).

The propositional part of our algebras, are $MV$ algebras. They are Boolean algebras without idempotency of
the conjuncts and disjuncts (but we have negation that behaves classically),
which indeed is a major difference, but one to be expected since we are dealing
with many valued logic. The values can be infinite, and indeed, uncountable.

When we impose idempotency we recover the analogous results
for the classical case, proved in many publications, dating back to the sixties of
the 20 century.

Such investigations started  with the pioneering paper of Daigneault
\cite{D}
(who was a student of Halmos the inventor of polyadic algebras)
followed by the paper of  of Daigneault and Monk \cite{DM}
proving a strong representation (completeness) theorem, proved independently also by Keisler \cite{K},
and finally,
Johnson's paper \cite{J}
proving the strong amalgamation property for polyadic (Boolean) algebras.
This algebraic result implies that Craig interpolation and Beth definability properties hold
for Keisler's logics \cite{K}.

Johnsson's proof, however,  is horribly complicated taking more than thirty pages in the Transactions of
the $AMS$.
This proof was considerably simplified and streamlined by Sayed-Ahmed,
proving even a stronger result, namely, that the class of polyadic algebras have
the super-amalgamation property \cite{super}. It is the ideas and techniques in the latter reference
that we use. We show that the technique survives lack of idempotency.

Some basics.
\begin{definition} An $MV$ algebra is an algebra
$$\A=(A, \oplus, \odot, \neg, 0,1)$$
where $\oplus$, $\odot$ are binary operations, $\neg$ is a unary operation and $0,1\in A$, such that the following identities hold:
\begin{enumerate}
\item $a\oplus b=b\oplus a,\ \ \  a\odot b=b\odot a.$
\item $a\oplus (b \oplus c)=(a\oplus b)\oplus c$,\  \  \ $a\odot (b \odot c)=(a\odot b)\odot c.$

\item $a\oplus 0=a$ ,\ \ \ $a\odot 1=a.$
\item $a\oplus 1=1$, \ \ \ $a\odot 0=a.$
\item  $a\oplus \neg a=1$,\ \ \  $a\odot \neg a=0.$
\item $\neg (a\oplus b)=\neg a\odot \neg b,$\ \ \ $\neg (a\odot b)=\neg a\oplus \neg b.$
\item $a=\neg \neg a$\ \ \ $\neg 0=1.$
\item $\neg(\neg a\oplus b)\oplus b=\neg(\neg b\oplus a)\oplus a.$
\end{enumerate}
\end{definition}
$MV$ algebras form a variety that is a subvariety of the variety of  $BL$ algebras introduced by Hajek,
in fact $MV$ algebras coincide with those $BL$ algebras satisfying double negation law,
namely that $\neg\neg x=x$, and contains all  Boolean algebras.
\begin{example}
A simple numerical example is $A=[0,1]$ with operations $x\oplus y=min(x+y, 1)$, $x\odot y=max(x+y-1, 0)$,  and $\neg x=1-x$.
In mathematical fuzzy logic, this $MV$-algebra is called the standard $MV$ algebra,
as it forms the standard real-valued semantics of Lukasiewicz logic.
\end{example}

$MV$ algebras also arise from the study of continuous $t$ norms.
\begin{definition}A $t$ norm is a binary operation $*$ on $[0,1]$, i.e $(t:[0,1]^2\to [0,1]$) such that
\begin{enumroman}
\item  $*$ is commutative and associative,
that is for all $x,y,z\in [0,1]$,
$$x*y=y*x,$$
$$(x*y)*z=x*(y*z).$$
\item $*$ is non decreasing in both arguments, that is
$$x_1\leq x_2\implies x_1*y\leq x_2*y,$$
$$y_1\leq y_2\implies x*y_1\leq x*y_2.$$
\item $1*x=x$ and $0*x=0$ for all $x\in [0,1].$

\end{enumroman}
\end{definition}
The following are the most important (known) examples of continuous $t$ norms.

\begin{enumroman}
\item Lukasiewicz $t$ norm: $x*y=max(0,x+y-1),$
\item Godel $t$ norm $x*y=min(x,y),$
\item Product $t$ norm $x*y=x.y$.
\end{enumroman}
We have the following known result \cite[lemma 2.1.6]{H}

\begin{theorem} Let $*$ be a continuous $t$ norm.
Then there is a unique binary operation $x\to y$ satisfying for all $x,y,z\in [0,1]$, the condition $(x*z)\leq y$ iff $z\leq (x\to y)$, namely
$x\to y=max\{z: x*z\leq y\}.$
\end{theorem}
The operation $x\to y$ is called the residuam of the $t$ norm. The residuam $\to$
defines its corresponding unary operation of precomplement
$\neg x=(x\to 0)$.
Abstracting away from $t$ norms, we get

\begin{definition} A residuated lattice is an algebra
$$(L,\cup,\cap, *, \to 0,1)$$
with four binary operations and two constants such that
\begin{enumroman}
\item $(L,\cup,\cap, 0,1)$ is a lattice with largest element $1$ and the least element $0$ (with respect to the lattice ordering defined the usual way:
$a\leq b$ iff $a\cap b=a$).
\item $(L,*,1)$ is a commutative semigroup with largest element $1$, that is $*$ is commutative, associative, $1*x=x$ for all $x$.
\item Letting $\leq$ denote the usual lattice ordering, we have $*$ and $\to $ form an adjoint pair, i.e for all $x,y,z$
$$z\leq (x\to y)\Longleftrightarrow x*z\leq y.$$
\end{enumroman}
\end{definition}
A result of Hajek, is that an $MV$ algebra is  a prelinear commutative bounded integral residuated lattice
satisfying the additional identity $x\cup y=(x\to y)\to y.$ In case of an $MV$ algebra, $*$ is the so-called strong conjunction which we
denote here following standard notation in the
literature by $\odot$. $\cap$ is called weak conjunction. The other operations are defined by
$\neg a=a\to 0$ and $a\oplus b=\neg(\neg a\odot \neg b).$ The operation $\cup$ is called weak disjunction, while $\oplus$ is called
strong disjunction. The presence of weak and strong conjunction is a common feature
of substructural logics without the rule of contraction, to which Lukasiewicz
logic belongs.

For an algebra $\A$, $End(\A)$ denotes the set of endomorphisms of $\A$, i.e homomorphisms from $\A$ into itself.

\begin{definition} A transformation system is a quadruple $(\A, I, G, S)$ where $\A$ is an algebra, $I$ is a set, $G$ is a subsemigroup of $^II$
and $S$ is a homomorphism from
$G$ into $End(\A)$
\end{definition}
We shall deal with three cases of $G$. When $G$ is the semigroup of finite transformations, $G$ is a countable subset of $^II$
satisfying certain conditions but containing infinitary substitutions and $G={}^II.$
$\A$ will always be an $MV$ algebra.
If we want to study predicate $MV$ logic, then we are naturally led to expansions of $MV$ algebras allowing quantification.
\begin{definition} Let $\A$ be an $MV$ algebra. An existential quantifier on $\A$ is a function
$\exists:A\to A$ that satisfies the following conditions for all $a, b\in A$:
\begin{enumerate}
\item $\exists 0=0.$
\item $a\leq \exists a.$
\item $\exists(a\odot \exists b)=\exists a\odot \exists b.$
\item $\exists(a\oplus \exists b)=\exists a\oplus \exists b.$
\item  $\exists (a\odot a)=\exists a\odot  \exists a.$
\item $\exists (a\oplus a)=\exists a\oplus \exists a.$
\end{enumerate}
\end{definition}
 Let $\A$ be an $MV$ algebra, with existential quantifier $\exists$. For $a\in A$, set $\forall a=\neg \exists \neg a$.
Then $\forall$ is a unary operation on $\A$ called  a universal quantifier and it satisfies all properties
of the existential quantifier, except for (2) which takes the form: $\forall a\leq a$. This follows directly from the axioms.
Now we define our algebras. Their similarity type depends on a fixed in advance semigroup.
We write $X\subseteq_{\omega} Y$ to denote that $X$ is a finite subset
of $Y$, and we write $\wp_{\omega}(Y)$ for $\{X: X\subseteq_{\omega}Y\}$.\footnote{There should be  no conflict with the notation $S_{\mathfrak{m}}V$
introduced earlier.}

\begin{definition} Let $\alpha$ be an infinite set. Let $G\subseteq {}^{\alpha}\alpha$ be a semigroup under the operation of composition of maps. Let
$T\subseteq \wp(\alpha)$.
An $\alpha$ dimensional polyadic $MV$ algebra of type $(G,T)$, an $MV_{G,T}$ for short,
is an algebra of the following
type
$$(A,\oplus,\odot,\neg, 0, 1, {\sf s}_{\tau}, {\sf c}_{(J)} )_{\tau\in G, J\in T}$$
where
$(A,\oplus,\odot, \neg, 0, 1)$ is an $MV$ algebra, ${\sf s}_{\tau}:\A\to \A$ is an endomorphism of $MV$ algebras,
${\sf c}_{(J)}$ is an existential quantifier, such that the following hold for all
$p\in A$, $\sigma, \tau\in G$ and $J,J'\in T:$
\begin{enumarab}
\item ${\sf s}_{Id}p=p,$
\item ${\sf s}_{\sigma\circ \tau}p={\sf s}_{\sigma}{\sf s}_{\tau}p$ (so that $S:\tau\mapsto {\sf s}_{\tau}$ defines a homomorphism from $G$ to $End(\A)$;
that is $(A, \oplus, \odot, \neg, 0, 1, G, S)$ is a transformation system),
\item ${\sf c}_{(J\cup J')}p={\sf c}_{(J)}{\sf c}_{(J')}p,$
\item If $\sigma\upharpoonright \alpha\sim J=\tau\upharpoonright \alpha\sim J$, then
${\sf s}_{\sigma}{\sf c}_{(J)}p={\sf s}_{\tau}{\sf c}_{(J)}p,$
\item If $\sigma\upharpoonright \sigma^{-1}(J)$ is injective, then
${\sf c}_{(J)}{\sf s}_{\sigma}p={\sf s}_{\sigma}{\sf c}_{\sigma^{-1}(J)}p.$
\end{enumarab}
\end{definition}
\begin{theorem} Let $\A=(A,\oplus, \odot, \neg, 0, 1, {\sf s}_{\tau}, {\sf c}_{(J)})_{\tau\in G, J\in T}$ be an $MV_{G,T}$.
For each $J\in T$ and $x\in A$, set ${\sf q}_{(J)}x=\neg {\sf c}_{(J)}\neg x.$
Then the following hold for each $p\in A$, $\sigma, \tau\in G$ and $J,J'\in T$.
\begin{enumarab}
\item ${\sf q}_{(J)}$ is a universal quantifier,
\item ${\sf q}_{(J\cup J')}p={\sf q}_{(J)}{\sf c}_{(J')}p,$
\item ${\sf c}_{(J)}{\sf q}_{(J)}p={\sf q}_{(J)}p , \ \  {\sf q}_{(J)}{\sf c}_{(J)}p={\sf c}_{(J)}p,$
\item If $\sigma\upharpoonright \alpha\sim J=\tau\upharpoonright \alpha\sim J$, then
${\sf s}_{\sigma}{\sf q}_{(J)}p={\sf s}_{\tau}{\sf q}_{(J)}p,$
\item If $\sigma\upharpoonright \sigma^{-1}(J)$ is injective, then
${\sf q}_{(J)}{\sf s}_{\sigma}p={\sf s}_{\sigma}{\sf q}_{\sigma^{-1}(J)}p.$
\end{enumarab}
\end{theorem}
\begin{demo}{Proof} Routine
\end{demo}
Here we depart from \cite{HMT2} by defining polyadic algebras on sets rather than on ordinals. In this way we follow the tradition of Halmos.
We refer to $\alpha$ as the dimension of $\A$ and we write $\alpha=dim\A$.
Borrowing terminology from cylindric algebras, we refer to ${\sf c}_{(\{i\})}$ by ${\sf c}_i$ and ${\sf q}_{(\{i\})}$ by ${\sf q}_i$
\begin{definition}
Let $\A$ be an $MV$ algebra. The $MV$ algebras of the form $\F(V, \A)$ with universe
$\{f: V\to \A$, $V\subseteq {}^{\alpha}U \text { for some set $U$ }\}$
and   $MV$ operations defined pointwise and the extra non Boolean operations are defined the
usual way, are called set algebras. When $V={}^{\alpha}U$ and $\A=[0,1]$ this is called a standard
$MV$ set algebra.
We can recover the classical case by taking $\A=2$, where $2$ is the two element Boolean $(MV)$ algebra.
\end{definition}

Our next theorem works when $G$ consists of finite transformations; except that, in this case,  we replace the notion of free algebras
by dimension restricted free
algebras \cite[theorems 2.5.35, 2.5.36, 2.5.37]{HMT1},
and when $G$ is strongly rich in the sense of \cite{AU}, and $T$ is the set of singletons. Here we prove
the most difficult case, namely, when $G$ is the set of all transformations, and $T$ the set of
all subsets of $\alpha$, so that we are dealing with full fledged polyadic $MV$ algebras.

\begin{theorem}\label{mv} Let $\alpha$ be infinite, $G={}^{\alpha}\alpha$ and $T=\wp(\alpha)$.
Then $\Fr_{\beta}MV_{G,T}$ has the interpolation property.
\end{theorem}
\begin{demo}{Proof}
\begin{enumarab}
\item The first part of the proof is identical to that in \cite{MLQ}, but we include it for the sake of completeness referring to op cit for detailed arguments.
Let $\mathfrak{m}$ be the local degree of $\A=\Fr_{\beta}MV_{G,T}$, $\mathfrak{c}$ its effective cardinality and $\mathfrak{n}$
be any cardinal such that $\mathfrak{n}\geq \mathfrak{c}$
and $\sum_{s<m}\mathfrak{n}^s=\mathfrak{n}$. Let $X_1, X_2\subseteq A$ $a\in \Sg^{\A}X_1$ and $b\in \Sg^{\A}X_2$ such that $a\leq b$.
We want to find an interpolant.
Then there exists $\B\in MV_{G_{\mathfrak{n}},\bar{T}}$  where $\bar{T}=\wp(\kappa)$ such that $\A\subseteq \Nr_{\alpha}\B$ and $A$ generates $\B$.
(It can be proved that $\B=\Fr_{\beta}MV_{G_{\mathfrak{n},\bar{T}}}).$
Being a minimal dilation of $\A$, the local degree of $\B$ is the same as that of $\A$, in particular each $x\in \B$ admits
a support of cardinality $<\mathfrak{m}$.
Also for all $X\subseteq A$, $\Sg^{\A}X=\Nr_{\alpha}\Sg^{\B}X$, this can be proved exactly like the polyadic case
\cite[theorems 3.1, 3.2 and p. 166]{DM}.

 Assume for contradiction
such an interpolant does not exist.
Then there exists no interpolant in
$\Sg^{\B}(X_1\cap X_2)$. Indeed, let $c$ be an interpolant in $\Sg^{\B}(X_1\cap X_2)$.
Let $\Gamma=(\beta\sim \alpha)\cap \Delta c$.
Let $c'={\sf c}_{(\Gamma)}c$. Then $a\leq c'$. Also $b={\sf c}_{(\Gamma)}b$, so that $c'\leq b$.
hence $a\leq c'\leq b$. But $$c'\in \Nr_{\alpha}\Sg^{\B}(X_1\cap X_2)=\Sg^{\Nr_{\alpha}\B}(X_1\cap X_2)
=\Sg^{\A}(X_1\cap X_2).$$
Let $\A_1=\Sg^{\B}X_1$ and $\A_2=\Sg^{\B}X_2$.
Let $Z_1=\{(J, p): J\subseteq \mathfrak{n}, |J|<\mathfrak{m}, p\in A_1\}$ and define $Z_2$ similarly with $A_2$ replacing $\A_1$.
Then $|Z_1|=|Z_2|\leq \mathfrak{n}$.
To show that $|Z_1|\leq \mathfrak{n}$, let $K$ be a subset of $\mathfrak{n}$ of cardinality $\mathfrak{e}$, the effective degree of $\A_1$.
Then every element $p$ of $\A_1$ is of the form ${\sf s}_{\sigma}q$ with $q\in A_{1K}$ and $\sigma\in {}^{\mathfrak{n}}\mathfrak{n}$.
The number of subsets $J$ of $\mathfrak{n}$ such that $|J|<\mathfrak{m}$ is at most $\sum_{s<\mathfrak{m}}\mathfrak{n}^s=\mathfrak{n}$.
Let $q\in A_{1K}$ have a support of cardinality $s<\mathfrak{m}$. Then the number of distinct
elements
${\sf s}_{\sigma}q$ with $\sigma\in {}^{\mathfrak{n}}\mathfrak{n}$ is at most $\mathfrak{n}^s\leq \mathfrak{n}$.
Hence there is at most $\mathfrak{n}\cdot\mathfrak{c}$ elements $s_{\sigma}q$ with $\sigma\in {}^{\mathfrak{n}}\mathfrak{n}$ and $q\in A_{1K}$.
Hence $|Z_1|\leq \mathfrak{n}\cdot\mathfrak{n}\cdot\mathfrak{c}=\mathfrak{n}.$
Let $$\langle (k_i,x_i): i\in \mathfrak{n}\rangle\text {  and  }\langle (l_i,y_i):i\in \mathfrak{n}\rangle$$
be enumerations of $Z_1$ and $Z_2$ respectively, possibly with repititions.
Now there are two functions $u$ and $v$ such that for each $i<\mathfrak{n}$ $u_i, v_i$ are elements of $^\mathfrak{n}\mathfrak{n}$
with
$$u_i\upharpoonright \mathfrak{n} \sim k_i=Id$$
$$u_i\upharpoonright k_i \text { is one to one }$$
$$x_j\text { and $y_j$ and $a$ and $c$ are independent of $u_i(k_i)$ for all $j\leq i$}$$
$${\sf s}_{u_j}x_j\text {  is independent of $u_i(k_i)$ for all $j<i$}$$
$$v_i\upharpoonright \mathfrak{n}\sim l_i=Id$$
$$v_i\upharpoonright l_i \text { is one to one }$$
$$x_j\text { and $y_j$ and $a$ and $c$ are independent of $v_i(l_i)$ for all $j\leq i$}$$
$${\sf s}_{v_j}y_j\text {  is independent of $v_i(l_i)$ for all $j<i$}$$
and
$$v_i(l_i)\cap u_j(k_j)=\emptyset\text {  for all  }j\leq i.$$
The existence of such $u$ and $v$ can be proved by by transfinite recursion \cite{MLQ}. Let
$$Y_1= \{a\}\cup \{-{\sf  c}_{(k_i)}x_i\oplus{\sf s}_{u_i}x_i: i\in \mathfrak{n}\},$$
$$Y_2=\{-b\}\cup \{-{\sf  c}_{(l_i)}y_i\oplus {\sf s}_{v_i}y_i:i\in \mathfrak{n}\},$$
$$H_1= fl^{\Rd_{MV}\Sg^{\B}(X_1)}Y_1,\  H_2=fl^{\Rd_{MV}\Sg^\B(X_2)}Y_2,$$ and
$$H=fl^{\Rd_{MV}\Sg^{\B}(X_1\cap X_2)}[(H_1\cap \Sg^{\B}(X_1\cap X_2)
\cup (H_2\cap \Sg^{\B}(X_1\cap X_2)].$$
Then $H$ is a proper filter of $\Sg^{\B}(X_1\cap X_2).$
To prove this it is sufficient to consider any pair of finite, strictly
increasing sequences of ordinals
$$\eta(0)<\eta(1)\cdots <\eta(n-1)<\mathfrak{n}\text { and } \xi(0)<\xi(1)<\cdots
<\xi(m-1)<\mathfrak{n},$$
and to prove that the following condition holds:
\begin{equation}\label{x0-9}
\begin{split}
\textrm{ For any} ~~&b_0, b_1\in \Sg^{\B}(X_1\cap X_2) ~~\textrm{such
that} \\
&a\odot \Pi_{i<n}[a^{l-1}\odot (-{\sf  c}_{(k_{\eta(i)})}x_{\eta(i)}\oplus {\sf s}_{u_{\eta(i)}}x_{\eta(i)})^{l_i}]\leq b_0 \\
\textrm{and}\\
&(-b)\odot \Pi_{i<m}[(-b)^{k-1}\odot (-{\sf c}_{(l_{\xi(i)})}y_{\xi(i))}\oplus{\sf
s}_{v_{\xi(i)}}y_{\xi(i)})^{k_i}]\leq b_1\\
\textrm{we have}\\
& b_0\odot b_1\neq 0.
\end{split}
\end{equation}
By induction on $n+m+l-1+k-1=n+m+l+k-2\geq 0$. This can be done by the above argument together with those in \cite{MLQ}.

Let $\mathfrak{m}$ be the local degree of $\A=\Fr_{\beta}PA_{\alpha}$, $\mathfrak{c}$ its effective
cardinality and $\mathfrak{n}$ be any cardinal such that $\mathfrak{n}\geq \mathfrak{c}$
and $\sum_{s<\mathfrak{m}}\mathfrak{n}^s=\mathfrak{n}$.
Then, by Theorem 7,  there exists $\B\in PA_{\mathfrak{n}}$ such that $\A\subseteq \Nr_{\alpha}\B$ and $\A$ generates $\B$.
Being a minimal dilation of $\A$, the local degree of $\B$ is the same as that of $\A$,
in particular each $x\in \B$ admits a support of cardinality $<\mathfrak{m}$.
By Theorem 9(i), for all $X\subseteq \A$, $\Sg^{\A}X=\Nr_{\alpha}\Sg^{\B}X$.
Let $X_1, X_2\subseteq \beta$, $a\in \Sg^{\A}X_1$ and $b\in \Sg^{\A}X_2$ such that $a\leq b$.
We want to find an interpolant. We assume that such an interpolant does not exist
and we eventually arrive at a contradiction.
Suppose that there exists no interpolant in $\Sg^{\A}(X_1\cap X_2)$. Then there exists no interpolant in
$\Sg^{\B}(X_1\cap X_2)$. Indeed, let $c$ be an interpolant in $\Sg^{\B}(X_1\cap X_2)$.
Let $\Gamma=\mathfrak{n}\sim\alpha$.
Let $c'={\sf c}_{(\Gamma)}c$. Then $a\leq c'$. Also $b={\sf c}_{(\Gamma)}b$, so that $c'\leq b$.
Hence $a\leq c'\leq b$. But $c'\in \Nr_{\alpha}\Sg^{\B}(X_1\cap X_2)=\Sg^{\Nr_{\alpha}\B}(X_1\cap X_2)
=\Sg^{\A}(X_1\cap X_2).$ That is, assuming the existence of an interpolant in $\Sg^{\B}(X_1\cap X_2)$
implies the existence of one in $\Sg^{\A}(X_1\cap X_2)$.

Now we prepare for constructing the Boolean ultrafilters $F_1$ and $F_2$ as indicated in the outline of proof.
Let $\A_1=\Sg^{\B}X_1$ and $\A_2=\Sg^{\B}X_2$. $F_1$ will be a Boolean ultrafilter of $\A_1$ and $F_2$ will be a Boolean ultrafilter of $\A_2$
such that $F_1$ and $F_2$ agree on the common part $\Sg^{\B}(X_1\cap X_2)$, see (\ref{tarek4}) below.

Let $Z_1=\{(J, p): J\subseteq \mathfrak{n}, |J|<\mathfrak{m}, p\in \A_1\}$ and define $Z_2$ similarly with $\A_2$ replacing $\A_1$.
Then $|Z_1|,|Z_2|\leq \mathfrak{n}$.
To show that $|Z_1|\leq \mathfrak{n}$, let $K$ be a subset of $\mathfrak{n}$ of cardinality $\mathfrak{e}$, the effective degree of $\A_1$.
Then every element $p$ of $\A_1$ is of the form ${\sf s}_{\sigma}q$ with $q\in \A_{1K}$ and $\sigma\in {}^{\mathfrak{n}}\mathfrak{n}$.
$\A_{1K}$ is defined as in definition 5 (i).
The number of subsets $J$ of $\mathfrak{n}$ such that $|J|<\mathfrak{m}$ is at most $\sum_{s<\mathfrak{m}}\mathfrak{n}^s=\mathfrak{n}$.
Let $q\in \A_{1K}$ have a support of cardinality $s<\mathfrak{m}$. Then the number of distinct
elements
${\sf s}_{\sigma}q$ with $\sigma\in {}^{\mathfrak{n}}\mathfrak{n}$ is at most $\mathfrak{n}^s\leq \mathfrak{n}$.
Hence there are at most $\mathfrak{n}\cdot\mathfrak{c}$ elements $s_{\sigma}q$ with $\sigma\in {}^{\mathfrak{n}}\mathfrak{n}$ and $q\in \A_{1K}$.
Here we are using that $\A_1$ is a minimal dilation of $\A_{1K}$.
Hence $|Z_1|\leq \mathfrak{n}\cdot\mathfrak{n}\cdot\mathfrak{c}=\mathfrak{n}.$
The proof that $|Z_2|\leq \mathfrak{n}$ is completely analogous.
Let $$\langle (k_i,x_i): i\in \mathfrak{n}\rangle\text {  and  }\langle (l_i,y_i):i\in \mathfrak{n}\rangle$$
be enumerations of $Z_1$ and $Z_2$ respectively, possibly with repetitions.
Now there are two functions $u$ and $v$ such that for each $i<\mathfrak{n},$ $u_i, v_i$ are elements of $^\mathfrak{n}\mathfrak{n}$
with
\begin{itemize}
\item $u_i\upharpoonright \mathfrak{n} \sim k_i=Id$
\item $u_i\upharpoonright k_i \text { is one to one }$
\item $x_j\text { and $y_j$ and $a$ and $b$ are independent of $u_i(k_i)$ for all $j\leq i$}$
\item ${\sf s}_{v_j}y_j\text {  is independent of $u_i(k_i)$ for all $j<i$}$
\item $v_i\upharpoonright \mathfrak{n}\sim l_i=Id$
\item $v_i\upharpoonright l_i \text { is one to one }$
\item $x_j\text { and $y_j$ and $a$ and $b$ are independent of $v_i(l_i)$ for all $j\leq i$}$
\item ${\sf s}_{u_j}x_j\text {  is independent of $v_i(l_i)$ for all $j\leq i$}$

and finally

\item $v_i(l_i)\cap u_j(k_j)=\emptyset\text {  for all  }j\leq i.$

\end{itemize}

We prove the existence of such $u$ and $v$ by transfinite recursion.
This depends on the following set theoretic fact.

\begin{athm}{Claim 1} If $\mathfrak{n}$ and $\mathfrak{m}$ are cardinals such that $\sum_{s<\mathfrak{m}}\mathfrak{n}^s=\mathfrak{n}$,
if $\mu<\mathfrak{n}$ and if for each $\eta<\mu,$ $a_{\eta}$ is a cardinal less than $\mathfrak{m}$, then
$$\sum_{\eta<\mu}a_{\eta}<\mathfrak{n}.$$
\end{athm}

We first prove the existence of $u_i$ and $v_i$ by recursion. Then we prove Claim 1. Let $\mu<\mathfrak{n}$ and assume inductively
that $h,k:\mu\to {}^{\mathfrak{n}}\mathfrak{n}$ are functions satisfying the induction hypothesis.
For each $j\leq \mu$, let $K_j$ be a support of $x_j$ and $M_j$ be a support of $y_j$ each of cardinality $<\mathfrak{m}$.
$a$ and $b$ admit such supports as well, call them $sup(a)$ and $sup(b).$
Similarly, for all $j<\mu$,
let $L_j$ be a support of ${\sf s}_{u_j}x_j$ of cardinality $<\mathfrak{m}$ and $J_j$ be a support of ${\sf s}_{v_j}y_j$ of cardinality $<\mathfrak{m}.$
Then by Claim 1, $\sum_{\eta<\mu+1} |K_{\eta}|$, $\sum_{\eta<\mu+1} |M_{\eta}|$, $\sum_{\eta<\mu}|L_{\eta}|$,
$\sum_{\eta<\mu}|J_{\eta}|<\mathfrak{n}$.
Hence $\mathfrak{n}\sim sup(a)\cup \sup(b)\cup \bigcup_{\eta<\mu+1} K_{\eta}
\cup \bigcup_{\eta<\mu} L_{\eta}\cup \bigcup_{\eta<\mu+1}M_{\eta}\cup \bigcup_{\eta<\mu}J_{\eta}$
has cardinality equal to
$\mathfrak{n}$. It follows that  $h,k$ can be extended to maps from $\mu+1\to {}^\mathfrak{n}\mathfrak{n}$ satisfying the required.
The rest follows from Zorn's Lemma.

\begin{demo}{Proof of Claim 1}
We have $\mathfrak{m}\leq \mathfrak{n}$, for otherwise we would have,
setting $s=\mathfrak{n}$, $\mathfrak{n}^{\mathfrak{n}}\leq \mathfrak{n}$ which is absurd.
If $\mathfrak{m}<\mathfrak{n}$ then
$$\sum_{\eta<\mu}a_{\eta}\leq \mathfrak{m}\cdot |\mu|<\mathfrak{n}.$$
Now suppose $\mathfrak{m}=\mathfrak{n}$ and $\sum_{\eta<\mu} a_{\eta}=\mathfrak{n}$. Then
$$\mathfrak{n}<2^{\mathfrak{n}}=\mathfrak{n}^{\mathfrak{n}}=\mathfrak{n}^{\sum_{\eta<\mu}a_{\eta}}=
\Pi_{\eta<\mu}\mathfrak{n}^{a_{\eta}}\leq \mathfrak{n}^{|\mu|}\leq \mathfrak{n}$$
which is impossible. We have proved Claim 1.
\end{demo}

Nw we actually construct our desired ultrafilters. For a polyadic algebra $\F$ we write $Bl\F$ to denote its Boolean reduct.
For a Boolean algebra $\B$  and $Y\subseteq \B$, we write
$fl^{\B}Y$ to denote the Boolean filter generated by $Y$ in $\B.$  Now let
$$Y_1= \{a\}\cup \{-{\sf  c}_{(k_i)}x_i+{\sf s}_{u_i}x_i: i\in \mathfrak{n}\},$$
$$Y_2=\{-b\}\cup \{-{\sf  c}_{(l_i)}y_i+{\sf s}_{v_i}y_i:i\in \mathfrak{n}\},$$
$$H_1= fl^{Bl\Sg^{\B}(X_1)}Y_1,\  H_2=fl^{Bl\Sg^\B(X_2)}Y_2,$$ and
$$H=fl^{Bl\Sg^{\B}(X_1\cap X_2)}[(H_1\cap \Sg^{\B}(X_1\cap X_2)
\cup (H_2\cap \Sg^{\B}(X_1\cap X_2)].$$

Then
\begin{athm}{Claim 2}
$H$ is a proper filter of $\Sg^{\B}(X_1\cap X_2).$
\end{athm}

\begin{demo}{Proof of Claim 2} The proof of Claim 2 is rather long, for indeed it is the heart of the proof. It is very similar to the corresponding part
in \cite{IGPL}.
To prove Claim 2, it is sufficient to consider any pair of finite, strictly
increasing sequences of ordinals
$$\eta(0)<\eta(1)\cdots <\eta(n-1)<\mathfrak{n}\text { and } \xi(0)<\xi(1)<\cdots
<\xi(m-1)<\mathfrak{n},$$
and to prove that the following condition holds:
\begin{equation}\label{x0-9}
\begin{split}
\textrm{ For any} ~~&b_1, b_2\in \Sg^{\B}(X_1\cap X_2) ~~\textrm{such
that} \\
&a\odot \Pi_{i<n}(a^{l-1}\odot-{\sf  c}_{(k_{\eta(i)})}x_{\eta(i)}(\oplus){\sf s}_{u_{\eta(i)}}x_{\eta(i)})\leq b_1 \\
\textrm{and}\\
&(-b)\odot\Pi_{i<m} ((-b)^{k-1}\odot-{\sf c}_{(l_{\xi(i)})}y_{\xi(i)}\oplus {\sf
s}_{v_{\xi(i)}}y_{\xi(i)})\leq b_2\\
\textrm{we have}\\
& b_1\odot b_2\neq 0.
\end{split}
\end{equation}
We prove this by induction on $n+m$. If $n+m=0$, then (\ref{x0-9})
simply expresses the fact that no interpolant of $a$ and $b$ exists
in $\Sg^{\B}(X_1\cap X_2).$ In more detail: if $n+m=0$, then $a\leq
b_1$ and $-b\leq b_2$. So if $b_1\cdot b_2=0$, we get $a\leq b_1\leq
-b_2\leq b.$ Now assume that $n+m>0$ and for the time being suppose
that $\eta(n-1)>\xi(m-1)$. Let $i<\mathfrak{n}$. Then $u_i\in {}^{\mathfrak{n}}\mathfrak{n}$, and $u_i(k_i)\subseteq \mathfrak{n}$.
Let $k_{\eta(n-1)}'=u_{\eta(n-1)}(k_{\eta(n-1)}).$
Apply ${\sf  c}_{(k_{\eta(n-1)}')}$ to both
sides of the first inclusion of (\ref{x0-9}). By
$a$ is independent of $k_{\eta(n-1)}'$ we have ${\sf  c}_{(k_{\eta(n-1)}')}a=a$,
and by noting that ${\sf c}_{(\Gamma)}({\sf  c}_{(\Gamma)}x\odot y)={\sf  c}_{(\Gamma)}x\odot {\sf c}_{(\Gamma)}y$, we get
\begin{equation}\label{x0-10}
\begin{split}
a\odot {\sf  c}_{(k_{\eta(n-1)}')}\Pi_{i<n}[a^{l-1}\odot -{\sf
c}_{(k_{\eta(i)})}x_{\eta(i)}\oplus{\sf
s}_{u_{\eta(i)}}^{l_i}x_{\eta(i)})]\leq {\sf
c}_{(k_{\eta(n-1)}')}b_1.
\end{split}
\end{equation}
Let ${\sf  c}_{(\Gamma)}^{\partial}(x)=-{\sf
c}_{(\Gamma)}(-x)$.  ${\sf  c}_{(\Gamma)}^{\partial}$ is the algebraic counterpart of
the universal quantifier. Now apply
${\sf c}_{(k_{\eta(n-1)}')}^{\partial}$ to the second inclusion of
(\ref{x0-9}). By noting that ${\sf  c}_{(\Gamma)}^{\partial}$, the dual of
${\sf  c}_{(\Gamma)}$, distributes over the Boolean meet and by
$b$ independent of $k_{\eta(n-1)}'$ we get
\begin{equation}\label{x0-11}
\begin{split}
(-b)\cdot \Pi_{j<m}((-b)^{k-1}\odot {\sf  c}_{(k_{\eta(n-1)}')}^{\partial}(-{\sf
c}_{(l_{\xi(j)})}y_{\xi(j)}\oplus {\sf s}_{v_{\xi(j)}}
y_{\xi(j)})^{k_i})\leq {\sf  c}_{(k_{\eta(n-1)}')}^{\partial}b_2.
\end{split}
\end{equation}
Before proceeding, we formulate (and prove) a subclaim that will enable us to
eliminate the generalized quantifier ${\sf  c}_{(k_{\eta(n-1)}')}$ (and its dual)
from (\ref{x0-9}) (and (\ref{x0-11})) above.

For the sake of brevity set for each $i<n$ and each $j<m:$
$$z_i=-{\sf  c}_{(k_{\eta(i)})}x_{\eta(i)}\oplus{\sf s}_{u_{\eta(i)}}x_{\eta(i)}$$  and
$$t_j=-{\sf  c}_{(l_{\xi(j)})}y_{\xi(j)}\oplus{\sf s}_{v_{\xi(j)}}y_{\xi(j)}.$$
Then (\ref{tarek10}) and (\ref{tarek11}) below hold:

\begin{equation}\label{tarek10}
\begin{split}{\sf  c}_{(k_{\eta(n-1)}')}z_i=z_i\text { for }i<n-1 \text { and }{\sf  c}_{(k_{\eta(n-1)}')}z_{n-1}=1
\end{split}
\end{equation}
\begin{equation}\label{tarek11}
\begin{split}{\sf  c}_{(k_{\eta(n-1)}')}^{\partial}t_j=t_j\text { for all }j<m.
\end{split}
\end{equation}

{\it Proof of ${\sf  c}_{(k_{\eta_{n-1}}')}z_i=z_i$ for  $i<n-1$.}

Both ${\sf  c}_{(k_{\eta(i)})}x_{\eta(i)}$ and ${\sf s}_{u_{\eta(i)}}^{k_{\eta(i)}}x_{\eta(i)}$ are independent of $k_{\eta(n-1)}'$, it thus follows that
$${\sf  c}_{(k_{\eta(n-1)}')}(-{\sf  c}_{(k_{\eta(i)})}x_{\eta(i)})=-{\sf  c}_{(k_{\eta(i)})}x_{\eta(i)}$$
and
$${\sf  c}_{(k_{\eta(n-1)}')} ({\sf
s}_{u_{\eta(i)}}x_{\eta(i)})={\sf s}_{u_{\eta(i)}}x_{\eta(i)}.$$

Finally, by ${\sf c}_{(k_{\eta(n-1)}')}$ distributing over the Boolean join, we get
$${\sf  c}_{(k_{\eta(n-1)}')} z_i=z_i \text { for }  i<n-1.$$

{\it Proof of ${\sf  c}_{(k_{\eta(n-1)}')}z_{n-1}=1.$}

Computing we get, by $x_{\xi(n-1)}$ independent of $k_{\eta(n-1)}'=u_{\eta(n-1)}(k_{\eta(n-1)}),$ the following:
\begin{equation*}
\begin{split}
&d={\sf  c}_{(k_{\eta(n-1)}')}(-{\sf  c}_{(k_{\eta(n-1)})}x_{\eta(n-1)}\oplus
{\sf s}_{u_{\eta(n-1)}}x_{\eta(n-1)})\\
&={\sf  c}_{(k_{\eta(n-1)}')}-{\sf  c}_{(k_{\eta(n-1)})}x_{\eta(n-1)}\oplus
{\sf  c}_{(k_{\eta(n-1)}')} {\sf
s}_{u_{\eta(n-1)}}x_{\eta(n-1)}\\
&=-{\sf  c}_{(k_{\eta(n-1)})}x_{\eta(n-1)}\oplus {\sf
c}_{(k_{\eta(n-1)}')}{\sf s}_{u_{\eta(n-1)}}
x_{\eta(n-1)}\\
\end{split}
\end{equation*}
To carry on with this computation, we let for the sake of brevity and better readability
$X=k_{\eta(n-1)}$, $u=u_{\eta(n-1)}$, $Y=u(X)=k_{\eta(n-1)}'$, and $x=x_{\eta(n-1)}.$
Choose $t\in {}^\mathfrak{n}\mathfrak{n}$ with $t\upharpoonright (\mathfrak{n}\sim Y)\cup X=u\upharpoonright (\mathfrak{n}\sim Y)\cup X$ and
$Y\cap t(Y\sim X)=\emptyset$.
Then $t^{-1}(Y)=X$ and $t\upharpoonright X$ is one to one. Now
\begin{equation*}
\begin{split}
&{\sf c}_{(Y)}{\sf s}_ux={\sf c}_{(Y)}{\sf s}_u{\sf c}_{(Y)}x\\
&={\sf c}_{(Y)}{\sf s}_t{\sf c}_{(Y)}x\\
&={\sf s}_t{\sf c}_{(X)}{\sf c}_{(Y)}x\\
&={\sf s}_{Id}{\sf c}_{(X\cup Y)}x\\
&={\sf c}_{(X)}x.\\
\end{split}
\end{equation*}

It follows that
$$d= -{\sf  c}_{k_{\eta(n-1)}}x_{\eta(n-1)}\oplus {\sf
c}_{k_{\eta(n-1)}}x_{\eta(n-1)}=1.$$

With this the proof of (\ref{tarek10}) in our subclaim is complete. Now we prove
(\ref{tarek11}). Let $j<m$ . Then we have
$${\sf  c}_{(k_{\eta(n-1)}')}^{\partial}(-{\sf  c}_{(l_{\xi(j)})}y_{\xi(j)})
=-{\sf  c}_{(l_{\xi(j)})}y_{\xi(j)}$$ and
$${\sf  c}_{(k_{\eta(n-1)}')}^{\partial}
({\sf s}_{v_{\xi(j)}}y_{\xi(j)})={\sf
s}_{v_{\xi(j)}}y_{\xi(j)}.$$ Indeed,  computing we get
\begin{equation*}
\begin{split}
{\sf  c}_{(k_{\eta(n-1)}')}^{\partial}(-{\sf c}_{(l_{\xi(j)})}y_{\xi(j)})
&=-{\sf  c}_{(k_{\eta_{n-1}}')}-(-{\sf c}_{(l_{\xi(j)})}y_{\xi(j)})\\
& = -{\sf  c}_{k_{\eta(n-1)}')}{\sf c}_{(l_{\xi(j)})}y_{\xi(j)} \\
&=-{\sf c}_{(l_{\xi(j)})}y_{\xi(j)}.
\end{split}
\end{equation*}

Similarly,  we have
 \begin{equation*}
\begin{split}
{\sf  c}_{(k_{\eta(n-1)}')}^{\partial} ({\sf
s}_{v_{\xi(j)}}y_{\xi(j)})& =-{\sf  c}_{(k_{\eta(n-1)}')}-
({\sf s}_{v_{\xi(j)}}y_{\xi(j)})\\
&=-{\sf  c}_{(k_{\eta(n-1)}')} ({\sf
s}_{v_{\xi(j)}}-y_{\xi(j)}) \\
&=- {\sf s}_{v_{\xi(j)}}-y_{\xi(j)} \\
&= {\sf s}_{v_{\xi(j)}}y_{\xi(j)}.
\end{split}
\end{equation*}
By ${\sf c}_{(J)}^{\partial}({\sf  c}_{(J)}^{\partial}x\oplus y)= {\sf
c}_{(J)}^{\partial}x\oplus{\sf  c}_{(J)}^{\partial}y$ we get from the above that
 \begin{equation*}
\begin{split}
{\sf  c}_{(k_{\eta(n-1)}')}^{\partial}(t_j) &= {\sf
c}_{(k_{\eta(n-1)}')}^{\partial}(-{\sf  c}_{(l_{\xi(j)})}y_{\xi(j)}\oplus{\sf
s}_{v_{\xi(j)}}y_{\xi(j)})\\
&={\sf  c}_{(k_{\eta(n-1)}')}^{\partial}-{\sf
c}_{(l_{\xi(j)})}y_{\xi(j)}\oplus{\sf  c}_{(k_{\eta(n-1)}')}^{\partial} {\sf
s}_{v_{\xi(j)}}y_{\xi(j)}\\
&=-{\sf  c}_{(l_{\xi(j)})}y_{\xi(j)}\oplus {\sf
s}_{v_{\xi(j)}}y_{\xi(j)}\\
&=t_j.
\end{split}
\end{equation*}
We have proved (\ref{tarek11}).
By the above proven subclaim, i.e. by (\ref{tarek10}), (\ref{tarek11}), we have
\begin{equation*}
\begin{split}
{\sf  c}_{(k_{\eta(n-1)}')}a^{l-1}\odot\Pi_{i<n}z_i
&=a^{l-1}\odot {\sf c}_{(k_{\eta(n-1)}')}[\Pi_{i<n-1}z_i^{l_i}\odot z_{n-1}^{l_{n-1}}]\\
&=a^{l-1}\odot {\sf  c}_{(k_{\eta(n-1)}')}\Pi_{i<n-1}z_i^{l_i}\odot
{\sf c}_{(k_{\eta(n-1)}')}z_{n-1}^{n-1}\\
&=a^{l-1}\odot\Pi_{i<n-1}z_i.
\end{split}
\end{equation*}

Combined with
(\ref{x0-10}) we obtain
$$a\odot  \Pi_{i<n-1}(a^{l-1}\odot-{\sf  c}_{(k_{\eta(i)})}x_{\eta(i)}\oplus{\sf s}_{u_{\eta(i)}}x_{\eta(i)})
 \leq {\sf  c}_{(k_{\eta(n-1)}')}b_1.$$
On the other hand, from (\ref{tarek10}), (\ref{tarek11}) and (\ref{x0-11}),
it follows that
$$(-b)\cdot \Pi_{j<m}
(-b)^{k-1}\odot (-{\sf  c}_{(l_{\xi(j)})}y_{\xi(j)}\oplus{\sf
s}_{v_{\xi(j)}}y_{\xi(j)}^{k_i})\leq {\sf c}_{(k_{\eta(n-1)}')}^{\partial}b_2.$$ Now making use of the induction
hypothesis we get
$${\sf  c}_{(k_{\eta(n-1)}')}b_1\odot {\sf  c}_{(k_{\eta(n-1)}')}^{\partial}b_2\neq 0;$$
and hence that
$$b_1\odot {\sf  c}_{(k_{\eta(n-1)}')}^{\partial}b_2\neq 0.$$
From
$$b_1\odot {\sf  c}_{(k_{\eta(n-1)}')}^{\partial}b_2\leq b_1\cdot b_2$$
we reach the desired conclusion, i.e. that
$$b_1\odot b_2\neq 0.$$
The other case, when $\eta(n-1)\leq \xi(m-1)$ can be treated
in a similar manner and is therefore left to the reader. We have proved that
$H$ is a proper filter, that is the proof of Claim 2 is complete.
\end{demo}

We proceed exactly as above. Proving that $H$ is a proper filter of $\Sg^{\B}(X_1\cap X_2)$,
let $H^*$ be a (proper $MV$) maximal filter of $\Sg^{\B}(X_1\cap X_2)$
containing $H.$
We obtain  maximal filters $F_1$ and $F_2$ of $\Sg^{\B}(X_1)$ and $\Sg^{\B}(X_2)$,
respectively, such that
$$H^*\subseteq F_1,\ \  H^*\subseteq F_2$$
and (*)
$$F_1\cap \Sg^{\B}(X_1\cap X_2)= H^*= F_2\cap \Sg^{\B}(X_1\cap X_2).$$
Now for all $x\in \Sg^{\B}(X_1\cap X_2)$ we have
$$x\in F_1\text { if and only if } x\in F_2.$$
Let $i\in \{1,2\}$. Then $F_i$ by construction satisfies the following:
for each $p\in \B$ and each subset $J\subseteq \mathfrak{n}$ with $|J|<\mathfrak{m}$,  there exists $\rho\in {}^{\mathfrak{n}}\mathfrak{n}$ such that
$$\rho\upharpoonright \mathfrak{n}\sim J=Id_{\mathfrak{n}-J}$$
and $$-{\sf c}_{(J)}p \oplus {\sf s}_{\rho}p\in F_i.$$
Since $${\sf s}_{\rho}p\leq {\sf c}_{(J)}p,$$
we have (**)
$${\sf c}_{(J)}p\in F_i\Longleftrightarrow{\sf s}_{\rho}p\in F_i.$$
Let $\D_i=\Sg^{\A_i}X_i$, $i=1, 2$.
Let $$\psi_i:\D_i\to \F(^{\alpha}\mathfrak{n}, \D_i/F_i)$$ be defined as follows:
$$\psi_i(a)(x)=s_{\bar{x}}^{\B}a/F$$
Note that  $\bar{\tau}=\tau\cup Id_{\beta\sim \alpha}$ is in $^{\mathfrak{n}}\mathfrak{n}$, so that substitutions are evaluated in the big algebra $\B$.
Then, we claim that  $\psi_i$ is a homomorphism. Then using freeness we paste the two maps $\psi_1$ $\psi_2$, obtaining that $\psi=\psi_1\cup \psi_2$
is homomorphism from the free algebra to
$\F(^{\alpha}\mathfrak{n}, [0,1])$ such that $\psi(a-b)\neq 0$ which is a contradiction. As usual,
we check cylindrifiers, and abusing notation for a while, we omit superscripts, in particular, we write $\psi$ instead of $\psi_1$.
Let $x\in {}^{\alpha}\mathfrak{n}$,
$M\subseteq \alpha,$ $p\in D$.  Then
$$\psi({\sf c}_{(M)}p)(x)= {\sf s}_x{\sf c}_{(M)}p/F.$$
Let $K$ be a support of $p$ such that  $|K|<\mathfrak{m}$
and let $J=M\cap K$. Then
$${\sf c}_{(J)}p={\sf c}_{(M)}p.$$
Let $$\sigma\in {}^\mathfrak{n}\mathfrak{n}, \sigma\upharpoonright \mathfrak{n}\sim J=\bar{\tau}\upharpoonright \mathfrak{n}\sim J,$$
$$\sigma J\cap \tau(K\sim J)=\emptyset$$
and
$$\sigma\upharpoonright J\text { is one to one }.$$
Then $${\sf s}_{x}{\sf c}_{(J)}p={\sf c}_{(\sigma J)}{\sf s}_{\sigma}p.$$
By (**), let $\rho$ be such that
$$\rho\upharpoonright \sigma J=Id_{\mathfrak{n}\sim \sigma J}$$
and
$${\sf c}_{(\sigma J)}{\sf s}_{\sigma}p\in F\Longleftrightarrow {\sf s}_{\rho}{\sf s}_{\sigma}p\in F
\Longleftrightarrow {\sf s}_{\rho\circ \sigma}p\in F.$$
It follows that
$${\sf c}_{(\sigma J)}{\sf s}_{\sigma}p/F={\sf s}_{\rho}{\sf s}_{\sigma}p/F={\sf s}_{\rho\circ \sigma}/F.$$
Let $y\in {}^{\mathfrak{n}}\mathfrak{n}$ such that
$$y\upharpoonright M=\rho\circ \sigma \upharpoonright M$$
and (***)
$$y\upharpoonright \mathfrak{n}\sim M=\bar{\tau}\upharpoonright \mathfrak{n}\sim M.$$
If $k\in K\sim M$, $\rho\circ \sigma(k)=\rho\circ \tau(k)$ and since $\tau k\notin \sigma J$ we have
$$\rho \tau(k)=\tau(k)=y(k).$$ So
$$y\upharpoonright K=\rho \circ \sigma \upharpoonright K.$$
Now we have using (**) and (***):
\begin{equation*}
\begin{split}
&\psi({\sf c}_{M}p)x\\
&={\sf s}_x{\sf c}_{(M)}p/F\\
&={\sf s}_x{\sf c}_{(J)}p/F\\
&={\sf c}_{(\sigma J)}{\sf s}_{\sigma}p/F\\
&={\sf s}_{\rho}{\sf s}_{\sigma}p/F\\
&={\sf s}_{\rho\circ \sigma}p/F\\
&={\sf s}_yp/F\\
&\leq {\sf c}_{M}(\psi(p)x\\
\end{split}
\end{equation*}
\end{enumarab}
The other inclusion is left to the reader.
The proof is complete.
\end{demo}

When the algebras considered are Boolean algebras, then we get  the results in
\cite{K, DM,D, AUamal, MLQ,  J}.

\subsection{Reducts of Heyting polyadic algebras}

Now we address an interpolation property for an intuitionistic quantifier logic,
the existential quantifiers (and their duals) are as usual, its signature is countable,
but it allows infinitary relation symbols, and infinitary substitutions as connectives,
but only countable.

We follow closely \cite{s}. For Heyting polyadic algebras, see \cite[definition 3]{s},  witness
\cite[definition 4]{s} for the algebraisation of a Kripke frame $\K$ and the concrete set algebra $\F_{\K}$, based
on $\K$. (In this context these are the representable algebras).
For dilations (defined like classical polyadic algebras) and their properties see
\cite[definition15]{s}, \cite[definition 11, theorems 12-13]{s}.

The proof of the next theorem is very close to that of \cite[theorem 21]{s};
the notation and concepts adopted are the same. But here we have diagonal elements, so
our obtained result will be weaker. It is a both a completeness and a Robinson's joint consistency theorem,
but not in the full strength as in the case
of first order logic. (The strong form is known {\it not to hold} in the intuitionistic context even without equality).

A novelty here, is that our  representation theorem has to respect diagonal elements,
and this seems to be an impossible task in the presence of infinitary substitutions,
unless we make a compromise, that is, from our point of view, acceptable.

The interaction of substitutions based on infinitary transformations,
together with the existence of diagonal elements tends to make matters `blow up'; indeed this even happens in the classical case,
when the class of (ordinary) set algebras ceases to be closed under ultraproducts \cite{Sain}.
The natural thing to do is to avoid those infinitary substitutions at the start, while finding the interpolant possibly using such substitutions.
It can also be shown that in some cases the interpolant has to use infinitary substitutions,
even if the original implication uses only finite transformations. This idea is used in \cite{typless}.

In our next theorem $G$ is a strongly rich semigroup, as defined in \cite{AUamal},
and $GPHAE_{\alpha}$ denotes the class of Heyting
polyadic equality algebras of dimension$\alpha$, a countable ordinal,
substitution operators are restricted to substitutions in $G$,
and we have the usual axioms for diagonal elements, interacting with cylindrifiers,
their duals and substitutions the standard way.

For an algebra $\A$, we let $\Rd\A$ denote its reduct when we discard infinitary substitutions. In particular, $\Rd\A$ satisfies
cylindric algebra axioms. Dilations for such algebras can be defined exactly like the classical case \cite[remark 2.8 p. 327]{AU}
and  \cite[p. 323-336]{AU}.

\begin{theorem}\label{heyting}
Let $\alpha$ be an infinite set. Let $G$ be a semigroup on $\alpha$ containing at least one infinitary transformation.
Let $\A\in GPHAE_{\alpha}$ be the free $G$ algebra generated by $X$, and suppose that $X=X_1\cup X_2$.
Let $(\Delta_0, \Gamma_0)$, $(\Theta_0, \Gamma_0^*)$ be two consistent theories in $\Sg^{\Rd\A}X_1$ and $\Sg^{\Rd\A}X_2,$ respectively.
Assume that $\Gamma_0\subseteq \Sg^{\A}(X_1\cap X_2)$ and $\Gamma_0\subseteq \Gamma_0^*$.
Assume, further, that
$(\Delta_0\cap \Theta_0\cap \Sg^{\A}X_1\cap \Sg^{\A}X_2, \Gamma_0)$ is complete in $\Sg^{\Rd\A}X_1\cap \Sg^{\Rd\A}X_2$.
Then there exist $\mathfrak{K}=(K,\leq \{X_k\}_{k\in K}\{V_k\}_{k\in K}),$ a homomorphism $\psi:\A\to \mathfrak{F}_K,$
$k_0\in K$, and $x\in V_{k_0}$,  such that for all $p\in \Delta_0\cup \Theta_0$ if $\psi(p)=(f_k)$, then $f_{k_0}(x)=1$
and for all $p\in \Gamma_0^*$ if $\psi(p)=(f_k)$, then $f_{k_0}(x)=0$.
\end{theorem}
\begin{demo}{Proof}
The first half of the proof is close to that of \cite[theorem 2.1]{s}. We highlight the main steps,
for the convenience of the reader, except that we only deal with the case
when $G$ is strongly rich.
Assume, as usual, that $\alpha$, $G$, $\A$ and $X_1$, $X_2$, and everything else in the hypothesis are given.
Let $I$ be  a set such that  $\beta=|I\sim \alpha|=max(|A|, |\alpha|).$
Let $(K_n:n\in \omega)$ be a family of pairwise disjoint sets such that $|K_n|=\beta.$
Define a sequence of algebras
$\A=\A_0\subseteq \A_1\subseteq \A_2\subseteq \A_2\ldots \subseteq \A_n\ldots$
such that
$\A_{n+1}$ is a minimal dilation of $\A_n$ and $dim(\A_{n+1})=\dim\A_n\cup K_n$.We denote $dim(\A_n)$ by $I_n$ for $n\geq 1$.

Now we prove the theorem when $G$ is a strongly rich semigroup.
Let $$K=\{((\Delta, \Gamma), (T,F)): \exists n\in \omega \text { such that } (\Delta, \Gamma), (T,F)$$
$$\text { is a a matched pair of saturated theories in }
\Sg^{\Rd\A_n}X_1, \Sg^{\Rd\A_n}X_2\}.$$
We have $((\Delta_0, \Gamma_0)$, $(\Theta_0, \Gamma_0^*))$ is a matched pair but the theories \cite[definition 19]{s},
are not saturated \cite[item (4), definition 15]{s}, hence \cite[lemma 18]{s}
there are $T_1=(\Delta_{\omega}, \Gamma_{\omega})$,
$T_2=(\Theta_{\omega}, \Gamma_{\omega}^*)$ extending
$(\Delta_0, \Gamma_0)$, $(\Theta_0, \Gamma_0^*)$, such that $T_1$ and $T_2$ are saturated in $\Sg^{\Rd\A_1}X_1$ and $\Sg^{\Rd\A_1}X_2,$
respectively. Let $k_0=((\Delta_{\omega}, \Gamma_{\omega}), (\Theta_{\omega}, \Gamma_{\omega}^*)).$ Then $k_0\in K,$
and   $k_0$ will be the desired world and $x$ will be specified later; in fact $x$ will be the identity map on some specified
domain.

If $i=((\Delta, \Gamma), (T,F))$ is a matched pair of saturated theories in $\Sg^{\Rd\A_n}X_1$ and $\Sg^{\Rd\A_n}X_2$, let $M_i=dim \A_n$,
where $n$ is the least such number, so $n$ is unique to $i$.
Let $${\bf K}=(K, \leq, \{M_i\}, \{V_i\})_{i\in \mathfrak{K}},$$
where $V_i=\bigcup_{p\in G_n, p\text { a finitary transformation }}{}^{\alpha}M_i^{(p)}$
(here we are considering only substitutions that move only finitely many points),
and $G_n$
is the strongly rich semigroup determining the similarity type of $\A_n$, with $n$
the least number such $i$ is a saturated matched pair in $\A_n$, and $\leq $ is defined as follows:
If $i_1=((\Delta_1, \Gamma_1)), (T_1, F_1))$ and $i_2=((\Delta_2, \Gamma_2), (T_2,F_2))$ are in $\bold K$, then set
$$i_1\leq i_2\Longleftrightarrow  M_{i_1}\subseteq M_{i_2}, \Delta_1\subseteq \Delta_2, T_1\subseteq T_2.$$
We are not yet there, to preserve diagonal elements we have to factor out $\bold K$
by an infinite family equivalence relations, each defined on the dimension of $\A_n$, for some $n$, which will actually turn out to be
a congruence in an exact sense.
As usual, using freeness of $\A$, we will  define two maps on $\A_1=\Sg^{\Rd\A}X_1$ and $\A_2=\Sg^{\Rd\A}X_2$, respectively;
then those will be pasted
to give the required single homomorphism.

Let $i=((\Delta, \Gamma), (T,F))$ be a matched pair of saturated theories in $\Sg^{\Rd\A_n}X_1$ and $\Sg^{\Rd\A_n}X_2$, let $M_i=dim \A_n$,
where $n$ is the least such number, so $n$ is unique to $i$.
For $k,l\in dim\A_n=I_n$, set $k\sim_i l$ iff ${\sf d}_{kl}^{\A_n}\in \Delta\cup T$. This is well defined since $\Delta\cup T\subseteq \A_n$.
We omit the superscript $\A_n$.
These are infinitely many relations, one for each $i$, defined on $I_n$, with $n$ depending uniquely on $i$,
we denote them uniformly by $\sim$ to
avoid complicated unnecessary notation.
We hope that no confusion is likely to ensue. We claim that $\sim$ is an equivalence relation on $I_n.$
Indeed,  $\sim$ is reflexive because ${\sf d}_{ii}=1$ and symmetric
because ${\sf d}_{ij}={\sf d}_{ji};$
finally $E$ is transitive because for  $k,l,u<\alpha$, with $l\notin \{k,u\}$,
we have
$${\sf d}_{kl}\cdot {\sf d}_{lu}\leq {\sf c}_l({\sf d}_{kl}\cdot {\sf d}_{lu})={\sf d}_{ku},$$
and we can assume that $T\cup \Delta$ is closed upwards.
For $\sigma,\tau \in V_k,$ define $\sigma\sim \tau$ iff $\sigma(i)\sim \tau(i)$ for all $i\in \alpha$.
Then clearly $\sigma$ is an equivalence relation on $V_k$.

Let $W_k=V_k/\sim$, and $\mathfrak{K}=(K, \leq, M_k, W_k)_{k\in K}$, with $\leq$ defined on $K$ as above.
We write $h=[x]$ for $x\in V_k$ if $x(i)/\sim =h(i)$ for all $i\in \alpha$; of course $X$ may not be unique, but this will not matter.
Let $\F_{\mathfrak K}$ be the set algebra based on the new Kripke system ${\mathfrak K}$ obtained by factoring out $\bold K$.

Set $\psi_1: \Sg^{\Rd\A}X_1\to \mathfrak{F}_{\mathfrak K}$ by
$\psi_1(p)=(f_k)$ such that if $k=((\Delta, \Gamma), (T,F))\in K$
is a matched pair of saturated theories in $\Sg^{\Rd\A_n}X_1$ and $\Sg^{\Rd\A_n}X_2$,
and $M_k=dim \A_n$, with $n$ unique to $k$, then for $x\in W_k$
$$f_k([x])=1\Longleftrightarrow {\sf s}_{x\cup (Id_{M_k\sim \alpha)}}^{\A_n}p\in \Delta\cup T,$$
with $x\in V_k$ and $[x]\in W_k$ is define as above.

To avoid cumbersome notation, we
write ${\sf s}_{x}^{\A_n}p$, or even simply ${\sf s}_xp,$ for
${\sf s}_{x\cup (Id_{M_k\sim \alpha)}}^{\A_n}p$.  No ambiguity should arise because the dimension $n$ will be clear from context.

We need to check that $\psi_1$ is well defined.
It suffices to show that if $\sigma, \tau\in V_k$ if $\sigma \sim \tau$ and $p\in \A_n$,
with $n$ unique to $k$,
then $${\sf s}_{\tau}p\in \Delta\cup T\text { iff } {\sf s}_{\sigma}p\in \Delta\cup T.$$

This can be proved by induction on the cardinality of
$J=\{i\in I_n: \sigma i\neq \tau i\}$, which is finite since we are only taking finite substitutions.
If $J$ is empty, the result is obvious.
Otherwise assume that $k\in J$. We recall the following piece of notation.
For $\eta\in V_k$ and $k,l<\alpha$, write
$\eta(k\mapsto l)$ for the $\eta'\in V$ that is the same as $\eta$ except
that $\eta'(k)=l.$
Now take any
$$\lambda\in \{\eta\in I_n: \sigma^{-1}\{\eta\}= \tau^{-1}\{\eta\}=\{\eta\}\}\smallsetminus \Delta x.$$
This $\lambda$ exists, because $\sigma$ and $\tau$ are finite transformations and $\A_n$ is a dilation with enough spare dimensions.
We have by cylindric axioms (a)
$${\sf s}_{\sigma}x={\sf s}_{\sigma k}^{\lambda}{\sf s}_{\sigma (k\mapsto \lambda)}p.$$
We also have (b)
$${\sf s}_{\tau k}^{\lambda}({\sf d}_{\lambda, \sigma k}\cdot {\sf s}_{\sigma} p)
={\sf d}_{\tau k, \sigma k} {\sf s}_{\sigma} p,$$
and (c)
$${\sf s}_{\tau k}^{\lambda}({\sf d}_{\lambda, \sigma k}\cdot {\sf s}_{\sigma(k\mapsto \lambda)}p)$$
$$= {\sf d}_{\tau k,  \sigma k}\cdot {\sf s}_{\sigma(k\mapsto \tau k)}p.$$
and (d)
$${\sf d}_{\lambda, \sigma k}\cdot {\sf s}_{\sigma k}^{\lambda}{\sf s}_{{\sigma}(k\mapsto \lambda)}p=
{\sf d}_{\lambda, \sigma k}\cdot {\sf s}_{{\sigma}(k\mapsto \lambda)}p$$

Then by (b), (a), (d) and (c), we get,
\begin{align*}
{\sf d}_{\tau k, \sigma k}\cdot {\sf s}_{\sigma} p&={\sf s}_{\tau k}^{\lambda}({\sf d}_{\lambda,\sigma k}\cdot {\sf s}_{\sigma}p)\\
&={\sf s}_{\tau k}^{\lambda}({\sf d}_{\lambda, \sigma k}\cdot {\sf s}_{\sigma k}^{\lambda}
{\sf s}_{{\sigma}(k\mapsto \lambda)}p)\\
&={\sf s}_{\tau k}^{\lambda}({\sf d}_{\lambda, \sigma k}\cdot {\sf s}_{{\sigma}(k\mapsto \lambda)}p)\\
&= {\sf d}_{\tau k,  \sigma k}\cdot {\sf s}_{\sigma(k\mapsto \tau k)}p.
\end{align*}

The conclusion follows from the induction hypothesis.
Now $\psi_1$ respects all quasi-polyadic equality operations, that is finite substitutions (with the proof as before;
recall that we only have finite substitutions since we are considering
$\Sg^{\Rd\A}X_1$)  except possibly for diagonal elements.
We check those:

Recall that for a concrete Kripke frame $\F_{\bold W}$ based on ${\bold W}=(W,\leq ,V_k, W_k),$ we have
the concrete diagonal element ${\sf d}_{ij}$ is given by the tuple $(g_k: k\in K)$ such that for $y\in V_k$, $g_k(y)=1$ iff $y(i)=y(j)$.

Now for the abstract diagonal element in $\A$, we have $\psi_1({\sf d}_{ij})=(f_k:k\in K)$, such that if $k=((\Delta, \Gamma), (T,F))$
is a matched pair of saturated theories in $\Sg^{\Rd\A_n}X_1$, $\Sg^{\Rd\A_n}X_2$, with $n$ unique to $i$,
we have $f_k([x])=1$ iff ${\sf s}_{x}{\sf d}_{ij}\in \Delta \cup T$ (this is well defined $\Delta\cup T\subseteq \A_n).$

But the latter is equivalent to ${\sf d}_{x(i), x(j)}\in \Delta\cup T$, which in turn is equivalent to $x(i)\sim x(j)$, that is
$[x](i)=[x](j),$ and so  $(f_k)\in {\sf d}_{ij}^{\F_{\mathfrak K}}$.
The reverse implication is the same.

We can safely assume that $X_1\cup X_2=X$ generates $\A$.
Let $\psi=\psi_1\cup \psi_2\upharpoonright X$. Then $\psi$ is a function since, by definition, $\psi_1$ and $\psi_2$
agree on $X_1\cap X_2$. Now by freeness $\psi$ extends to a homomorphism,
which we denote also by $\psi$ from $\A$ into $\F_{\mathfrak K}$.
And we are done, as usual, by $\psi$, $k_0$ and $Id\in V_{k_0}$.
\end{demo}

Now let ${\mathfrak L}_G$ be the corresponding logic (defined the usual way as in abstract algebraic logic).
Then we can now infer that:

\begin{theorem} The logic ${\mathfrak L}$ has the weak interpolation property.
\end{theorem}

\begin{demo}{Proof}  Assume that $\theta_1\in \Sg^{\Rd\A}X_1$ and $\theta_2\in \Sg^{\Rd\A}X_2$ such that $\theta_1\leq \theta_2$.
Let $\Delta_0=\{\theta\in \Sg^{\A}(X_1\cap X_2): \theta_1\leq \theta\}.$
If for some $\theta\in \Delta_0$ we have $\theta\leq \theta_2$, then we are done.
Else $(\Delta_0, \{\theta_2\})$ is consistent, hence $(\Delta_0\cap \Sg^{\Rd\A}X_2,\theta_2)$ is consistent.
 Extend this to a complete theory $(\Delta_2, \Gamma_2)$ in $\Sg^{\Rd\A}X_2$; this is possible since $\theta_2\in \Sg^{\Rd\A}X_2$.
Consider $(\Delta, \Gamma)=(\Delta_2\cap \Sg^{\A}(X_1\cap X_2), \Gamma_2\cap \Sg^{\A}(X_1\cap X_2))$.
It is complete in the `common language',  that is, in $\Sg^{\A}(X_1\cap X_2)$.
Then $(\Delta\cup \{\theta_1\}), \Gamma)$ is consistent in $\Sg^{\Rd\A}X_1$ and  $(\Delta_2, \Gamma_2)$
is consistent in $\Sg^{\Rd\A}X_2$, and
$\Gamma\subseteq \Gamma_2.$
Applying the previous theorem, we get $(\Delta_2\cup \{\theta_1\}, \Gamma_2)$ is satisfiable.
Let $\psi_1, \psi_2$ and $\psi$ and $k_0$ be as in the previous proof. Then
$\psi\upharpoonright \Sg^{\Rd\A}X_1=\psi_1$ and $\psi\upharpoonright \Sg^{\Rd\A}X_2=\psi_2$.
But $\theta_1\in \Sg^{\Rd\A}X_1$, then $\psi_1(\theta_1)=\psi(\theta_2)$.
Similarly, $\psi_2(\theta_2)=\psi(\theta_2).$
So, it readily follows that  $(\psi(\theta_1))_{k_0}(Id)=1$ and $(\psi(\theta_2))_{k_0}(Id)=0$.
This contradicts that $\psi(\theta_1)\leq \psi(\theta_2),$
and we are done.
\end{demo}
When $G$ is the semigroup of finite transformations then we get a usual interpolation theorem,
for here $\Rd\A$ is just $\A$ (we do not have infinitary substitutions,
as above). In particular, usual quantifier intuitionistic logic has the interpolation property.

\section{The use of category theory}

On the algebraic level, the dichotomy between cylindric world and polyadic world
manifests itself blatantly in the following theorem:

\begin{theorem}\label{Sc}
\begin{enumarab}
\item \label{two} For all $\alpha>2$ for all $r\in \omega, k\geq 1$,
there exists $\B^r\in S\Nr_{\alpha}\QEA_{\alpha+k}$ such that $\Rd_{sc}\B^r\notin S\Nr_{\alpha}\Sc_{\alpha+k+1}$,
but $\Pi \B^r/F\in S\Nr_{\alpha}\RQEA_{\alpha+k+1}$. If  $\K\in \{\CA, \QEA\}$, then $\B^r$ can be chosen so that
$\Rd_{ca}\B^r\notin S\Nr_{\alpha}\CA_{\alpha+k+1}$, and
$\Pi_r\B^r\in \RQEA_{\alpha}$. In particular, for any $\K$ of $\CA$ and $\QEA$, and for any finite $k$,
$\sf RK_{\alpha}$ is not axiomatizable by a finite schema over
${\sf S}\Nr_{\alpha}\K_{\alpha+k}$.

\item Let $\alpha$ be an infinite ordinal, then there exists an algebra $\A\in {\RQEA}_{\alpha}$ such that
$\Rd_{\Sc}\A\subseteq \Nr_{\alpha}\B_i$, $\B_i\in \Sc_{\alpha+\omega}$ for $i\in \{1,2\}$, $\A$ generates $\B_i$ using
the $\Sc$ operations, but
there is no $\Sc$ isomorphism from $\B_1$ to $\B_2$ that fixes $\A$ pointwise.
Furthermore, for $\K\in \{\Sc, \CA, \QEA, \QA\}$,
${\sf K}_{\alpha}$ does not have the amalgamation property

\item Let $\alpha$ be an infinite ordinal.  For $\K=\PA_{\alpha}$ or $\K=\PEA_{\alpha}$,
$\alpha$ an infinite ordinal and $\beta>\alpha$, $S\Nr_{\alpha}\K_{\beta}=\Nr_{\alpha}\K_{\beta}=\K_{\alpha}$
and any $\A\in \K_{\alpha}$ has the neat  amalgamation property. Furthermore,  $\PA_{\alpha}$ has the superamalgamation property
\end{enumarab}
\end{theorem}

\begin{proof}
\begin{enumarab}

\item For $\CA$ and $\QEA$ as before. For any diagonal
free reduct of $\PEA$ containing $\Sc$ follows from the proof of the main theorem in
\cite{t}, sketched after theorem \ref{thm:cmnr}, where algebras $\C(m,n, r)\in \Nr_n\PEA_m$
are constructed,  such that $\Rd_{sc}\C(m,n,r)\notin S\Nr_{n}\Sc_m$ but
$\Pi_r \C(m,,n,r)\in \Nr_n\QEA_m$. Then use the lifting argument in \ref{2.12}
\item  Here we will be even more sketchy. Details can be found in \cite{conference}.
Let $\A=\Fr_4\CA_{\alpha}$ with $\{x,y,z,w\}$ its free generators. Let $X_1=\{x,y\}$ and $X_2=\{x,z,w\}$.
Let $r, s$ and $t$ be defined as follows:
$$ r = {\sf c}_0(x\cdot {\sf c}_1y)\cdot {\sf c}_0(x\cdot -{\sf c}_1y),$$
$$ s = {\sf c}_0{\sf c}_1({\sf c}_1z\cdot {\sf s}^0_1{\sf c}_1z\cdot -{\sf d}_{01}) + {\sf c}_0(x\cdot -{\sf c}_1z),$$
$$ t = {\sf c}_0{\sf c}_1({\sf c}_1w\cdot {\sf s}^0_1{\sf c}_1w\cdot -{\sf d}_{01}) + {\sf c}_0(x\cdot -{\sf c}_1w),$$
where $ x, y, z, \text { and } w$ are the first four free generators
of $\A$.
Then $r\leq s\cdot t$.
Let $\D=\Fr_4\RCA_{\alpha}$ with free generators $\{x', y', z', w'\}$.
Let  $\psi: \A\to \D$ be defined by the extension of the map $t\mapsto t'$, for $t\in \{x,y,x,w\}$.
For $i\in \A$, we denote $\psi(i)\in \D$ by $i'$.
Let $I=\Ig^{\D^{(X_1)}}\{r'\}$ and $J=\Ig^{\D^{(X_2)}}\{s'.t'\}$, and let
$$L=I\cap \D^{(X_1\cap X_2)}\text { and }K =J\cap \D^{(X_1\cap X_2)}.$$
Then $L=K$, and $\A_0=\D^{(X_1\cap X_2)}/L$  can be embedded into
$\A_1=\D^{(X_1)}/I$ and $\A_2=\D^{(X_2)}/J$,
but there is no amalgam even in $\CA_{\omega}.$ In particular,
$\CA_{\omega}$ with $MGR$ does not have $AP$.
If $\A_0$ has the unique neat embedding property that it lies in the amalgamation
base of $\RCA_{\alpha}$ \cite{Sayedneat}.
Hence the algebra $\A_0$ does not have the unique neat embedding property.
More generally, if $\D_{\beta}$ is taken as the free
$\RCA_{\alpha}$ on $\beta$ generators, so that our algebra in the previous theorem is just $\D_4$,
where $\beta\geq 4$, then the algebra constructed from $\D_{\beta}$ as above,
will not have the unique neat embedding property.

\item From \cite{SL} and \cite{super}.
\end{enumarab}

\end{proof}

The theme of using category theory to approach problems in algebraic logic, and in particular, on neat embeddings
was initiated by the author in \cite{conference},
viewing neat embeddings as adjoint situation. But there each paradigm (cylindric-polaydic)
was approached separately. In particular (2) above was proved
equivalent to that the neat reduct functor has no right adjoint while (3) ie equivalent
to invertibility of the neat reduct functor.

Now that we have the notion of Hamos' general schema at hand,
we have the opportunity to give a categorial adjoint situation for both
cylindric and polyadic algebras simultaneously.

In our categorial notation
we follow \cite{conference}, which in turn follows \cite{cat}.
We start by well known categorial definitions and results.

\begin{definition} Let $L$ and $K$ be two categories.
Let $G:K\to L$ be a functor and let $\B\in Ob(L)$. A pair $(u_B, \A_B)$ with $\A_B\in Ob(K)$ and $u_B:\B\to G(\A_B)$ is called a universal map
with respect to $G$
(or a $G$ universal map) provided that for each $\A'\in Ob(K)$ and each $f:\B\to G(\A')$ there exists a unique $K$ morphism
$\bar{f}: \A_B\to \A'$ such that
$$G(\bar{f})\circ u_B=f.$$
\end{definition}

\begin{displaymath}
    \xymatrix{
        \mathfrak{B} \ar[r]^{u_B} \ar[dr]_f & G(\mathfrak{A}_\mathfrak{B}) \ar[d]^{G(f)}  &\mathfrak{A}_\mathfrak{B} \ar[d]^{\hat{f}} \\
             & G(\mathfrak{A}')  & \mathfrak{A}'}
\end{displaymath}

The above definition is strongly related to the existence of adjoints of functors.
Functors are often defined by universal properties examples are the tensor product, the direct sum, and direct product of groups or vector spaces,
construction of free groups and modules, direct and inverse. The concepts of limit and colimit  generalize several of the above.
Universal constructions often give rise to pairs of adjoint functors.
And indeed, the above definition is strongly related to the existence of adjoints of functors, as we proceed to show.
For undefined notions in the coming definition, the reader is referred to \cite{cat}

For undefined notions in the coming definition, the reader is referred to \cite[theorem 27.3, p. 196]{cat}.

\begin{theorem} Let $G:K\to L$.
\begin{enumarab}
\item If each $\B\in Ob(K)$ has a $G$ universal map $(\mu_B, \A_B)$, then there exists a unique adjoint situation $(\mu, \epsilon):F\to G$
such that $\mu=(\mu_B)$ and for each $\B\in Ob(L),$
$F(\B)=\A_B$.
\item Conversely, if we have an adjoint situation $(\mu,\epsilon):F\to G$ then for each $\B\in Ob(K)$ $(\mu_B, F(\B))$ have a $G$ universal map.
\end{enumarab}
\end{theorem}

The neat reduct operator can be viewed as functor from a
certain subcategory of $\sf L$ of $\CA_{\alpha+\omega}$
to $\RCA_{\alpha}$, where $\sf L=\{\A\in \CA_{\alpha+\omega}: \A=\Sg^{\A}\Nr_{\alpha}\A\},$
in the natural way. Each object is taken to its neat $\alpha$ reduct and
morphisms are restrictions of injective homomorphisms between algebras
in ${\sf L}$ to their images. Note here that morphisms are restricted
only to injective maps.
The following theorem is proved in \cite{conference}.

\begin{theorem} Let $\bold L=\{\A\in \RCA_{\alpha+\omega}: \A=\Sg^{\A}\Nr_{\alpha}\A\}$. Then the neat reduct functor $\Nr_{\alpha}$
from $\bold L$ to $\RCA_{\alpha}$ with morphisms restricted to injective homomorphisms
does not have a right adjoint. Conversely, the neat reduct functor from
$\{\A\in \PA_{\beta}: \Sg^{\A}\Nr_{\alpha}\A=\A\}$
to $\PA_{\alpha}$, $\beta>\alpha\geq \omega$, with the same restrictions,  is strongly invertible.
\end{theorem}
\begin{proof} \cite{conference}. Can also
be easily distilled from the proof of our next theorem \ref{cat}.
\end{proof}

Category theory has the supreme advantage of putting
many existing results scattered in the literature, in their proper perspectives highlighting interconnections, illuminating differences and similarities,
despite the increasing tendencies toward fragmentation and specialization, in mathematical logic  in general,
and in even more specialized fields like algebraic logic. Our next thorem emphasizes this viewpoint, and unifies
the categorial results proved for both cylindric and polyadic algebras of infinite
dimension \cite{conference}, in the context of generalized systems of varieties.

For the well-known definition of amalgamation and super-amalgamation properties, the reader is referred to \cite{MS},
\cite{conference}.
Call a class $\K$ of algebras is (very) nice if it has the (super) amalgamation property.
We write $AP$ for the amalgamation property and $SUPAP$ for super amalgamation property.

Our final theorem says:
\begin{theorem}\label{cat} Let $\K=(\K_{\alpha}: \alpha\in \mu)$ be a system of varieties, such that $\omega$ and $\omega+\omega\in \mu$.
Assume that $\bold M =\{\A\in \K_{\omega+\omega}: \A=\Sg^{\A}\Nr_{\omega}\A\}$
has $SUPAP$. Assume further that for any injective homomorphism $f:\Nr_{\alpha}\B\to \Nr_{\alpha}\B'$, there exists an injective
homomorphism $g:\B\to \B'$ such that $f\subseteq g$. Then the following two conditions are equivalent.
\begin{enumarab}
\item $\Kn_{\omega}$ is (very) nice.
\item $\Nr_{\omega}$ is (strongly) invertible
\end{enumarab}
\end{theorem}
\begin{demo}{Proof}

\begin{enumarab}
\item Assume that $\Kn_{\omega}$ has the amalgamation property. We first show that
$\Kn_{\omega}$ has the following {\it neat amalgamation property}:
If $i_1:\A\to \Nr_{\alpha}\B_1$, $i_2:\A\to \Nr_{\alpha}\B_1$
are such that $i_1(A)$ generates $\B_1$ and $i_2(A)$ generates
$\B_2$, then there is an isomorphism $f:\B_1\to \B_2$ such that $f\circ i_1=i_2$.

By assumption, there is an amalgam, that is there is  $\D\in \Kn_{\omega}$, $m_1:\Nr_{\alpha}\B_1\to \D$, $m_2:\Nr_{\alpha}\B_2\to \D$ such that
$m_1\circ i_1=m_2\circ i_2$. We can assume that $m_1:\Nr_{\alpha}\B\to \Nr_{\alpha}\D^+$ for some $\D^+\in \bold M$, and similarly for $m_2$.
By hypothesis, let $\bar{m_1}:\B_1\to \D^+$ and $\bar{m_2}:\B_2\to \D^+$ be isomorphisms extending $m_1$ and $m_2$.
Then since $i_1A$ generates $\B_1$ and $i_2A$ generates $\B_2$, then $\bar{m_1}\B_1=\bar{m_2}\B_2$. It follows that
$f=\bar{m}_2^{-1}\circ \bar{m_1}$ is as desired.
From this it easily follows that $\Nr$ has universal maps and we are done.

Since the neat amalgamation is equivalent to existence of universal maps; this can be proved without difficulty \cite{conference},
to  prove the converse, we assume the neat amalgamation property, and we set out
to prove that
$\Kn_{\omega}$ has the amalgamation property

Let $\A,\B\in \Kn_{\omega}$. Let $f:\C\to \A$ and $g:\C\to \B$ be injective homomorphisms.
Then there exist $\A^+, \B^+, \C^+\in \K_{\alpha+\omega}$, $e_A:\A\to \Nr_{\alpha}\A^+$
$e_B:\B\to  \Nr_{\alpha}\B^+$ and $e_C:\C\to \Nr_{\alpha}\C^+$.
We can assume that $\Sg^{\A^+}e_A(A)=\A^+$ and similarly for $\B^+$ and $\C^+$.
Let $f(C)^+=\Sg^{A^+}e_A(f(C))$ and $g(C)^+=\Sg^{B^+}e_B(g(C)).$
By the neat amalgamation property, there exist $\bar{f}:\C^+\to f(C)^+$ and $\bar{g}:\C^+\to g(C)^+$ such that
$(e_A\upharpoonright f(C))\circ f=\bar{f}\circ e_C$ and $(e_B\upharpoonright g(C))\circ g=\bar{g}\circ e_C$.
Now $\bold M$ as $SUPAP$, hence there is a $\D^+$ in $\bold M$ and $k:\A^+\to \D^+$ and $h:\B^+\to \D^+$ such that
$k\circ \bar{f}=h\circ \bar{g}$. Then $k\circ e_A:\A\to \Nr_{\alpha}\D^+$ and
$h\circ e_B:\B\to \Nr_{\alpha}\D^+$ are one to one and
$k\circ e_A \circ f=h\circ e_B\circ g$.

\item  Now for the second equivalence. Assume that $\Kn_{\omega}$ has $SUPAP$. Then, {\it a fortiori}, it has $AP$
hence, by the above argument, it has neat amalgamation property.
We first show that if $\A\subseteq \Nr_{\alpha}\B$ and $\A$ generates $\B$ then equality holds,
we call this property $NS$, short for neat reducts commuting with forming
subalgebras.

If not, then $\A\subseteq \Nr_{\alpha}\B$, $\B\in K,$ $A$ generates $\B$ and $\A\neq \Nr_{\alpha}\B$.
Then $\A$ embeds into $\Nr_{\alpha}\B$ via the inclusion map $i$ . Let $\C=\Nr_{\alpha}\B$.
By $SUPAP$, there exists $\D\in \Kn_{\omega}$ and $m_1$, $m_2$ monomorphisms
from $\C$ to $\D$ such that $m_1(\C)\cap m_2(\C)=m_1\circ i(\A)$. Let $y\in \C\sim A$.
Then $m_1(y)\neq m_2(y)$ for else $d=m_1(y)=m_2(y)$ will be in
$m_1(\C)\cap m_2(\C)$ but not in $m_1\circ i(\A)$. Assume that
$\D\subseteq \Nr_{\alpha}\D^+$ with $\D^+\in K$. By hypothesis, there exist injections
$\bar{m_1}:\B\to \D^+$ and $\bar{m_2}:\B\to \D^+$ extending $m_1$ and $m_2$.
But $A$ generates $\B$ and so $\bar{m_1}=\bar{m_2}$.
Thus $m_1y=m_2y$ which is a contradiction.

Now let  $\beta=\alpha+\omega$. Let $\bold M=\{\A\in \K_{\beta}: \A=\Sg^{\A}\Nr_{\omega}\A$\}.
Let $\Nr:\bold M\to \Kn_{\omega}$ be the neat reduct functor.
We show that  $\Nr$ is strongly invertible, namely there is a functor $G:\Kn_{\omega}\to \bold M$ and natural isomorphisms
$\mu:1_{\bold M}\to G\circ \Nr$ and $\epsilon: \Nr\circ G\to 1_{\Kn_{\omega}}$.
Let $L$ be a system of representatives for isomorphism on $Ob(\bold M)$.
For each $\B\in Ob(\Kn_{\omega})$ there is a unique $G(\B)$ in $\bold M$ such that $\Nr(G(\B))\cong \B$.
Then $G:Ob(\Kn_{\omega})\to Ob(\bold M)$ is well defined.
Choose one isomorphism $\epsilon_B: \Nr(G(\B))\to \B$. If $g:\B\to \B'$ is a $\Kn_{\omega}$  morphism, then the square

\begin{displaymath}
    \xymatrix{ \Nr(G(B)) \ar[r]^{\epsilon_B}\ar[d]_{\epsilon_B^{-1}\circ g\circ \epsilon_{B'}} & B \ar[d]^g \\
               \Nr(G(B'))\ar[r]_{\epsilon_{B'}} & B' }
\end{displaymath}
commutes. There is a unique morphism $f:G(\B)\to G(\B')$ such that $\Nr(f)=\epsilon_{\B}^{-1}\circ g\circ \epsilon$.
We let $G(g)=f$. Then it is easy to see that $G$ defines a functor. Also, by definition $\epsilon=(\epsilon_{\B})$
is a natural isomorphism from $\Nr\circ G$ to $1_{\Kn_{\omega}}$.
To find a natural isomorphism from $1_{\bold M}$ to $G\circ \Nr,$ observe that that for each $\A\in Ob(\bold M)$,
$\epsilon_{\Nr\A}:\Nr\circ G\circ \Nr(\A)\to \Nr(\A)$ is an isomorphism.
Then there is a unique $\mu_A:\A\to G\circ \Nr(\A)$ such that $\Nr(\mu_{\A})=\epsilon_{\Nr\A}^{-1}.$
Since $\epsilon^{-1}$ is natural for any $f:\A\to \A'$ the square
\bigskip
\bigskip
\begin{displaymath}
    \xymatrix{ \Nr(A) \ar[r]^{\epsilon_{\Nr(A)}^{-1}=\Nr(\mu_A)}\ar[d]_{\Nr(f)} & \Nr\circ G\circ \Nr(A) \ar[d]^{\Nr\circ G\circ \Nr(f)} \\
               \Nr(A')\ar[r]_{\epsilon_{\Nr\A}^{-1}=\Nr(\mu_{A'})} & \Nr\circ G\circ \Nr(A') }
\end{displaymath}

commutes, hence the square

\bigskip
\begin{displaymath}
    \xymatrix{ A \ar[r]^{\mu_A}\ar[d]_f & G\circ \Nr(A) \ar[d]^{G\circ \Nr(f)} \\
               A'\ar[r]_{\mu_{A'}} & G\circ \Nr(A') }
\end{displaymath}

commutes, too. Therefore $\mu=(\mu_A)$ is as required.

Conversely, assume that the functor $\Nr$ is invertible. Then we have the neat amalgamation property
and the $NS$.

The  former
follows from the fact that the functor has
a right adjoint, and so it has universal maps. To prove that it has $NS$ assume for contradiction
that there exists $\A$ generating subreduct of $\B$ and
$\A$ is not isomorphic to $\Nr_{\alpha}\B$. This means that $Nr$ is not invertible, because
had it been invertible, with inverse $Dl$, then  $Dl(\A)=Dl(\Nr_{\alpha}\B)$ and this cannot happen.

Now we prove that $\Kn_{\omega}$ has $SUPAP$.
We obtain (using the notation in the first part)
$\D\in \Nr_{\alpha}\K_{\alpha+\omega}$
and $m:\A\to \D$ $n:\B\to \D$
such that $m\circ f=n\circ g$.
Here $m=k\circ e_A$ and $n=h\circ e_B$.  Denote $k$ by $m^+$ and $h$ by $n^+$.
Suppose that $\C$ has $SNEP$. We further want to show that if $m(a) \leq n(b)$,
for $a\in A$ and $b\in B$, then there exists $t \in C$
such that $ a \leq f(t)$ and $g(t) \leq b$.
So let $a$ and $b$ be as indicated.
We have  $m^+ \circ e_A(a) \leq n^+ \circ e_B(b),$ so
$m^+ ( e_A(a)) \leq n^+ ( e_B(b)).$
Since $\bold M$ has $SUPAP$, there exist $ z \in C^+$ such that $e_A(a) \leq \bar{f}(z)$ and
$\bar{g}(z) \leq e_B(b)$.
Let $\Gamma = \Delta z \sim \alpha$ and $z' =
{\sf c}_{(\Gamma)}z$. So, we obtain that
$e_A({\sf c}_{(\Gamma)}a) \leq \bar{f}({\sf c}_{(\Gamma)}z)~~ \textrm{and} ~~ \bar{g}({\sf c}_{(\Gamma)}z) \leq
e_B({\sf c}_{(\Gamma)}b).$ It follows that $e_A(a) \leq \bar{f}(z')~~\textrm{and} ~~ \bar{g}(z') \leq e_B(b).$ Now by hypothesis
$$z' \in \Nr_\alpha \C^+ = \Sg^{\Nr_\alpha \C^+} (e_C(C)) = e_C(C).$$
So, there exists $t \in C$ with $ z' = e_C(t)$. Then we get
$e_A(a) \leq \bar{f}(e_C(t))$ and $\bar{g}(e_C(t)) \leq e_B(b).$ It follows that $e_A(a) \leq e_A \circ f(t)$ and
$e_B \circ g(t) \leq
e_B(b).$ Hence, $ a \leq f(t)$ and $g(t) \leq b.$

\end{enumarab}
\end{demo}

\section{Another adjoint situation for finite dimensions, cylindric like algebras that have the polyadic spirit}

The pairing technique due to Alfred Tarski, and substantially generalized by Istv\'an N\'emeti, consists of  defining
a pair of quasi-projections.
$p_0$ and $p_1$
so that in a model $\cal M$ say of  a certain sentence $\pi$, where $\pi$ is built out of these quasi-projections, $p_0$ and $p_1$
are functions and for any element $a,b\in {\cal M}$, there is a $c$
such that $p_0$ and $p_1$
map $c$ to $a$ and $b,$ respectively.
We can think of $c$ as representing the ordered pair $(a,b)$
and $p_0$ and $p_1$ are the functions that project the ordered pair onto
its first and second coordinates.

Such a technique, ever since introduced by Tarski, to formalize, and indeed successfully so, set theory, in the calculus of relations
manifested itself in several re-incarnations in the literature some of which are quite subtle and
sophisticated.
One is Simon's proof of the representability of quasi-relation algebras $\QRA$
(relation algebras with quasi projections) using a neat embedding theorem for cylindric algebras \cite{Andras}.
The proof consists of stimulating a neat embedding theorem
via the quasi-projections, in short it is actually a
{\it a completeness proof}.

The idea implemented  is that quasi-projections, on the one hand, generate extra dimensions, and on the other it has
control over such a stretching. The latter property does not come across very much in Simon's proof, but below we will give an exact rigorous
meaning to such property. This method can is used by Simon to
apply a Henkin completeness construction.

We shall use Simon's technique to further show that $\QRA$ has the superamalgamation property;
this is utterly unsurprising because Henkin constructions also prove interpolation theorems. This is the case, e.g.
for first order logics and several of its non-trivial extensions arising from the process of algebraizing first order logic,
by dropping the condition of local finiteness reflecting the fact
that formulas contain only finitely many (free) variables. A striking example in this connection is the algebras studied by Sain and Sayed Ahmed
\cite{Sain}, \cite{Sayed}. This last condition of local finiteness is  unwarranted from the algebraic point of view,
because it prevents an equational formalism of first order logic.

The view, of capturing extra dimensions, using also quasi-projections comes along also very much so,
in N\'emetis directed cylindric algebras (introduced as a $\CA$ counterpart of
$\QRA$). In those, S\'agi defined quasi-projections also to achieve a completeness theorem for higher order logics.
The technique used is similar to Maddux's proof of representation of ${\QRA}$s, which further emphasizes the correlation.These algebras were
studied  by many authors
including Andr\'eka, Givant, N\'emeti, Maddux, S\'agi, Simon, and others.
The reference \cite{Andras} is recommended for other references in the topic.
It also has reincarnations in Computer Science literature
under the name of Fork algebras.

We have already made the notion of extra dimensions explicit. Its dual notion (in an exact categorial sense presented above)
that  of compressing dimensions, or taking neat reducts.

The definition of neat reducts in the standard definition adopted by Henkin, Monk and Tarski in their monograph,
deals only with initial segments, but it proves useful to widen the definition a little allowing
arbitrary subsets of $\alpha$ not just initial segments. This is no more than a notational tactic.

However, it will enable us, using a deep result of Simon, to present an equivalence between algebras that are finite dimensional.
We infer from our definition
that such algebras, referred to in the literature as {\it directed cylindric algebras}, actually
belong to the polyadic paradigm, in this context the neat reduct functor establishes an equivalence between algebras in
every  dimension $>2$.
To achieve this equivalent we use a transient category, namely, that of quasi-projective relation algebras.

\begin{definition} Let ${\C}\in \CA_{\alpha}$ and $I\subseteq \alpha$, and let $\beta$ be the order type of $I$. Then
$$Nr_IC=\{x\in C: {\sf c}_ix=x \textrm{ for all } i\in \alpha\sim I\}.$$
$$\Nr_{I}{\C}=(Nr_IC, +, \cdot ,-, 0,1, {\sf c}_{\rho_i}, {\sf d}_{\rho_i,\rho_j})_{i,j<\beta},$$
where $\beta$ is the unique order preserving one-to-one map from $\beta$ onto $I$, and all the operations
are the restrictions of the corresponding operations on $C$. When $I=\{i_0,\ldots i_{k-1}\}$
we write $\Nr_{i_0,\ldots i_{k-1}}\C$. If $I$ is an initial segment of $\alpha$, $\beta$ say, we write $\Nr_{\beta}\C$.
\end{definition}
Similar to taking the $n$ neat reduct of a $\CA$, $\A$ in a higher dimension, is taking its $\Ra$ reduct, its relation algebra reduct.
This has universe consisting of the $2$ dimensional elements of $\A$, and composition and converse are defined using one spare dimension.
A slight generalization, modulo a reshufflig of the indices:

\begin{definition}\label{RA} For $n\geq 3$, the relation algebra reduct of $\C\in \CA_n$ is the algebra
$$\Ra\C=(Nr_{n-2, n-1}C, +, \cdot,  1, ;, \breve{}, 1').$$
where $1'={\sf d}_{n-2,n-1}$, $\breve{x}={\sf s}_{n-1}^0{\sf s}_{n-1}^{n-2}{\sf s}_0^{n-1}x$ and $x;y={\sf c}_0({\sf s}_0^{n-1}x. {\sf s}_0^{n-2}y)$.
Here ${\sf s}_i^j(x)={\sf c}_i(x\cdot {\sf d}_{ij})$ when $i\neq q$ and ${\sf s}_i^i(x)=x.$
\end{definition}
But what is not obvious at all is that an $\RA$ has a $\CA_n$ reduct for $n\geq 3$.
But Simon showed that certain relations algebras do; namely the $\QRA$s.

\begin{definition} A relation algebra $\B$ is a $\QRA$
if there are elements $p,q$ in $\B$ satisfying the following equations:
\begin{enumarab}
\item $\breve{p};p\leq \Id, \breve{q}; q\leq \Id;$
\item  $\breve{p};q=1.$
\end{enumarab}
\end{definition}
We say that $\B$ is a $\QRA$  with quasi-projections $p$ and $q$.
We give a deep application of quasi projections; concerning  G\"odel's first incompleteness theorem in finite variable fragments of
first order logic.
Let $\pi$ be the formula that says that the partial functions $p_0$ and $p_1$ form a pair of quasi-projections.
$$\forall x\forall y \forall z[(p_0(x,y)\land p_0(x, z)\to y=z\land
(p_1(x,y)\land p_1(x,z)\to y=z \land$$
$$\exists z(p_0(z,x)\land p_1(z,y))]$$

\begin{definition} $\pi$ is the pairing formula using partial function $p_0$ and $p_1$ as defined above,
$\pi'$ is saying that the witness of the quasi projections is unique (it uses $4$ variables), and $\pi_{\mathsf{RA}}$ is the quasi projective axiom
formulated in the language of relation algebras.
\end{definition}

A sentence $\lambda$ is {\it inseparable} iff there is no set of sentences which recursively separates the theorems
of $\lambda$ from the refutable sentences of $\lambda$.

\begin{theorem}There exists an inseparable formula $\lambda$ and quasi projections
$p_0, p_1$ such that $\lambda \land \pi'$ is a valid in Peano arithmetic (these can be formulated in the language of set theory, that is we need
only one binary relation),
and there is a translation function $\tr: \Fm_{\omega}\to \Fm_4$
such that for any formulas $\psi$ and $\phi$,
whenever $\psi\models \phi$, then $\tr(\psi)\vdash_4 \tr(\phi)$. Furthermore, there are no atoms below $\tr(\lambda)/\vdash_4$ in
$\Fm/_{{\vdash_4}}$.
\end{theorem}
\begin{demo}{Proof}  $\lambda$ is Robinson's arithmetic interpreted in set theory, so it is equivalent to $ZFC$ without the axiom of infinity,
 and $p_0$ and $p_1$ are the genuine projection function defined
also (by brute force) in the language of set theory. The translation function is implemented using three
functions $h$, $r$ and $f$; it is defined at $\phi$
by $h[rf(\phi).\pi_{\mathsf{RA}}]$. Here $f$ is
Tarski function $f:\Fm_{\omega}^0\to \Fm_3^0$, so that $\pi\models \phi \leftrightarrow f(\phi)$
and $f$ preserves meaning,
$r$ is a recursive function translating $3$ formulas to $\mathsf{RA}$ terms
recursively, also preserving meaning, and $h$ is the natural
homomorphism from $\Fr_1\mathsf{RA} \to \Fm_4$, taking the generator to $E$ the only binary relation.
\end{demo}

\begin{theorem} Let $\sf K$ be any class such that $\mathsf{RCA}_n\subseteq \sf K\subseteq \mathsf{CA}_n$.
Then for any finite $n>3$, $\mathsf{EqK}$ is undecidable.
\end{theorem}

\begin{proof} Let $\sf K$ be any of the above classes.
Let $\lambda$ be an inseparable formula such that $\M\models \lambda\land \pi'$, where $\M$ is a model of Peano
arithmetic.
Then $\M\models k(\lambda)$. Let $\tau$ be the translation function of $n$ variable formulas to $\mathsf{CA}_n$ terms.
Let $T=\{\phi \in \Fm_{\omega}^0: \sf K\models \tau(k(\lambda))\leq \tau(k(\phi))\}$.
Assume that $\lambda \models \phi$. Then $k(\lambda)\vdash_4 k(\phi).$
Hence $\phi\in T$.
Assume now that $\lambda\models \neg \phi$. Then $\tr(\lambda)\vdash _4 \neg \tr(\phi)$.
By $\M\models \tr(\lambda)$ and ${\sf Cs}_n^{\M}\in \sf K$, we have
it is not the case that $\sf K\models \tau(k(\lambda)=0$,
hence $\phi\notin T$.
We have seen that $T$ separates the theorems of $\lambda$ from the refutable
sentences of $\lambda$, hence $T$ is not recursive, but
$\tau$ and $k$ are recursive, hence ${\sf Eq K}$ is not recursive.
\end{proof}

To construct cylindric algebras of higher dimensions 'sitting' in a $\QRA$,
we need to define certain terms. Seemingly rather complicated, their intuitive meaning
is not so hard to grasp.

\begin{definition} Let $x\in\B\in \RA$, then $\dom(x)=1';(x;\breve{x})$ and $\rng(x)=1';(\breve{x}; x)$, $x^0=1'$, $x^{n+1}=x^n;x$. $x$
is a functional element if $x;\breve{x}\leq 1'$.
\end{definition}
Given a $\QRA$, which we denote by $\bold Q$, we have quasi-projections $p$ and $q$ as mentioned above.
Next we define certain terms in ${\bf Q}$, cf. \cite{Andras}:
\begin{align*}
\epsilon^{n}&=\dom q^{n-1}\\
\pi_i^n&=\epsilon^{n};q^i;p,  i<n-1, \pi_{n-1}^{(n)}\\
&=q^{n-1}\\
\xi^{(n)}&=\pi_i^{(n)}; \pi_i^{(n)}\\
t_i^{(n)}&=\Pi_{i\neq j<n}\xi_j^{(n)}, t^{(n)}\\
&=\Pi_{j<n}\xi_j^{(n)}\\
 {\sf c}_i^{(n)}x&=x;t_i^{(n)}\\
{\sf d}_{ij}^{(n)}&=1;(\pi_i^{(n)}.\pi_j^{(n)})\\
1^{(n)}&=1;\epsilon^{(n)}.
\end{align*}and let
$$\B_n=(B_n, +, \cdot, -, 0,1^{(n)}, {\sf c}_i^{(n)}, {\sf d}_{ij}^{(n)})_{i,j<n},$$
where $B_n=\{x\in B: x=1;x; t^{(n)}\}.$
The intuitive meaning of those terms is explained in {Andras}, right after their definition on p. 271.

\begin{theorem} Let $n>1$
\begin{enumerate}
\item Then ${\B}_n$ is closed under the operations.
\item ${\B}_n$ is a $\CA_n$.
\end{enumerate}
\end{theorem}
\begin{proof} (1) is  proved in \cite[lemma 3.4, p. 273-274]{Andras} where the terms are definable in a $\QRA$.
 That it is a $\CA_n$ can be proved as \cite[theorem 3.9]{Andras}.
\end{proof}

\begin{definition} Consider the following terms.
$${\sf suc} (x)=1; (\breve{p}; x; \breve{q})$$
and
$${\sf pred}(x)=\breve{p}; \rng x; q.$$
\end{definition}
It is proved in \cite{Andras} that $\B_n$ neatly embeds into $\B_{n+1}$ via ${\sf suc}$. The successor function thus codes
extra dimensions. The thing to observe here is that  we will see that $pred$; its inverse;
guarantees a condition of commutativity of two operations: forming neat reducts and forming subalgebras;
it does not make a difference which operation we implement first, as long as we implement both one after the other.
So the function ${\sf suc}$ {\it captures the extra dimensions added.}. From the point of view of {\it definability} it says
that terms definable in extra dimensions add nothing, they are already term definable.
And this indeed is a definability condition, that will eventually lead to strong interpolation property we want.
Worthy of note that the same phenomena can be discerned in \cite{AUamal}, but in a more explicit setting; it is closer to
the surface of the proof.
In such a context the successor
and predecessor functions, have the same role, they faithfully code extra dimensions, but they are
more explicit appearing in the signature of the algebras in question. They are not term definable, as is the case here, and in a while with
the case of directed cylindric algebras. In the last two cases, it is indeed more difficult to discern their existence from the surface of the formalism, but
they are there, very much so,  lurking in the background.

\begin{theorem}\label{neat} Let $n\geq 3$. Then ${\sf suc}: {\B}_n\to \{a\in {\B}_{n+1}: {\sf c}_0a=a\}$
is an isomorphism into a generalized neat reduct of ${\B}_{n+1}$.
Strengthening the condition of surjectivity,  for all $X\subseteq \B_n$, $n\geq 3$, we have (*)
$${\sf suc}(\Sg^{\B_n}X)\cong \Nr_{1,2,\ldots, n}\Sg^{\B_{n+1}}{\sf suc}(X).$$
\end{theorem}

\begin{proof} The operations are respected \cite[theorem 5.1]{Andras}.
The last condition follows  because of the presence of the
functional element ${\sf pred}$, since we have ${\sf suc}({\sf pred} x)=x$ and ${\sf pred}({\sf suc}x)=x$, when ${\sf c}_0x=x$
[lemmas 4.6-4.10] \cite{Andras}.

\end{proof}
\begin{theorem}
Let $n\geq 3$. Let ${\C}_n$ be the algebra obtained from ${\B}_n$ by reshuffling the indices as follows;
set ${\sf c}_0^{{\C}_n}={\sf c} _n^{{\B}_n}$ and ${\sf c}_n^{{\C}_n}={\sf c}_0^{{\cal B}_n}$. Then ${\C}_n$ is a cylindric algebra,
and ${\sf suc}: {\C}_n\to \Nr_n{\C}_{n+1}$ is an isomorphism for all $n$.
Furthermore, for all $X\subseteq \C_n$ we have
$${\sf suc}(\Sg^{\C_n}X)\cong \Nr_n\Sg^{\C_{n+1}}{\sf suc}(X).$$
\end{theorem}

\begin{proof} immediate from \ref{neat}
\end{proof}
\begin{theorem} Let ${\C}_n$ be as above. Then $succ^{m}:{\C_n}\to \Nr_n\C_m$ is an isomorphism, such that
for all $X\subseteq A$, we have
$${\sf suc}^{m}(\Sg^{\C_n}X)=\Nr_n\Sg^{\C_m}{\sf suc}^{n-1}(X).$$
\end{theorem}
\begin{proof} By induction on $n$.
\end{proof}
Now we want to neatly embed our $\QRA$ in $\omega$ extra dimensions. At the same time,
we do not want to lose our control over the stretching;
we still need the commutativing of taking, now $\Ra$  reducts with forming subalgebras; we call this property the $\Ra S$ property.
To construct the big $\omega$ dimensional algebra, we use a standard ultraproduct construction.
So here we go.
For $n\geq 3$, let  ${\C}_n^+$ be an algebra obtained by adding ${\sf c}_i$ and ${\sf d}_{ij}$'s for $\omega>i,j\geq n$ arbitrarity and with
$\Rd_n^+\C_{n^+}={\B}_n$. Let ${\C}=\Pi_{n\geq 3} {\C}_n^+/G$, where $G$ is a non-principal ultrafilter
on $\omega$.
In our next theorem, we show that the algebra $\A$ can be neatly embedded in a locally finite algebra $\omega$ dimensional algebra
and we retain our $\Ra S$ property.

\begin{theorem} Let $$i: {\A}\to \Ra\C$$
be defined by
$$x\mapsto (x,  {\sf suc}(x),\ldots {\sf suc}^{n-1}(x),\dots n\geq 3, x\in B_n)/G.$$
Then $i$ is an embedding ,
and for any $X\subseteq A$, we have
$$i(\Sg^{\A}X)=\Ra\Sg^{\C}i(X).$$
\end{theorem}
\begin{proof} The idea is that if this does not happen, then it will not happen in a finite reduct, and this impossible.
\end{proof}

Note that Simon's theorem, actually says that in every $\QRA$, there sits an $\RCA_n$ for evry $n\geq 3$.

\begin{theorem} Let $\bold Q\in {\QRA}$. Then for all $n\geq 4$, there exists a unique
$\A\in S\Nr_3\CA_n$ such that $\bold Q=\Ra\A$,
such that for all $X\subseteq A$, $\Sg^{\bf Q}X=\Ra\Sg^{\A}X.$
\end{theorem}
\begin{proof} This follows from the previous theorem together with $\Ra S$ property.
\end{proof}

\begin{corollary} Let $\bf Q\in \QRA$ be such that $\bf Q=\Ra\A\cong \Ra\B$, where $\A$ and $\B$ are locally finite, and are generated by $\bf Q$,
then this isomorphism lifts to an isomorphism from $\A$ to $\B$
that fixes $\bf Q$ pointwise.
\end{corollary}
The previous theorem says that $\Ra$ as a functor establishes an equivalence between ${\QRA}$
and a reflective subcategory of $\Lf_{\omega}.$
We say that $\A$ is the $\omega$ dilation of ${\bf Q}$.
Now we are ready for:

\begin{theorem}\label{SUPAP} $\QRA$ has $SUPAP$.
\end{theorem}
\begin{proof}  We have encountered the idea before.
We form the unique dilations of the given algebras required to be superamalgamated.
These are locally finite so we can find a superamalgam $\D$.
Then $\Ra\D$ will be required superamalgam; it contains quasiprojections because the base algebras
does.
In more detail, let $\A,\B\in \QRA$. Let $f:\C\to \A$ and $g:\C\to \B$ be injective homomorphisms .
Then there exist $\A^+, \B^+, \C^+\in \CA_{\alpha+\omega}$, $e_A:\A\to \Ra{\alpha}\A^+$
$e_B:\B\to  \Ra\B^+$ and $e_C:\C\to \Ra\C^+$.
We can assume, without loss,  that $\Sg^{\A^+}e_A(A)=\A^+$ and similarly for $\B^+$ and $\C^+$.
Let $f(C)^+=\Sg^{\A^+}e_A(f(C))$ and $g(C)^+=\Sg^{\B^+}e_B(g(C)).$
Since $\C$ has $UNEP$, there exist $\bar{f}:\C^+\to f(C)^+$ and $\bar{g}:\C^+\to g(C)^+$ such that
$(e_A\upharpoonright f(C))\circ f=\bar{f}\circ e_C$ and $(e_B\upharpoonright g(C))\circ g=\bar{g}\circ e_C$. Both $\bar{f}$ and $\bar{g}$ are
monomorphisms.
Now $\Lf_{\omega}$ has $SUPAP$, hence there is a $\D^+$ in $K$ and $k:\A^+\to \D^+$ and $h:\B^+\to \D^+$ such that
$k\circ \bar{f}=h\circ \bar{g}$. $k$ and $h$ are also monomorphisms. Then $k\circ e_A:\A\to \Ra\D^+$ and
$h\circ e_B:\B\to \Ra\D^+$ are one to one and
$k\circ e_A \circ f=h\circ e_B\circ g$.
Let $\D=\Ra\D^+$. Then we obtained $\D\in \QRA$
and $m:\A\to \D$ $n:\B\to \D$
such that $m\circ f=n\circ g$.
Here $m=k\circ e_A$ and $n=h\circ e_B$.
Denote $k$ by $m^+$ and $h$ by $n^+$.
Now suppose that $\C$ has $NS$. We further want to show that if $m(a) \leq n(b)$,
for $a\in A$ and $b\in B$, then there exists $t \in C$
such that $ a \leq f(t)$ and $g(t) \leq b$.
So let $a$ and $b$ be as indicated.
We have  $(m^+ \circ e_A)(a) \leq (n^+ \circ e_B)(b),$ so
$m^+ ( e_A(a)) \leq n^+ ( e_B(b)).$
Since ${\sf Lf}_{\omega}$ has $SUPAP$, there exist $z \in C^+$ such that $e_A(a) \leq \bar{f}(z)$ and
$\bar{g}(z) \leq e_B(b)$.
Let $\Gamma = \Delta z \sim \alpha$ and $z' =
{\sf c}_{(\Gamma)}z$. (Note that $\Gamma$ is finite.) So, we obtain that
$e_A({\sf c}_{(\Gamma)}a) \leq \bar{f}({\sf c}_{(\Gamma)}z)~~ \textrm{and} ~~ \bar{g}({\sf c}_{(\Gamma)}z) \leq
e_B({\sf c}_{(\Gamma)}b).$ It follows that $e_A(a) \leq \bar{f}(z')~~\textrm{and} ~~ \bar{g}(z') \leq e_B(b).$ Now by hypothesis
$$z' \in \Ra\C^+ = \Sg^{\Ra\C^+} (e_C(C)) = e_C(C).$$
So, there exists $t \in C$ with $ z' = e_C(t)$. Then we get
$e_A(a) \leq \bar{f}(e_C(t))$ and $\bar{g}(e_C(t)) \leq e_B(b).$ It follows that $e_A(a) \leq (e_A \circ f)(t)$ and
$(e_B \circ g)(t) \leq
e_B(b).$ Hence, $ a \leq f(t)$ and $g(t) \leq b.$
We are done.
\end{proof}

\subsection{ Pairing functions in N\'emetis directed $\CA$s}

We recall the definition of what is called weakly higher order cylindric algebras, or directed cylindric algebras invented by N\'emeti
and further studied by S\'agi and Simon \cite{Andras}.
Weakly higher order cylindric algebras are natural expansions of cylindric algebras.
They have extra operations that correspond to a certain kind of bounded existential
quantification along a binary relation $R$. The relation $R$ is best thought of as the `element of relation' in a model of some set theory.
It is an abstraction of the membership relation. These cylindric-like algebras
are the cylindric counterpart of quasi-projective relation algebras, introduced by Tarski.
We start by recalling the concrete versions of directed cylindric algebras from \cite{Sagi}:

\begin{definition}(P--structures and extensional structures.) \\
Let $U$ be a set and let $R$ be a binary relation on $U$. The structure
$\langle U; R \rangle$ is defined to be a P--structure\footnote{``P'' stands for ``pairing'' or ``pairable''.} iff for every
elements $a,b \in U$ there exists an element $c \in U$ such that $R(d,c)$ is
equivalent with $d=a$ or $d=b$ (where $d \in U$ is arbitrary) , that is, \\
\\
\centerline{ $\langle U; R \rangle \models (\forall x,y)(\exists z)(\forall w)( R(w,z) \Leftrightarrow (w=x$ or $w=y))$.} \\
\\
The structure $\langle U; R \rangle $ is defined to be a weak P--structure iff \\
\\
\centerline{ $ \langle U; R \rangle \models (\forall x,y)(\exists z)(R(x,z) $ and $ R(y,z))$.} \\
\\
The structure $\langle U; R \rangle$ is defined to be {extensional}
iff every two points $a,b \in U$ coincide whenever they have the same
``$R$--children'', that is, \\
\\
\centerline{ $\langle U; R \rangle \models (\forall x,y)(((\forall z) R(z,x) \Leftrightarrow R(z,y)) \Rightarrow x=y) $.}
\end{definition}

\noindent
We will see that if $\langle U; R \rangle$ is a P--structure then one can
``code'' pairs of elements of $U$ by a single element of $U$ and whenever
$\langle U; R \rangle$ is extensional then this coding is ``unique''. In fact,
in $\RCA_{3}^{\uparrow}$ (see the definition below) one can define terms similar
to quasi--projections and, as with the class of $\QRA$'s, one can equivalently
formalize many theories of first order logic as equational theories of certain
$\RCA_{3}^{\uparrow}$'s. Therefore $\RCA_{3}^{\uparrow}$ is in our main interest.
$\RCA_{\alpha}^{\uparrow}$ for bigger $\alpha$'s behave in the same way, an
explanation of this can be found in \cite{Sagi} and can be deduced from our proof, which shows that $\RCA_{3}^{\uparrow}$ has implicitly $\omega$
extra dimensions.
\begin{definition}
\label{canyildef}
(${\sf Cs}^{\uparrow}_{\alpha}$, $\RCA^{\uparrow}_{\alpha}$.) \\
Let $\alpha$ be an ordinal. Let $U$ be a set and let $R$ be a binary relation on $U$
such that $\langle U; R \rangle$ is a weak P--structure.
Then the
{full w--directed cylindric set algebra} of dimension $\alpha$ with base
structure $\langle U; R \rangle$ is the algebra: \\
\\
\centerline{$\langle {\cal P}({}^{\alpha}U); \cap, -, {\sf C}_{i}^{\uparrow(R)}, {\sf C}_{i}^{\downarrow(R)}, {\sf D}_{i,j}^{U} \rangle_{i,j \in \alpha}$,} \\
\\
where $\cap$ and $-$ are set theoretical intersection and complementation (w.r.t. ${}^{\alpha}U$),
respectively, ${\sf D}^{U}_{i,j} = \{ s \in {}^{\alpha}U: s_{i}=s_{j} \}$ and
$ {\sf C}_{i}^{\uparrow(R)}, {\sf C}_{i}^{\downarrow(R)}$ are defined as follows. For every
$X \in {\wp}({}^{\alpha}U)$: \\
\\
\indent $ {\sf C}_{i}^{\uparrow(R)}(X) = \{ s \in {}^{\alpha}U: (\exists z \in X)( R(z_{i},s_{i})$ and $(\forall j \in \alpha)(j \not=i \Rightarrow s_{j}=z_{j})) \},$ \\
\indent $ {\sf C}_{i}^{\downarrow(R)}(X) = \{ s \in {}^{\alpha}U: (\exists z \in X)( R(s_{i},z_{i})$ and $(\forall j \in \alpha)(j \not=i \Rightarrow s_{j}=z_{j})) \}.$ \\
\\
The class of {w--directed cylindric set algebras} of dimension $\alpha$
and the class of {directed cylindric set algebras} of dimension $\alpha$
are defined as follows. \\
\\
\centerline{$ w-{\sf Cs}^{\uparrow}_{\alpha} ={\bf S} \{ {\cal A}: \ {\cal A}$ is a full
w--directed cylindric set algebra of dimension $\alpha$} \\
\centerline{ \indent \indent \indent \indent with base structure $\langle U; R \rangle$,
for some weak P--structure $\langle U; R \rangle \}$.} \\
\\
\centerline{$ {\sf Cs}^{\uparrow}_{\alpha} ={\bf S} \{ {\cal A}: \ {\cal A}$ is a full
w--directed cylindric set algebra of dimension $\alpha$} \\
\centerline{ \indent \indent \indent \indent with base structure $\langle U; R \rangle$,
for some extensional P--structure $\langle U; R \rangle \}$.} \\
\\
The class $\RCA^{\uparrow}_{\alpha}$ of {representable directed cylindric algebras} of
dimension $\alpha$ is defined to be $\RCA^{\uparrow}_{\alpha} = {\bf SP}{\sf Cs}^{\uparrow}_{\alpha}$.
\end{definition}

The main result of S\'agi in \cite{Sagi} is a direct proof for the following: \\
\\
\begin{theorem}\label{rep}
{\em $\RCA^{\uparrow}_{\alpha}$ is a finitely axiomatizable
variety whenever $\alpha \geq 3$ and $\alpha$ is finite}
\end{theorem}

$\CA^{\uparrow}_3$ denotes the variety of directed cylindric algebras of dimension $3$
as defined in \cite[definition 3.9]{Sagi}. In \cite{Sagi}, it is proved that
$\CA^{\uparrow}_3=\RCA^{\uparrow}_3.$ A set of axioms is postulated  \cite[dfinition 8.6.8]{Sagi}.
Let $\A\in \CA^{\uparrow}_3$.
Then we have quasi-projections
$p,q$ defined on $\A$ as defined in \cite[p.878, 879]{Sagi}. We recall their definition, which is a little bit complicated because
they are defined as formulas in the corresponding second
order logic.
Let $\cal L$ denote the untyped logic corresponding to directed $\CA_3$'s as  defined \cite[ p.876-877]{Sagi}. It has only $3$ variables.
There is a correspondence between formulas (or formula schemes)  in this language and $\CA^{\uparrow}_3$ terms.
This is completely analogous to the correspondence between $\RCA_n$ terms and first order formulas containing only $n$ variables.
For example $v_i=v_j$ corresponds to ${\sf d}_{ij}$, $\exists^{\uparrow} v_i(v_i=v_j)$ correspond to ${\sf c}^{\uparrow}_i {\sf d}_{ij}$.
In \cite{Sagi} the following formulas (terms) are defined:

\begin{definition} Let $i,j,k \in 3$ distinct elements.
We define variable--free $\RCA^{\uparrow}_{3}$ terms as follows:
\begin{tabbing}
\indent \= $ v_{i} = \{ \{ v_{j} \}_{R} \}_{R}$ \ \ \= is \indent \= $\exists v_{k}( v_{k} = \{ v_{j} \}_{R} \wedge v_{i} = \{ v_{k} \}_{R})$, \kill
\> $ v_{i} \in_R v_{j}$ \indent \> is \> $\exists^{\uparrow}v_{j}(v_{i}=v_{j})$, \\
\> $ v_{i} = \{ v_{j} \}_{R}$ \> is \> $\forall v_{k}( v_{k} \in_R v_{j} \Leftrightarrow v_{k}=v_{j}) $, \\
\> $ \{ v_{i} \}_{R} \in_R v_{j}$ \> is \> $\exists v_{k}( v_{k} \in_R v_{j} \wedge v_{k} = \{ v_{i} \}_{R})$, \\
\> $ v_{i} = \{ \{ v_{j} \}_{R} \}_{R}$ \> is \> $\exists v_{k}( v_{k} = \{ v_{j} \}_{R} \wedge v_{i} = \{ v_{k} \}_{R})$ ,\\
\> $ v_{i} \in_R \cup v_{j}$ \> is \> $\exists v_{k}(v_{i} \in_R v_{k} \wedge v_{k} \in_R v_{j})$.
\end{tabbing}
\end{definition}

Therefore $pair_{i}$ (a pairing function) can be defined as follows: \\

\indent $\exists v_{j} \forall v_{k}( \{ v_{k} \}_{R} \in_R v_{i} \Leftrightarrow v_{j} = v_{k}) \ \wedge $ \\
\indent $\forall v_{j} \exists v_{k}( v_{j} \in_R v_{i} \Rightarrow v_{k} \in_R v_{j} ) \ \wedge $ \\
\indent $\forall v_{j} \forall v_{k}( v_{j} \in_R \cup v_{i} \ \wedge \ \{ v_{j} \} \not\in_{R} v_{i} \ \wedge \ v_{k} \in_R \cup v_{i} \ \wedge \ \{ v_{k} \} \not\in_{R} v_{i} \Rightarrow v_{j} = v_{k} )$. \\

It is clear that this is a term built up of diagonal elements and directed cylindrifications.
The first quasi-projection  $v_{i} = P(v_{j})$ can be chosen as: \\
\\
\indent $pair_{j} \ \wedge \ \forall^{\downarrow} v_{j} \exists^{\downarrow} v_{j} (v_{i} = v_{j})$. \\
\\
and the second quasiprojection  $v_{i} = Q(v_{j})$ can be chosen as: \\
\\
\centerline{$pair_{j} \ \wedge \ (( \forall v_{i} \forall v_{k} ( v_{i} \in_R v_{j} \ \wedge \ v_{k} \in_R v_{j}  \Rightarrow v_{i} = v_{k} )) \Rightarrow v_{i} = P(v_{j})) \ \wedge $} \\
\centerline{ $(\exists v_{i} \exists v_{k}( v_{i} \in_R v_{j} \ \wedge \ v_{k} \in_R v_{j} \ \wedge \  v_{i} \not= v_{k}) \Rightarrow (v_{i} \not= P(v_{j}) \ \wedge \ \exists^{\downarrow} v_{j} \exists^{\downarrow} v_{j}(v_{i} = v_{j})))$.}

\begin{theorem}
Let $\B$ be the relation algebra reduct of $\A$; then $\B$ is a relation algebra, and the
variable free terms corresponding to the formulas  $v_i=P(v_j)$ and $v_j=Q(v_j)$
call  them $p$ and $q$, respectively, are quasi-projections.
\end{theorem}
\begin{proof}
One proof is very tedious, though routine. One translates the functions as variable free terms in the language of $\CA_3$ and use
the definition of composition and converse in the $\RA$ reduct, to verify that they are quasi-projections.
Else one can look at their meanings on set algebras, which we recall from S\'agi \cite{Sagi}.
Given a cylindric set algebra $\cal A$ with base $U$ and accessibility relation $R$
$$(v_i=P(v_j))^A=\{s\in {}^3U: (\exists a,b\in U)(s_j=(a,b)_R, s_i=a\}$$
$$(v_i=Q(v_j)^A=\{s\in {}^3U: (\exists a,b\in U)(s_j=(a,b)_R, s_i=b\}.$$
First $P$ and $Q$ are functions, so they are functional elements.
Then it is clear that in this set algebras that $P$ and $Q$ are quasi-projections.
Since $\RCA^{\uparrow}_3$ is the variety generated by set algebras, they have the same meaning in the class $\CA^{\uparrow}_3.$
\end{proof}

Now we can turn the class around. Given a $\QRA$ one can define a directed $\CA_n$, for every finite $n\geq 2$.
This definition is given by N\'emeti and Simon in \cite{NS}.
It is very similar to Simon's definition above (defining $\CA$ reducts in a $\QRA$,
except that directed cylindrifiers along a relation $R$ are
implemented.)

\begin{theorem}\label{directed} The concrete category $\QRA$
with morphisms injective homomorphisms, and that of $\CA^{\uparrow}$  with morphisms also injective homomorphisms are equivalent.
in particular $\CA^{\uparrow}_3$ is categorically equivalent to $\CA^{\uparrow}_n$ for $n\geq 3$.
\end{theorem}
\begin{proof} Given $\A$ in $\QRA$ we can associate a directed $\CA^{\uparrow}_3$, injective
homomorphism are restrictions and vice versa; these are inverse functors.
However, when we pass from an $\QRA$ to a $\CA^{\uparrow}_3$ and then take the $\QRA$ reduct,
we get back exactly to the $\QRA$ we started off with,
but possibly with  new quasi projections which are  definable from the old ones.
Via this equivalence, we readily  conclude
that $\CA^{\uparrow}_3$  and $\CA^{\uparrow}_n$, with the transient category $\QRA$ establishing this categorial
equivalence.
\end{proof}

We have encountered dimension restricted free algebras of infinite dimension in the context of the interpolation property.
Now, in what follows we address finite dimensional ones from a completely different angle.
We show that such (finitely generated) dimension restricted free algebras are not atomic, which is yet another algebraic
reflection of G\'odel's first incompleteness
theorem, which we stumbled on earlier.

\subsection{Dimension restricted free algebras}
\begin{definition}
\cite[2.5.31]{HMT1}
Let $\delta$ be a cardinal.
Let $\alpha$ be an ordinal.
Let $_{\alpha} \Fr_{\delta}$ be the absolutely free algebra on $\delta$
generators and of type $\CA_{\alpha}.$
For an algebra
$\A,$ we write $R\in Con\ A$ if $R$ is a congruence relation on $\A.$
Let $\rho\in {}^{\delta}\wp(\alpha)$.
Let $L$ be a class having the same similarity type as
$\CA_{\alpha}.$ $SL$ denotes the class of all subalgebras of members of $L$.
Let
$$Cr_{\delta}^{(\rho)}L=\bigcap\{R: R\in Co_{\alpha}\Fr_{\delta},
{}_{\alpha}\Fr_{\delta}/R\in SL,
c_k^{_{\alpha}\Fr_{\delta}}{\eta}/R=\eta/R \text { for each }$$
$$\eta<\delta \text
{ and each }k\in \alpha\smallsetminus
\rho{\eta}\}$$
and
$$\Fr_{\delta}^{\rho}L={}_{\alpha}\Fr_{\beta}/Cr_{\delta}^{(\rho)}L.$$
\end{definition}

The ordinal $\alpha$ does not appear in $Cr_{\delta}^{(\rho)}L$ and $\Fr_{\delta}^{(\rho)}L$
though it is involved in their definition.
However, $\alpha$ will be clear from context so that no confusion is likely to ensue.
The algebra $\Fr_{\delta}^{(\rho)}L$ is referred to \cite{HMT1}
as a dimension restricted free algebra over $K$ with $\beta$
generators. Also $\Fr_{\delta}^{(\rho)}L$ is said to be
dimension restricted by the function $\rho$, or simply,
$\rho$-dimension-restricted.

\begin{definition} Let $\delta$ be a cardinal.
Assume that $L\subseteq
\CA_{\alpha}$, $x=\langle x_{\eta}: \eta<\delta\rangle\in {}^{\delta}A$
and $\rho\in {}^{\delta}\wp(\alpha)$. Then we say that the sequence
$x$ $L$-freely generates $\A$ under the
dimension restricting function $\rho,$ if the following two conditions are satisfied:
\begin{enumroman}
\item $\A$ is generated by $\rng x$,
and $\Delta x_{\eta}\subseteq \rho(\eta)$ for every $\eta<\delta.$

\item Whenever $\B\in L$, $y=\langle y_{\eta}: \eta<\delta\rangle\in   {}^{\delta}B$
and $\Delta y_{\eta}\subseteq \rho(\eta)$
for every $\eta <\delta,$ there is a homomorphism
$h$ from $\A$ to $\B$
such that $h\circ x=y.$
\end{enumroman}
\end{definition}
For an algebra $\A$ and $X\subseteq A,$ we write,
following \cite{HMT1}, $\Sg^{\A}X$, or even simply $\Sg X,$
for the subalgebra of $\A$
generated by $X$. We have the following which almost follow from the definitions.

\begin{theorem} Let $L\subseteq \CA_{\alpha}.$ Let $\delta$ be a cardinal.
Let $\rho\in {}^{\delta}\wp(\alpha)$.
Then the sequence $\langle \eta/Cr_{\delta}^{(\rho)}L: \eta<\delta\rangle$
$L$-freely generates $\Fr_{\delta}^{\rho}L$
under the dimension restricting function
$\rho$.
\end{theorem}
\begin{demo}{Proof}
\cite[2.5.37]
{HMT1} and \cite[2.6.45]{HMT1}.
\end{demo}
We shall need:
\begin{definition}\cite[2.6.28]{HMT1}
\begin{enumroman}
\item Let $1<\alpha<\beta$. Let $\A=(A, +, \cdot, 0, 1, {\sf c}_i, {\sf d}_{ij})_{i,j<\beta}$
be a $\CA_{\beta}$.
For $x\in A$, let $\Delta x=\{i\in \alpha: {\sf c}_ix\neq x\}.$
Then the neat $\alpha$ reduct
of $\A$, in symbols $\Nr_{\alpha}\A$, is the algebra with
universe $\{a\in A: \Delta x\subseteq \alpha\}.$
The Boolean operations of $\Nr_{\alpha}\A$ are inherited from $\A$ and the non Boolean
operations are those of $\A$ up to the index $\alpha.$
\item For $L\subseteq \CA_{\beta}$, $\Nr_{\alpha}L$ denote the class $\{\Nr_{\alpha}\A: \A\in L\}$.
\end{enumroman}
\end{definition}

The notation $_m\Fm_r^{\Lambda_n}$ in the coming theorem is
the syntactic Tarski-Linenbaum algebra where the number of variables available are $n$, but $m>n$ variables are
used to define the congruence relation on
the absolutely free algebras, which is restricted by the rank function $\rho$ of the language.
Then rather cumbersome notation is taken from \cite{HMT2}.
\begin{theorem} If $\Lambda_n=(n,  R, \rho)$ is any language
with $\dom R=\dom\rho=\beta$ and $n\leq m$,
then
$_{m}\Fm_r^{\Lambda_n}\cong \Fr_{\beta}^{\rho}S\Nr_n\CA_{m}.$
Here $S$ stands for the operation of forming subalgebras.
\end{theorem}

\begin{demo}{Proof}.
Let $n<\omega$. For brevity, let $\B={}_{m,p}\Fm_r^{\Lambda_n}$
Then $\B\in \CA_m$.
Let $x=\langle R_i/{}_{\cong_m}: i<\beta\rangle$.
Consider any $\A\in \Nr_n\CA_m$ and $y\in {}^{\beta}A$ such that $\Delta y_i\subseteq \rho(i)$ for $i<\beta$.
Let $\C\in \CA_m$ such that $\A=\Nr_n\C$. Clearly $y\in {}^nC$ and $\Delta y_i\subseteq n$ for all $i<\beta$.
Now $x$ $\CA_n$ freely generates $\B$ under the dimension restricting function $\rho$,
and so there is an $h\in Hom(\B, \C)$ such that $h\circ x=y$.
Therefore $h\in Hom(\Rd_{n}\B, \Rd_{n}\C)$. But $\Sg^{\Rd_{n}\B}[\rng x]\subseteq \Rd_{n}\B$
and $$h(\Sg^{\Rd_{n}\B}[\rng x])=\Sg^{\Rd_{n}\C}(h [\rng x])$$
and so
$$h\in Hom(\Sg^{\Rd_{n}\B}[\rng x], \Sg^{\Rd_n\C}h [\rng x]).$$
But $\rng x\subseteq \Nr_n\B$ and $h(\rng x)=[\rng y]\subseteq \Nr_n\C$. We readily obtain that
$h\in Hom(\Sg^{\Nr_{n}\B}[\rng x], \A)$. But
$$\Sg^{\Nr_{n}\B}[\rng x]\cong{} _{m,p}\Fm_r^{\Lambda_n},$$
hence the desired conclusion.
\end{demo}
In fact,  dimension restricted free algebras were designed specially to
represent (algebraically) such formula algebras of pure logic.
When $K\subseteq \CA_n$ and $\rho(i)=n$ for all $i<\beta$, then $\Fr_{\beta}^{\rho}K$ is isomorphic to
$\Fr_{\beta}K$. (In this case $\rho$ is not restricting anything).

The next theorem says that the dimension restricted free algebras of infinitely many pairwise distinct
$3$ dimensional  varieties, namely, $S\Nr_3\CA_m$ for $m>3$,
are not  atomic. The result for dimensions $>3$ also holds, with an easier proof.
\begin{theorem}
\begin{enumroman}
\item Let $\omega\geq m\geq 3$. Let $\beta$ be a non-zero cardinal $<\omega$
and $\rho:\beta\to \wp(3)$ such that $\rho(i)\geq 2$ for some $i\in \beta.$
Then $\Fr_{\beta}^{\rho}S\Nr_3\CA_m$ is not atomic.
In particular, $\Fr_{\beta}^{\rho}\CA_3$ and $\Fr_{\beta}^{\rho}\RCA_3$
are not atomic.

\item Let $m\geq n>3$. Let $\beta$ be a non-zero cardinal $<\omega$ and let $\rho:\beta\to \wp(n)$ such that
$\rho(i)\geq 2$ for some $i\in \beta$.
Then $\Fr_{\beta}^{\rho}S\Nr_n\CA_m$ is not atomic.
In particular, $\Fr_{\beta}\CA_4$ and $\Fr_{\beta}\RCA_4$ are not atomic.
\end{enumroman}
\end{theorem}

\begin{demo}{Proof} Follows from the above.
\end{demo}
The non-atomicity of $\Fr_{\beta}\RCA_3$ which follows from the last theorem,
a result of N\'emeti's, solves \cite[problem 4.14]{HMT1}.
The far more difficult syntactic version addressing $\CA_3$ and even its diagonal free reduct can be found
in \cite{1}. The very rough idea is
that quasi-projections can be stimulated in such weak fragments of first order logic, which means that they are not that weak
after all.

\subsection{G\"odel's second incompleteness theorem}

There has been some debate over the impact of G\"odel's incompleteness theorems on Hilbert's Program,
and whether it was the first or the second incompleteness theorem that delivered the coup de grace. Through a careful G\"odel
coding of sequences of symbols (formulas, proofs), G\"odel showed that in theories $T$ which contain a sufficient amount of arithmetic,
it is possible to produce a formula ${\sf pr}(x, y)$ which "says" that $x$ is (the code of) a proof of (the formula with code) $y$.
Specifically, if $0 = 1$ is the code of the formula $0 = 1$,
then ${\sf Con} (T) = \forall (x \neg pr(x,0 = 1))$ may be taken to "say" that $T$ is consistent (no number is the code of a derivation in $T$ of $0 = 1$).
The second incompleteness theorem $(G2)$ says
that under certain assumptions about $T$ and the coding apparatus, $T$ does not prove ${\sf Con}(T)$.

In fact, G\"odel's second incompleteness
theorem follows from the first by formalizing the meta mathematical proof of it into the formal system
whose consistency is at stake. So  such theories should
be strong enough to encode the proof of the first incompleteness theorem. Roughly the provability relation ${\sf pr}(x,y)$ ($x$ proves $y$)
not only proves, when it does, but it can also prove that it proves. Given a theory $T$ containing arithmetic, let
$Prb_T(\sigma)$ denotes $\exists x {\sf pr}(x, \sigma)$.

We have seen a we can use the  translation function defined by  interpreting Robinson's arithmetic in four variable
fragments of first order logic. But Tarski and Givant went further; they showed
that they can interpret the whole of $ZFC$ in the calculus of relations, hence in $L_4.$

Thus we know a we can interpret strong enough theories
in $\CA_4$ like, for instance, Peano arithmetic.
Now we work out a G\"odel's second incompleteness theorem for finite variable fragments of first order logic.

\begin{definition} A theory $T$ is strong enough if when
$T$ proves $\phi$ then $T$ proves that $T$ proves $\phi$
In more detail,
\begin{enumarab}
\item $T$ contains Robinson's arithmetic

\item for any sentence $\sigma$, $T\vdash \sigma$, then $T\vdash Prb_T(\sigma)$
\item for any sentence $\sigma$, $T\vdash (Prb_T(\sigma)\to Prb_TPrb_T(\sigma)$)
\item For any sentences $\rho$ and $\sigma$, $T\vdash Prb_T(\rho\to \sigma)\to (Prb_T\rho\to Prb_T \sigma).$
\end{enumarab}
\end{definition}
Strong theories are strong enough not to prove their consistency, if they are consistent.
Robinson's arithmetic is not strong enough but $PA$ and $ZF$ are.
So we need to capture at least $PA$ in $L_4$.
Though we can capture the whole of $ZF$, we will be content only with $PA$, which is sufficient for our process.
We can translate
Peano arithmetic in $L_4$ using the translation function denoted by $\sf tr$ above, obtaining $T^+$
using only $4$ variables. Unlike G\"odel's first incompleteness theorem,
we do not have a clear algebraic counterpart of G\"odel's second incompleteness theorem,
at least not an obvious one,
but, all the same,  there is an obvious metalogical counterpart.

\begin{theorem}\label{second}
\begin{enumarab}
\item  There is a formula ${\sf Con}(T^+)$
using only $4$ variables, such that in each model $\frak{M}\vDash T^+$
this formula expresses the consistency of $T^+$. Furthermore,
$$T^+\nvdash  {\sf Con(T^+)}$$
and
$$T^+\nvdash \neg {\sf Con(T^+)}.$$
\item There is a formula $\varphi $
using $4$ variables and an extension $T^{++}$ of $T^+$ in $L_4$ such that truth of statement (i)
below is independent of $ZFC$.

\begin{description}
\item  (i)
\[
T^{++}\vDash \varphi
\]
\end{description}
\end{enumarab}

\end{theorem}
\begin{proof}

\begin{enumarab}

\item In $T^+$ like $PA$, there is formula ${\sf pr}(x,y)$
expressing that $x$ is the G\"{o}del number of a proof from $T^+$
of a formula $\varphi $ of  whose G\"{o}del number is $y$.
Now, $\exists x{\sf pr}(x,y)$ is a provability
formula $\pi (y)$ which in $T^+$ expresses that $y$ is the G\"{o}%
del number of an $L_4$ formula provable in $T^+$. Furthermore, one can
easily check that the L\"{o}b conditions are satisfied by $\pi (y)$ and by $T^+$. Now, we
choose ${\sf Con}(T^+)$ to be $\urcorner \pi $($False$).
The rest follows the standard proof.
Also, the generalization for (consistent) extensions of $T^{+}$ with
finitely many new axioms can be proved like the classical case; if we have a $%
\sigma_{1}$ definition of the G\"{o}del numbers of the axioms of $T^+$
then we can extend this $\sigma_{1}$-definition to ``$T^+$ an
extra (concrete) axiom, say $\varphi $'', since $\varphi $ has a concrete G\"{o}del number $\lceil \varphi \rceil .$

\item Our next theorem says that truth in our theory is independent of $ZF$:
Choose $T^{++}$ such that
the theory of full first-order arithmetic can be
interpreted  in $T^{++}$.  In $\sf Th(\N)$ there exist a formula,
$\psi $, such that the statement "$\sf Th\N\vDash \psi $''
is independent of $ZFC$ (assuming $ZFC$ is consistent). Such a $\psi $ is the G%
\"{o}delian formula Con($ZF$),  then ``$\tr(\psi)\in  T^{++}$'' or equivalently ``
$T^{++}\vDash \tr(\psi )$'' is a statement about  $T^{++}$ whose
truth is independent from $ZFC.$

\end{enumarab}
\end{proof}

\subsection{Forcing in relation and cylindric algebras}

Tarski used the theory of relation algebras to express Zermelo-Fraenkel set theory as a system of equations without variables.
Representations of relation algebras will take us back to set-theoretic relational systems.

On the other hand, Cohen's method of forcing provides us a way to build new models of set theory and to establish the independence
of many set-theoretic statements. In \cite{z} a way of building the missing link to connect relation algebras and the method of forcing is presented.

The idea is that distinct set theories (for example one in which $CH$ is an axiom, and one in which $\neg CH$ is provable),
give rise to equationally  distinct simple countable quasi projective relation algebras.
By the duality established above, the same can be said about N\'emeti's directed cylindric
algebras of any finite dimensional $n>2$.

Maddux showed that every $\QRA$ is representable. The first proof is due to Tarski.
Unlike Simon's proof, appealing to a neat embedding theorem,
which can be easily discerned below the surface of our above discussion,
Maddux style representations are built using  a step-by-step argument.
Moreover, such representations has the supreme
advantage that it preserves {\it well-foundedness}, witness \cite{z},
which allows an algebraic reflection of independence results in set theory,
baring in mind that ${\sf QRA}s$ were originally designed to capture
the whole of set theory without the use of variables.
We formulate this fact in a theorem.
\begin{theorem}\label{forcing}
Let $\A$ be a simple countable $\QRA$
that is based on a model $(M,\in)$ of set theory.
Let $h$ be a Maddux style representation of $\A$. If
$d\in A$ is well founded relation on $M$, then $h(d)$ is well founded
\end{theorem}

So Maddux's and also Sagi's style representations, the latter proving a strong representability
result for directed cylindric algebras,  in fact preserves well foundedness of relations, which is not an elementary fact.
In \cite[theorem 14,  p. 61]{z}, a characterization of simple ${\QRA}$'s with a distinguished element
that are  isomorphic to an algebra of relations arising
from a countable transitive model of enough set theory is given.

So let $h$ be the Maddux style representation of such an $\A$,
on a set algebra with base $U$. Then $U$ is countable, and $h(e)$ "set like", meaning that it
behaves like the membership relation, at least as far as well-foundness, is concerned.

By Mostowski Collapsing theorem, there is a transitive $M$ and a one to one
map $g$ from $U$ onto $M$, such that $g$
is an isomorphism between $(U, h(e))$ and $(M,\in)$, where
$\in$ is the real membership.
$(M,\in)$ is also, a model of enough set theory. Let $M[G]$ is generic extension of $M$, formed by the methods of forcing,
and take the ${\sf QRA}$, call it $\A[G]$
corresponding to  $(M[G],\in).$  Assume for example that $\A$ models the translation of the continuum hypothesis,
while $M[G]$ models its negation.
Then we can conclude that $\A$ and $\A[G]$ are simple countable quasi projective
relation algebras that are equationally distinct.

One can carry similar investigations in the context of directed cylindric algebras instead of ${\QRA}$,
since representations of such algebras defined by
Sagi also preserves well foundness.

Summarizing, we have:

Let ${\sf M}$ denote the class of all countable transitive models of set theories, and $\sf QRA_{s}$ denotes
the class of countable simple $\sf QRA$s. Then there is a one to one correspondence
between ${\sf M}$ and $\sf QRA_{s}$, that is, in turn, definable in set theory.

For any finite $n>2$, there is also a one to one correspondence between ${\sf M}$ and
simple countable directed cylindric algebras of dimension $n$.

The logic corresponding to N\'emet's directed cylindric algebras is a higher order (untyped) logic,
so here to make the connection with set theory more
explicit one views set theory formulated in higher
order logic.

Now summarizing the above discussion, we have a triple duality
between countable transitive models of set theory, namely, $\sf M$,
$\sf QRA_s$s and simple directed cylindric algebras of any finite dimension
$>2$.

We can also descend one dimension, and work without equality. Then even in this case
quasi-projections {\it can be created} in $\Df_3$ yielding that the whole of $ZFC$ is formalizable in
the equality free  version of $L_3$
with only one ternary relation.
But this is a very long story briefly narrated in
Andr\'eka and N\'emeti's first chapter in \cite{1}.

\end{document}